\documentclass[12 pt]{article}
\usepackage{amsmath}
\usepackage{amssymb}
\usepackage[english]{babel}
\usepackage[margin=2.5cm]{geometry}
\usepackage{multicol, multirow}
\usepackage{mathrsfs}
\usepackage{color}
\usepackage{amsthm}
\usepackage{amsfonts}
\usepackage{bbm} %para la función indicadora de un conjunto
\usepackage{mdframed} %para enmarcar enunciados
\usepackage{lipsum}
\usepackage{tikz}
\usepackage{tikz-cd}
\usepackage[shortlabels]{enumitem} %Para poder personalizar las etiquetas en las listas enumeradas
\usetikzlibrary{matrix}
\usepackage[colorlinks=true, linkcolor=azulf, citecolor=esmeralda, urlcolor=black]{hyperref}
\usepackage{longtable}

\newtheorem{Lem}{Lemma}[section]
\newtheorem{Fact}[Lem]{Fact}
\newtheorem{Prop}[Lem]{Proposition}
\newtheorem{Theo}[Lem]{Theorem}
\newtheorem{Cor}[Lem]{Corollary}
\newtheorem{Rem}[Lem]{Remark}

\newtheorem*{Teo*}{Theorem}
\newtheorem{Ej*}{Example}

\newtheorem{Prob}[Lem]{Problem}

\newcommand{\AAA}{\mathbb{A}}
\newcommand{\BB}{\mathbb{B}}
\newcommand{\CC}{\mathbb{C}}

\newcommand{\QQ}{\mathbb{Q}}
\newcommand{\RR}{\mathbb{R}}
\newcommand{\SSS}{\mathbb{S}}
\newcommand{\ZZ}{\mathbb{Z}}
\newcommand{\ind}{\mathbbm{1}}

\newcommand{\Cc}{\mathcal{C}}
\newcommand{\Dc}{\mathcal{D}}

\newcommand{\Fc}{\mathcal{F}}
\newcommand{\Hc}{\mathcal{H}}

\newcommand{\Nc}{\mathcal{N}}
\newcommand{\Oc}{\mathcal{O}}

\newcommand{\Qc}{\mathcal{Q}}
\newcommand{\Rc}{\mathcal{R}}
\newcommand{\Uc}{\mathcal{U}}
\newcommand{\Vc}{\mathcal{V}}

\newcommand{\Zc}{\mathcal{Z}}

\newcommand{\Bcal}{\mathscr{B}}

\newcommand{\Jne}{\textbf{J}}

\newcommand{\Lne}{\textbf{L}}

\newcommand{\SOne}{\textbf{SO}}
\newcommand{\Spinne}{\textbf{Spin}}
\newcommand{\GLne}{\textbf{GL}}

\newcommand{\att}{\mathtt{a}}
\newcommand{\Att}{\mathtt{A}}
\newcommand{\btt}{\mathtt{b}}
\newcommand{\Btt}{\mathtt{B}}
\newcommand{\ctt}{\mathtt{c}}
\newcommand{\Ctt}{\mathtt{C}}
\newcommand{\dtt}{\mathtt{d}}
\newcommand{\ftt}{\mathtt{f}}

\newcommand{\ntt}{\mathtt{n}}

\newcommand{\dd}{\text{d}}%para el diferencial en una integral
\newcommand{\ad}{\text{ad}\,}%representación adjunta
\newcommand{\cov}{\text{cov }}

\newcommand{\hgt}{\mathscr{H}}

\newcommand{\hgot}{\mathfrak{h}}
\newcommand{\ggot}{\mathfrak{g}}
\newcommand{\wgot}{\mathfrak{w}}

\newcommand{\norm}[1]{||#1||}
\newcommand{\normp}[1]{||#1||_{\scriptscriptstyle p}} %norma p-adica
\newcommand{\normpo}[1]{||#1||_{ p_0}} %norma p_0-adica
\newcommand{\normi}[1]{||#1||_{\scriptscriptstyle \infty}} %norma sup
\newcommand{\normnu}[1]{||#1||_{\nu}} %norma sup nu
\newcommand{\normop}[1]{||#1||_{op}} %norma de operador
\newcommand{\normeuc}[1]{||#1||_{euc}} %norma euclidiana de R^d

\newcommand{\inv}{^{-1}}
\newcommand{\tra}{\,^t}
\newcommand{\diag}{\text{diag }}

\newcommand{\sieS}[2]{\mathscr{S}_{d,S}^{#1,#2}}
\newcommand{\sieR}[2]{\mathscr{S}_{d,\infty}^{#1,#2}}

\newcommand{\muY}{\mu_{\scriptscriptstyle Y}}

\newcommand{\Haar}[1]{\lambda_{#1}}

\newcommand{\scalar}[2]{\langle #1,#2\rangle } 

\newcommand{\ncc}{^\circ}
\newcommand{\aut}[1]{\widehat{#1}^{Aut}}

\newcommand{\consSmallStandMatForQF}[2]{\mathtt{A}_{#1,#2}}

\newcommand{\cteZSEquiv}[1]{A_{#1}}
\newcommand{\cteGenSetOQZS}[1]{B_{#1}}

\newcommand{\consDynStRiso}[1]{C_{#1}}
\newcommand{\consDynStRani}[1]{F_{#1}}
\newcommand{\consDecayHarishReal}{\mathcal{D}_1}
\newcommand{\consMixingReal}{\mathcal{D}}
\newcommand{\consSmoothBump}[1]{\mathcal{N}_{#1}}

\newcommand{\consVolClosedOrb}[1]{C^{(2)}_{#1}}
\newcommand{\consCRecurrence}[2]{C_{#1, #2}}
\newcommand{\consAlfaRecurrence}[1]{\vartheta_{#1}}
\newcommand{\consDefBigCompact}[1]{\mathcal{E}_{#1}}
\newcommand{\consCoefTRecurrence}[1]{C^{(4)}_{#1}}
\newcommand{\consEll}[1]{\ell_{#1}}
\newcommand{\consBigOrbits}[1]{A_{#1}}

\newcommand{\consZSEquivCritRiso}[1]{\mathcal{C}_{i, #1}}
\newcommand{\consZSEquivCritRani}[1]{\mathcal{C}_{a, #1}}

\newcommand{\consCoefVolHTransversalInf}[1]{V^{-}_{#1}}
\newcommand{\consCoefVolHTransversalSup}[1]{V^{+}_{#1}}
\newcommand{\consExpVolHTransversal}[1]{c_{#1}}
\newcommand{\consRecurrenceBox}[1]{B_{#1}}
\newcommand{\consVolXUno}[2]{\Vc_{#1,#2}}

\newcommand{\consTQSiso}[1]{\mathcal{G}_{#1}}
\newcommand{\consTQSani}[1]{\mathcal{H}_{#1}}
\newcommand{\consSmallGensRiso}[1]{\mathcal{K}_{#1}}
\newcommand{\consSmallGensRani}[1]{\mathcal{L}_{ #1}}

\newcommand{\consChangeReducedPDQF}[1]{W_{#1}}
\newcommand{\consExtremalVectorsBound}[1]{W_{1,#1}}
\newcommand{\consTransSiegel}[1]{W_{2,#1}}
\newcommand{\consReducedIntegralQF}[1]{W_{3,#1}}

\newcommand{\consBumpFuncRealOG}[1]{\mathcal{M}_{#1}}
\newcommand{\consInfVolROG}[1]{\mathtt{R}_{#1}}
\newcommand{\consSupVolROG}[1]{\mathtt{S}_{#1}}
\newcommand{\consBumpFuncRealOGbis}[1]{\mathcal{M}_{#1,1}}

\newcommand{\GL}[2]{G_{#1,#2}}
\newcommand{\GLS}[1]{G_{#1,S}}
\newcommand{\GLUnoS}[1]{G^1_{#1,S}}
\newcommand{\GLUnobis}[2]{G'_{#1,#2}}
\newcommand{\LowTMat}[2]{W_{#1, #2}}
\newcommand{\GammaS}[1]{\Gamma_{#1,S}}

\newcommand{\LatSpaceS}[1]{X_{#1,S}}
\newcommand{\LatSpaceUnoS}[1]{X^{1}_{#1,S}}
\newcommand{\OrbitS}[1]{Y_{#1,S}}
\newcommand{\OrbitUnoS}[1]{Y^{1}_{#1,S}}
\newcommand{\BasePointS}[1]{x_{#1,S}}
\newcommand{\BasePointUnoS}[1]{x^{1}_{#1,S}}

\usepackage{color}
\definecolor{naranja}{cmyk}{0,.42,1,0}  
\definecolor{durazno}{cmyk}{0,.46,.50,0}  
\definecolor{fresa}{cmyk}{0,1,.50,0} 
\definecolor{ladrillo}{cmyk}{0,.77,.87,0}  
\definecolor{violeta}{cmyk}{.07,.90,0,.34}  
\definecolor{purpura}{cmyk}{.45,.86,0,0}  
\definecolor{aguamarina}{cmyk}{.85,0,.33,0}    
\definecolor{esmeralda}{cmyk}{.91,0,.88,.12}  
\definecolor{pino}{cmyk}{.92,0,.59,.25}  
\definecolor{oliva}{cmyk}{.64,0,.95,.40}  
\definecolor{canela}{cmyk}{.14,.42,.56,0}  
\definecolor{cafe}{cmyk}{0,.81,1,.60}    
\definecolor{marron}{cmyk}{0,.72,1,.45}  
\definecolor{gris-claro}{cmyk}{0,0,0,.30}  
\definecolor{gris-oscuro}{cmyk}{0,0,0,.50}  
\definecolor{dorado}{cmyk}{0,.10,.84,0}  
\definecolor{melon}{cmyk}{0,.29,.84,0} 
\definecolor{turquesa}{rgb}{.1,.7,.4}
\definecolor{morado}{rgb}{.6,0,.9}  
\definecolor{azulf}{rgb}{0.19, 0.55, 0.91}
\definecolor{burgundy}{rgb}{0.5, 0.0, 0.13}
\definecolor{bostonred}{rgb}{0.8, 0.0, 0.0}
\definecolor{bleudefrance}{rgb}{0.19, 0.55, 0.91}
\definecolor{cadmiumred}{rgb}{0.89, 0.0, 0.13}
\definecolor{cadmiumyellow}{rgb}{1.0, 0.96, 0.0}
\definecolor{chartreuse(web)}{rgb}{0.5, 1.0, 0.0}
\definecolor{applegreen}{rgb}{0.55, 0.71, 0.0}
\definecolor{skobeloff}{rgb}{0.0, 0.48, 0.45}
\definecolor{sinopia}{rgb}{0.8, 0.25, 0.04}
\definecolor{uclagold}{rgb}{1.0, 0.7, 0.0}
\definecolor{britishracinggreen}{rgb}{0.0, 0.26, 0.15}
\definecolor{burntorange}{rgb}{0.8, 0.33, 0.0}

\usepackage{pstricks}
\usepackage{pgf,tikz}
\usetikzlibrary{arrows}

\title{$S$-integral quadratic forms and homogeneous dynamics}
\author{Irving Calderón}

\begin{document}
\maketitle

\selectlanguage{english}
\begin{abstract}
Let $S = \{ \infty \} \cup S_f$ be a finite set  of places of $\QQ$. Using homogeneous dynamics, we establish two new quantitative and explicit results about integral quadratic forms in three or more variables: The first is a criterion of $S$-integral equivalence. The second determines a finite generating set of any $S$-integral orthogonal group. Both theorems—which extend results of H. Li and G. Margulis for $S = \{ \infty\}$—are given by polynomial bounds on the size of the coefficients of the quadratic forms.
\end{abstract}

\tableofcontents

\section{Introduction}\label{sec_intro}

	\subsection{A criterion of $\ZZ_S$-equivalence for integral quadratic forms}

The first result of this work is motivated by the classical problem of classifying integral quadratic forms. Let $Q_1$ and $Q_2$ be integral quadratic forms in $d$ variables and let $\mathcal{R}$ be a commutative ring with $1$ containing $\ZZ$. We say that $Q_1$ and $Q_2$ are \textit{$\mathcal{R}$-equivalent} if $Q_1 \circ g = Q_2$ for some $g \in GL(d,\mathcal{R})$. The ultimate goal is to produce a list containing precisely one representative of each $\ZZ$-equivalence class. This has been achieved only for quadratic forms of small determinant. See for example the tables in Chapter X of the classical \cite{dickson_studies_1931} of L.E. Dickson, or in Chapter 15 of the more recent \cite{conway_sphere_1999} of J.H. Conway and N.A. Sloane. The last reference has a very complete discussion of the main developments around the $\ZZ$-classification of integral quadratic forms. Here we merely highlight some aspects relevant to this work.

The \textit{Reduction Theory of Quadratic Forms}, which originates from the pioneering works of C.F. Gauss, C. Hermite and H. Minkowski, is an important first step towards the $\ZZ$-classification. Reduced quadratic forms are defined by simple inequalities in its coefficients. The two important facts are: Any integral quadratic form is $\ZZ$-equivalent to a reduced one, and the norm of a reduced, integral quadratic form $R$ in $d$ variables is $\ll (\det R) ^{2d}$. We can thus write down all the reduced quadratic forms of given determinant. To achieve the desired classification, it only remains to remove the repetitions in this list. It is thus important to have a practical solution to the \textit{$\ZZ$-equivalence problem}: decide if any two given integral quadratic forms $Q_1$ and $Q_2$ in $d$ variables are $\ZZ$-equivalent. 

In Disquisitiones Arithmeticae \cite{gauss_disquisitiones_1965}, Gauss gives an efficient algorithm that settles the problem for $d = 2$. But the $\ZZ$-equivalence problem  for $d \geq  3$ is much harder, even when $Q_1$ and $Q_2$ have small coefficients. For instance, rather than risking having repetitions in the table of quadratic forms in  \cite{dickson_studies_1931} mentioned earlier, L.E. Dickson and A.E. Ross left in blank the line corresponding to $\det Q = 68$ because they couldn't tell if
 \[Q_1(x) = x_1^2 - 3 x_2^2 - 2x_2 x_3 -23 x_3^2 \quad \text{and} \quad Q_2(x) = x_1^2 - 7 x_2^2 - 6x_2 x_3 - 11 x_3^2\]
 are $\ZZ$-equivalent\footnote{It turns out they are. See \cite[p. 403]{conway_sphere_1999}}.

Here we are interested in an approach to the $\ZZ$-equivalence problem based on the following remarkable result of Siegel---see  \cite[Satz 11]{siegel_zur_1972}. 

\begin{Theo}\label{Existence_of_Z-search_bounds}
For any $d \geq 2$ there is an explicit function $M_d$ with the next property: any integral quadratic forms $Q_1$ and $Q_2$ in $d$ variables are $\ZZ$-equivalent if and only if there is $\gamma_0 \in GL(d,\ZZ)$ with $\normi{\gamma_0} \leq M_d(Q_1, Q_2)$, such that $Q_1 \circ \gamma_0 = Q_2$.
\end{Theo}
 An $M_d$ as in Theorem \ref{Existence_of_Z-search_bounds}  will be called \textit{search bound}. We can thus tell if $Q_1$ and $Q_2$ are $\ZZ$-equivalent by checking if the equation $Q_1 \circ \gamma = Q_2$ has a solution $\gamma$ in the finite subset of $GL(d,\ZZ)$ determined by $M_d (Q_1, Q_2)$.   One of the first explicit search bounds is due to S. Straumann, who following Siegel's proof of Theorem \ref{Existence_of_Z-search_bounds} arrives in \cite{straumann_s_aquivalenzproblem_1999} to
\begin{equation}\label{Strau_search_bound} 
M_d(Q_1, Q_2) = \exp \left( A_d |\det Q_1 |^\frac{d^3 + d^2}{2} \right)  \max \{\normi{Q_1}, \normi{Q_2} \}^\frac{d^3 - d^2}{2},  
\end{equation}      
where $A_d$ is some constant that depends only on $d$, and $\normi{Q_i}$ is the maximum of the absolute values of the coefficients of $Q_i$. The challenge now is to find search bounds that grow as slowly as possible. The expected optimal $M_d$ is different for $d = 2$ and $d \geq 3$. For $d = 2$, one can't do much better than the Siegel–Straumann search bound \eqref{Strau_search_bound}, which is exponential in $\normi{Q_1}, \normi{Q_2}$. The reason is the existence of a family Pell-type equations $a_n x^2 - b_n y^2 = \pm 1$ for which the size of the smallest solution grows faster than any polynomial in $a_n, b_n$, as $n \to \infty$. See \cite[p. 486]{lagarias_computational_1980}. As for $d \geq 3$, D. Masser predicts in \cite[p. 252]{masser_search_2002} that there is an $M_d$ polynomial in $\normi{Q_1}, \normi{Q_2}$. Masser's Conjecture for $d = 3$ was verified by R. Dietmann in \cite{dietmann_small_2003} using tools from analytic number theory. In a later work \cite{dietmann_polynomial_2007}, he obtains a partial answer for any $d \geq 3$. Namely, under certain hypotheses\footnote{That allow him to deduce the result from his search bound for ternary quadratic forms.} on $Q_1$ and $Q_2$, one can roughly take 
\[ M_d(Q_1, Q_2) \approx A_d \max \{ \normi{Q_1}, \normi{Q_2}  \}^{5^d d^{d + 1}} .\]
The proof of  Masser's Conjecture in full generality is due to H. Li and G. Margulis. Their method is based on homogeneous dynamics and automorphic representations. Here is a simplified version of their \cite[Theorem 1]{li_effective_2016}.

	\begin{Theo}\label{Z-equivalence_criterion}
For any integer $d \geq 3$ there is a constant $A_d$ with the next property: for any non-degenerate, integral quadratic forms $Q_1$ and $Q_2$ in $d$ variables, $Q_1$ is $\ZZ$-equivalent to $Q_2$ if and only if there is $\gamma_0 \in GL(d,\ZZ)$ with
\[ \normi{\gamma_0} \leq  A_d (\normi{Q_1} \normi{Q_2})^{\frac{13}{40} d^3}, \]
such that  $Q_1 \circ \gamma_0 = Q_2$.  
\end{Theo}

The first main result of this work is a criterion of $\ZZ[1/n]$-equivalence of integral quadratic forms analogous to Theorem \ref{Z-equivalence_criterion}, for any $n \geq 1$. We need a couple of definitions for the statement. Consider a finite set $S_f$ of prime numbers and let $S = \{\infty \} \cup S_f$.  We define $p_S$ as 1 if $S_f = \emptyset$, and otherwise as the product of the primes in $S_f$. Then $\ZZ_S = \ZZ[1/p_S]$ is the \textit{ring of $S$-integers}.  We now want a method to decide if any two given $Q_1$ and $Q_2$ are $\ZZ_S$-equivalent. Theorem \ref{Z-equivalence_criterion} gives an answer for $S = \{\infty\}$. For a general $S$, we will bound both the archimedean absolute value $|\cdot|_\infty$, and the denominator of the coefficients of some equivalence matrix $\gamma_0 \in GL(d,\ZZ_S)$ between $Q_1$ and $Q_2$. We define the $S$-norm of a rational matrix $g = (g_{ij})$ as 
\[ \norm{g}_S = \max_{\nu \in S} \normnu{g} ,\]
where $\normnu{g} = \max_{i,j} |g_{ij}|_\nu$. Here is our first result.

\begin{Theo}\label{Z_S-equiv_criterion_intro}
For any integer $d \geq 3$ there is a constant $\cteZSEquiv{d}$ with the next property: For any   non-degenerate, integral quadratic forms $Q_1$ and $Q_2$ in $d$ variables, and for any finite set $S = \{\infty\} \cup S_f$ of places of $\QQ$, $Q_1$ is $\ZZ_S$-equivalent to $Q_2$ if and only if there is $\gamma_0 \in GL(d,\ZZ_S)$ with
\[ \norm{\gamma_0}_S \leq \cteZSEquiv{d}  p_S^{19 d^6} (\normi{Q_1} \normi{Q_2})^{2d^3}, \]
such that $Q_1 \circ \gamma_0 = Q_2$. 
\end{Theo}

The statement above follows from the more precise Theorem \ref{Z_S-equivalence_R-isotropic} and Theorem \ref{Z_S-equivalence_R-anisotropic}.

We close with a word on the proofs. Li and Margulis reformulate the $\ZZ$-equivalence problem in terms of the action of $O(Q_1, \RR)$ on the space $X_{d, \infty}$ of lattices of $\RR^d$. This lets them deduce Theorem \ref{Z-equivalence_criterion} from quantitative results on the distribution of closed $O(Q_1, \RR)$-orbits in $X_{d,\infty}$---which in turn follow from profound theorems on automorphic representations---, and from an upper bound for the volume of the orbit $O(Q_1, \RR) \ZZ^d$ in terms of $|\det Q_1|_\infty$. For a general $S$, we follow closely their strategy. Let $\QQ_S = \prod_{\nu \in S} \QQ_\nu$. We prove Theorem \ref{Z_S-equiv_criterion_intro} using a quantitative distribution of closed $O(Q_1, \QQ_S)$-orbits in the space $\LatSpaceS{d}$ of lattices of $\QQ_S^d$, and a bound of the volume of $O(Q_1, \QQ_S) \ZZ_S^d$ in terms of
\[ \hgt_S(\det Q_1) = \prod_{\nu \in S} |\det Q_1 |_\nu. \]

	\subsection{Small generators of $S$-integral orthogonal groups}
	
Let $Q$ be a non-degenerate integral quadratic form $d$ variables. It's a classical result of Siegel that $O(Q,\ZZ)$ is finitely generated---see \cite[Satz 11]{siegel_einheiten_1939}. Li and Margulis obtain in  \cite[Theorem 2]{li_effective_2016} the following effective refinement of Siegel's theorem.

\begin{Theo}\label{Generators_O(Q,Z)}
For any integer $d \geq 3$ there is a constant $B_d$ with the next property: for any non-degenerate, integral quadratic form $Q$ in $d$ variables, the group $O(Q,\ZZ)$ is generated by the $\xi \in O(Q, \ZZ)$ with
\[ \normi{\xi} \leq B_d \normi{Q}^{3d^4} |\det Q|_\infty^{d^6} . \]
\end{Theo}

The second main result of this work is an extension of Theorem \ref{Generators_O(Q,Z)} to $S$-integral orthogonal groups.

\begin{Theo}\label{Generators_O(Q,Z_S)_intro}
For any integer $d \geq 3$ there is a constant $\cteGenSetOQZS{d}$ with the next property: for any non-degenerate, integral quadratic form $Q$ in $d$ variables, and for any finite set $S = \{\infty\} \cup S_f$ of places of $\QQ$, the group $O(Q,\ZZ_S)$ is generated by the $\xi 	\in O(Q, \ZZ_S)$ with 
\begin{equation}\label{Height_bound_gen_O(Q,Z_S)} 
 \norm{\xi}_S \leq \cteGenSetOQZS{d} p_S^{20 d^7} \normi{Q}^{5d^6} . 
\end{equation}
\end{Theo}

This statement is a simplified combination of Theorem \ref{Small_generators_R-isotropic} and Theorem \ref{Small_generators_R-anisotropic}. Our proof of Theorem \ref{Generators_O(Q,Z_S)_intro} is inspired by the strategy of Li and Margulis for Theorem \ref{Generators_O(Q,Z)}. The key ingredients are an effective reduction theory of quadratic forms, and Theorem \ref{Z_S-equiv_criterion_intro}. Here is an outline of argument: We consider $O(Q, \ZZ_S)$ embedded diagonally in $H_S := O(Q,\QQ_S)$. As we show in Subsection \ref{subsec_two_basic_lemmas}, 
\[ \mathscr{G}_{U,M} : = (U\inv M U) \cap O(Q, \ZZ_S)\] 
generates $O(Q, \ZZ_S)$  for any fundamental set $U$ of $O(Q,\ZZ_S)$ in $H_S$, and any generating set $M$ of $H_S$.  Reduction theory tells us that $Q$ is $\ZZ_S$-equivalent to finitely many reduced, integral quadratic forms $R_1, \ldots, R_k$, and that an $U$ can be obtained from any choice of  equivalence matrices $\gamma_i \in GL(d, \ZZ_S)$  between the $R_i$'s and $Q$. We show that any $\xi \in \mathscr{G}_{U,M}$ verifies \eqref{Height_bound_gen_O(Q,Z_S)} when the $\gamma_i$'s are taken as in Theorem \ref{Z_S-equiv_criterion_intro}.

Theorem \ref{Generators_O(Q,Z_S)_intro} is a contribution to the general problem—raised by T. Chinburg and M. Stover in  \cite{chinburg_small_2014}—of bounding the height of generators of $S$-arithmetic subgroups of linear algebraic groups defined over number fields.

	\subsection{Extension of the main results to any number field}

The two main results of this work—about quadratic forms over $\QQ$—can be extended to any number field $K$. We chose not to present them in this generality for two reasons: the first is that the key ideas for $K=\QQ$ are the same for any $K$, but the computations needed to get statements as explicit as Theorem \ref{Z_S-equiv_criterion_intro} and Theorem \ref{Generators_O(Q,Z_S)_intro} in the general case would have substantially increased the length of the article. The second reason is that the results we obtain for $K=\QQ$ are sharper because the approximations to the Generalized Ramanujan Conjecture for $\GLne(2)$ over $\QQ$ are finer than those for any $K$. Nonetheless, let us briefly describe the key technical ingredients needed to extend Theorem \ref{Z_S-equiv_criterion_intro}—the criterion of $\ZZ_S$-equivalence of quadratic forms—to any number field. 

Let $V^K$ be the set of places of $K$. We denote the completion of $K$ relative to $\nu \in V^K$ as $K_\nu$. For any finite subset $S$ of $V^K$ containing all the archimedean places of $K$ we define the \textit{ring of $S$-integers} as 
\[ \Oc_S = \{x\in K \mid |x|_\nu \leq 1 \text{ for all } \nu \in V^K - S\}, \]
and we set $K_S = \prod_{\nu \in S} K_\nu $. Consider two $\Oc_S$-equivalent $K$-quadratic forms $Q_1$ and $Q_2$ in $d\geq 3$ variables. The goal is to bound above the size of an equivalence matrix in $GL(d,\Oc_S)$ between $Q_1$ and $Q_2$. The reformulation of the problem in terms of the dynamical system 
\[ \mathcal{D}_{Q_1}: O(Q_1, K_S)  \curvearrowright O(Q_1,K_S)/O(Q_1,\Oc_S)\] 
we present in Section \ref{sec_dynamics_Z_S-equiv} for $K = \QQ$ is valid also for any $K$. Having done this, the proof of the equivalence criterion divides into two main parts: a mixing speed argument for $\mathcal{D}_{Q_1}$ and an upper bound of the volume of $O(Q_1,K_S)/O(Q,\Oc_S)$. 

The first part for $d \geq 3$ can be deduced from the case $d=3$, so we focus on the latter. 
What we need for the mixing speed argument is a decay speed of coefficients of $L^2$-functions orthogonal to all the one-dimensional, $SO(Q_1,K_S)$-invariant subspaces of 
\[ \mathcal{H}_{Q_1} := L^2(O(Q_1,K_S)/O(Q_1,\Oc_S)).\] 
It suffices to look at the irreducible unitary representations of $SO(Q,K_S)$—of dimension $>1$—contained in $\mathcal{H}_{Q_1}$. Any such representation extends to a \textit{cuspidal representation} 
\[ \pi^*=\bigotimes_{\nu \in V^K} \pi^*_\nu \] 
of $SO(Q_1,\AAA_K)$—$\AAA_K$ is the ring of \textit{ad\`eles} of $K$—. The local factors $\pi_\nu^*$ are irreducible unitary representations of $SO(Q_1,K_\nu)$ and—with finitely many exceptions—spherical. It suffices to find a $\nu_0 \in S$ such that $\pi^*_{\nu_0}$ has fast decay of coefficients for all the relevant $\pi^*$. The Jacquet-Langlands correspondence—valid for any $K$—sends $\pi^*$ to a cuspidal representation $\pi=\otimes_\nu \pi_\nu$ of $GL(2,\AAA_K)$. Moreover, $\pi^*_\nu$ and $\pi_\nu$ are unitary equivalent for all but finitely many $\nu$. The decay speed of coefficients of $\pi_\nu$ implied by \cite[equations (19), (20)]{sarnak_notes_2005}—an approximation to the Generalized Ramanujan Conjecture for $\GLne (2)$ over $K$—of Sarnak's notes suffices for the mixing speed argument on $\mathcal{D}_{Q_1}$. 

The dynamical argument to establish the bound of the volume of $O(Q_1,K_S)/O(Q_1,\Oc_S)$ in Section \ref{sec_vol_closed_orbits} for $K=\QQ$ extends almost line by line to any $K$. The main technical ingredient—a quantitative recurrence for unipotent flows on the space of lattices of $\QQ_S^d$ from \cite{kleinbock_flows_2007}—is in fact established for any $K$ in \cite{kleinbock-tomanov_flows_2003}. An alternative approach that works for all $K$ is computing the volume of $O(Q_1,K_S)/O(Q_1,\Oc_S)$ using Prasad's formula \cite{prasad_volumes_1989}, and then work out an upper bound in terms of the coefficients of $Q_1$.

	\subsection{Organization of the article}

In Section \ref{sec_qf} we cover the preliminaries on quadratic forms over completions $\QQ_\nu$ of $\QQ$. The reader can skip it, and go back to it only when necessary. 

The next three sections are devoted to the proof of Theorem \ref{Z_S-equiv_criterion_intro}. The reformulation of the $\ZZ_S$-equivalence problem in terms of the action $O(Q_1, \QQ_S) \curvearrowright \LatSpaceS{d}$ is explained in Section \ref{sec_dynamics_Z_S-equiv}, where we also prove two results---Proposition \ref{Dynamical_statement_RR-isotropic} and Proposition \ref{Dynamical_statement_R-anisotropic}---on the dynamics of $O(Q_1, \QQ_S)$ on any closed $O(Q_1, \QQ_S)$-orbit $Y$ in $\LatSpaceS{d}$. These dynamical statements depend on the volume of the $Y$.  For the $\ZZ_S$-equivalence problem, $Y$ is essentially $O(Q_1, \QQ_S) \ZZ_S^d$. That's why we prove  an estimate of the volume of $O(Q_1, \QQ_S) \ZZ_S^d$ in Section \ref{sec_vol_closed_orbits}. Having paved the way, we complete the proof  of  Theorem \ref{Z_S-equiv_criterion_intro} in Section \ref{sec_Zs-equiv_criteria}.

Our quantitative result on generators of $S$-integral orthogonal groups, Theorem \ref{Generators_O(Q,Z_S)_intro}, is established in Section \ref{sec_small_gens}. 

There are six appendices at the end of the article. In Appendices \ref{app_Decay_coefficients} through  \ref{app_reduction_theory} we gather several auxiliary results used in the article. Many of these give the explicit constants of our main results, so the proofs are sometimes lengthy—though not particularly difficult—. This explains why the appendices constitute almost a half of the article. A list of the explicit values of the constants in our statements can be found in Appendix \ref{app_constants}, the last one.
	\subsection{Acknowledgments}

This article contains the results of my PhD thesis. I would like to thank my advisor, Yves Benoist, for introducing me to the exciting world of homogeneous dynamics and its applications, and for its generosity during the three and a half years that I spent at Orsay working with him. I am also indebted to A. Gorodnik and G. Tomanov, who read thoroughly my thesis, kindly accepted to write a report on it and made valuable suggestions.  Thanks also to E. Ullmo, N. Bergeron and F. Maucourant,for being part of my thesis' jury and their valuable comments. Finally, I thank the anonymous referees for helping me to improve this work with their remarks and comments.

 This project has received funding from \textit{Fondation CFM pour la recherche}, as well as the European Research Council (ERC) under the European Union's Horizon 2020 research and innovation programme (grant agreement No 949143).

\section{Quadratic forms over $\QQ_\nu$}\label{sec_qf}

The objective of this section is to discuss basic aspects of quadratic forms, the main object of study of this work. We start by giving the general definitions in Subsection \ref{subsec_basicdefs}. Then, in Subsection \ref{subsec_sqf} we specialize to quadratic forms with coefficients in completions $\QQ_\nu$ of $\QQ$, and we fix representatives in each $\QQ_\nu$-equivalence class. We will refer to these as the \textit{standard $\QQ_\nu$-quadratic forms}. Finally, we establish in Subsection \ref{subsec_three_lemmas} three auxiliary results that will be used in the proofs of the main results of the article.  

	\subsection{Basic definitions}\label{subsec_basicdefs}

Let $\mathcal{R}$ be a commutative ring with 1. An $\mathcal{R}$-quadratic form in $d$ variables is a homogeneous polynomial of degree 2
\[ Q(x) = \sum_{\scriptscriptstyle \substack{i, j = 1 \\ i \leq j}}^d a_{ij} x_i x_j, \]
with coefficients in $\mathcal{R}$. We say that $Q$ is $\mathcal{R}$-\textit{isotropic} if there is $v \in \Rc^d - \{0\}$ such that $Q(v) = 0$, and $Q$ is $\Rc$-\textit{anisotropic} if $Q(v) \neq 0$ for any $v \in \Rc^d-\{0\}$. Let $Q, Q_1$ and $Q_2$ be $\Rc$-quadratic forms in $d$ variables. $Q_1$ and $Q_2$ are $\Rc$-\textit{equivalent}, denoted $Q_1\underset{\Rc}{\sim} Q_2$, if and only if there is $g \in GL(d,\Rc)$---the group of $d \times d$ matrices with coefficients in $\Rc$ whose determinant is invertible in $\Rc$---such that $Q_2 = Q_1 \circ g $. $Q$ is non-degenerate if it's not $\Rc$-equivalent to a quadratic form in less than $d$ variables.  When 2 is a unit of $\Rc$, there is a correspondence between quadratic forms in $d$ variables and symmetric bilinear forms on $\Rc^d$: the map
\[ \scalar{x}{y}_Q = \frac{1}{2} (Q(x+y) - Q(x) - Q(y))\]
is a symmetric bilinear form on $\Rc^d$. Conversely, a  symmetric bilinear form $\scalar{\cdot}{\cdot}$ on $\Rc^d$ defines a quadratic form $x \mapsto \scalar{x}{x}$. We denote by $b_Q$ the matrix $(\scalar{e_i}{e_j}_Q)_{i,j}$ of $\scalar{\cdot}{\cdot}_Q$ with respect to the standard basis $e_1,\cdots,e_d$ of $\Rc^d$, and we denote $\det b_Q$ by $\det Q$. 

The rings $\Rc$ relevant to this work are $\QQ_S$ and $\ZZ_S$ for some finite set $S = \{\infty \} \cup S_f$ of places of $\QQ$. We will briefly recall some facts about absolute values of $\QQ$ and fix some notation we'll use. An \textit{absolute value} on $\QQ$ is a map $|\cdot|:\QQ \to [0,\infty)$ such that the next conditions hold for any $s,t \in \QQ$:
\begin{enumerate}[$(i)$]
\item $|s| = 0 \Leftrightarrow s = 0$,
\item $|st| = |s| \, |t|$,
\item $|s + t| \leq |s| + |t|$.
\end{enumerate}
When an absolute value $|\cdot|$  verifies the ultrametric triangle inequality
\[ |s + t| \leq \max\{ |s|, |t| \} \quad \text{for any }s,t \in \QQ,\]
we say it is \textit{non-archimedean. }Otherwise, $|\cdot|$ is \textit{archimedean}. Let's see three examples: The first is the standard absolute value of $\QQ$, denoted $|\cdot|_\infty$. It is archimedean. For any prime number $p$, the \textit{$p$-adic absolute value} of a nonzero integer $n$ is defined as
\[|n|_p = p^{-a(n)},\]
where $a(n)$ is the unique natural number such that $n$ is in $p^{a(n)} \ZZ - p^{a(n)+1}\ZZ$. We extend $|\cdot|_p$ to $\QQ$ multiplicatively. Finally,  
\[|s|_{\mathtt{t}} = 1 \quad \text{for any } s \in \QQ - \{0\}~\] 
is the \textit{trivial absolute value on $\QQ$}. Both $|\cdot|_p$ and $|\cdot|_{\mathtt{t}}$ are nonarchimedean. In fact, these three are essentially all the absolute values on $\QQ$. We say that $|\cdot|$ and $|\cdot|'$ are equivalent if and only if their topologies on $\QQ$ coincide. A \textit{place} of $\QQ$ is an equivalence class of absolute values on $\QQ$. The map $\nu \mapsto |\cdot|_\nu$ induces a bijection between 
\[\mathscr{P} = \{\infty\} \cup \{p \in \mathbb{N} \mid p \text{ is prime}, \}\] 
and the set of nontrivial places of $\QQ$---see \cite[Theorem 1, p. 3]{koblitz_p-adic_1984}. From now on we will consider only nontrivial places of $\QQ$, which we identify with elements of $\mathscr{P}$. For any place $\nu$ of $\QQ$, we denote by $\QQ_\nu$ the completion of $\QQ$ with respect to $|\cdot|_\nu$. We will stick to the classical notation $\RR$ for the field of real numbers rather than $\QQ_\infty$. We endow any finite dimensional $\QQ_\nu$-vector space $V$ with a preferred basis $(v_i)_i$ with the norm
\[ \normnu{\textstyle \sum_i a_i v_i} = \max_i |a_i|_\nu. \]
The preferred bases for $\QQ_\nu^d$, and the space $M_d(\QQ_\nu)$ of $d\times d$ matrices with coefficients in $\QQ_\nu$ are respectively the canonic basis $(e_1, \ldots, e_d)$ and  $(E_{ij})$\footnote{The $i,j$ coefficient of $E_{ij}$ is 1, and the rest are 0.}. The $\nu$-norm of a quadratic form $Q$ on $\QQ_\nu^d$ is defined as  $\normnu{Q} = \normnu{b_Q}$.

	\subsection{Standard quadratic forms over $\QQ_\nu$}\label{subsec_sqf}

Let $\nu$ be a place of $\QQ$. Here we fix a unique representative of each $\QQ_\nu$-equivalence class of $\QQ_\nu$-quadratic forms. We call these \textit{standard $\QQ_\nu$-quadratic forms}. They will be constantly used in the article.

First let's review the basic theory of $\QQ_\nu$-quadratic forms. The reader can find in \cite{serre_cours_1995} and \cite{cassels_rational_1978} the proofs of the facts in the discussion that follows. The non-degenerate quadratic forms on $\QQ_\nu^d$ are classified by means of certain invariants. We will consider only diagonal quadratic forms since any $\QQ_\nu$-quadratic form is $\QQ_\nu$-equivalent to a diagonal one. Consider such a $Q(x) = a_1 x_1^2 + \cdots + a_d x_d^2$. For $\nu = \infty$, the $\RR$-equivalence class of $Q$ depends only on its \textit{signature} $(p_Q, n_Q)$, where $p_Q$ and $n_Q$ are respectively the number of positive and negative $a_i$'s. For the $p$-adic case we need two invariants. The \textit{Hilbert symbol} of any $a,b \in \QQ_p^\times$ is defined as
\[ (a,b)_p = \begin{cases}
1 & \text{if } a x_1^2 + b x_2^2 - x_3^2 \text{ is } \QQ_p \text{-isotropic}, \\
-1 & \text{if } a x_1^2 + b x_2^2 - x_3^2 \text{ is } \QQ_p \text{-anisotropic}.
\end{cases}
 \]
Let $(\QQ_p^\times)^2$ be the set of nonzero squares of $\QQ_p$, and let $\Ctt$ be the projection $\QQ_p^\times \to \QQ_p^\times / (\QQ_p^\times)^2$. We define $\varepsilon(Q) = \prod_{i<j} (a_i, a_j)_p$, and the \textit{discriminant} of $Q$ as $disc\, Q = \Ctt(\det Q)$. It turns out that $\varepsilon(Q)$ and $disc\,Q$ depend only on the $\QQ_p$-equivalence class $[Q]$ of $Q$. Moreover, $[Q]$ is completely determined by these two invariants. Conversely, when $d \geq 3$, any combination of $disc\,Q$ and $\varepsilon(Q)$ is possible. 

Now we'll choose the standard $\QQ_\nu$-quadratic forms. Let's fix once and for all a system $\Cc_\nu$ of representatives of $\QQ_\nu^\times/ (\QQ_\nu^\times)^2$: for any odd prime number $p$, let $\ntt_p$ be an integer that is not a square in $\ZZ / p \ZZ$. We define 
\[\Cc_\nu = \begin{cases}
\{ \pm 1, \pm 3, \pm 2, \pm 6\} & \text{if } \nu = 2, \\
\{ 1, \ntt_p, p, \ntt_p p \} & \text{if } \nu = p \, \text{odd}, \\
\{ \pm 1\} & \text{if } \nu = \infty.
\end{cases}\] 
Let's start with the standard, anisotropic $\QQ_\nu$-quadratic forms. For $\nu = \infty$ there are two in $d \geq 1$ variables:
\[x_1^2 + \cdots + x_d^2 \quad \text{and} \quad -x_1^2 - \cdots - x_d^2.\]  
For $\nu = p < \infty$, $\QQ_p$-anisotropic quadratic forms exist only in $d \leq 4$ variables. The standard ones are gathered in Table \ref{table_standard_ani_qf}. Finally, the standard, isotropic $\QQ_\nu$-quadratic forms are either a direct sum of hyperbolic planes
\[x_1^2 - x_2^2 + \cdots +x_{2m-1}^2 - x_{2m}^2,\] 
or a direct sum of hyperbolic planes and a standard, anisotropic $\QQ_\nu$-quadratic form. For example, the standard, isotropic quadratic forms on $\QQ_\nu^4$ are 
\[x_1^2 - x_2^2 + x_3^2 - x_4^2 \quad \text{and} \quad x_1^2 - x_2^2 + A(x_3, x_4),\]
with $A(x_3,x_4)$  standard, anisotropic on $\QQ_\nu^2$. There are respectively 15, 7 and 3 of these for $\nu = 2, \nu = p$ odd, and $\nu = \infty$. Most of the time we will need to work with quadratic forms $Q_\nu$ over various $\QQ_\nu$'s. An efficient way to do this is to gather all the $Q_\nu$'s into a single quadratic form over the product ring of the $\QQ_\nu$'s. Let $S$ be a finite set of places of $Q$. We say that a quadratic form $P = (P_\nu)_{\nu \in S}$ on $\QQ_S^d$ is standard if and only if $P_\nu$ is a standard quadratic form on $\QQ_\nu^d$ for each $\nu \in S$. We will use freely the next fact throughout the article.

\begin{Fact}
Let $S$ be a finite set of places of $\QQ$. Any nonzero $\QQ_S$-quadratic form is $\QQ_S$-equivalent to a unique standard quadratic form. 
\end{Fact}

\begin{longtable}[c]{|c|c|c|}
\hline
 & $p = 2$ & $p > 2$  \\
\hline \hline
 \multirow{7}{2.5em}{$d = 2$} & $x_1^2 + x_2^2$, \quad $-x_1^2 - x_2^2$, & \\
 & $x_1^2 + 3 x_2^2$, \quad $2 x_1^2 + 6 x_2^2$, &  \\
& $x_1^2 - 3 x_2^2$, \quad $-x_1^2 + 3 x_2^2,$  & $x_1^2 - \ntt_p x_2^2$, \quad $p x_1^2 - p \ntt_px_2^2$,      \\
&   $x_1^2 + 2 x_2^2$, \quad $-x_1^2 - 2 x_2^2$, & $x_1^2 - p x_2^2$,  \quad $\ntt_p x_1^2 - p \ntt_px_2^2$,    \\
&   $x_1^2 - 2 x_2^2$, \quad $3 x_1^2 - 6 x_2^2$, & $x_1^2 - p \ntt_p x_2^2$, \quad $\ntt_p x_1^2 - p x_2^2$ \\
&  $x_1^2 + 6 x_2^2$, \quad $2 x_1^2 + 3 x_2^2$, & \\
&   $x_1^2 - 6 x_2^2$, \quad $-x_1^2 + 6 x_2^2$ &  \\
 \hline
  &  & $x_1^2 - \ntt_p x_2^2 + p x_3^2$ \\
$d = 3$ & $c(x_1^2 + x_2^2 + x_3^2)$ for any $c \in \mathcal{C}_2$ & $x_1^2 - \ntt_p x_2^2 + \ntt_p p x_3^2$ \\
 & & $x_1^2 + p x_2^2 - \ntt_p p x_3^2$ \\
 & & $\ntt_p x_1^2 + p x_2^2 - \ntt_p p x_3^2$ \\
 \hline
 $d = 4$ & $x_1^2 + x_2^2 + x_3^2 + x_4^2$ & $x_1^2 - \ntt_p x_2^2 + p x_3^2 -\ntt_p p x_4^2$ \\
\hline 
\caption{Standard anisotropic $\QQ_p$-quadratic forms}
\label{table_standard_ani_qf}
\end{longtable}

	\subsection{Three lemmas}\label{subsec_three_lemmas}
In this subsection we gather three lemmas about $\QQ_\nu$-quadratic forms that will be used later. The reader can safely skip this subsection,  and come back to it when needed.

We introduce new notation needed for the statement of the first lemma. For any prime number $p$ we define
\[\Att_p = \begin{cases}
2 & \text{if } p=2,\\
1 & \text{if } p > 2.
\end{cases} 
\]
For any place $\nu$ of $\QQ$ and any integer $d \geq 2$ we define
\[\consSmallStandMatForQF{d}{\nu} = \begin{cases}
\Att_p \sqrt{p}& \text{if } \nu = p, \\
d & \text{if } \nu = \infty. 
\end{cases} \]

\begin{Lem}\label{SmallStandMatForQF}
Consider a place $\nu$ of $\QQ$ and an integer $d \geq 2$. For any non-degenerate quadratic form $R$ on $\QQ_\nu^d$, there are a standard quadratic form $P$ and $g \in GL(d,\QQ_\nu)$ with 
\[ \normnu{g} \leq \consSmallStandMatForQF{d}{\nu} \normnu{R}^{\frac{1}{2}}, \]
such that $R = P \circ g$. 
\end{Lem}

To establish Lemma \ref{SmallStandMatForQF} we'll use the next four facts that can be easily verified. 

\begin{Fact}\label{Aux_anisotropic_1}
Let $p$ be a prime number and let $P$ be a standard, anisotropic quadratic form on $\QQ_\nu^d$. Any $t \in \QQ_p^d$ such that $|P(t)|_p \leq 1$ verifies $\normp{t} \leq \Att_p$.
\end{Fact}

\begin{Fact}\label{Aux_isotropic_1}
Let $p$ be a prime number and let $I(x) = x_1^2 - x_2^2$. For any $a \in \ZZ_p$ there is $t \in \QQ_p^2$ with $\normp{t} \leq \Att_{p}$, such that $I(t) = a$. 
\end{Fact}

\begin{Fact}\label{Aux_binary_change_sign}
Let $p$ be an odd prime number. Consider $a,b \in \QQ_p^\times$ such that $|a|_p = |b|_p$ and $ab$ is a square in $\QQ_p^\times$. Then there is $k \in GL(2,\ZZ_p)$ taking the quadratic form $-a x_1^2 - b x_2^2$ to $a x_1^2 + b x_2^2$. 
\end{Fact}

\begin{Fact}\label{Aux_binary_force_coefficient}
Consider an odd prime number $p$, a diagonal $\QQ_p$-anisotropic quadratic form $Q'$ in $d_0 \leq 2$ variables with $\normp{Q'} = 1$, and $c \in \QQ_p$ with $p\inv \leq |c|_p \leq 1$. If $Q'$ represents $c$, there is $k \in GL(d_0,\ZZ_p)$ such that $Q'  \circ k$ is a diagonal quadratic form whose $x_1^2$ coefficient is $c$.  
\end{Fact}

\begin{proof}[Proof of Lemma \ref{SmallStandMatForQF}]

Throughout the proof, $P$ is the standard quadratic form $\QQ_\nu$-equivalent to $R$ (and later $Q$). We begin with the case $\nu = \infty$. Using Cauchy-Schwartz inequality one can prove that for any $k_0 \in O(d,\RR)$ and for any $A \in M_d(\RR)$,  $\normi{k_0A}$ and $\normi{Ak_0}$ are both less than $\sqrt{d} \normi{A}$. There is always $k \in O(d,\RR)$ such that $R_1 = R \circ k$ is diagonal, say 
\[ R_1(x) = a_1 x_1^2 + \cdots + a_d x_d^2. \]  
Since permutation matrices are in $O(d,\RR)$, we assume further that $a_1, \ldots , a_s$ are positive, and the rest are negative. Note that 
\[\normi{R'} = \normi{b_{R'}} = \normi{\tra k b_R k} \leq d \normi{R}. \] 
Consider
\[g' = diag(|a_1|_\infty ^\frac{1}{2}, \ldots, |a_d|_\infty ^\frac{1}{2}). \]
 Then $g := g' k$ takes $P(x) = x_1^2 + \cdots + x_s^2 - \cdots - x_d^2$ to $R$ and verifies the bound of the statement. 

We pass to the case $\nu = p$. We will prove that for any diagonal quadratic form $Q(x) = b_1 x_1^2 + \cdots + b_d x_d^2$ such that $p\inv \leq |b_i|_p \leq 1$ for any $i$, there is a standard quadratic form $P$ and $g_0 \in GL(d,\QQ_p)$ with
\begin{equation} \label{quasi-standard_qf}
\normp{g_0} \leq \Att_p,
\end{equation}
such that $Q = P \circ g_0$. Let's see how to deduce form this the result for a general $R$. By \cite[Fact 5.4]{benoist_polar_2007}, there is a diagonal quadratic form $R_1(x) = a_1 x_1^2 + \cdots + a_d x_d^2$ and some $k_1 \in GL(d,\ZZ_p)$ such that $R_1 \circ k_1 = R$. Note that $\normp{R} = \normp{R_1}$ since $k_1$ is an isometry of $(\QQ_p^d, \normp{\cdot})$. For each $i$ we write $a_i = b_i c_i^2$ for some $b_i \in \ZZ_p^\times \cup p \ZZ_p^\times$. Consider $Q(x) = b_1 x_1^2 + \cdots + b_d x_d^2$ and $g_1 = \diag (c_1,\ldots, c_d)$. Note that
\begin{equation}\label{NNN1}
\normp{g_1} = \max_i |c_i|_p \leq \max_i \sqrt{p} |a_i|_p^\frac{1}{2} = \sqrt{p} \normp{R}^\frac{1}{2}.
\end{equation}   
Consider $g_0 \in GL(d,\QQ_p)$ as in \eqref{quasi-standard_qf}, such that $Q = P \circ g_0$. Then $g =  g_0 g_1 k_1$ takes $P$ to $R$, and by \eqref{quasi-standard_qf} and \eqref{NNN1} we have
\[\normp{g} = \normp{g_0 g_1} \leq \normp{g_0} \normp{g_1} \leq \sqrt{p} \Att_{p} \normp{R}^\frac{1}{2}. \]
Now let's prove \eqref{quasi-standard_qf}. Suppose first that $Q$ is $\QQ_p$-anisotropic. Let $g_0$ be any matrix in $GL(d,\QQ_p)$ taking $P$ to $Q$. Any column $t$ of $g_0$ verifies $\normp{t} \leq \Att_p$ by Fact \ref{Aux_anisotropic_1}, since $P(t) = b_i$ for some $i$ and $|b_i|_p \leq 1$. Suppose now that $Q$ is $\QQ_p$-isotropic. We will complete the proof only for $p$ odd. The case $p = 2$ can be settled in a similar way, but there are more subcases to consider. We proceed by induction on the number $d$ of variables of $Q$.  $P(x) = \ctt x_1^2$ to $Q$.  Let $\Ctt$ be the projection $\QQ_p^\times \to \QQ_p^\times / (\QQ_p^\times)^2$. In the base case, $d = 2$ and $P(x) = x_1^2 - x_2^2$, so the discriminant of $Q$ is $\Ctt(-1)$. Hence there is $u \in \QQ_p$ such that $b_2 = -b_1 u^2$. Moreover, $|u|_p = 1$ since $|b_i|_p$ is either 1 or $p\inv$. By Fact \ref{Aux_isotropic_1}, there is $t \in \QQ_p^2$ with $\normp{t} \leq 1$ such that $P(t) = b_1$. Then
\[ g_0 = \begin{pmatrix}
t_1 & ut_2 \\
t_2 & ut_1
\end{pmatrix} 
\]
takes $P$ to $Q$ and $\normp{g_0} \leq 1$. Now consider $d \geq 3$ and suppose that \eqref{quasi-standard_qf} is true for any diagonal, $\QQ_p$-isotropic quadratic form $Q'$ in $<d$ variables with nonzero coefficients in $\ZZ_p^\times \cup p \ZZ_p^\times$. We consider the following subcases.
\begin{enumerate}[C1.]
\item Suppose that $Q_1(x) = b_i x_i^2 + b_j x_j^2$ is $\QQ_p$-isotropic for some $i \neq j$. We may even assume that $i = 1$ and $j = 2$ since the permutation matrices are contained in $GL(d,\ZZ_p)$. Let $Q_2(x) = b_3 x_3^2 + \cdots + b_d x_d^2$. Let $P_\ell$ be the standard quadratic form $\QQ_p$-equivalent to $Q_\ell$. Note that $P = P_1 \oplus P_2$. By inductive hypothesis\footnote{If $Q_2$ is $\QQ_p$-isotropic, it is covered by the inductive hypothesis. Otherwise, it is $\QQ_p$-anisotropic and we already have proved the result in that case.}, there are $g_1 \in GL(2,\QQ_p)$ and $g_2 \in GL(d-2,\QQ_p)$ with $\normp{g_\ell} \leq 1$ such that $Q_\ell = P_\ell \circ g_\ell$. Then $g_0 := g_1\oplus g_2$ takes $P$ to $Q$ and verifies \eqref{quasi-standard_qf}. 
 
\item Suppose now that $b_i x_i^2 + b_j x_j^2$ is $\QQ_p$-anisotropic for any $i \neq j$. We'll consider two further subcases. First assume there are three $b_k$'s with the same $p$-absolute value. Say that $|b_1|_p = |b_2|_p = |b_3|_p$.  Since $|b_i b_j|_p$ is either $1$ or $p^{-2}$ for any $i \neq j$ in $\{1,2,3\}$, then $\Ctt(b_i b_j)$ is either $\Ctt(-1)$ or $\Ctt(-\ntt_p)$. But $b_i x_i^2 + b_j x_j^2$ is $\QQ_p$-anisotropic, so $\Ctt(b_i b_j) \neq \Ctt(-1)$. Hence $\Ctt(b_1 b_2) = \Ctt(b_2 b_3) = \Ctt(b_3 b_1) = \Ctt(-\ntt_p)$. It follows that $\Ctt(b_1) = \Ctt(b_2) = \Ctt(b_3)$, and hence $b_1 b_2, b_2 b_3$ and $b_3 b_1$ are squares in $\QQ_p^\times$. Let $Q_1(x) = b_1 x_1^2 -b_2 x_2^2 - b_3 x_3^2 + b_4 x_4^2 + \cdots + b_d x_d^2$. By Fact \ref{Aux_binary_change_sign}, there is $k \in GL(2,\ZZ_p)$ such that $k_1 = I_1 \oplus k \oplus I_{d-3}$ takes $Q_1$ to $Q$. Since $k_1$ is in $GL(d,\ZZ_p)$ and $Q_1$ is covered by C1, we are done.

Finally, suppose there are no three $b_k$'s with the same $p$-absolute value. Hence $d$ is either 3 or 4 and, up to a permutation of variables, $|b_1|_p = 1$ and $|b_2|_p = p\inv$.  Let $Q_1(x) = b_1 x_1^2 + b_2 x_2 ^2$ and $Q_2(x) = b_3 x_3^2 + \cdots + b_d x_d^2$. Note that $Q_1$ and $Q_2$ are $\QQ_p$-anisotropic. We claim there is $c \in \QQ_p^\times$ such that $Q_1$ represents $c$ and $Q_2$ represents $-c$. Indeed, since $Q$ is $\QQ_p$-isotropic, there is a nonzero $t \in \QQ_p^d$ such that $Q(t) = 0$. Let $d_0 = d-2$. We write $t = (t_1, t_2)$ with $t_1 \in \QQ_p^2$ and $t_2 \in \QQ_p^{d_0}$.  We know that $t_1 \neq 0$ or $t_2 \neq 0$. Say that $t_1 \neq 0$. Then $c: = Q_1(t_1) \neq 0$ and $Q_2(t_2) = -c$.   By Fact \ref{Aux_binary_force_coefficient},  there are $k_1 \in GL(2,\ZZ_p)$ and $k_2 \in GL(d_0,\ZZ_p)$ and a diagonal quadratic form $Q_3$ in $d$ variables, whose coefficients of $x_1^2$ and $x_3^2$ are respectively $c$ and $-c$, such that $k_1 \oplus k_2$ takes $Q_3$ to $Q$. C1 covers $Q_3$, so we are done.
\end{enumerate}     

\end{proof}

We pass to our second lemma. For any place $\nu$ of $\QQ$ we set 
\[ \Btt_\nu = \begin{cases}
4 & \text{if } \nu = 2, \\
p & \textit{if } \nu \text{ is an odd prime } p,\\
1 & \text{if } \nu = \infty. 
\end{cases} \] 
Let $P$ be a quadratic form on $\QQ_\nu^d$. We denote by $O(P,\QQ_\nu)\ncc$ the image in $SO(P,\QQ_\nu)$ of the universal covering $Spin(P,\QQ_\nu)$ of $SO(P, \QQ_\nu)$\footnote{See \cite[Chapitre II]{dieudonne_geometrie_1971} for the construction of $Spin(P)$ using the Clifford algebra of $P$.}. 

\begin{Lem}\label{Small_rep_H/Hncc}
Consider a place $\nu$ of $\QQ$ and an integer $d \geq 2$. Let $P(x) = a_1 x_1^2 + \cdots + a_d x_d^2$ be a non-degenerate, diagonal quadratic form on $\QQ_\nu^d$ with $a_1 = - a_2 = 1$. Any $O(P, \QQ_\nu)\ncc$-coset in $O(P,\QQ_\nu)$ has a representative $\eta$ with $\normnu{\eta} \leq \Btt_\nu$.  
\end{Lem}

To prepare for the proof of Lemma \ref{Small_rep_H/Hncc}, let's give a more concrete description of $O(P,\QQ_\nu)\ncc$ in terms of the \textit{spinor norm} of $P$. What follows holds for any non-degenerate quadratic form $P$ on a finite-dimensional vector space $V$ over a field $k$. For any $v \in V$ with $P(v)\neq 0$---in which case we say that $v$ is $P$-anisotropic---we denote by $r_v$ the linear operator of $V$ that fixes the $P$-orthogonal complement $v^\perp$ of $v$ pointwise, and sends $v$ to $-v$. We say that a linear operator $T$ on $V$ is a $P$-reflection if $T = r_v$ for some $P$-anisotropic $v \in V$. It's a classical fact that $O(P,k)$ is generated by the $P$-reflections of $V$---see \cite[Theorem 3.20, p. 129]{artin_geometric_1988}. The \textit{spinor norm} of the quadratic space $(V,P)$ is the unique group morphism $\mathcal{S}_P: O(P,k) \to k^\times / (k^\times)^2$ such that 
\[ \mathcal{S}_P(r_v) = P(v)(k^\times)^2 \]
for any $P$-reflection $r_v$---see \cite[Definition 3.4, p. 336]{scharlau_quadratic_1985}. The sequence
\begin{equation}\label{exact_seq}
Spin(P,k) \longrightarrow SO(P,k) \overset{\mathcal{S}_P}{\longrightarrow} k^\times / (k^\times)^2 
\end{equation}
is exact---see \cite[p. 336]{scharlau_quadratic_1985}---, then any two $h_1, h_2 \in O(P,k)$ are in the same $O(P,k)\ncc$-coset if and only if $\mathcal{S}_P(h_1) = \mathcal{S}_P(h_2)$ and $\det h_1 = \det h_2$.  

\begin{proof}[Proof of Lemma \ref{Small_rep_H/Hncc}]
Let's prove first the result for $P_0(x) = x_1^2 - x_2^2$. For any $c \in \Cc_\nu$ we'll choose a $v_c = (\att_c, \btt_c) \in \QQ_\nu^2$ such that $P_0(v_c) = c$. Let's fix $v_1 = (1,0)$. The matrix of the $P_0$-reflection $r_{v_c}$ in the standard basis of $\QQ_\nu^2$ is 
\[h_c = \frac{1}{c} \begin{pmatrix}
-(\att_c^2 + \btt_c^2) & 2 \att_c \btt_c \\
-2 \att_c \btt_c & \att_c^2 + \btt_c^2
\end{pmatrix},\]
hence $\normnu{h_c} \leq \frac{\normnu{v_c}^2}{|c|_\nu}$. The matrices $h_c$ and $h_1 h_c$ with $c \in \Cc_\nu$ represent all the $O(P,\QQ_\nu)\ncc$-cosets in $O(P,\QQ_\nu)$ by the remark right after \eqref{exact_seq}. To prove the result it suffices to show that $\normnu{h_c} \leq \Btt_\nu$ since $\normnu{h_1 h_c} = \normnu{h_c}$. 

It's easy to see that we can choose $v_c$ with $\normnu{v_c} = 1$ if $\nu \neq 2$, or if $\nu = 2$ and $c \in \Cc_2$ is odd. Hence $\normnu{h_c} \leq \Btt_\nu$ in those cases. Suppose now that $\nu = 2$ and $c \in \{\pm 2, \pm 6 \}$. We consider
\begin{equation}\label{Spin_ex_seq}
v_2 = \left(\frac{3}{2}, \frac{1}{2} \right), \quad v_6 = \left( \frac{5}{2}, \frac{1}{2} \right),
\end{equation}
and $v_{-c} = (\btt_c, \att_c)$ for any $c \in \{2,6\}$. We see right away that $\norm{h_c}_2 = 4$ in the four cases.

Now consider $P(x) = x_1^2 - x_2^2 + \cdots + a_d x_d^2$ with $d \geq 3$. Let $\varphi: O(P_0,\QQ_\nu) \to O(P,\QQ_\nu)$ be the homomorphism sending any $h$ in $O(P_0,\QQ_\nu)$ to the linear map that acts respectively as $h$ and the identity on $\QQ_\nu e_1 \oplus \QQ_p e_2$ and $\QQ_\nu e_3 \oplus \cdots \oplus \QQ_\nu e_d$. Note that $\varphi$ induces a bijection $O(P_0,\QQ_\nu) / O(P_0,\QQ_\nu)\ncc \to O(P,\QQ_\nu) / O(P,\QQ_\nu)\ncc$, and that $\normnu{\varphi(h_c)} \leq \Btt_\nu$ for our choices of $h_c$, so we are done.   
\end{proof}

We need new notation for the statement of the third lemma. For any $t \in \RR$ we define
\[ a_{\infty, t} = \begin{pmatrix}
e^{t/2} & 0 \\
0 & e^{-t/2}
\end{pmatrix} \quad \text{and} \quad 
b_{\infty, t} = \begin{pmatrix}
\cosh t & - \sinh t   \\
-\sinh t & \cosh t 
\end{pmatrix}.  \]
For any prime $p$ and any integer $t$ we define 
\[ a_{p,t} = \begin{pmatrix}
p^{-t} & 0 \\
0 & p^{t}
\end{pmatrix} \quad \text{and} \quad
b_{p,t} = \frac{1}{2} \begin{pmatrix}
p^{2t} + p^{-2t} & p^{2t} - p^{-2t}  \\
p^{2t} - p^{-2t} & p^{2t} + p^{-2t} 
\end{pmatrix}. \]
Let $m$ and $n$ be positive integers. We denote by $O_{m \times n}$ the $m \times n$ zero matrix and by $I_m$ the $m \times m$ identity matrix. If $A$ and $B$ are respectively $m \times m$ and $n \times n$ matrices, we denote by $A \oplus B$ the $(m+n) \times (m+n)$ matrix
\[\begin{pmatrix}
A & O_{m\times n} \\
O_{n \times m} & B
\end{pmatrix}.\]
For any positive integers $d, p$ and $n$, with $p$ prime, we denote the kernel of the natural morphism $SL(d,\ZZ_p) \to SL(d,\ZZ / p^n \ZZ)$ by $K_{d,p}(n)$.

\begin{Lem}\label{Covering_SL(2)->SO(P)} 
Consider a place $\nu$ of $\QQ$ and a non-degenerate, diagonal $\QQ_\nu$-quadratic form $P(x) = x_1^2 - x_2^2 + a_3 x_3^2 + \cdots + a_d x_d^2$ with $d \geq 3$. There is a continuous morphism with finite kernel $\rho_P : SL(2,\QQ_\nu) \to O(P,\QQ_\nu)\ncc$ with the following properties:

	\begin{enumerate}
	\item  $\rho_P$ sends $a_{\nu,t}$ to $b_{\nu,t} \oplus I_{d-2}$ for any $t\in \RR$ if $\nu = \infty$, and any $t \in \ZZ$ if $\nu < \infty$.
	\item When $\nu = p$ and $p\inv \leq |a_3|_p \leq 1$, $\rho_P$ sends $K_{2,p}(n+1)$ to $K_{d,p}(n - 1)$ for any $n > 1$. 
	\end{enumerate}
 
\end{Lem}

\begin{proof}
Let's begin with $d = 3$. The morphism $\iota_P$ will be obtained writing the adjoint representation of $SL(2,\QQ_\nu)$ in an appropriate basis of the Lie algebra $\mathfrak{sl}(2,\QQ_\nu)$ of $SL(2,\QQ_\nu)$. The Killing form  of $\mathfrak{sl}(2,\QQ)$ in the basis
\[\beta = \left( \begin{pmatrix}
0 & 1 \\
0 & 0
\end{pmatrix}, \quad  \begin{pmatrix}
0 & 0 \\
1 & 0
\end{pmatrix}, \begin{pmatrix}
1 & 0 \\
0 & -1
\end{pmatrix}
\right)
\]
is $\mathscr{K} (x) = 8(x_1 x_2 + x_3^2)$. For any $g = \begin{pmatrix}
\att & \btt \\
\ctt & \dtt 
\end{pmatrix} \in SL(2,\QQ_\nu)$, the matrix of $Ad\, g$ with respect to $\beta$ is
\begin{equation}\label{Ad_SL(2)}
 [Ad\, g]_{\beta} = \begin{pmatrix}
\att^2 & -\btt^2 & -2\att \btt \\
-\ctt^2 & \dtt^2 & 2\ctt \dtt \\
-\att \ctt & \btt \dtt & \att \dtt + \btt \ctt
\end{pmatrix}.
\end{equation} 
Consider the matrix
\[ f_{a_3} = \begin{pmatrix}
a_3\inv & -a_3\inv & 0\\
1 & 1 & 0 \\
0 & 0 & 1 
\end{pmatrix} \quad \text{and its inverse} \quad f_{a_3}\inv = 
\begin{pmatrix}
a_3/2 & 1/2 & 0 \\
-a_3 / 2 & 1/2 & 0 \\
0 & 0 & 1
\end{pmatrix}. \]
Since $\mathscr{K} \circ f_{a_3} = \frac{8}{a_3} P$, the map
\begin{align}
\notag \rho_P(g) &= f_{a_3}\inv [Ad\, g]_{\beta} f_{a_3} \\
\label{rhop} & = \begin{pmatrix}
\frac{1}{2}(\att^2-a_3 \btt^2 - a_3\inv \ctt^2 + \dtt^2) & \frac{1}{2}(-\att^2 -a_3 \btt^2 + a_3\inv \ctt^2 + \dtt^2) & -a_3 \att \btt + \ctt \dtt \\
\frac{1}{2}(-\att^2 + a_3 \btt^2 - a_3\inv \ctt^2 + \dtt^2) & \frac{1}{2} (\att^2 + a_3 \btt^2 + a_3\inv \ctt^2 + \dtt^2) & a_3 \att \btt + \ctt \dtt \\
-a_3\inv \att \ctt + \btt \dtt & a_3\inv \att \ctt + \btt \dtt & \att \dtt + \btt \ctt~ 
\end{pmatrix} 
\end{align} 
is a morphism $SL(2,\QQ_\nu) \to SO(P,\QQ_\nu)$ with kernel $\{\pm I_2\}$. Since $SL(2,\QQ_\nu)$ is simply connected, the image of $\rho_P$ is $O(P,\QQ_\nu)\ncc$ by the uniqueness of the universal covering. A direct computation shows that $\rho_P(a_{\nu,t}) = b_{\nu,t} $ for any $t$.  Suppose now that $d>3$. Let $P'(x) = x_1^2 - x_2^2 + a_3 x_3^2$ and consider the map $\rho_{P'} : SL(2,\QQ_\nu) \to SO(P', \QQ_\nu)$ as above. The morphism $\rho_P$ obtained by composing $\rho_{P'}$ with the embedding $SO(P', \QQ_\nu) \to SO(P,\QQ_\nu), h \mapsto h \oplus I_{d-3},$ verifies point $1.$ When $g$  is in $K_{2,p}(n+1)$ and $p\inv \leq |a_3|_p \leq 1$, from $p\inv \leq |a_3|_p \leq 1$ and \eqref{rhop} we see that $\normp{ \rho_P(g) - I_d} \leq p^{-n}$, so point 2 follows. 
\end{proof}

\section{The dynamics of $\ZZ_S$-equivalence}\label{sec_dynamics_Z_S-equiv}

The goal of this section is to give a dynamical interpretation of the problem of $\ZZ_S$-equivalence of integral quadratic forms, and to establish two dynamical results---propositions \ref{Dynamical_statement_RR-isotropic} and \ref{Dynamical_statement_R-anisotropic}---that will be key in the proof of our $\ZZ_S$-equivalence criteria in Section \ref{sec_Zs-equiv_criteria}. Below we explain the dynamics of the $\ZZ_S$-equivalence problem and we outline the rest of the section.

Suppose that the non-degenerate, integral quadratic forms $Q_1$ and $Q_2$ in $d$ variables are $\ZZ_S$-equivalent. We want to bound, for any $\nu \in S$, the $\nu$-norm of a $\gamma_0 \in GL(d,\ZZ_S)$ taking $Q_1$ to $Q_2$, so we'll think the $Q_i$'s as $\QQ_\nu$-quadratic forms. In order to do this efficiently we'll work with the ring $\QQ_S$. For any rational quadratic form $Q$ in $d$ variables, we denote by $Q_S$ the quadratic form with coefficients in $\QQ_S$ determined by $Q$ via the diagonal embedding $\QQ \to \QQ_S$. Consider the groups $\GLS{d} = GL(d,\QQ_S)$ and the diagonal copy $\GammaS{d}$ of $GL(d,\ZZ_S)$ in $\GLS{d}$. Since $Q_1$ and $Q_2$ are $\ZZ_S$-equivalent, then  $(Q_1)_S$ and $(Q_2)_S$ are $\QQ_S$-equivalent. Thus there are $f,g \in \GLS{d}$ and a standard quadratic form $P$ on $\QQ_S^d$ such that
\[ (Q_1)_S = P \circ f \quad \text{and} \quad (Q_2)_S = P \circ g. \]
We denote by $H_S$ the group $O(P,\QQ_S)$. Note that any $g' \in \GLS{d}$ taking $(Q_1)_S$ to $(Q_2)_S$ is of the form $f\inv h g$ for some $h \in H_S$. Since $Q_1$ and $Q_2$ are $\ZZ_S$-equivalent, there is an $h \in H_S$ such that $f\inv h g$ belongs to $\GammaS{d}$. To give a dynamical reformulation of this condition, let's consider the action of $H_S$ on the homogeneous space $\LatSpaceS{d} = \GLS{d} / \GammaS{d}$. Let $\BasePointS{d}$ be the base point $\GammaS{d} / \GammaS{d}$ of $\LatSpaceS{d}$. Note that $f\inv h g$ is in $\GammaS{d}$ if and only if $hg\BasePointS{d} = g \BasePointS{d}$, so $f \BasePointS{d}$ and $g \BasePointS{d}$ are in the same $H_S$-orbit $Y$ in $\LatSpaceS{d}$. One can show that $Y$ is closed in $\LatSpaceS{d}$---see Lemma \ref{Y_Q,S_is_closed}---, so our original arithmetic problem is intimately related to the next dynamical problem.

\begin{Prob} \label{Dynamical_problem}
Given two points $y_1$ and $y_2$ in a closed $H_S$-orbit in $\LatSpaceS{d}$, bound the size of the smallest $h^\star \in H_S$ moving $y_2$ to $y_1$.
\end{Prob} 

Li and Margulis answer in  \cite[Theorem 5]{li_effective_2016} Problem \ref{Dynamical_problem} for $S= \{\infty\}$. We will extend their result to any finite set $S = \{\infty\} \cup S_f$ of places of $\QQ$ with the two statements below. To see that these are meaningful one needs to know that any closed $H_S$-orbit $Y$ in $\LatSpaceS{d}$ has finite volume with respect to its $H_S$-invariant measure $\muY$\footnote{The measure $\muY$ is unique up to multiplication by a positive constant. We pick the normalization of $\muY$ determined by the Haar measure $\Haar{H_S} = \otimes_{\nu \in S} \Haar{H_\nu}$ of $H_S$. The $\Haar{H_\nu}$'s are fixed in \eqref{Haar_H_infty} and \eqref{Haar_H_p} of Appendix \ref{app_volume_computations}.}---see Lemma \ref{Closed_implies_finite_volume}. We denote $\muY(Y)$ simply by $vol\, Y$. For any $g' \in  \GLS{d}, \nu \in S$ and $S' \subseteq S$ we define
\[T_\nu(g') = \frac{\normnu{g'_\nu}^d}{|\det g'_\nu |_\nu} \quad \text{and} \quad T_{S'}(g') = \prod_{\nu \in S'} T_\nu(g').   \]
Here are our two dynamical statements.

\begin{Prop}\label{Dynamical_statement_RR-isotropic}
Let $S = \{\infty\} \cup S_f$ be a finite set of places of $\QQ$ and let $H_S$ be the orthogonal $\QQ_S$-group of a standard quadratic form on $\QQ_S^d$ with $d \geq 3$. Suppose that $H_\infty$ is noncompact. Consider $f,g \in  \GLS{d}$ such that $f \BasePointS{d}$ and $g \BasePointS{d}$ are in a closed $H_S$-orbit $Y$ in $\BasePointS{d}$. Then there is an $h^\star \in H_S$ with  
\begin{align*}
\normi{h^\star_\infty} & < \consDynStRiso{d} p_S^{9d^3} (T_\infty(f) T_\infty(g))^{\frac{3}{2}d(d-1)+6} (T_{S_f}(f) T_{S_f}(g))^{3d^2}  (vol\, Y)^6, \\
\normp{h^\star_p} & \leq p \quad  \text{for any odd } p \in S_f, \\
\norm{h^\star_2} & \leq 4 \quad \text{if } 2 \in S_f,
\end{align*}   
such that $h^\star g \BasePointS{d} = f \BasePointS{d}$.
\end{Prop}

\begin{Prop}\label{Dynamical_statement_R-anisotropic} 
 Let $S = \{\infty \} \cup S_f$ be a finite set of places of $\QQ$ and let $H_S$ be the orthogonal $\QQ_S$-group of a standard quadratic form on $\QQ_S^d$ with $d \geq 3$. Suppose that $H_\infty$ is compact and that $H_{p_0}$ is noncompact for some $p_0$ in $S_f$. Consider $f,g \in  \GLS{d}$ such that $f \BasePointS{d}$ and $g \BasePointS{d}$ are in a closed $H_S$-orbit $Y$ in $\BasePointS{d}$. Then there is an $h^\star \in H_S$ with
\begin{align*}
 \normpo{h^\star_{p_0}} & < \consDynStRani{d} p_S^{6 d^3} (T_{p_0}(f) T_{p_0}(g))^6 (T_S(f) T_S(g))^{d(d-1)} (vol\, Y)^4, \\
\normp{h^\star_p} & \leq p \quad  \text{for any odd } p \in S_f-\{p_0\},\\
\norm{h^\star_2}_2 & \leq 4 \quad \text{if } 2 \in S_f - \{p_0\},
\end{align*}   
such that $h^\star g \BasePointS{d} = f \BasePointS{d}$. 
\end{Prop} 

\begin{Rem}
When $H_S$ is as in Proposition \ref{Dynamical_statement_R-anisotropic}, any closed $H_S$-orbit in $\LatSpaceS{d}$ is in fact compact. 
\end{Rem}

To prove propositions \ref{Dynamical_statement_RR-isotropic} and \ref{Dynamical_statement_R-anisotropic} we'll proceed as follows\footnote{There are minor technical inaccuracies in this outline that we'll fix in due time.}: Consider points $y_1, y_2$ in a closed $H_S$-orbit $Y$ in $X_{d,S}$. We want to estimate the size of an $h^\star \in H_S$ moving $y_1$ to $y_2$. With a mixing speed argument we'll find an $h' \in H_S$ sending $h' y_2$ very close to $y_1$. More precisely, we'll exhibit an explicit function $F: H_S \to [0,1]$ vanishing at $\infty$ with the next property: for any closed $H_S$-orbit $Y$ in $\LatSpaceS{d}$ and for any smooth $L^2$-functions $\varphi_1, \varphi_2$ on $Y$, there is a constant $C_{\varphi_1, \varphi_2} > 0$ such that  
\begin{equation}\label{Mixing_speed_cartoon} 
\left|\int_Y \overline{\varphi_1} (\varphi_2 \circ h\inv) \dd \muY - \frac{1}{vol\, Y} \int_Y \overline{\varphi_1} \dd \muY \int_Y \varphi_2 \dd \muY \right| \leq C_{\varphi_1, \varphi_2} F(h),
\end{equation}
for any $h \in H_S$. Suppose now that the support of $\varphi_i$ is a tiny neighborhood $\Uc_i$ of $y_i$. We'll choose $h' \in H_S$ such that
\[ C_{\varphi_1, \varphi_2} F(h') <  \frac{1}{vol\, Y} \int_Y \overline{\varphi_1} \dd \muY \int_Y \varphi_2 \dd \muY.\]
Since $\int_Y \varphi_1 (\varphi_2 \circ (h')\inv) \dd \muY$ is positive by \eqref{Mixing_speed_cartoon}, $h'\Uc_2$ meets $\Uc_1$. In other words, $h'$ moves $y_2$ near to $y_1$. Hence there is an $h^\star \in H_S$ of about the same size as $h'$ that moves $y_2$ to $y_1$.

The mixing speed \eqref{Mixing_speed_cartoon} will be deduced from fine information about the spectral decomposition of the unitary representation of $H_S$ on $L^2(Y)$. Let's make some comments about the main elements of \eqref{Mixing_speed_cartoon}. Suppose there is $\nu_0 \in S$ such that $H_{\nu_0}$ is noncompact. The function $F$ will be defined in a finite-index subgroup $H\ncc_{\nu_0}$ of $H_{\nu_0}$, not on all $H_S$. The smoothness condition of the $\varphi_i$'s is with respect to $H\ncc_{\nu_0}$. Finally, the constant $C_{\varphi_1, \varphi_2}$ is essentially the product of some Sobolev norms of $\varphi_1$ and $\varphi_2$.    

Here is the road map of the proof of Proposition \ref{Dynamical_statement_RR-isotropic} and Proposition \ref{Dynamical_statement_R-anisotropic}. The basic dynamical facts---that integral quadratic forms are associated to closed, finite-volume orbits in $\LatSpaceS{d}$---are established in Subsection \ref{subsec_dyn_lemmas}. Then, we recall the definitions and results needed from the theory of unitary representations in Subsection \ref{subsec_unitary_reps}. We will  also explain why a mixing speed for a measure-preserving dynamical system $H' \curvearrowright (Y', \mu')$ is equivalent to an estimate of the decay of the coefficients of $L^2_0(Y', \mu')$. In the context that interests us, the decay speed comes from profound results on automorphic representations. The basic definitions of automorphic representations are given in Subsection \ref{subsec_automorphic_reps}, and the technical results about them we need are stated in Subsection \ref{subsec_spectral_gap}. These will be translated in Subsection \ref{subsec_mixing_speed} into explicit mixing rates for the action of $H_S$ on any closed $H_S$-orbit $Y$ in $\LatSpaceS{d}$. By then, we'll almost be ready to prove propositions \ref{Dynamical_statement_RR-isotropic} and \ref{Dynamical_statement_R-anisotropic}. Remember that we'll apply the mixing speed to functions $\varphi_1, \varphi_2$ supported respectively on small neighborhoods of $y_1 = g \BasePointS{d}$ and $y_2 = f \BasePointS{d}$. Thanks to the results of subsections \ref{subsec_inj_radius} and \ref{subsec_bump_functions} we will be able to replace $C_{\varphi_1, \varphi_2}$ in \eqref{Mixing_speed_cartoon} by something in terms of the volume of $Y$ and the matrices $f, g$. Finally, we prove the main propositions in Subsection \ref{subsec_proofs_dyn_st}.
  
	\subsection{Basic dynamical lemmas}\label{subsec_dyn_lemmas}

Here we establish the link between integral quadratic forms and closed orbits in the space of lattices of $\QQ_S^d$. 

Let $Q$ be a non-degenerate integral quadratic form in $d$ variables and let $S = \{\infty\} \cup S_f$ be a finite set of places of $\QQ$. We introduce a notation for the orbit in $\LatSpaceS{d}$ associated to $Q$. Let $P$ be the standard quadratic form on $\QQ_S^d$ that is $\QQ_S$-equivalent to $Q_S$. Consider any $g \in \GLS{d}$ such that $Q_S = P \circ g$. We denote by $\OrbitS{Q}$ the set $O(P,\QQ_S) g \BasePointS{d}$. 

\begin{Lem}\label{Y_Q,S_is_closed}
Let $Q$ be a non-degenerate integral quadratic form in $d \geq 2$ variables. Then $\OrbitS{Q}$ is closed in $\LatSpaceS{d}$ for any finite set $S = \{\infty\} \cup S_f$ of places of $\QQ$. 
\end{Lem}

\begin{proof}
We write $Q_S = P \circ g$ with $g \in  \GLS{d}$ and $P$ a standard quadratic form on $\QQ_S^d$. Let $H_S = O(P,\QQ_S)$. Suppose that $h_n g \BasePointS{d} \underset{n \to \infty}{\longrightarrow} f \BasePointS{d}$ for some $h_n \in H_S$ and some $f \in  \GLS{d}$. There are $\gamma_n \in \Gamma_{d,S}$ such that $h_n g \gamma_n \to f$, so 
\[P \circ f = \lim_{n\to \infty} P \circ (h_n g \gamma_n) = \lim_{n\to \infty} Q_S \circ \gamma_n. \]
The diagonal copy $M_d(\ZZ_S)^\Delta$ in $M_d(\QQ_S)$ of $M_d(\ZZ_S)$ is discrete and closed. Since each $b_{Q\circ \gamma_n}$ is in $M_d(\ZZ_S)^\Delta$, then the matrix of $P \circ f$ is as well and $P \circ f = Q_S \circ \gamma_n$ for any big enough $n$. Since $Q_S = P \circ g$, we have $f = hg\gamma_n$ for some $h\in H_S$ and some big enough $n$. In other words, $f \BasePointS{d}$ is in $  \OrbitS{Q}$. 
\end{proof}

Consider again a non-degenerate, integral quadratic form $Q$ in $d$ variables. Lemma \ref{Y_Q,S_is_closed} implies that $O(Q_S, \QQ_S) \BasePointS{d}$ is closed in $\LatSpaceS{d}$\footnote{Because $O(Q_S, \QQ_S) \BasePointS{d} = g\inv \OrbitS{Q}$ for any $g \in \GLS{d}$ taking $P$---the standard quadratic form on $\QQ_S^d$ that is $\QQ_S$-equivalent to $Q_S$---to $Q_S$.}. The next lemma is a partial converse. It will be used in the proofs of  Proposition \ref{Mixing_speed_R-isotropic} and Proposition \ref{Mixing_speed_R-anisotropic}.

\begin{Lem}\label{Closed_implies_integral}
Consider a finite set  $S = \{\infty\} \cup S_f$ of places of $\QQ$ and let $R$ be a non-degenerate quadratic form in $d \geq 3$ variables with coefficients in $\QQ_S$. If $R$ is $\QQ_S$-isotropic and $SO(R, \QQ_S) \BasePointS{d}$ is closed in $\LatSpaceS{d}$, then $SO(R,\QQ_S) = SO(Q_S, \QQ_S)$ for a non-degenerate integral quadratic form $Q$ in $d$ variables. 
\end{Lem}

In the proof of Lemma \ref{Closed_implies_integral} we use the next well-known result. 

\begin{Lem}\label{Closed_implies_finite_volume}
Consider a finite set  $S = \{\infty\} \cup S_f$ of places of $\QQ$ and let $H_S$ be the orthogonal $\QQ_S$-group of a non-degenerate quadratic form in $d \geq 3$ variables with coefficients in $\QQ_S$. Any closed $H_S$-orbit $Y$ in $\LatSpaceS{d}$ admits a finite $H_S$-invariant measure $\muY$. 
\end{Lem}

In fact, Lemma \ref{Closed_implies_finite_volume} is valid replacing $H_S$ by any $H'_S = \prod_{\nu \in S} H'_\nu$, where $H'_\nu$ is a semisimple closed subgroup of $GL(d,\QQ_\nu)$. The case $S = \{\infty\}$ is due to Dani and Margulis---see \cite[Proposition 3.1]{benoist_arithmeticity_2020} for a proof. We are ready to prove Lemma \ref{Closed_implies_integral}

\begin{proof}[Proof of Lemma \ref{Closed_implies_integral}]
For $\nu \in S$, let $R_\nu$ be the component of $R$ in $\QQ_\nu$. Since $R$ is $\QQ_S$-isotropic, then $R_{\nu_0}$ is isotropic for some $\nu_0 \in S$. We'll prove first that $R_{\nu_0}$ has an integral multiple $Q$. Let $H_S = SO(R, \QQ_S)$, which is semisimple since $d \geq 3$. Then $\Lambda_S = \Gamma_{d,S} \cap H_S$ is a lattice in $H_S$ by Lemma \ref{Closed_implies_finite_volume}. For $S_0 \subset S$, let $\Lambda_{S_0}$ be the projection of $\Lambda_S$ to $G_{S_0,d}$. If we show that $\Lambda_{\nu_0}$---which is contained in $SO(R_{\nu_0},\QQ)$---is Zariski-dense, so $R_{\nu_0}$ has a nontrivial integral multiple $Q$. Let $T$ be the subset of $\nu \in S$ for which $R_\nu$ is isotropic. Note that $\Lambda_T$ is still a lattice in $H_T$ because $H_{S-T}$ is compact. $H_T$ is semisimple, Zariski-connected and has no compact factors, hence $\Lambda_T$ is Zariski-dense in $H_T$ by Borel's Density Theorem---see \cite[p. 41 and Remark in p. 42]{zimmer_ergodic_1984}. $\Lambda_T$ projects to $\Lambda_{\nu_0}$, so this last one is Zariski-dense  in $H_{\nu_0}$. 

Let $S' = S - \{\nu_0\}$. To show that $H_S = SO(Q_S, \QQ_S)$ it suffices to prove that $H_{S'}$ contains a neighborhood of the identity in $SO(Q,\QQ_{S'})$. Let $\Delta_{S'}$ be the diagonal copy of $SO(Q,\ZZ_S)$ in $G_{d, S'}$. Since $SO(Q,\QQ_{\nu_0})$ is noncompact, by the Strong Approximation Theorem \footnote{See \cite[Theorem 7.12]{platonov_algebraic_1994}}the closure---with respect to the analytic topology---of $\Delta_{S'}$ is a clopen subgroup $U_{S'}$ of $SO(Q,\QQ_{S'})$. Write $ \GLS{d} = G_{\nu_0} \times G_{d, S'}$. Note that
\[ (1 \times \Delta_{S'}) \BasePointS{d} = (SO(Q,\ZZ_S) \times 1) \BasePointS{d} \subset H_S \BasePointS{d},\]
hence $(1 \times U_{S'}) \BasePointS{d}$ is also contained in $H_S \BasePointS{d}$, since this last is closed in $\BasePointS{d}$. This implies also that there is a neighborhood of the identity $W_S = \prod_{\nu \in S} W_\nu$ in $ \GLS{d}$ such that $w \mapsto w \BasePointS{d}$ is an homeomorphism $W_S \to W_S \BasePointS{d}$ and $(W_S \BasePointS{d}) \cap (H_S \BasePointS{d}) = (W_S \cap H_S) \BasePointS{d}$. Then $H_{S'}$ contains $U_{S'} \cap W_{S'}$.

\end{proof}

We close this subsection with a reformulation of Lemma \ref{Closed_implies_integral} in terms of the orbits $\OrbitS{Q}$. 

\begin{Cor}
Consider a finite set $S = \{\infty\} \cup S_f$ of places of $\QQ$ and $d \geq 3$. Let $H_S$ be the orthogonal $\QQ_S$-group of a $\QQ_S$-isotropic standard quadratic form on $\QQ_S^d$. Then any closed $H_S$-orbit in $\LatSpaceS{d}$ is of the form $\OrbitS{Q}$ for some integral quadratic form $Q$.
\end{Cor}

\begin{proof}
Let $Y$ be a closed $H_S$-orbit in $\LatSpaceS{d}$ and take $g \BasePointS{d} \in Y$. Consider $R = P \circ g$. The set $g\inv Y = O(R,\QQ_S)\BasePointS{d}$ is also closed in $\LatSpaceS{d}$. Since $R$ is isotropic, Lemma \ref{Closed_implies_integral} tells us that $O(R, \QQ_S) = O(Q_S,\QQ_S)$ for some integral quadratic form $Q$, so $Y = \OrbitS{Q}$. 
\end{proof}

	\subsection{Background on unitary representations}\label{subsec_unitary_reps}
In this section we recall estimates of the decay of coefficients of unitary representations of $SL(2,\QQ_\nu)$ satisfying an integrability condition. These estimates will be used in the mixing speed argument outlined in \eqref{Mixing_speed_cartoon}. This section is organized as follows: We start by recalling the definitions we need from the theory of unitary representations. Then we explain briefly the reformulation of the mixing property of a measure-preserving dynamical system in terms of the associated regular representation. Finally, we state the decay estimates: Proposition \ref{Decay_smooth_vectors_almost_L2m_explicit} for $SL(2,\RR)$ and Proposition \ref{Decay_speed_K_n-inv-vectors_explicit} for $SL(2,\QQ_p)$. The proofs are postponed to Appendix \ref{app_Decay_coefficients}.

Let $\Hc$ be a Hilbert space---always assumed to be complex. We denote by $U(\Hc)$ the group of unitary transformations of $\Hc$. A unitary representation of a locally compact group $G$ on $\Hc$ is a group morphism $\pi: G \to U(\Hc)$ such that $g \mapsto \pi(g)v$ is continuous for any $v \in \Hc$. Let $\pi$ be a unitary representation of $G$. We'll often denote by $\Hc_\pi$ the Hilbert space of $\pi$. For $v,w \in \Hc_\pi$, the map $g \mapsto \scalar{\pi(g)v}{w}$ is the \textit{coefficient of $v$ and $w$}. When $v = w$ we call it the \textit{diagonal coefficient of $v$}. The unitary representations $\pi_1$ and $\pi_2$ of $G$ are \textit{unitary equivalent} if there is a $G$-equivariant bijective isometry $\Hc_{\pi_1} \to \Hc_{\pi_2}$. We say that $\pi$ is irreducible if $0$ and $\Hc_\pi$ are the only $G$-invariant closed subspaces of $\Hc_\pi$. The set of equivalence classes of unitary representations of $G$, denoted by $\widehat{G}$, is known as the \textit{unitary dual of $G$}. We denote by $[\pi]$ the unitary equivalence class $\pi$.  A unitary representation $\sigma$ of $G$ is weakly contained in $\pi$ if any diagonal coefficient of $\sigma$ can be approximated uniformly on compact subsets by finite sums of diagonal coefficients of $\pi$. The support $supp\, \pi$ of $\pi$ consists of the $[\sigma] \in \widehat{G}$ weakly contained in $\pi$. 

Here is the most important example of unitary representation for us: Let $Y$ be a topological space endowed with a finite Borel measure $\mu$. Suppose that $\alpha$ is a measure-preserving action of a locally compact group $G$ on $Y$. The formula
\[ \pi_\alpha (g) f(y) = f (\alpha(g\inv)y) \]
defines a unitary representation $\pi_\alpha$ of $G$ on $L^2(Y, \mu)$. Recall that $\alpha$ is mixing if and only if for any $\varphi, \psi \in L^2(Y)$,
\[ \lim_{g \to \infty} \scalar{\pi_\alpha (g) \varphi}{\psi} = \frac{1}{\mu(Y)} \int_{Y} \varphi \dd\mu \int_{Y} \overline{\psi} \dd\mu.\]
The fact that $\alpha$ is mixing can be reformulated in terms of certain coefficients of $\pi_\alpha$. We denote by $\pi_\alpha\ncc$ the restriction of $\pi_\alpha$ to 
\[L^2_0(Y) = \left\{ f \in L^2(Y) \mid \int_Y f \dd \mu = 0 \right\}.  \]
The orthogonal projection of $\varphi \in L^2(Y)$ to $L^2_0(Y)$ is $\varphi_0 = \varphi - \frac{1}{\mu(Y)}\int_Y \varphi \dd \mu$ and
\[ \scalar{\pi_\alpha\ncc (g) \varphi_0}{\psi_0} =
\scalar{\pi_\alpha (g) \varphi}{\psi} - \frac{1}{\mu(Y)} \int_{Y} \varphi \dd\mu \int_{Y} \overline{\psi} \dd\mu.\]
Thus $\alpha$ is mixing if an only if any coefficient of $\pi\ncc_\alpha$ vanishes at $\infty$. Moreover, a decay speed for the coefficients of $\pi_\alpha\ncc$ is equivalent to a mixing speed for the action $\alpha$. 

In the situation that interests us, $G$ will be an orthogonal group $O(P,\QQ_S)$ and $Y$ will be a closed $O(P,\QQ_S)$-orbit in the space of lattices of $\QQ_S^d$. The next two propositions are the first ingredient to deduce a mixing speed for $O(P,\QQ_S) \curvearrowright (Y, \muY)$. We need various definitions for the statements. Let $k \in [2, \infty)$. A unitary representation $\pi$ of a locally compact group $G$ is almost $L^k$ if there is a dense subset $\mathscr{D}$ of $\Hc_\pi$ such that the coefficient of any two vectors in $\mathscr{D}$ is an $L^{k + \varepsilon}$-function on $G$ for any $\varepsilon > 0$---see the article \cite[p. 125]{shalom_rigidity_2000} of Y. Shalom for more on this concept. A unitary representation of $G$ is \textit{tempered} if and only if it is weakly contained in $L^2(G)$. Tempered and almost $L^2$ unitary representations are the same thing when $G$ is the group of $\QQ_\nu$-points of a semisimple linear $\QQ_\nu$-group---see \cite[Theorem 1, Theorem 2]{cowling_almost_1988}.  

We need additional notation to state the estimate for $SL(2,\RR)$. For any $t \in \RR$ we denote
\[a_{\infty, t} = diag(e^{\frac{t}{2}}, e^{-\frac{t}{2}}). \]
Consider 
\[A^+_\infty = \{a_{\infty,t} \mid t \geq 0 \} \quad \text{and} \quad K_{2,\infty} = SO(2,\RR).\]
Let $\pi$ be a unitary representation of $SL(2,\RR)$. A vector $v \in \Hc_\pi$ is $K_{2,\infty}$-smooth if and only if the map $K_{2,\infty} \to \Hc_\pi, k \mapsto \pi(k)v$ is smooth. Consider 
\[ \Zc = \begin{pmatrix}
0 & -1 \\
1 & 0
\end{pmatrix} \]
as element of the Lie algebra of $K_{2,\infty}$. If $v \in \Hc_\pi$ is $K_{2,\infty}$-smooth, we define its  \textit{Sobolev norm} as 
\begin{equation}\label{Sobolev_norm}
\norm{v}_{\Zc} = (\norm{v}^2 + \norm{\pi(\Zc)v}^2)^\frac{1}{2},
\end{equation}
where
\[\pi(\Zc) v =  \frac{\dd}{\dd t} \Big|_{t=0} \pi(e^{t\Zc})v.\]

\begin{Prop}\label{Decay_smooth_vectors_almost_L2m_explicit}
Let $\pi$ be an almost $L^{2m}$ unitary representation of $SL(2,\RR)$, with $m$ a positive integer. For any $K_{2,\infty}$-smooth vectors $v, w \in \mathcal{H}_\pi$ and for any $t \geq 0$ we have
\[|\scalar{\pi(a_{\infty,t})v}{w}| \leq e^{-\frac{t}{3m}} (5 \consDecayHarishReal^{\frac{1}{m}} \norm{v}_{\Zc}  \norm{w}_{\Zc}). \]

\end{Prop}

We pass to the decay estimate for $SL(2,\QQ_p)$. For any $m \in \ZZ$ we denote
\[ a_{p,m} = diag(p^{-m}, p^{m}). \]
Consider 
\[A^+_p = \{a_{p,m} \mid m \in \mathbb{N} \}, \quad K_{2,p} = SL(2,\ZZ_p),\]
and for any positive integer $n$ let
 \[K_{2,p}(n) = ker(K_{2,p} \to SL(2,\ZZ/ p^n \ZZ)).\]

\begin{Prop}\label{Decay_speed_K_n-inv-vectors_explicit}
Let $\pi$ be a tempered unitary representation of $SL(2,\QQ_p)$. Suppose that $v_1,v_2 \in \mathcal{H}_\pi$ are respectively $K_{2,p}(n_1)$ and $K_{2,p}(n_2)$-invariant. Then
\[|\scalar{\pi(a_{p,m})v_1}{v_2}| \leq p^{-\frac{m}{2}} (10 p^{\frac{3}{2}(n_1 + n_2)} \norm{v_1} \, \norm{v_2} ),\]
for any $m \geq 1$.
\end{Prop}

The proofs of Proposition \ref{Decay_smooth_vectors_almost_L2m_explicit} and Proposition \ref{Decay_speed_K_n-inv-vectors_explicit} are given in Appendix \ref{app_Decay_coefficients}.

	 \subsection{Automorphic representations of $\QQ$-groups}\label{subsec_automorphic_reps}
In this short subsection we define the notion of automorphic representation for linear algebraic groups defined over $\QQ$. Lemma \ref{aut_reps_via_congruence_subgroups} shows how to obtain examples of such representations from congruence subgroups of $S$-arithmetic groups. Finally, we cite a result saying that automorphic representations behave well with respect to restriction to $\QQ$-subgroups.	 
	 
Let $S = \{\infty\} \cup S_f$ be a finite set of places of $\QQ$. Recall that $\GammaS{d}$ is the diagonal embedding of $GL(d,\ZZ_S)$ in $GL(d,\QQ_S)$. For any positive integer $N$ relatively prime to $p_S$, the \textit{$N$-th principal congruence subgroup} of $\GammaS{d}$, denoted $\GammaS{d}(N)$, is the kernel of the reduction map $\GammaS{d} \to GL(d, \ZZ / N\ZZ)$. A \textit{congruence subgroup} of $\GammaS{d}$ is a subgroup that contains some principal congruence subgroup. Let $\Jne$ be a $\QQ$-subgroup of $\textbf{GL}(d)$. We denote $J_S \cap \GammaS{d}$ by $\Lambda_S$. For any $N$ as above, the $N$-th principal congruence subgroup $\Lambda_S(N)$ of $\Lambda_S$ is $\Lambda_S \cap \GammaS{d}$. The congruence subgroups of $\Lambda_S$ are those containing some $\Lambda_S(N)$. For any $\nu \in S$, we denote by $\text{supp}_\nu  L^2(J_S / \Lambda_S(N))$ the support of the unitary representation of $J_\nu$ on $L^2(J_S / \Lambda_S(N))$. It is a closed subset of $\widehat{J_\nu}$. For any place $\nu$ of $\QQ$ let $S_\nu = \{\infty, \nu\}$. We define the \textit{automorphic dual} of $J_\nu$ as
\[\aut{J_\nu} = \overline{\bigcup_{(N, p_{S_\nu}) = 1} \text{supp}_\nu L^2(J_{S_\nu} / \Lambda_{S_\nu}(N))},\]
where we take the closure with respect to the Fell topology of $\widehat{J_\nu}$. When $\Jne$ is semisimple, the class of the trivial representation of $J_\nu$ is in $\aut{J_\nu}$, since $J_{S\nu} / \Lambda_S$ has finite volume by a theorem of Borel and Harish-Chandra---see \cite[p. 50]{benoist_five_2009}. A unitary representation of $J_\nu$ is said to be \textit{automorphic} if its support is contained in $\aut{J_\nu}$. The next lemma provides many natural examples of automorphic representations.

\begin{Lem}\label{aut_reps_via_congruence_subgroups}
Let $\Jne$ be a semisimple, simply-connected $\QQ$-subgroup of $\textbf{GL}(d)$, and let $S = \{\infty\} \cup S_f$ be a finite set of places of $\QQ$. Suppose that all the simple factors of $J_{\nu_0}$ are noncompact for some $\nu_0 \in S$. Then, for any congruence subgroup $\Lambda_0$ of $J_S \cap \GammaS{d}$, the unitary representation of $J_{\nu_0}$ on $L^2(J_S / \Lambda_0)$ is automorphic. 
\end{Lem}

\begin{proof}
First, let's see that it is enough to prove the result when $\Lambda_0$ is a principal congruence subgroup of $\Lambda_S$. Let $\Lambda_1$ be a normal, finite-index subgroup of $\Lambda_0$. The group $\Lambda_0 / \Lambda_1$ acts  on the right on $J_S / \Lambda_1$, and the orbit space of this action is isomorphic, as $J_S$-spaces, to $J_S / \Lambda_0$. Hence we can identify $L^2(J_S / \Lambda_0)$ with the subspace of $\Lambda_0 / \Lambda_1$-invariant vectors of $L^2(J_S / \Lambda_1)$. Thus, if $J_{\nu_0} \curvearrowright L^2(J_S / \Lambda_1)$ is automorphic, then $J_{\nu_0} \curvearrowright  L^2(J_S/ \Lambda_0)$ is as well. 

Suppose then that $\Lambda_0 = \Lambda_S(N)$ for some $N$ relatively prime to $p_S$. Let's assume that $\nu_0$ is a prime number $q$---the case $\nu_0 = \infty$ is similar. Let $S_q = \{\infty, q\}$ and $T = S - \{q\}$. We'll show that
\[\text{supp}_{q}  L^2(J_S / \Lambda_0) = \overline{\bigcup_{n \geq 1} \text{supp}_{q} L^2(J_{S_q}/ \Lambda_{S_q}(N p_T^n)) }.\]
It suffices to show that for any $n \geq 1$, there is a $J_q$-invariant subspace $\Hc_n$ of $\Hc := L^2(J_S / \Lambda_0)$  such that $\cup_{n \geq 1} \Hc_n$ is dense in $\Hc$, and such that the unitary representations of $J_q$ on $L^2(J_{S_q} / \Lambda_{S_q}(N p_T^n))$ and $\Hc_n$  are unitary equivalent. 

For any prime number $p$ consider $K_p = GL(d,\ZZ_p), K_p^n = ker(K_p \to GL(d,\ZZ/ p^n \ZZ))$ and $U_p^n = J_p \cap K_p^n$. Suppose that $T_f = \{p_1, \ldots, p_\ell \}$. We'll denote $U^n_{p_1} \times \cdots \times U_{p_\ell}^n$ by $U_{T_f}^n$. The group $J_{S_q} \times U_{T_f}^n$ acts transitively on $J_S / \Lambda_S(N)$. Indeed, since $J_q$ has no compact simple factors, $J_q \Lambda_0$ is dense in $J_S$ by the Strong Approximation Theorem \cite[Theorem 7.2]{platonov_algebraic_1994}. Hence $J_S = (J_{S_q} \times U^n_{T_f}) \Lambda_0$ since $J_{S_q} \times U^n_{T_f}$ is open in $J_S$. Note that $(J_{S_q} \times U_{T_f}^n) \cap \Lambda_S(N) = \Lambda_{S_q}(N p_T^n)$, so we identify $J_S / \Lambda_0$ with $(J_{S_q} \times U^n_{T_f}) / \Lambda_{S_q}(N p_T^n)$. The isomorphism of $J_{S_q}$-spaces
\[J_{S_q} / \Lambda_{S_q}(N p_T^n) \simeq U_{T_f}^n \backslash (J_{S_q} \times U_{T_f}^n) / \Lambda_{S_q}(N p_T^n) \]
identifies $L^2(J_{S_q} / \Lambda_{S_q}(N p_T^n))$ with the subspace $\Hc_n$ of $U_{T_f}^n$-invariant vectors of $\Hc$. Since the sequence $(U_{T_f}^n)_{n \geq 1}$ shrinks to $\{ I_d \}$, then $\cup_{n \geq 1} \Hc_n$ is dense in $\Hc$\footnote{Let $\Cc_c(J_S / \Lambda_0)$ be the ring of continuous, $\CC$-valued functions on $J_S / \Lambda_0$ with compact support. Any $F \in \Cc_c(J_S / \Lambda_0)$ is the uniform limit, as $n\to \infty$, of the $U_{T_f}^n$-invariant functions $F_n: x \mapsto \int_{U_{T_f}^n} F(ux) \dd u$. This proves the claim since $\Cc_c(J_S / \Lambda_0)$ is dense in $L^2(J_S / \Lambda_0)$.}. 
\end{proof}

The next theorem tells us that automorphic representations of $J_\nu$ behave well when we restrict them to $\QQ$-subgroups. It was originally obtained by M. Burger and P. Sarnak for $\nu = \infty$, and later L. Clozel and E. Ullmo extended it to any $\nu$---see \cite[Theorem 1.1]{burger_ramanujan_1991} and \cite[Théorème 5.1]{clozel_equidistribution_2004}.

\begin{Theo}\label{BS-CU_restriction}
Let $\Jne' \subseteq \Jne$ be simply connected, semisimple $\QQ$-groups and let $\nu$ be a place of $\QQ$. The restriction of any automorphic representation of $\Jne_\nu$ to $\Jne'_\nu$ is automorphic. 
\end{Theo}

	\subsection{Uniform spectral gap for automorphic representations}\label{subsec_spectral_gap}
Let $R$ be a non-degenerate integral quadratic form in 3 variables. Here we present the key technical tool for the proof of our $\ZZ_S$-equivalence criteria: two spectral gap results for automorphic representations of $\textbf{O}(R)$. These will be translated in Subsection \ref{subsec_mixing_speed} into more concrete statements about mixing rates for the action of a noncompact orthogonal $\QQ_S$-group on a closed orbit in $\LatSpaceS{d}$. The two results of this subsection are written in terms of $\textbf{Spin}(R)$ rather than $\textbf{O}(R)$ because the former is simply connected, and hence the statements are cleaner.

The first result is for $\RR$-isotropic quadratic forms. It follows from the Jacquet-Langlands Correspondence---see \cite[Theorem 3.4, p. 163]{lubotzky_discrete_1994} or \cite[Theorems 10.1 and 10.2]{gelbart_automorphic_1975}---and the approximations to the Ramanujan-Petersson Conjecture for $\textbf{SL}(2)$ over $\QQ$. 
 	
	\begin{Prop}\label{Spectral_gap_at_infty}
	Let $R$ be a non-degenerate, integral quadratic form in 3 variables that is $\RR$-isotropic. Any irreducible, automorphic representation of $Spin(R,\RR)$ is either trivial or almost $L^{4}$. 
	\end{Prop}
	
\begin{Rem}
Using the best bound towards the Ramanujan-Petersson Conjecture for $\textbf{GL}(2)$ over $\QQ$ currently available, namely the Kim-Sarnak bound---see \cite[Appendix 2]{kim_functoriality_2003}---, Li and Margulis state a better spectral gap \cite[Lemma 5]{li_effective_2016} than Proposition \ref{Spectral_gap_at_infty}. The reason why we content ourselves with the weaker bound above is that our proof of the decay of coefficients of almost $L^p$ representations of $SL(2,\RR)$, Lemma \ref{Decay_smooth_vectors_almost_L2m_explicit}, works only when $p$ is a positive even number.  
\end{Rem}		
	
The second result deals with $\RR$-anisotropic quadratic forms, and it follows from the Jacquet-Langlands Correspondence and a celebrated theorem of P. Deligne about holomorphic modular forms---see \cite[Theorem 6.1.2]{lubotzky_discrete_1994} for a representation-theoretic formulation.	
	
	\begin{Prop}\label{Spectral_gap_at_finite_places}
	Let $R$ be an integral quadratic form in 3 variables that is $\RR$-anisotropic and $\QQ_p$-isotropic. Any irreducible automorphic representation of $Spin(R,\QQ_p)$ is either trivial or tempered. 
	\end{Prop}

	\subsection{Uniform mixing speed for closed $H\ncc_S$-orbits}\label{subsec_mixing_speed}

Let $S = \{\infty\} \cup S_f$ be a finite set of places of $\QQ$ and let $d \geq 3$. Consider a noncompact orthogonal $\QQ_S$-group $H_S$ of a standard quadratic form on $\QQ_S^d$. In this section we give an explicit mixing rate for the action of certain one-parameter subgroups of $H_S$ on any closed $H\ncc_S$-orbit $Y'$ in $\LatSpaceS{d}$. Moreover, the mixing rate doesn't depend on $Y'$. These results are a key ingredient in the proofs of Proposition \ref{Dynamical_statement_RR-isotropic} and Proposition \ref{Dynamical_statement_R-anisotropic}. 

Let $H_\infty$ be the real orthogonal group of a non-degenerate, diagonal quadratic form of the form $x_1^2 - x_2^2 + a_3 x_3^2 + \cdots + a_d x_d^2$, and let $\rho: SL(2,\RR) \to H_\infty$ be as in Lemma \ref{Covering_SL(2)->SO(P)}. The derivative of $\rho$ at $I_2$ sends the infinitesimal rotation $\Zc = \begin{pmatrix} 0 & -1 \\ 1 & 0 \end{pmatrix}$ to
\[\Zc_{H_\infty} := \begin{pmatrix}
0 & 0 & a_3 + 1 \\
0 & 0 & -a_3 + 1 \\
-(a_3\inv + 1) & a_3\inv - 1 & 0
\end{pmatrix}.\]
Recall that for any $t \in \RR$, we denote by $a_{\infty, t}$ the matrix $diag(e^{t/2}, e^{-t/2})$. The Sobolev norm $\norm{\cdot}_{\Zc_{H_\infty}}$ was defined in \eqref{Sobolev_norm}. 
\begin{Prop}\label{Mixing_speed_R-isotropic}
Consider a finite set $S = \{\infty\} \cup S_f$ of places of $\QQ$ and $d \geq 3$. Let $H_S$ be the orthogonal $\QQ_S$-group of a standard quadratic form on $\QQ_S^d$ with $H_\infty$ noncompact, and let $\rho: SL(2,\RR) \to H\ncc_\infty$ be as in Lemma \ref{Covering_SL(2)->SO(P)}. Suppose that $Y'$ is a closed $H\ncc_S$-orbit in $\LatSpaceS{d}$. For any $H\ncc_\infty$-smooth $L^2$-functions $\varphi_1, \varphi_2$ on $Y'$ and for any $t \geq 0$ we have 
\begin{equation}\label{MS_R-isotropic}
\left| \int_{Y'} (\varphi_1 \circ \rho(a_{\infty,-t})) \overline{\varphi_2} \dd \mu_{Y'} - 
\frac{\int_{Y'} \varphi_1 \dd\mu_{Y'} 
\int_{Y'} \overline{\varphi_2} \dd\mu_{Y'}}{vol\,Y'}  
\right| \leq 
(\consMixingReal  
\norm{\varphi_1}_{\mathcal{Z}_{H_\infty}} \norm{\varphi_2}_{\mathcal{Z}_{H_\infty}}) e^{-t/6}.
\end{equation} 
\end{Prop}

\begin{proof}
Let $\pi$ be the unitary representation of $H\ncc_\infty$ on $L^2_0(Y')$, and let $V = \RR e_1 \oplus \RR e_2 \oplus \RR e_3$. Recall that $H_S = O(P,\QQ_S)$ for some standard quadratic form $P = (P_\nu)_{\nu \in S}$ on $\QQ_S^d$. Let $R$ be the restriction of $P_\infty$ to $V$. We denote by $H_\infty^V$ the image of the morphism $h \mapsto h \oplus I_{d-3}$ from $O(R,\RR)$ to $H_\infty$. The image of $\rho:SL(2,\RR) \to H\ncc_\infty$ is $H^{V}_\infty \cap H_\infty\ncc$---see Lemma \ref{Covering_SL(2)->SO(P)}---, which we'll denote as $H^{V\circ}_\infty$. First let's see that it's enough to show that the restriction of $\pi$ to $H^{V\circ}_\infty$---which we'll denote by $\pi_0$---is almost $L^4$. If that's the case, then $\pi_0 \circ \rho$ is an almost $L^4$ unitary representation of $SL(2,\RR)$, since $\rho$ is a finite covering from $SL(2,\RR)$ to $H^{V\circ}_\infty$. The orthogonal projection of $\varphi_i$ to $L^2_0(Y')$ is 
\[ \psi_i = \varphi_i - \frac{1}{vol\, Y'} \int_{Y'} \varphi_i \dd \mu_{Y'}.\]
Note that the left-hand side of \eqref{MS_R-isotropic} is equal to $|\scalar{\pi_0 \circ \rho (a_{\infty, t}) \psi_1}{\psi_2} |$, which by Proposition \ref{Decay_smooth_vectors_almost_L2m_explicit} is less or equal than
\begin{align*}
(5 \Dc_1^{1/2} \norm{\psi_1}_{\Zc} \norm{\psi_2}_\Zc) e^{-t/6} 
& \leq (5 \Dc_1^{1/2} \norm{\varphi_1}_{\Zc} \norm{\varphi_2}_\Zc) e^{-t/6} \\
& = (5 \Dc_1^{1/2} \norm{\varphi_1}_{\Zc_{H_\infty}} \norm{\varphi_2}_{\Zc_{H_\infty}}) e^{-t/6},
\end{align*}
so we are done. 

It remains to show that $\pi_0$ is almost $L^4$. We'll work with a unitary representation equivalent to $\pi_0$. Consider $g \in \GLS{d}$ such that $g \BasePointS{d}$ is in $Y'$. Since $H_S$ is noncompact and $g\inv Y' = g\inv H_S g \BasePointS{d}$ is closed in $\LatSpaceS{d}$, by Lemma \ref{Closed_implies_integral} there is an integral quadratic form $Q$ such that $g\inv H_S g = O(Q_S, \QQ_S)$. Let $\Phi:H_S \to O(Q_S,\QQ_S)$ be the conjugation by $g\inv$. The left multiplication by $g\inv$ is a $\Phi$-equivariant map $Y' \to O(Q_S,\QQ_S)\ncc \BasePointS{d}$, hence $\pi_0$ and the unitary representation $\pi_1$ of $\Phi(H_\infty^{V\circ}) = SO(Q,\RR)^{W_\infty \circ}$ on $L^2_0(SO(Q_S,\QQ_S)\ncc \BasePointS{d})$ are unitary equivalent, where $W_\infty = g_\infty \inv V$. We'll show that $\pi_1$ is almost $L^4$. We can assume that $W_\infty$ is a rational subspace of $\RR^d$. Indeed, we can replace $g$ by $hg$ for any $h \in H_S$, changing $W_\infty$ by $h_\infty \inv W_\infty$. Using Witt's Theorem---see \cite[p. 58]{serre_cours_1995}---we choose $h_\infty \in H_\infty \ncc$ such that $h_\infty\inv W_\infty$ is rational. 

 The restriction $R$ of $Q$ to $W_\infty$ is a rational quadratic form because $W_\infty$ is rational. Let $\iota: \textbf{SO}(R) \to \textbf{SO} (Q)$ be the natural morphism of $\QQ$-groups that sends $SO(R,k)$ to $SO(Q,k)^{W_k}$ for any field extension $k$ of $\QQ$. We denote respectively by $\Jne$ and $\Lne$ the $\QQ$-groups $\textbf{Spin}(R)$ and $\textbf{Spin}(Q)$, and let $\Qc: \Lne \to \textbf{SO}(Q)$ and $\Rc: \Jne \to \textbf{SO}(R)$ be the covering maps. The morphism $\iota \circ \Rc$ lifts to $\widetilde{\iota}:\Jne \to \Lne$, so we have the commutative diagram
\begin{center}
\begin{tikzpicture}
  \matrix (m) [matrix of math nodes,row sep=3em,column sep=4em,minimum width=2em]
  {
     \Jne & \Lne \\
     \textbf{SO}(R) & \textbf{SO}(Q) \\};
  \path[-stealth]
    (m-1-1) edge node [left] {$\Rc$} (m-2-1)
            edge node [above] {$\widetilde{\iota}$} (m-1-2)
    (m-2-1)  edge node [below] {$\iota$} (m-2-2)
    (m-1-2) edge node [right] {$\Qc$} (m-2-2);
            
\end{tikzpicture} 
\end{center}  
We denote by $\Lambda$ the congruence subgroup $\Qc_S\inv (SO(Q_S,\QQ_S) \cap \Gamma_{d,S})$ of $L_S$. To see that $\pi_1$ is almost $L^4$, it suffices to show that $\sigma: J_\infty \curvearrowright L^2_0(L_S/ \Lambda)$ is almost $L^4$, because $\Rc_\infty$ has finite kernel and by the commutativity of the diagram. We see $\Jne$ as a $\QQ$-subgroup of $\Lne$ through $\widetilde{\iota}$. Note that $L_\infty \curvearrowright L^2_0(L_S/ \Lambda)$ is automorphic by Lemma \ref{aut_reps_via_congruence_subgroups}, hence its restriction $\sigma$ is automorphic by Theorem \ref{BS-CU_restriction}. Thus if we show that $\sigma$ does not weakly contain the trivial representation $1_{J_\infty}$ of $J_\infty$, Proposition \ref{Spectral_gap_at_infty} will imply that $\sigma$ is almost $L^4$. If $1_{J_\infty} \prec \sigma$, then $\sigma$ would have a nonzero $J_\infty$-invariant vector\footnote{Because $1_{J_\infty}$ is an isolated point in $\aut{J_\infty}$ by Proposition \ref{Spectral_gap_at_infty}.}, which is impossible. Indeed, if $\varphi \in L^2_0(L_S / \Lambda)$ is $J_\infty$-invariant, then $\varphi$ is $L_\infty$-invariant by the Howe-Moore phenomenon---see Lemma \ref{Vectors_fixed_by_a_unipotent_are_globally_fixed}. Notice that, as function on $L_S$, $\varphi$ is invariant on the right by $L_\infty$ (because $L_\infty$ is normal in $L_S$) and $\Lambda$. Since $L_\infty$ is noncompact, $L_\infty \Lambda$ is dense in $L_S$ by the Strong Approximation Theorem---see \cite[Theorem 7.12]{platonov_algebraic_1994}. This shows that $\varphi$ is almost surely constant, but $\int_{Y'} \varphi = 0$, so $\varphi = 0$. 
\end{proof}

For any positive integers $d, p$ and $n$, with $p$ prime, we denote $SL(d,\ZZ_p)$ and the kernel of the natural map $SL(d,\ZZ_p) \to SL(d,\ZZ / p^n \ZZ)$ respectively by $K_{d,p}$ and $K_{d,p}(n)$. 

\begin{Prop}\label{Mixing_speed_R-anisotropic}
 Consider a finite set $S = \{\infty\} \cup S_f$ of places of $\QQ$ and $d \geq 3$. Let $H_S$ be the orthogonal $\QQ_S$-group of a standard quadratic form on $\QQ_S^d$. Suppose that $H_\infty$ is compact and $H_{p_0}$ is noncompact for some $p_0 \in S_f$. Let $\rho: SL(2,\QQ_{p_0}) \to H\ncc_{p_0}$ be as in Lemma \ref{Covering_SL(2)->SO(P)}. Consider a closed $H\ncc_S$-orbit $Y'$ in $\LatSpaceS{d}$ and $L^2$-functions $\varphi_1$ and $\varphi_2$ on $Y'$ that are respectively $H\ncc_{p_0} \cap K_{d,p_0}(n_1)$ and $H\ncc_{p_0} \cap K_{d,p_0}(n_2)$-invariant. For any integer $m \geq 0$ we have
\begin{equation}\label{MS_R-anisotropic}
\left| \int_{Y'} (\varphi_1 \circ \rho(a_{p_0,-m})) \overline{\varphi_2} \dd \mu_{Y'} - 
\frac{\int_{Y'} \varphi_1 \dd\mu_{Y'} 
\int_{Y'} \overline{\varphi_2} \dd\mu_{Y'}}{vol\, Y'}  
\right| \leq 
 \left( 10 p_0^{\frac{3}{2}(n_1 + n_2+2)}  
\norm{\varphi_1}_{L^2}  \norm{\varphi_2}_{L^2}
\right) p_0^{-m/2}.
\end{equation}
 
\end{Prop}

\begin{proof}
Since this proof is very similar to the proof of Proposition \ref{Mixing_speed_R-isotropic}, we'll dwell less into details. Let $V = \QQ_{p_0} e_1 \oplus \QQ_{p_0} e_2 \oplus \QQ_{p_0} e_3$. The image of $\rho$ is $H_{p_0}^{V\circ}$---see Lemma \ref{Covering_SL(2)->SO(P)}. Let $\pi_0$ be the unitary representation of $H^{V\circ}_{p_0}$ on $L^2_0(Y')$ given by the action of $H^{V \circ}_{p_0}$ on $Y'$ by left multiplication.  We claim that it suffices to show that $\pi_0$ is tempered. If so, then $\pi_0 \circ \rho$ is a tempered unitary representation of $SL(2,\QQ_{p_0})$ since $\rho$ has finite kernel. Then \eqref{MS_R-anisotropic} is obtained by applying Proposition \ref{Decay_speed_K_n-inv-vectors_explicit} to $\pi_0$ and the orthogonal projections of $\varphi_1$ and $\varphi_2$ on $L^2_0(Y')$, which are respectively $K_{2,p_0}(n_1 + 1)$ and $K_{2,p_0}(n_2+1)$-invariant by Lemma \ref{Covering_SL(2)->SO(P)}. 

Now we show that $\pi_0$ is tempered. Choose $g \in \GLS{d}$ such that $g \BasePointS{d}$ is in $Y'$ and $W_{p_0} = g_{p_0}\inv V$ is a rational subspace of $\QQ_{p_0}^d$. Since $Y'$ is closed in $\LatSpaceS{d}$ and $H_S$ is noncompact, by Lemma \ref{Closed_implies_integral} there is an integral quadratic form $Q$ such that $g\inv H_S g = O(Q_S, \QQ_S)$. The restriction $R$ of $Q$ to $W_{p_0}$ is a rational quadratic form in $3$ variables because $W_{p_0}$ is rational. Consider again the morphism of $\QQ$-groups $\iota: \SOne(R) \to \SOne(Q)$. Note that $R$ is $\RR$-anisotropic and $\QQ_{p_0}$-isotropic\footnote{Indeed, let $P = (P_\nu)_{\nu \in S}$ be the standard quadratic form on $\QQ_S^d$ such that $H_S = O(P,\QQ_S)$. Since $P_{p_0}$ is isotropic and standard, then $P_{p_0}(x) = x_1^2 - x_2^2 + a_3 x_3^2 + \cdots$. Thus the restriction of $P_{p_0}$ to $V$ is non-degenerate and isotropic. Since $Q = P_{p_0} \circ g_{p_0}$ and $R$ is the restriction of $Q$ to $W_{p_0} = g_{p_0}\inv V$, then $R$ is $\QQ_{p_0}$-isotropic. Similarly, $\iota_\infty$ is an isomorphism from $SO(R,\RR)$ to  $SO(Q,\RR)^{W_\infty}$ for some 3-dimensional linear subspace $W_\infty$ of $\RR^d$. By assumption, $P_\infty$ is $\RR$-anisotropic, so $Q$ is as well. It follows that $R$ is also $\RR$-anisotropic.}. We denote again $\Spinne(R)$ and $\Spinne(Q)$ by $\Jne$ and $\Lne$. Let $\Rc: \Jne \to \SOne(R)$ and $\Qc: \Lne \to \SOne(Q)$ be the covering maps. Consider the congruence lattice $\Lambda = \Qc_S\inv (SO(Q_S, \QQ_S) \cap \Gamma_{d,S})$ of $L_S$. As before, it suffices to show that $\sigma: J_{p_0} \curvearrowright L^2_0(L_S / \Lambda)$ is tempered. Note that $\sigma$ is automorphic by Lemma \ref{aut_reps_via_congruence_subgroups} and Theorem \ref{BS-CU_restriction}. The irreducible automorphic unitary representations of $J_{p_0}$ are either tempered or trivial by Proposition \ref{Spectral_gap_at_finite_places}. Thus we only have to see that $L^2_0(L_S / \Lambda)$ doesn't have $J_{p_0}$-invariant vectors. Again, since $J_{p_0}$ and $L_{p_0}$ are noncompact, this follows from the Howe-Moore phenomenon (Lemma \ref{Vectors_fixed_by_a_unipotent_are_globally_fixed}) and the Strong Approximation Theorem.   
\end{proof}

	\subsection{Injectivity radius on $\LatSpaceS{d}$}\label{subsec_inj_radius}
The goal of this section is to give, for any $g \in G_{d,S}$, a neighborhood $\Bcal_S^g$ of the identity in $G_{d,S}$ such that the map $\Bcal_S^g \to \LatSpaceS{d}, b \mapsto b g \BasePointS{d}$ is injective. This will allow us to compute on $G_{d,S}$ integrals of functions on $\LatSpaceS{d}$ with small support.

	For any $r > 0$ we define
\[G_{\infty,d}(r) = \{ g_\infty \in G_{\infty,d}(r) \mid \normi{g_\infty - I_d} < r \text{ and } \normi{g_\infty \inv - I_d} < r \}, \]
and
\[G_{d,p}(r) = \{ g_p \in G_{d,p} \mid \normp{g_p - I_d} \leq r \text{ and } \normp{g_p\inv - I_d} \leq r \}. \]
 For $g \in G_{d,S}$ and $\nu \in S$ we define
\[ r_\nu(g) = T_\nu \inv (g) =  \frac{| \det g_\nu |_\nu}{\normnu{g_\nu}^d}, \] 
and 
\begin{equation} \label{inj_set}
\Bcal_S^g = G_{\infty,d} \left( \frac{r_\infty(g)}{3d^2 \cdot d!} \right) \times \prod_{p \in S_f} G_{d,p}(r_p(g)).
\end{equation}	
	
	\begin{Lem}\label{Injectivity_radius_X_S}
The map $\Bcal_S^g \to \LatSpaceS{d}$, $f \mapsto fg \BasePointS{d}$ is injective for any $g \in G_{d,S}$. 
\end{Lem}

\begin{proof}
We have to show that 
\[ (g\inv (\Bcal_S^g)\inv \Bcal_S^g g) \cap \Gamma_{d,S} = \{ I_d\} \]
for any $g \in G_{d,S}$. Suppose that $f,h \in \Bcal_S^g$ and $\gamma = (\gamma_0, \ldots, \gamma_0) \in \GammaS{d}$ verify $\gamma = g\inv f\inv h g$. It suffices to prove that the matrix $\gamma_0 - I_d$ in $GL(d,\ZZ_S)$ has integral coefficients and $\normi{\gamma_0 - I_d} < 1$. Take $p \in S_f$. Since $f_p$ and $h_p$ are in $G_{p,d}(r_p(g))$ and $r_p(g) < 1$, then $f_p\inv h_p$ is also in $G_{d,p}(r_p(g))$. We have
\begin{align*}
\normp{\gamma_0 - I_d} &= \normp{g_p\inv (f_p\inv h_p - I_d) g_p} \\
						&\leq \normp{g_p\inv} \normp{g_p}  \normp{f_p\inv h_p - I_d}   \\
						& \leq \frac{\normp{g_p}^{d}}{|\det g_p|_p} \cdot  r_p(g)   = 1,
\end{align*}
and this holds for any $p \in S_f$, so $\gamma - I_d$ has integral coefficients. The inequality for the real coordinate follows the same lines. We have
\begin{align*}
\normi{\gamma_0 - I_d} & = \normi{g_\infty\inv( f_\infty\inv h_\infty - I_d) g_\infty } \\
							& \leq d^2 \normi{g_\infty\inv} \normi{g}  \normi{f_\infty\inv h_\infty - I_d} \\
							& \leq \frac{d \cdot d!}{r_\infty(g)} (\normi{f_\infty \inv h_\infty- f_\infty\inv} + \normi{f_\infty\inv - I_d}) \\
							& < \frac{d^2 \cdot d!}{r_\infty(g)} ( \normi{f\inv_\infty} \normi{h_\infty - I_d} + \normi{f\inv_\infty-I_d}) \\
							& <\frac{d^2 \cdot d!}{r_\infty(g)} \cdot \frac{r_\infty(g)}{3d^2 \cdot d!} \left( \frac{r_\infty(g)}{3d^2 \cdot d!} + 2 \right) \\
							& \leq \frac{1}{3} \left( \frac{1}{3d^2} + 2 \right) < 1,  
\end{align*} 
which completes the proof.
\end{proof}

	\subsection{Bump functions}\label{subsec_bump_functions}
Consider a finite set $S = \{\infty\} \cup S_f$ of places of $\QQ$, the orthogonal $\QQ_S$-group $H_S$ of a standard quadratic form $P = (P_\nu)_{\nu \in S}$ on $\QQ_S^d$, and $g \in \GLS{d}$. Suppose that $Y' = H\ncc_S g \BasePointS{d}$ is closed in $\LatSpaceS{d}$. The goal of this section is to construct an $H_\infty$-smooth function $\varphi_g: Y' \to \RR$ supported on a small neighborhood of $g \BasePointS{d}$ in $Y'$, and to estimate a Sobolev-type norm of $\varphi_g$. The estimates are used in the proof of Proposition \ref{Dynamical_statement_RR-isotropic} given in Section \ref{subsec_proofs_dyn_st}.

Let $\Bcal^g_S \subset G_{d,S}$ be as \eqref{inj_set}. We define $r_g = \frac{r_\infty(g)}{3d^2 \cdot d!}$ and
\begin{equation}\label{vecindad}
\Uc^g = (\mathscr{B}^g_S \cap H_S) g \BasePointS{d}
\end{equation}
 Consider the smooth function $\psi_{r_g}:H\ncc_\infty \to [0, \infty)$ as in Lemma \ref{Smooth_bump_functions} and define $\varphi_g: Y' \to [0, \infty)$ as
\[ \varphi_g(y) = \begin{cases}
\psi_{r_g}(b_\infty) & \text{if } y = bg \BasePointS{d} \text{ with } b \in H_S \cap \Bcal^g_S, \\
0 & \text{if } y \in Y' - \Uc^g.
\end{cases}\]
The map $\varphi_g$ is well-defined since $\mathscr{B}^g_S \to \LatSpaceS{d}, b \mapsto bg \BasePointS{d}$ is injective by Lemma \ref{Injectivity_radius_X_S}. Recall that when $H_\infty$ is noncompact, we defined in Section \ref{subsec_mixing_speed} a smooth morphism $\rho_{H_\infty}: SL(2,\RR) \to H_\infty$. We denote by $\Zc_{H_\infty}$ the image in $\hgot_\infty = Lie(H_\infty)$ of the infinitesimal rotation
\[ \begin{pmatrix}
0 & -1 \\
1 & 0
\end{pmatrix} \in \mathfrak{sl}(2,\RR)\]
under the derivative of $\rho_{H_\infty}$ at the identity. Let $\varphi: Y' \to \RR$ be an $H_\infty$-smooth map. Recall that for any $\Zc \in \hgot_\infty$, $\norm{\varphi}_{\Zc}$ stands for the Sobolev norm
\[(\norm{\varphi}_{L^2(Y')}^2 + \norm{\Zc(\varphi)}_{L^2(Y')}^2)^{\frac{1}{2}}.\]
The next lemma gathers some properties of $\varphi_g$ that we'll use in the proof of Proposition \ref{Dynamical_statement_RR-isotropic}.
 	
\begin{Lem}\label{S-adic_smooth_bump_functions}
Consider a finite set $S= \{\infty\} \cup S_f$ of places of $\QQ$ and $d \geq 3$. Let $H_S$ be the orthogonal $\QQ_S$-group of a standard quadratic form on $\QQ_S^d$. Suppose that $H_\infty$ is noncompact. Take $g \in \GLS{d}$ such that $Y' = H_S\ncc g \BasePointS{d}$ is closed. The function $\varphi_g: Y' \to [0,\infty)$ has support in $\Uc^g$, it is $H\ncc_\infty$-smooth,
\[\norm{\varphi_g}_{L^1(Y')} = (p_S^{-3} r_{S_f}(g))^{\frac{1}{2}d(d-1)} < 1, \]
and 
\[\norm{\varphi_g}_{\Zc_{H_\infty}} \leq \consSmoothBump{d} r_\infty(g)^{-(\frac{1}{4}d(d-1) + 1)}. \]
\end{Lem}	

\begin{proof}
We'll use freely the properties of $\psi_{r_g}$ proved in Lemma \ref{Smooth_bump_functions}. Before starting, we remind the reader that if $P_\nu(x) = a_1 x_1^2 + \cdots + a_d x_d^2$, we endow $H_\nu$ with the Haar measure induced by the basis 
\[E_{ij}-a_i a_j\inv E_{ji},\quad 1 \leq i<j \leq d \] 
of the Lie algebra of $H_\nu$---see \eqref{Haar_H_infty} in Subsection \ref{subsec_ROG}. We have
\begin{align*}
\int_{Y'} \varphi_g \dd \mu_{Y'} &= \int_{H_S \cap \Bcal^g_S} \psi_{r_g}(b_\infty) \dd \lambda_{H_S} (b) \\
		&=\lambda_{H_{S_f}}(H_{S_f} \cap \Bcal^g_{S_f}) \int_{H_\infty(r_g)} \psi_{r_g}(b_\infty) \dd \lambda_{H_\infty}(b_\infty)  \\
		& = (p_S^{-3} r_{S_f}(g))^{\frac{1}{2}d(d-1)} < 1,
\end{align*}
where $r_{S_f}(g) = \prod_{p \in S_f} r_p(g)$.   To get the last line we used the volume formula of Corollary \ref{Volume_balls_standard_p-adic-orthogonal_groups}. Note that $r_p(g) \leq 1$ for any $p \in S_f$, hence $p_S^{-3} r_{S_f}(g) < 1$. Similarly we have
\begin{align}
\notag \norm{\varphi_g}_{L^2(Y')} & = \left( \int_{H_S \cap \Bcal^g_S} \psi_{r_g}^2(b_\infty) \dd \lambda_{H_S} (b) \right)^\frac{1}{2} \\
\notag		& =  \lambda_{H_{S_f}}(H_{S_f} \cap \Bcal^g_{S_f})^\frac{1}{2} \norm{\psi_{r_g}}_{L^2(H_\infty)} \\
\notag		& <  \consBumpFuncRealOG{d} r_g^{-(\frac{1}{4}d(d-1) + 1)} \\
\label{EI1}		& = (3d^2 \cdot d!)^{\frac{1}{4}d(d-1) + 1} \consBumpFuncRealOG{d} r_\infty(g)^{-(\frac{1}{4}d(d-1) + 1)},   
\end{align}
where $\consBumpFuncRealOG{d}$ is as in Lemma \ref{Smooth_bump_functions}, and\footnote{Here $\Zc_{H_\infty} \in \mathfrak{h}_\infty$ is as in Proposition \ref{Mixing_speed_R-isotropic}.}
\begin{align}
\notag \norm{\Zc_{H_\infty}(\varphi_g)}_{L^2(Y')} & \leq \norm{\Zc_{H_\infty} (\psi_{r_g})}_{L^2(H_\infty)}\\
\notag		& \leq (3d^2 \cdot d!)^{\frac{1}{4}d(d-1) + 1} \consBumpFuncRealOG{d} \normi{\Zc_{H_\infty}} r_\infty(g)^{-(\frac{1}{4}d(d-1) + 1)} \\
\label{EI2}		& = 2 (3d^2 \cdot d!)^{\frac{1}{4}d(d-1) + 1} \consBumpFuncRealOG{d}  r_\infty(g)^{-(\frac{1}{4}d(d-1) + 1)}.
\end{align}
Combining \eqref{EI1} and \eqref{EI2} we obtain
\[\norm{\varphi_g}_{\Zc_{H_\infty}} \leq \consSmoothBump{d} r_\infty(g)^{-(\frac{1}{4}d(d-1)+1)}, \]
where $\consSmoothBump{d} = 3 (3d^2 \cdot d!)^{\frac{1}{4}d(d-1) + 1} \consBumpFuncRealOG{d}$.
\end{proof}	
	
	\subsection{The main proofs}\label{subsec_proofs_dyn_st}

Here we finally establish our two dynamical statements. Although the arguments are exactly the same, we keep the proofs separated to make the computations are easier to follow.
	
\begin{proof}[Proof of Proposition \ref{Dynamical_statement_RR-isotropic}]
By Lemma \ref{Small_rep_H/Hncc} there is $\eta \in H_S$ such that $\eta g \BasePointS{d}$ and $f \BasePointS{d}$ are in the same closed $H\ncc_S$-orbit $Y' \subseteq Y$ in $\LatSpaceS{d}$, $\eta_\infty$ is a diagonal matrix with $\pm 1$ in the main diagonal, $\normp{\eta_p} \leq p$ for odd $p \in S_f$ and $\norm{\eta_2}_2 \leq 4$ if $2 \in S_f$. 

Consider the $H\ncc_\infty$-smooth functions $\varphi_1 := \varphi_{\eta g}, \varphi_2 := \varphi_{f}: Y' \to [0, \infty)$ of Lemma \ref{S-adic_smooth_bump_functions}, supported respectively in the open subsets $\Uc^{\eta g}$ and $\Uc^{f}$ of $Y'$. By Proposition \ref{Mixing_speed_R-isotropic} and Lemma \ref{S-adic_smooth_bump_functions} we have
\begin{align}
\notag \left| \int_{Y'} (\varphi_1 \circ \rho(a_{\infty,-t})) \overline{\varphi_2} \dd \muY - \frac{(p_S^{-6} r_{S_f}(f) r_{S_f}(\eta g))^{\frac{1}{2}d(d-1)} }{vol \, Y'} \right|_\infty & \leq
(\Dc  \norm{\varphi_1}_{\Zc_{H_\infty}}  \norm{\varphi_2}_{\Zc_{H_\infty}}) e^{-t/6} \\
\label{DS1}	& \leq (\Dc \Nc_d^2  (r_\infty(f) r_\infty(g))^{-(\frac{1}{4}d(d-1) + 1)}) e^{-t/6}.
\end{align}
Recall that $\rho = \rho_{H_\infty}$ is the morphism $SL(2,\RR) \to H_\infty$ of Proposition \ref{Mixing_speed_R-isotropic}. Let's assume that $(\rho(a_{\infty,t}) \Uc^{\eta g}) \cap \Uc^f = \emptyset$ for any $t \in [0,1]$\footnote{Otherwise there is an $h^\star \in H_S$ with \mbox{$\normi{h^\star_\infty} < 12d^2$} and $\normp{h^\star_p} \leq 2p$ for any $p \in S_f$, such that $h^\star g \BasePointS{d} = f \BasePointS{d}$.}. Then, for any such $t$, the map $(\varphi_1 \circ \rho(a_{\infty, -t}))\overline{\varphi_2}$ is identically 0, so \eqref{DS1} yields
\begin{equation}\label{DS2}
\frac{(p_S^{-6} r_{S_f}(f) r_{S_f}(\eta g))^{\frac{1}{2}d(d-1)} }{vol \, Y'} \leq
(\Dc \Nc_d^2  (r_\infty(f) r_\infty(g))^{-(\frac{1}{4}d(d-1) + 1)}) e^{-t/6}.
\end{equation} 
 Let $ t_0 - 1$ be the positive number for which we have equality in $\eqref{DS2}$ for $t = t_0 - 1$. Then
\begin{equation}\label{Lola}
(\Dc \Nc_d^2  (r_\infty(f) r_\infty(g))^{-(\frac{1}{4}d(d-1) + 1)}) e^{-\frac{t_0}{6}} < \frac{(p_S^{-6} r_{S_f}(f) r_{S_f}(\eta g))^{\frac{1}{2}d(d-1)} }{vol \, Y'}. 
\end{equation}
Let $h'_\infty = \rho(a_{\infty, t_0})$. From \eqref{Lola} and \eqref{DS1} with $t = t_0$ we deduce that 
\[ \int_{Y'} (\varphi_1 \circ (h'_\infty)\inv) \overline{\varphi_2} \dd \muY \neq 0,\] 
so $h'_\infty \Uc^{\eta g}$ meets $\Uc^f$. Recall that $\Uc^f = \Bcal^f_S f \BasePointS{d}$ with $\Bcal^f_S \subset G_{d,S}$ as in \eqref{inj_set}, and similarly for $\Uc^{\eta g}$. Thus there are 
\[s \in \Bcal_S^{\eta g} \cap H\ncc_S \quad \text{and} \quad t \in \Bcal^f_S \cap H\ncc_S\] 
such that $(t\inv h'_\infty s) \eta g \BasePointS{d} = f \BasePointS{d}$. We set $h^\star = t\inv h'_\infty s \eta$. For $p \in S_f$ we have
\[\normp{h^\star_p} = \normp{t\inv_p s_p \eta_p} \leq \normp{\eta_p} \leq 
\begin{cases} p &\text{for any odd } p, \\ 4 &\text{if p=2}. \end{cases} 
\]
It remains only to prove the bound for $\normi{h^\star_\infty}$. Before doing so, note that by the choice of $t_0$ we have
\[ (p_S^{-6} r_{S_f}(f) r_{S_f}(\eta g))^{\frac{1}{2}d(d-1)} =
 \Dc \Nc_d^2 e^{\frac{1}{6}} e^{-\frac{t_0}{6}} (r_\infty(f) r_\infty(g))^{-(\frac{1}{4}d(d-1) + 1)} vol\, Y',   \]
so 
\[\normi{h'_\infty} \leq
e^{t_0}  < 3 \Dc^6 \Nc_d^{12} p_S^{18d(d-1)}  
(r_\infty(f) r_\infty(g))^{-(\frac{3}{2}d(d-1) + 6)}
(r_{S_f}(f) r_{S_f}(\eta g))^{-3d(d-1)}
 (vol \, Y)^6 
\]
Recall that for any $g' \in G_{d,S}$ and any $\nu \in S$ we defined 
\[T_\nu(g') = r_\nu(g')\inv = \frac{\normnu{g'_\nu}^d}{|\det g'_\nu |_\nu}.\] 
Also, for any $S' \subset S$, $T_{S'}(g')$ is the product of the $T_\nu(g')$ for $\nu \in S$. By the choice of $\eta$ we have $T_{S_f}(\eta g) \leq (2p_S)^d T_{S_f}(g)$. Thus
\[
\normi{h'_\infty} < (3 \cdot 2^{3d^2(d-1)} \Dc^6 \Nc_d^{12}) 
p_S^{9d^3} 
(T_\infty(f) T_\infty(g))^{(\frac{3}{2}d(d-1) + 6)}
(T_{S_f}(f) T_{S_f}( g))^{3d(d-1)}  
(vol \, Y)^6.
\]
Now
\begin{align*}
\normi{h^\star_\infty} &= \normi{t_\infty\inv h'_\infty s_\infty \eta_\infty} \\
	&\leq d^2 \normi{t\inv_\infty} \normi{s_\infty} \normi{h'_\infty} \\
	& \leq 4d^2 \normi{h'_\infty} \\
	& <  (12 \cdot 2^{3d^2(d-1)} d^2 \consMixingReal^6 \consSmoothBump{d}^{12}) p_S^{9d^3}
	 (T_\infty(f) T_\infty(g))^{\frac{3}{2}d(d-1) + 6}(T_{S_f}(g) T_{S_f}( g))^{3d^2}  (vol \, Y)^6,
\end{align*}
 which completes the proof. 
\end{proof}

\begin{proof}[Proof of Proposition \ref{Dynamical_statement_R-anisotropic}] 
By Lemma \ref{Small_rep_H/Hncc} there is $\eta \in H_S$ with $\normp{\eta_p} \leq p$ for any odd $p \in S$, $\norm{\eta_2} \leq 4$ if $2 \in S$ and $\eta_\infty = diag(\pm1 ,1,\dots,1)$, such that $\eta g \BasePointS{d}$ and $f \BasePointS{d}$ are in the same $H\ncc_S$-orbit $Y' \subseteq Y$.

Consider the neighborhoods $\Uc^f$ and $\Uc^{\eta g}$ of $f \BasePointS{d}$ and $\eta g\BasePointS{d}$ in $Y'$ as in \eqref{vecindad}. Let $n_2 = \log_{p_0} (T_{p_0}(f)) + 4$. Consider $\rho: SL(2,\QQ_{p_0}) \to H_{p_0}$ as in Lemma \ref{Covering_SL(2)->SO(P)} and Proposition \ref{Mixing_speed_R-anisotropic}.  Note that $\Uc^f$ is invariant under 
\[H\ncc_{p_0}\cap G_{d,p_0}(p_0^{-3} r_{p_0}(f)) = H\ncc_{p_0} \cap K_{d,p_0}(p_0^{-(n_2-1)}).\] 
In other words, $\varphi_2 = \ind_{\Uc^f}$ is an $H\ncc_{p_0} \cap K_{d,p_0}(p_0^{-(n_2-1)})$-invariant vector of $L^2(Y')$. By the same token, if $n_1 = \log_{p_0}(T_{p_0}(\eta g)) + 4$, then $\varphi_1 = \ind_{\Uc^{\eta g}}$ is $H\ncc_{p_0} \cap K_{d,p_0}(p_0^{-(n_1-1)})$-invariant. Proposition \ref{Mixing_speed_R-anisotropic} applied to $\varphi_1$ and $\varphi_2$ yields
\begin{align}
\notag \left| \muY((\rho(a_{p_0,m})\Uc^{\eta g}) \cap \Uc^f) - \frac{\muY(\Uc^{\eta g}) \muY(\Uc^f)}{\muY(Y')} \right|_\infty 
&\leq 
 \left(10 p_0^{\frac{3}{2}(n_1 + n_2 + 2)} \norm{\varphi_1}_{L^2} \norm{\varphi_2}_{L^2} \right) p_0^{-m/2} \\
\label{Mix1} 	& = \left(10 p_0^{15} (T_{p_0}(f) T_{p_0}(\eta g))^{\frac{3}{2}} (\muY(\Uc^{\eta g}) \muY(\Uc^f))^\frac{1}{2} \right)p_0^{-m/2},
\end{align}   
for any $m \geq 1$. Suppose that $\rho(a_{p_0,1}) \Uc^{\eta g}$ and $\Uc^f$ are disjoint\footnote{Otherwise there is an $h^\star \in H_S$ with $\normpo{h^\star_{p_0}} \leq p_0^4$ and $\normp{h^\star_p} \leq 2p$ for any $p \in S_f- \{p_0\}$, such that $h^\star g \BasePointS{d} = f \BasePointS{d}$.}. Let $m_0$ be the smallest positive integer such that the right-hand side of \eqref{Mix1} is strictly smaller than $\frac{\muY(\Uc^{\eta g}) \muY(\Uc^f)}{\muY(Y')}$ and set $h'_{p_0} = \rho(a_{p_0,m_0})$. From \eqref{Mix1} we deduce that that $h'_{p_0} \Uc^{\eta g}$ meets $\Uc^f$, hence there are 
\[s \in \Bcal^f_S \cap H\ncc_S \quad \text{and} \quad t \in \Bcal^{\eta g}_S \cap H\ncc_S\] 
such that $(t\inv h'_{p_0} s) \eta g \BasePointS{d} = f \BasePointS{d}$.  We set $h^\star = t\inv h'_{p_0} s \eta$, which is in $H_S$. For any $p \in S_f - \{p_0\}$ we have
\[\normp{h^\star_p} = \normp{t\inv_p s_p \eta_p} \leq \normp{\eta_p} \leq 
\begin{cases} p & \text{if }p \text{ is odd,}\\ 4 & \text{if }p = 2. \end{cases} \]
Before bounding $h^\star_{p_0}$ note that by the choice of $m_0$ 
\[ \frac{\muY(\Uc^{\eta g}) \muY(\Uc^f)}{\muY(Y')} \leq 
\left(10 p_0^{\frac{31}{2}} (T_{p_0}(f) T_{p_0}(\eta g))^{\frac{3}{2}} (\muY(\Uc^{\eta g}) \muY(\Uc^f))^\frac{1}{2} \right) p_0^{-m_0/2},\]
thus
\begin{equation}\label{Mix2} 
p_0^{m_0} \leq 10^2 p_0^{31} (T_{p_0}(f) T_{p_0}(\eta g))^{3} (\muY(\Uc^{\eta g}) \muY(\Uc^f))\inv (vol\, Y)^2. 
\end{equation}
Since $\Uc^f$ and $\Bcal^f_S$ have the same volume by Lemma \ref{Injectivity_radius_X_S}, using the volume estimates of Lemma \ref{Volume_small_balls_real_orthogonal_groups} and Corollary \ref{Volume_balls_standard_p-adic-orthogonal_groups} we get
\[\muY(\Uc^f)\inv \leq \Fc_d p_S^{\frac{3}{2}d(d-1)} T_S(f)^{\frac{1}{2}d(d-1)}, \]
where $\Fc_d =  (9 d^3 \cdot d!)^{\frac{1}{2}d(d-1)}$. Similarly 
\[\muY(\Uc^{\eta g})\inv \leq \Fc_d p_S^{\frac{3}{2}d(d-1)} T_S(\eta g)^{\frac{1}{2}d(d-1)}.\]
Now we go back to \eqref{Mix2}:
\[p_0^{m_0} \leq (10 \Fc_d)^2 p_0^{31} p_S^{3d(d-1)} (T_{p_0}(f) T_{p_0}(\eta g))^{3} (T_S(f) T_S(\eta g))^{\frac{1}{2}d(d-1)} (vol \, Y)^2.\]
Note that $T_p(\eta g) \leq \normp{\eta_p}^d T_p(g)$. Thus
\begin{align*}
p_0^{m_0}  & \leq 
2^{d^3} (10 \Fc_d)^2  p_0^{3d + 31} p_S^{\frac{1}{2}(d^3 + 5d^2 -6d)} (T_{p_0}(f) T_{p_0}(g))^{3} ( T_S(f) 
T_S(g) )^{\frac{1}{2}d(d-1)} (vol\, Y)^2 \\
& \leq 2^{d^3} (10 \Fc_d)^2   p_S^{\frac{1}{2}(d^3 + 5d^2 + 62)} (T_{p_0}(f) T_{p_0}(g))^{3} ( T_S(f) 
T_S(g) )^{\frac{1}{2}d(d-1)} (vol\, Y)^2
\end{align*}
Recall that $\normpo{h'_{p_0}} = \normpo{\rho(a_{p_0,m_0})} \leq p_0^{2m_0 + 1}$ by Lemma \ref{Covering_SL(2)->SO(P)}. We are ready to bound $h^\star_{p_0}$:
\begin{align*}
\normpo{h^\star_{p_0}} &= \normpo{t\inv_{p_0} h'_{p_0} s_{p_0} \eta_{p_0}} \\
	& \leq p_0^{2m_0 + 2} \\
	& \leq 2^{2d^3} (10 \Fc_d)^4   p_S^{d^3 + 5d^2 + 64} (T_{p_0}(f) T_{p_0}(g))^{6} ( T_S(f) T_S(g) )^{d(d-1)} (vol \, Y)^4 \\
	& < 10^4 \cdot 2^{2d^3} (9d^3 \cdot d!)^{2d(d-1)}  p_S^{6 d^3} (T_{p_0}(f) T_{p_0}(g))^{6} ( T_S(f) T_S(g) )^{d(d-1)} (vol \, Y)^4.
\end{align*}

\end{proof}

\section{Volume of closed $H_S$-orbits}\label{sec_vol_closed_orbits}

Let $Q$ be a non-degenerate quadratic form in $d \geq 3$ variables and let $S = \{\infty\} \cup S_f$ be a finite set of places of $\QQ$. The main result of this section is an upper bound of the volume of the $H_S$-orbit $\OrbitS{Q}$ in the space $\LatSpaceS{d}$ of lattices of $\QQ_S^d$ associated to $Q$. The bound is polynomial in $\hgt_S(\det Q)$. It will allow us to prove the criteria of $\ZZ_S$-equivalence of integral quadratic forms from propositions \ref{Dynamical_statement_RR-isotropic} and \ref{Dynamical_statement_R-anisotropic}, which are written in terms of the action of $H_S$ on $\LatSpaceS{d}$. We follow the dynamical approach that Li and Margulis used to establish the bound in the case $S = \{\infty\}$—see \cite[Theorem 6]{li_effective_2016}—. We point out that, alternatively, one can compute the volume of $\OrbitS{Q}$, which is equal to the volume of $O(Q,\QQ_S) / O(Q,\ZZ_S)$, applying Prasad's Volume Formula---see \cite{prasad_volumes_1989}.

	\subsection{Main statement and strategy of proof}

Before stating the main result, let's recall some notation. Consider the groups $G_{d,S} = GL(d,\QQ_S)$, the diagonal copy $\Gamma_{d,S}$ of $GL(d,\ZZ_S)$ in $G_{d,S}$, and the space of lattices $X_{d,S} = G_{d,S} / \Gamma_{d,S}$ of $\QQ_S^d$ with  base point $x_{d,S} = \Gamma_{d,S}/\Gamma_{d,S}$. Let $Q$ be a non-degenerate integral quadratic form in $d$ variables and let $P$ be the standard quadratic form on $\QQ_S^d$ that is $\QQ_S$-equivalent to $Q_S$\footnote{Recall that $Q_S$ is the quadratic form on $\QQ_S^d$ determined by $Q$ via the diagonal embedding $\QQ \to \QQ_S$.}. Recall that
\[Y_{Q,S} = H_S g x_{d,S},\]
where $H_S$ is the orthogonal $\QQ_S$-group of $P$ and $g$ is any matrix in $G_{d,S}$ taking $P$ to $Q_S$. The $H_S$-orbit $Y_{Q,S}$ is closed in $X_{d,S}$ by Lemma \ref{Y_Q,S_is_closed}, hence it admits an $H_S$-invariant measure $\mu_{Y_{Q,S}}$ by Lemma \ref{Closed_implies_finite_volume}. Recall that our choice of $\mu_{\OrbitS{Q}}$ is determined by the Haar measure $\Haar{H_S} = \otimes_{\nu \in S} \Haar{H_\nu}$ of $H_S$---the $\Haar{H_\nu}$'s are fixed in \eqref{Haar_H_infty} and \eqref{Haar_H_p} of Appendix \ref{app_volume_computations}. Remember that $\det Q$ is the determinant of the matrix of $Q$ in the standard basis of $\QQ^d$ and, when $S_f \neq \emptyset$, $p_S$ is the product of the primes in $S_f$.

\begin{Prop}\label{Main_volume_H_S-orbits}
Consider a finite set $S = \{\infty\} \cup S_f$ of places of $\QQ$ and $d \geq 3$. Let $Q$ be a non-degenerate integral quadratic form in $d$ variables such that $Q_S$ is $\QQ_S$-isotropic. Then
\[ vol\, \OrbitS{Q} < 
\begin{cases}
\consVolClosedOrb{d} 2^{2d^6} |\det Q|_\infty^\frac{d+1}{2} & \text{if } S = \{\infty\},\\
\consVolClosedOrb{d} p_S^{3d^6} \hgt_S(\det Q)^\frac{d+1}{2} & \text{if } S \neq \{\infty\}.
\end{cases} \]
\end{Prop}

To prove Proposition \ref{Main_volume_H_S-orbits} we will study the behavior of $\OrbitS{Q}$ when we move it transversally in $\LatSpaceS{d}$ for a certain ``time'' $T$\footnote{The \textit{time} $T$ will in fact be a point in a transversal to $H_S$ in $\GLS{d}$ near the identity $I_d$.}. On the one hand, we'll see that $\OrbitS{Q}$ is isolated in the following sense: the first return time $T_1$ to $\OrbitS{Q}$ is bounded from below by a constant times a negative power of $\hgt_S(\det Q)$. On the other hand, if the orbit $\OrbitS{Q}$ is big enough, we do to come back to $\OrbitS{Q}$: there is an upper bound $T_1$ of the form a constant times a negative power of $vol\, Y$. We'll obtain the bound of $vol\, \OrbitS{Q}$ combining the two estimates of $T_1$.

It will be convenient to replace $X_{d,S}$ by the space $X_{d,S}^1$ of covolume 1 lattices of $\QQ_S^d$ because the latter has finite volume. Since for any $g \in \GLS{d}$, the covolume of the lattice $g \ZZ_S^d$ is $\hgt_S (\det g)$, we will identify $X_{d,S}^1$ with $G_{d,S}^1 / \Gamma_{d,S}$, where
\[ G_{d,S}^1 := \{ g \in G_{d,S} \mid \hgt_S( \det g) = 1 \}. \]
Let $x_{d,S}^1 = \Gamma_{d,S} / \Gamma_{d,S} \in X_{d,S}^1$, and let  $\beta_{d,S}$ be the $G_{d,S}^1$-invariant measure on $X_{d,S}^1$ determined by our choice of Haar measure on $G_{d,S}^1$---see Subsection \ref{subsec_vol_X1}. Consider $Q, P$ and $H_S$ as before. Instead of $Y_{Q,S}$, we'll work with an $H_S$-orbit in $X_{d,S}^1$: Write $Q = P \circ f'$  with $f' \in G_{d,S}$. We replace $f'$ by a multiple $f'$ in $\GLUnoS{d}$. Let
\begin{equation}\label{MSQ_def}
M_S(Q) =  \left( \frac{\hgt_S(\det Q)}{ \hgt_S(\det P)} \right)^\frac{1}{2}.
\end{equation}
We define $N = N_S(Q) \in \QQ_S$ as $N_\infty =  M_S(Q)^{-\frac{1}{d}}$ and $N_p = 1$ for any $p \in S_f$. Then $f = N_S(Q)f'$ is in $G_{d,S}^1$, and we set
\[ Y^1_{Q,S} = H_S f x_{d,S}^1.  \]
Notice that $Y_{Q,S} \subseteq X_{S}$ and $Y_{Q,S}^1 \subseteq X_{d,S}^1$ have the same volume. Indeed, both are identified with
$H_S / (H_S \cap (f\inv \Gamma_{d,S} f))$
since the $f\inv$ and $(f')\inv$-conjugates of $\Gamma_{d,S}$ are the same.

	\subsection{Transversal isolation of closed $H_S$-orbits}\label{subsec_Transversal_isolation}

Consider a finite set $S = \{\infty\} \cup S_f$ of places of $\QQ$ and the orthogonal $\QQ_S$-group $H_S$ of a standard quadratic on $\QQ_S^d$. The goal of this subsection is to establish Lemma \ref{Transversal_isolation}, which says that the $H_S$-orbits in $\LatSpaceUnoS{d}$ of the form $\OrbitS{Q}$, for some integral quadratic form $Q$, are transversally isolated. These $H_S$-orbits are closed in $\LatSpaceUnoS{d}$ by Lemma \ref{Y_Q,S_is_closed} and moreover, when $H_S$ is non-compact and $d \geq 3$, any closed $H_S$-orbit in $\LatSpaceUnoS{d}$ is of this form---see Lemma \ref{Closed_implies_integral}. Hence, in that case all the closed $H_S$-orbits in $\LatSpaceUnoS{d}$ are transversally isolated.  

We use two new definitions in the main statement. The \textit{$S$-height} of any $g \in M_d(\QQ_S)$ is
\[\hgt_S(g) = \prod_{\nu \in S} \normnu{g_\nu}. \]
For any non-degenerate quadratic form $R$ on $\QQ_S^d$, $O(R,\QQ_S)$ is conjugated in $\GLS{d}$ to a unique orthogonal $\QQ_S$-group of a standard quadratic form on $\QQ_S^d$, which we call the \textit{standard conjugate} of $O(R,\QQ_S)$.
	
\begin{Lem}\label{Transversal_isolation}
Let $S = \{ \infty \} \cup S_f$ be a finite set of places of $\QQ$. Consider a non-degenerate integral quadratic form $Q$ in $d \geq 3$ variables and the standard conjugate $H_S$ of $O(Q_S,\QQ_S)$. Take $g \in \GLUnoS{d}$ and $u \in \GLUnoS{d} - H_S$ with $\normp{u_p} \leq 1$ for any $p \in S_f$. If $g \BasePointUnoS{d}$ and $ug \BasePointUnoS{d}$ are in $\OrbitUnoS{Q}$, then 
\[\normi{u_\infty - I_d} \geq \frac{1}{2d^3} p_S\inv \hgt_S(g)^{-2} \hgt_S(\det Q)^{-\frac{1}{d}}. \]
\end{Lem}

The proof of Lemma \ref{Transversal_isolation} is based on the next four auxiliary results.	

\begin{Lem}\label{TI_aux_1}
For any $g_\infty \in GL(d,\RR)$ we have
\[ \normi{g_\infty} \geq \frac{|\det g_\infty|_\infty ^\frac{1}{d}}{\sqrt{d}}.\]
\end{Lem}

\begin{proof}
Consider $f = |\det g_\infty|_\infty^{-\frac{1}{d}} g$. Notice that $f$ is in $SL^{\pm}(d,\RR)$. Thanks to the Iwasawa Decomposition we write $f = kan$ for some $k \in O(d,\RR),$
\[ a = diag(a_1, \ldots, a_d), \]
and $n$ unipotent, upper-triangular. Since $|a_1 \cdots a_d|_\infty = 1$, then $\normi{an} \geq 1$. Thus
\begin{align*}
1 \leq \normi{an} = \normi{k\inv f} & = |\det g_\infty|_\infty^{-\frac{1}{d}} \normi{k\inv g_\infty} \\
	& \leq \sqrt{d} \cdot |\det g_\infty|_\infty^{-\frac{1}{d}} \normi{g_\infty},
\end{align*}
which is what we wanted. 
\end{proof}

\begin{Lem}\label{TI_aux_2}
For any $g_p \in GL(d,\QQ_p)$ we have
\[ \normp{g_p} \geq |\det g_p|_p^\frac{1}{d}.\]
\end{Lem}

\begin{proof}
We write $g = kan$ with $k \in GL(d,\ZZ_p)$, 
\[a = diag(p^{n_1}, \ldots, p^{n_d}),\]
and $n$ unipotent, upper-triangular. Then
\[
\normp{g_p}  = \normp{an} \geq \max_{i} |p^{n_i}|_p  .
\]
The product of the positive real numbers $|p^{n_i}|_p  |\det g_p|_p^{-\frac{1}{d}}$ for $1\leq i \leq d$ is 1, so at least one is $\geq 1$. 
\end{proof}

Recall that $M_S(Q)$ was defined in \eqref{MSQ_def}.

\begin{Lem}\label{MSQ}
Let $Q$ be a non-degenerate integral quadratic form in $d \geq 2$ variables. For any finite set $S = \{\infty\} \cup S_f$ of places of $\QQ$ we have
\[1 \leq M_S(Q) \leq p_S \hgt_S(\det Q)^\frac{1}{2}.\]
\end{Lem}

\begin{proof}
Let $P = (P_\nu)_{\nu \in S}$ be the standard quadratic form on $\QQ_S^d$ that is $\QQ_S$-equivalent to $Q_S$. We have  
\[\det P_\infty = \pm 1 \quad  \text{and} \quad p^{-2} \leq |\det P_p|_p \leq 1 \]
for any $p\in S_f$, thus 
\[p_S^{-2} \leq \hgt_S(\det P) \leq 1.\] 
Since $M_S(Q) = \left( \frac{\hgt_S(\det Q)}{\hgt_S(\det P)} \right) ^\frac{1}{2}$ and $\hgt_S(\det Q)$ is a positive integer, the inequality we want follows.
\end{proof}

\begin{Lem}\label{TI_aux_3}
Any $t \in [0, \frac{1}{2}]$ verifies 
\[\sqrt{t+1} - 1 \geq \frac{2}{5} t. \] 
\end{Lem}

\begin{proof}
Since $F(t) = \sqrt{t+1} - 1$ is concave, it suffices to verify the inequality for $t \in \{0, 1/2\}$. We have $F(0) = 0$ and 
\[ F\left(\frac{1}{2} \right) =  2\left( \sqrt{\frac{3}{2}} - 1 \right) = 0.44\ldots > \frac{2}{5}.  \] 
\end{proof}

We are ready to prove the transversal isolation of the $H_S$-orbits $\OrbitUnoS{Q}$.

\begin{proof}[Proof of Lemma \ref{Transversal_isolation}]
Let $Q$ be a non-degenerate integral quadratic form in $d \geq 3$ variables. The strategy we'll follow is: points in $Y^1_{Q_S}$ correspond to quadratic forms $\ZZ_S$-equivalent to $Q$. The ones associated to $g x_{d,S}^1$ and $u g x_{d,S}$ are different because $u \notin H_S$, so the $S$-height of the difference of their matrices is at least 1. From this we'll deduce the bound for $u_\infty$. 

First we recover the matrices with coefficients in $\ZZ_S$ corresponding to points in $Y^1_{Q,S}$.  We'll recall briefly the definition of $Y^1_{Q,S}$. Let $P$ be the standard quadratic form $\QQ_S$-equivalent to $Q_S$, $H_S = O(P, \QQ_S)$ and consider $f' \in G_{d,S}$ such that $Q = P \circ f'$.  Let $f = N_S(Q) f'$, where $N_S(Q) \in \QQ_S$ is defined as: 
\[N_S(Q)_\infty = M_S(Q)^{-\frac{1}{3}},\]
and  $N_S(Q)_p = 1$ for $p \in S_f$. See \eqref{MSQ_def} for the definition of $M_S(Q)$. Then $f$ is in $G_{d,S}^1$ and
\[Y^1_{Q,S} = H_S f x_{d,S}^1.\]
Let $b \in GL(d,\QQ_S)$ be the matrix of $P$ in the standard basis of $\QQ_S^d$. If $g'$ is in $H_S f' \Gamma_{d,S}$, then
\[\tra g' b g' = \tra \gamma \tra f' b f' \gamma = \tra \gamma b_{Q_S} \gamma \]
for some $\gamma \in \Gamma_{d,S}$. It follows that the matrix $\tra g' b g' \in M_d(\QQ_S)$ is the diagonal image of a matrix in $M_d(\ZZ_S)$. This implies that if $g_1 x_{d,S}^1$ is in $Y^1_{Q,S}$---for $g_1 \in G_{d,S}^1$---then $N_S(Q)^{-2} \tra g_1 b g_1 \in M_d(\QQ_S)$ is the diagonal image of a matrix with coefficients in $\ZZ_S$.

Now we compare the matrices $B,C \in M_d(\QQ_S)$ associated to the two points of the statement. We'll normalize them to make the estimates in $M_d(\RR)$. Let $g, u \in G_{d,S}^1$ as in the statement. Then $g x_{d,S}^1$ and $u g x_{d,S}^1$ are in $Y_{Q,S}$. We consider
\[ B = N_S(Q)^{-2} (\tra g b g), \quad C = N_S(Q)^{-2} (\tra g \tra u b u g). \]
For any $p \in S_f$ we have 
\begin{align*}
\normp{C_p} & = \normp{\tra g_p \tra u_p b_p u_p g_p} \\
			& \leq \normp{\tra g_p}  \normp{\tra u_p}  \normp{P_p}  \normp{u_p}  \normp{g_p} \\		
			& \leq \normp{g_p}^2,
\end{align*}  
and similarly $\normp{B_p} \leq \normp{g_p}^2$. It follows that $\hgt_{S_f}(g)^2 B_\infty$ and $\hgt_{S_f}(g)^2 C_\infty$ have integral coefficients, where 
\[\hgt_{S_f}(g) = \prod_{p \in S_f} \normp{g_p}. \]
These two matrices are different because $\tra u b u \neq b$, hence the $\infty$-norm of their difference is at least 1:
\begin{align*}
1 & \leq \normi{\hgt_{S_f}(g)^2 C_\infty - \hgt_{S_f}(g)^2 B_\infty} \\
	& = \hgt_{S_f}(g)^2 M_S(Q)^{\frac{2}{d}} \normi{\tra g_\infty ( \tra u_\infty b_\infty u_\infty - b_\infty) g_\infty }.
\end{align*} 
We rearrange this inequality and we work with the right-hand side: 
\begin{align*}
\hgt_{S_f}(g)^{-2} M_S(Q)^{-\frac{2}{d}} & \leq \normi{\tra g_\infty ( \tra u_\infty b_\infty u_\infty - b_\infty) g_\infty } \\
			& \leq d^2 \normi{\tra g_\infty} \cdot \normi{\tra u_\infty b_\infty u_\infty - b_\infty} \cdot \normi{g_\infty} \\
			& = d^2 \normi{ g_\infty}^2 \cdot \normi{\tra(u_\infty - I_d) b_\infty (u_\infty - I_d) + \tra(u_\infty - I_d) b_\infty + b_\infty (u_\infty - I_d) } \\
			& \leq d^2 \normi{ g_\infty}^2 ( d \normi{u_\infty - I_d} \cdot \normi{b_\infty (u_\infty - I_d)} + 2 \normi{u_\infty - I_d}) \\
			& \leq d^3 \normi{g_\infty}^2 ( \normi{u_\infty - I_d}^2 + \normi{u_\infty - I_d}). 
\end{align*}
Hence 
\[ \normi{u_\infty - I_d}^2 + \normi{u_\infty - I_d} \geq C_g, \]
where $C_g = d^{-3} \hgt_S(g)^{-2} M_S(Q)^{-\frac{2}{d}}$. We obtain that $\normi{u_\infty - I_d}$ is greater or equal than the positive root of $t^2 + t - C_g$, that is
\[ \normi{u_\infty - I_d} \geq \frac{1}{2} (\sqrt{4 C_g + 1} - 1). \]
Using \eqref{MSQ} and lemmas \ref{TI_aux_1}, \ref{TI_aux_2} we deduce that
\begin{align*}
4 C_g = 4 \cdot d^{-3} \hgt_S(g)^{-2} M_S(Q)^{-\frac{2}{d}} & \leq 4 \cdot d^{-3} (d \hgt_S(\det g)^{-\frac{2}{d}}) \\
		& = 4 \cdot d^{-2} < \frac{1}{2}. 
\end{align*}
We use now Lemma \ref{TI_aux_3} and the lower bound of \eqref{MSQ}:
\begin{align*}
\normi{u_\infty - I_d} & \geq \frac{1}{5} \cdot 4 C_g \\
			& = \frac{4}{5 d^3} \hgt_S(g)^{-2} M_S(Q)^{-\frac{2}{d}} \\
			& \geq \frac{4}{5 d^3} p_S^{-\frac{2}{d}} \hgt_S(g)^{-2} \hgt_S(\det Q)^{-\frac{1}{d}} \\
			& \geq \frac{1}{2d^3}  p_S^{-1} \hgt_S(g)^{-2} \hgt_S(\det Q)^{-\frac{1}{d}},
\end{align*}
which is what we wanted. 
\end{proof}
	
	\subsection{Uniform recurrence of closed $H_S$-orbits}\label{subsec_Uniform_recurrence}

	Let $S = \{\infty\} \cup S_f $ be a finite set of places of $\QQ$ and let $d \geq 3$. The purpose of this subsection is to give an explicit compact subset $\Omega_{d,S}$ of $\LatSpaceS{d}$ with the next property: for any non-compact orthogonal $\QQ_S$-group $H_S$ of a non-degenerate quadratic form on $\QQ_S^d$, any closed $H_S$-orbit in $\LatSpaceS{d}$ spends most of its time in $\Omega_{d,S}$. The result for $S = \{\infty\}$ was established by Li and Margulis in \cite[Lemma 13]{li_effective_2016}, following the proof of Einsiedler, Margulis and Venkatesh of \cite[Lemma 3.2]{einsiedler_effective_2009}, which is an effective uniform recurrence of closed $H$-orbits in a slightly more general setting than real orthogonal groups\footnote{They consider the action of $H$ on $G/ \Gamma$, where $H \subset G$ are real semisimple Lie groups and $\Gamma$ is an arithmetic lattice of $G$.}.

Let's introduce the definitions that we need for the main result. For any $M > 0$ we define 
\begin{equation}\label{def_lift_big_compact} \widetilde{ \Omega}_{d,S}(M) = \left\{ g \in SL^{\pm}(d, \RR) \times \prod_{p \in S_f} GL(d, \ZZ_p) : \normi{g_\infty} \leq M \right\}, 
\end{equation}
and $\Omega_{d,S}(M) = \widetilde{ \Omega}_{d,S}(M) x_{d,S}^1$. Consider $\consDefBigCompact{d} = 2^{d^3} \cdot 3^{2d^4} d^{3d^3}$. We introduce the following compact subset of $X_{d,S}^1$ 
\[ \Omega_{d,S} = \begin{cases}
\Omega_{\infty,d} ( \consDefBigCompact{d} 2^{d^4}) & \text{if } S = \{\infty\}, \\
\Omega_{d,S}(\consDefBigCompact{d} p_S^{2d^4}) & \text{if } S \neq \{\infty\}.
\end{cases}
\]  	
	
\begin{Prop} \label{Compact_meeting_closed_H_S-orbits}
Consider a finite set $S = \{\infty\} \cup S_f$ of places of $\QQ$ and $d \geq 3$. Let $H_S$ be the orthogonal group of a non-degenerate,  $\QQ_S$-isotropic quadratic form on $\QQ_S^d$. For any closed $H_S$-orbit $Y$ in $\LatSpaceUnoS{d}$ we have
\[\muY(Y \cap \Omega_{d,S}) \geq \frac{1}{2} vol\, Y . \]
\end{Prop}

The key tool in the proof of \cite[Lemma 13]{li_effective_2016} is the effective recurrence of unipotent flows on $X^1_{d,\infty}$ of Kleinbock and Margulis---see \cite[Theorem 5.3]{kleinbock_flows_1998}. The extension by Kleinbock and Tomanov of this result to any $\LatSpaceUnoS{d}$ will allow us to establish Proposition \ref{Compact_meeting_closed_H_S-orbits}. Here is the outline of the proof: the action of (almost) any one-parameter unipotent subgroup $U$ of $H_S$ on any closed $H_S$-orbit $Y$ in $\LatSpaceUnoS{d}$ is ergodic, hence by Birkhoff's Theorem, for any measurable subset $E$ of $\LatSpaceUnoS{d}$ and for $\muY$-almost any $y \in Y$, we can approximate $\muY( Y \cap E)$ by averages of $\ind_{E}$ along pieces of $Uy$. Thanks to the effective uniform recurrence of unipotent flows, we can give an explicit compact subset $E$ of $\LatSpaceUnoS{d}$ that traps most of the $U$-orbit $Uy$, and hence most of $Y$\footnote{There are minor imprecisions in our sketch of the proof that we'll correct in due time.}.  

We start by proving that one-parameter unipotent subgroups of $H_S$ act ergodically on closed $H\ncc_S$-orbits. Although the arguments involved are standard in the area and well-known for $S = \{\infty\}$, we feel that writing a complete proof of this fact for any $S$ may be useful. The only extra ingredient to extend the proof for $S = \{\infty\}$ to any $S = \{\infty\} \cup S_f$ is the Strong Approximation Theorem.

\begin{Lem}\label{Unipotent_groups_act_ergodically}
Consider a finite set $S = \{\infty\} \cup S_f$ of places of $\QQ$ and $d \geq 3$. Let $H_S$ be the orthogonal group of a non-degenerate quadratic form on $\QQ_S^d$ such that $H_{\nu_0}$ is non-compact for some $\nu_0 \in S$. Let $U_{\nu_0}$ be a one-parameter unipotent subgroup of $H_{\nu_0}$ with non-trivial projection to each simple factor of $H_{\nu_0}$\footnote{In fact $H_{\nu_0}$ is always simple, except when $d = 4$ and $H_{\nu_0}$ is locally isomorphic to $SL(2,\QQ_{\nu_0}) \times SL(2,\QQ_{\nu_0})$.}. The action of $U_{\nu_0}$ on any closed $H\ncc_S$-orbit in $\LatSpaceUnoS{d}$ is ergodic. 
\end{Lem}

To prove Lemma \ref{Unipotent_groups_act_ergodically} we'll use the well-known \textit{Howe-Moore phenomenon} below.

\begin{Lem}\label{Vectors_fixed_by_a_unipotent_are_globally_fixed}
Consider the group  of $\QQ_\nu$-points $J$ of a semisimple $\QQ_\nu$-group. Let $\pi$ be a unitary representation of $J$ and let $J\ncc$ be the subgroup of $J$ generated by the unipotent elements. If $v \in \mathcal{H}_\pi$ is fixed by a unipotent element with non-trivial projection to each simple factor of $J$, then $v$ is $J\ncc$-invariant. 
\end{Lem}

\begin{proof}
The case $J = SL(2,\RR)$ is done in \cite[Proposition 3.4]{benoist_five_2009}, and the same proof works for  $SL(2,\QQ_\nu)$. Now consider a general $J$. 

We prove first that a vector $v \in \Hc_\pi$ is $J\ncc$-invariant if it is fixed by a hyperbolic element\footnote{$h\in J$ is hyperbolic if $Ad(h): Lie(J) \to Lie(J)$ is diagonalizable over $\QQ_\nu$.} $h\in J$ with non-trivial projection to each simple factor of $J$. We take an $h$-invariant vector $v$ of unit length. Consider the subgroup
\[ U_h^+ = \left\{ g \in J \mid \lim_{n \to \infty} h^n g h^{-n} = e \right \}. \]
Since $\pi(h)v = v$, then  
\[\scalar{\pi(g)v}{v} = \scalar{\pi(h^n g h^{-n})v}{v}\] 
for any $n\in \ZZ$. If $g$ is in $U_h^+$, we obtain that $\scalar{\pi(g)v}{v} = 1$ by letting $n \to \infty$, so $v$ is fixed by $g$. This proves that $v$ is $U_h^+$-invariant. In a similar way we see that $v$ is $U^-_h$-invariant, where
\[ U^-_h = \left\{ g \in J \mid \lim_{n \to \infty} h^{-n} g h^{n} = e \right \}. \]
The groups $U^\pm_h$ have non-trivial projection to each simple factor of $J$ since $h$ has this property. Then $J\ncc$ is generated by $U_h^+$ and $U_h^-$---see \cite[Proposition 1.5.4 $(ii)$]{margulis_discrete_1991}---, so $v$ is $J\ncc$-invariant. 

Suppose now that $v$ is fixed by a non-trivial unipotent element $u$ of $J$. By Jacobson-Morozov's Theorem $u$ is in the image of a group morphism $\psi: SL(2,\QQ_\nu) \to J$ with finite kernel. The vector $v$ is then $SL(2,\QQ_\nu)$-invariant because it is fixed by a non-trivial unipotent element of $SL(2, \QQ_\nu)$. The image of $\psi$ has non-trivial projection to each simple factor of $J$ because it's generated by conjugates of $u$, which have this property. Since $\psi(SL(2,\QQ_\nu))$ has non-trivial hyperbolic elements, $v$ is $J\ncc$-invariant thanks to the previous paragraph. 
\end{proof}

We are ready to prove that one-paramenter unipotent groups act ergodically on closed $H\ncc_S$-orbits.

\begin{proof}[Proof of Lemma \ref{Unipotent_groups_act_ergodically}]
Let $Y = H\ncc_S g x_{d,S}^1$ be a closed $H\ncc_S$-orbit in $X_{d,S}^1$. Since $H_S$ is non-compact, then $g\inv H_S g = O(Q_S,\QQ_S)$ for some non-degenerate integral quadratic form in $d$ variables by Lemma \ref{Closed_implies_integral}. Let $J_S = g\inv H_S g$, $Y' = J_S x_{d,S}^1$ and $U'_{\nu_0} = g\inv U_{\nu_0} g$. We'll prove that $U'_{\nu_0} \curvearrowright Y'$ is ergodic. 

Let $\pi$ be the unitary representation of $J\ncc_S$ on $L^2(Y')$. Suppose that $\varphi \in L^2(Y')$ is $U'_{\nu_0}$-invariant. Then $\varphi$ is $J\ncc_{\nu_0}$-invariant by Lemma \ref{Vectors_fixed_by_a_unipotent_are_globally_fixed} because $U_{\nu_0}$ has non-trivial projection in each simple factor of $H_{\nu_0}$. To see that $\varphi$ is $J\ncc_S$-invariant, consider the function $\Phi: J\ncc_S \to \CC, h \mapsto \varphi(h x^1_{d,S})$. $\Phi$ is $(J\ncc_S \cap \Gamma_{d,S})$-invariant on the right and $J\ncc_{p_0}$-invariant on the left. Since $J\ncc_{p_0}$ is normal in $J\ncc_S$, then $\Phi$ is also $J\ncc_{p_0}$-invariant on the right. By the Strong Approximation Theorem---see \cite[Theorem 7.12]{platonov_algebraic_1994}---$J\ncc_{p_0} (J\ncc_S \cap \Gamma_{d,S})$ is dense in $J\ncc_S$, so $\Phi$ is $J\ncc_S$-invariant on the right. This proves that $\varphi$ is $\mu_{Y'}$-almost surely constant, thus the action of $U'_{p_0}$ on $Y'$ is ergodic. 
\end{proof}

We need a couple of new definitions to state the effective recurrence of unipotent flows in $\LatSpaceUnoS{d}$. Any discrete subgroup of $\RR^d$ is a free abelian group of rank at most $d$. More generally, for any finite set $S = \{\infty\} \cup S_f$ of places of $\QQ$, any discrete $\ZZ_S$-submodule\footnote{Here $\ZZ_S$ is considered as subring of $\QQ_S$ via the diagonal embedding $\QQ \to \QQ_S$.} $\Delta$ of $\QQ_S^d$ is free, of rank at most $d$. Let $W_\Delta$ be the $\QQ_S$-module generated by $\Delta$. The natural Haar measure  Haar measure\footnote{This comes from the natural Haar measure of any linear subspace of $\QQ_\nu^d$, which we now describe. Let $V$ be a $d_0$-dimensional linear subspace of $\QQ_\nu^d$. Let $\Haar{\QQ_\nu^{d_0}}$ be the Haar measure of $\QQ_\nu^{d_0} \times \{0\}$ normalized in the usual way. Let $K_\nu$ be respectively $O(d,\RR)$ and $GL(d,\ZZ_\nu)$ if $\nu = \infty$ and $\nu < \infty$. Consider any $k \in K_\nu$ sending $\QQ_\nu^{d_0}$ to $V$. We set $\Haar{V} = k_* \Haar{\QQ_\nu^{d_0}}$.} $\Haar{W_\Delta}$ of $W_\Delta$  induces a measure on $\Delta \backslash W_\Delta$ in the usual way. We define the covolume of $\Delta$ as the volume of $\Delta \backslash W_\Delta$, which we denote as $\cov \Delta$. It is finite since $\Delta$ is cocompact in $W_\Delta$. We will denote by $\Sigma_{<1}(\Delta)$ the set of non-zero $\ZZ_S$-submodules of $\Delta$ of covolume strictly less than 1. Note that $\Delta$ is a lattice in $\QQ_S^d$ precisely when it has $\ZZ_S$-rank $d$. We identify the space of covolume 1 lattices of $\QQ_S^d$ with $\LatSpaceUnoS{d}$. Let $\normeuc{\cdot}$ be the standard euclidean norm on $\RR^d$. We define the \textit{$S$-height} of any $v \in \QQ_S^d$ as
\[ \hgt_S(v) = \normeuc{v_\infty} \, \prod_{p \in S_f} \normp{v_p}.\]
We define also the \textit{systole} of a lattice $\Delta$ of $\QQ_S^d$ as
\[ \alpha_1(\Delta) = \min_{v \in \Delta-\{0\}} \hgt_S(v). \]
In Appendix \ref{app_Mahler} we prove an effective version of Mahler's Criterion for $\LatSpaceUnoS{d}$, which in particular allows us to detect compact subsets of $\LatSpaceUnoS{d}$ using $\alpha_1$. Here is finally the statement of recurrence of unipotent flows that we'll use.   

\begin{Prop}\label{Effective_recurrence_unipotent_flows_p-adic}
Let $S= \{\infty\} \cup S_f$ be a finite set of places of $\QQ$, $\nu \in S$ and $d \geq 2$. Consider  a one-parameter, unipotent subgroup $U_\nu = (u_t)_t$ of $SL(d,\QQ_\nu)$ and a covolume 1 lattice $\Delta$ of $\QQ_S^d$. Suppose that for any $\Lambda \in \Sigma_{<1}(\Delta)$, $U_\nu$ doesn't preserve the $\QQ_S$-module generated by $\Lambda$. Then, there is a positive constant $T_0 = T_0(U_\nu, \Delta)$ such that for any $T \geq T_0$ and any $\varepsilon \in (0,1)$,
\[ \Haar{\QQ_\nu} (\{t \in B_{\nu}(T) \mid \alpha_1(u_t \Delta) < \varepsilon \}) \leq \consCRecurrence{\nu}{d} \varepsilon^{\consAlfaRecurrence{d}} \Haar{\QQ_\nu}(B_\nu(T)).\]
\end{Prop}

To avoid a big detour here, we explain in Appendix \ref{app_recurrence_unipotent_flows} how to obtain Proposition \ref{Effective_recurrence_unipotent_flows_p-adic} from more general statements of D. Kleinbock and G. Tomanov \cite{kleinbock_flows_2007}. The next two auxiliary results will allow us to apply Proposition \ref{Effective_recurrence_unipotent_flows_p-adic} to prove Proposition \ref{Compact_meeting_closed_H_S-orbits}.

\begin{Lem}\label{Few_conjugates_preserve_a_proper_subspace}
Consider a prime $\nu$ and $d \geq 3$. Let $H_\nu$ be the orthogonal group of a non-degenerate isotropic quadratic form on $\QQ_\nu^d$ and let $U_\nu$ be a one-parameter unipotent subgroup of $H_{\nu}$ with non-trivial projection to each simple factor of $H_\nu$. For any proper linear subspace $V$ of $\QQ_{\nu}^d$, the subset 
\[\{h \in H_{\nu} \mid h\inv U_{\nu} h \text{ preserves } V \} \]
of $H_{\nu}$ has measure 0. 
\end{Lem}

\begin{proof}
We denote $\mathscr{C}(V)$ the set in the statement. Since $\mathscr{C}(V)$ is Zariski-closed, it has measure 0 or it contains a Zariski-connected component of $H_{\nu}$. We'll show that the latter case implies $V=0$ or $V= \QQ_{\nu}^d$. Let $H'$ be the Zariski-connected component of the identity of $H_{\nu}$. If $\mathscr{C}(V)$ contains $h_0 H'$, then $V$ is stable under the groups
\[(h') \inv(h_0\inv U_{\nu} h_0) h' \]
with $h' \in H'$. Let $Z$ be an infinitesimal generator of $h_0\inv U_{\nu} h_0$. $V$ is invariant under $Ad\,h' (Z)$ for $h' \in H'$. Note that the lie algebra $\mathfrak{h}_\nu$ of $H_\nu$ is generated by the $Ad\, h' (Z)'$s for $h' \in H'$ since $Z$ has non-trivial projection to each simple factor of $\mathfrak{h}_\nu$. Thus $V$ is $\mathfrak{h}_{\nu}$-invariant. Then $V=0$ or $V= \QQ_{\nu}^d$ because the natural action of $\hgot_\nu$ on $\QQ_\nu^d$ is irreducible.
\end{proof}

\begin{Lem}\label{Sigma<1_finite}
Let $S = \{\infty\} \cup S_f$ be a finite set of places of $\QQ$ and let $\Delta$ be a lattice of $\QQ_S^d$. Then $\Sigma_{<1}(\Delta)$ is finite. 
\end{Lem}

\begin{proof}
For any $\Delta' \in \Sigma(\Delta)$, let $W_{\Delta'}$ be the $\QQ_S$-submodule of $\QQ_S^d$ generated by $\Delta'$ and consider
\[\mathscr{W} = \{W_{\Delta'} \mid \Delta' \in \Sigma_{<1}(\Delta)\}. \]
For $W \in \mathscr{W}$, let $\Delta'_W = \Delta \cap W$. We'll show that the map $\Sigma_{<1}(\Delta) \to \mathscr{W}, \Delta' \mapsto W_{\Delta'}$ is finite to one and that $\mathscr{W}$ is finite.

Take $W \in \mathscr{W}$ and $\Delta' \in \Sigma_{<1}(\Delta)$ such that $W_{\Delta'} = W$. Then $\Delta'$ is contained in $\Delta'_W$, so
\[ [\Delta'_W : \Delta'] = \frac{\text{cov } \Delta'}{\text{cov } \Delta'_W} < \frac{1}{\text{cov } \Delta'_W}. \]
To conclude note that $\Delta'_W$ has finitely many subgroups $\Lambda$ of index, say $N >0$. Indeed, any such $\Lambda$ contains $N \Delta'_W$, and $\Delta'_W / (N \Delta'_W)$ is finite. 

Let's prove that $\mathscr{W}$ is finite. It suffices to see that the subset $\mathscr{W}_k$ of elements of $\mathscr{W}$ of $\QQ_S$-rank $k$ is finite for any $1 \leq k \leq d-1$. If $W \in \mathscr{W}_k$, let $v_1,\ldots, v_k$ be a $\ZZ_S$-basis of $\Delta'_W$. Then $v_1 \wedge \cdots \wedge v_k$ belongs to $\bigwedge^k \Delta$ and its $S$-height is $\text{cov } \Delta'_W < 1$. Since $\Delta$ is a lattice in $\QQ_S^d$, $\bigwedge^k \Delta$ is a lattice in $\bigwedge^k \QQ_S^d$. Moreover, $\ZZ_S^\times(v_1 \wedge \cdots \wedge v_k)$ doesn't depend on the chosen $\ZZ_S$-basis of $\Delta'_W$ and the map $\mathscr{W}_k \to \ZZ_S^\times \backslash \bigwedge^k \Delta$ is injective. To conclude note that there are finitely many $\ZZ_S^\times v \in \ZZ_S^\times \backslash \bigwedge^k \Delta$ with $\hgt_S(v) < 1$\footnote{This is a fact valid for any lattice $\Lambda$ in $\QQ_S^m$. Take $v \in \Lambda-\{0\}$ with $\hgt_S(v) < 1$. We'll see that $\ZZ_S^\times v$ has a representative in the finite set $A = \Lambda \cap ([-1,1] \times \prod_{p \in S_f} \ZZ_p)$.  Since $\ZZ_S v$ is discrete in $\QQ_S^m$, then  $v_\nu \neq 0$ for any $\nu \in S$, so $\norm{v}_{S_f} = \prod_{p \in S_f} \normp{v_p}$ is a unit in $\ZZ_S \hookrightarrow \QQ_S$. We set $v' = \norm{v}_{S_f} v$. Note that $\normp{v'_p} = 1$ for any $p \in S_f$ and $\normi{v'_\infty} = \hgt_S(v) < 1$, hence $v'$ is in $A$. }.
 
\end{proof}

We are ready for the main proof of this subsection. 

\begin{proof}[Proof of Proposition \ref{Compact_meeting_closed_H_S-orbits}]

We'll first give a compact subset $\mathfrak{O}_{d,S}$ of $\LatSpaceS{d}$ with the desired property in terms of $\alpha_1$, and then we'll use our Effective Mahler's Criterion---Lemma \ref{CS_aux_2}---to recover the set $\Omega_{d,S}$ of the statement. 

For any $\varepsilon > 0$, the set
\[ \mathfrak{S}_{d,S}(\varepsilon) = \{ \Delta \in X^1_{d,S} \mid \alpha_1(\Delta) \geq \varepsilon \} \]
is compact by Corollary \ref{Mahlers_crit}. We define
\[\varepsilon_{\infty,d} = \frac{1}{2} \cdot \left( \frac{1}{2 \cdot 3^{2d} d^3 2^{d+2} } \right)^{(d-1)^2}, 
\quad \quad 
\varepsilon_{p,d} = \frac{1}{2} \cdot \left( \frac{1}{2 \cdot 3^{2d} d^3 p^{2d+1} } \right)^{(d-1)^2}, \]
and $\varepsilon_{d,S} = \min_{\nu \in S} \varepsilon_{\nu,d}. $
We set $\mathfrak{O}_{d,S} = \mathfrak{S}_{d,S} (\varepsilon_{d,S}).$
Let $H_S$ be the orthogonal $\QQ_S$-group of a non-degenerate, isotropic quadratic form on $\QQ_S^d$.  We'll show that $\mathfrak{O}_{d,S}$ meets at least half of any closed $H_S\ncc$-orbit in $\LatSpaceS{d}$. This easily implies a similar statement for closed $H_S$-orbits in $\LatSpaceS{d}$. Indeed, $H_S\ncc$ is a normal, finite-index subgroup of $H_S$, hence any closed $H_S$-orbit in $\LatSpaceS{d}$ is a finite union of closed $H\ncc_S$-orbits. 

Consider $\nu_0 \in S$ such that $H_{\nu_0}$ is non-compact. Let $C = \consCRecurrence{\nu_0}{d}$ and $\vartheta = \consAlfaRecurrence{d}$ be as in Proposition \ref{Effective_recurrence_unipotent_flows_p-adic}. Note that
\[\varepsilon_{d,S} \leq \varepsilon_{\nu_0, d} = \frac{1}{2} \cdot \left( \frac{1}{2 C}\right)^{\vartheta\inv}. \]
Let $\varepsilon_1 = 2 \varepsilon_{\nu_0, d}$. Then $0 < \varepsilon_{d,S} < \varepsilon_1 < 1$ and $C\varepsilon_1 ^\vartheta = \frac{1}{2}$. 

Let $Y$ be a closed $H\ncc_S$-orbit in $\LatSpaceS{d}$ and take $\Delta \in Y$.  Let $U_{\nu_0}$ be a one-parameter unipotent subgroup of $H_{\nu_0}$ with non-trivial projection to every simple factor of $H_{\nu_0}$. The action of $U_{\nu_0}$ on $Y$ by left multiplication is ergodic by Lemma \ref{Unipotent_groups_act_ergodically}. By Birkhoff's Theorem---see \cite[Chapter 6, Corollary 3.2]{tempelman_ergodic_1992}---there is a co-null subset $E$ of $H\ncc_S$ such that for any $h \in E$,
\begin{align*}
\frac{\muY(\mathfrak{O}_{d,S} \cap Y)}{vol\, Y} & = \lim_{T \to \infty} \frac{\Haar{\QQ_{\nu_0}}( \{t \in B_{\nu_0}(T) \mid u_t h \Delta \in \mathfrak{O}_{d,S} \})}{\Haar{\QQ_{\nu_0}}(B_{\nu_0}(T))}  \\
			& = \lim_{T \to \infty} \frac{\Haar{\QQ_{\nu_0}}( \{t \in B_{\nu_0}(T) \mid h\inv u_t h \Delta \in h\inv \mathfrak{O}_{d,S} \})}{\Haar{\QQ_{\nu_0}}(B_{\nu_0}(T))}.
\end{align*} 
Notice that $\mathfrak{S}_{d,S}(\varepsilon_1)$ is contained in the interior of $\mathfrak{O}_{d,S}$ because $\varepsilon_{d,S} < \varepsilon_1$. We choose $h_0 \in E$ close enough to $I_d$ so that $\mathfrak{S}_{d,S}(\varepsilon_1)$ is still contained in $h_0\inv \mathfrak{O}_{d,S}$. Moreover, we ask that $h_0\inv U_{\nu_0} h_0$ does not preserve the $\QQ_S$-module generated any $\Delta' \in \Sigma_{<1}(\Delta)$. This is possible since, by Lemma \ref{Few_conjugates_preserve_a_proper_subspace}, the $h \in H_{\nu_0}$ such that $h\inv U_{\nu_0} h$ preserve $\langle \Delta' \rangle_{\QQ_S}$ form a null subset of $H_{\nu_0}$, and $\Sigma_{<1}(\Delta)$ is finite by Lemma \ref{Sigma<1_finite}.  Thus
 \[ \Haar{\QQ_{\nu_0}}( \{t \in B_{\nu_0}(T) \mid h_0\inv u_t h_0 \Delta \in h_0 \inv \mathfrak{O}_{d,S} \})
 \geq \Haar{\QQ_{\nu_0}}( \{t \in B_{\nu_0}(T) \mid h_0\inv u_t h_0 \Delta \in \mathfrak{S}_{d,S}(\varepsilon_1) \}). \]
By Proposition \ref{Effective_recurrence_unipotent_flows_p-adic}, for $T \gg 1$ we have
\[ \frac{\lambda_{\QQ_{\nu_0}}( \{t \in B_{\nu_0}(T) \mid h_0\inv u_t h_0 \Delta \in \mathfrak{S}_{d,S}(\varepsilon_1) \})}{\lambda_{\QQ_{\nu_0}}(B_{\nu_0}(T))} 
\geq 1 - C \varepsilon_1^\vartheta = \frac{1}{2}, \]
so $\muY(\mathfrak{O}_{d,S} \cap Y) \geq \frac{1}{2}\, vol \, Y$. To conclude, note that $\mathfrak{O}_{d,S}$ is contained in $\Omega_{d,S}$ by Lemma \ref{CS_aux_2}, so $\muY(Y \cap \Omega_{d,S}) \geq \muY(\mathfrak{O}_{d,S} \cap Y)$.

\end{proof}
	\subsection{Transversal recurrence of closed $H_S$-orbits}\label{subsec_Transversal_recurrence}

Here $H_S$ is a non-compact orthogonal $\QQ_S$-group of a non-degenerate, diagonal quadratic form on $\QQ_S^d$, for some $d \geq 3$. We will establish the transversal recurrence of sufficiently big closed $H_S$-orbit in $\LatSpaceUnoS{d}$. This result is inspired by \cite[Lemma 15]{li_effective_2016}.  

\begin{Lem}\label{Transversal_recurrence}
Consider a finite set $S = \{\infty\} \cup S_f$ of places of $\QQ$ and $d\geq 3$. Let $H_S$ be the orthogonal group of a non-degenerate diagonal quadratic form on $\QQ_S^d$. Suppose that $H_S$ is non-compact. For any closed $H_S$-orbit $Y$ in $\LatSpaceUnoS{d}$ with $vol\, Y > \consBigOrbits{d} p_S^{4c_d}$, there is an $u \in \GLUnoS{d} - H_S$ such that $u (Y \cap \Omega_{d,S})$ meets $Y$, $\normp{u_p} \leq 1$ for any $p \in S_f$, and 
\[ \normi{u_\infty - I_d} \leq \consCoefTRecurrence{d} p_S^4 (vol\, Y)^{-\frac{1}{\consExpVolHTransversal{d}}}. \]
\end{Lem}	
	
To prove Lemma \ref{Transversal_recurrence} we will thicken $Y$  using a small transversal $W_S(r)$ to $H_S$ in $\GLUnoS{d}$, whose size depends on a parameter $r > 0$. If $wY$ and $Y$ are disjoint for any $w \in W_S(r)-\{I_d\}$, the volume $v_r$ of the box $W_S(r)Y$ is the product of the volumes of $W_S(r)$ and $Y$. But $v_r$ is at most the volume of $\LatSpaceUnoS{d}$, hence $r$ can't be too big. 	
	
Before formalizing this proof, let's give a description of $\LatSpaceUnoS{d}$ better suited for the computations in this subsection. For any place $\nu$ of $\QQ$ we define
	\[G'_{d,\nu} = \{g \in G_{d,\nu} \mid |\det g|_\nu = 1 \}, \]
and $G'_{d,S} = \prod_{\nu \in S} G'_{d,\nu}$. The Haar measure of $\GLUnobis{d}{S}$ is fixed in \eqref{Haar_SL(d,R)} and \eqref{Haar_G_d,p} of Subsection \ref{subsec_vol_X1}. Thanks to the next lemma we can identify $\LatSpaceUnoS{d}$ with $G'_{d,S}/\Gamma'_{d,S}$, where $\Gamma'_{d,S} = \Gamma_{d,S} \cap G'_{d,S}$.

\begin{Lem}\label{G'_covers_X1}
Let $S = \{\infty\} \cup S_f$ be a finite set of places of $\QQ$ and let $d \geq 2$. The group $G'_{d,S}$ acts transitively on $\LatSpaceUnoS{d}$.
\end{Lem}

\begin{proof}
For any $g$ in $G_{d,S}^1$ we have 
\[|\det g_\infty|_\infty \cdot \prod_{p\in S_f} |\det g_p|_p = \hgt_S(\det g) = 1,\] 
so $\det g_\infty$ is a unit in $\ZZ_S$. Then
\[diag(\det g_\infty \inv, 1,\ldots, 1)\]
is in $GL(d,\ZZ_S)$. Let $\gamma_g$ be the diagonal image  of this matrix in $\Gamma_{d,S}$. Then $g x_{d,S}^1 = (g\gamma_g) x_{d,S}^1$, and $g \gamma_g$ is in $G'_{d,S}$. 
\end{proof}

	Now we'll see that a subgroup $\LowTMat{d}{S}$ of lower-triangular matrices of $G'_{d,S}$ is transversal to $H_S$ and  that the Haar measure of $G'_{d,S}$ is the product of the Haar measures of $\LowTMat{d}{S}$ and $H_S$. This will justify that the volume of the box $\LowTMat{d}{S}(r)Y$ mentioned in the introduction is equal to the product of the volumes of $\LowTMat{d}{S}(r)$ and $Y$. We work separately on $G'_{d,\nu}$ for every $\nu \in S$. 
	
	Consider the subgroup $W_{d, \infty}$ of lower-triangular matrices of $G'_{d, \infty}$ with positive entries in the main diagonal. See Subsection \ref{subsec_triangular_gps} for the choice of Haar measure on $\LowTMat{d}{\infty}$.

\begin{Lem}\label{Transversal_to_H_real}
Let $H_\infty$ be the real orthogonal group of a non-degenerate, diagonal quadratic form on $\RR^d$.  
\begin{enumerate}[$(i)$]
	\item The multiplication map $\LowTMat{d}{\infty} \times H_\infty \to \GLUnobis{d}{\infty}$ is injective and the image $\LowTMat{d}{\infty} H_\infty$ is open in $\GLUnobis{d}{\infty}$.
	\item On $\LowTMat{d}{\infty} H_\infty$ we have $\Haar{\GLUnobis{d}{\infty}} = \Haar{\LowTMat{d}{\infty}} \otimes \Haar{H_\infty}$. 
\end{enumerate}
\end{Lem}

\begin{proof}
Since $P_\infty$ is diagonal, the only lower-triangular matrices in $H_\infty$ are those of the form $diag(\varepsilon_1,\ldots, \varepsilon_d)$, with $\varepsilon_i = \pm 1$. Hence $H_\infty \cap W_{d, \infty} = 1$. Take $w_1, w_2 \in W_{d, \infty}$ and $h_1, h_2 \in H_\infty$. Then
\[w_1 h_1 = w_2 h_2 \Leftrightarrow w_2\inv w_1 = h_2 h_1\inv, \]
but this last element is in $H_\infty \cap W_{d, \infty}$, so the equality holds if and only if $w_1 = w_2$ and $h_1 = h_2$. This proves that the multiplication map $\mathcal{M}: W_{d, \infty} \times H_\infty \to W_{d, \infty} H_\infty$ is injective. 

We prove now that $W_{d, \infty} H_\infty$ is open. The group $W_{d, \infty} \times H_\infty$ acts on $G'_{d, \infty}$ by  
\[ (w,h) \cdot g = w g h\inv, \]
and $W_{d, \infty} H_\infty$ is an orbit, thus it suffices to prove that $W_{d, \infty} H_\infty$ contains an open neighborhood of $I_d$ in $G'_{d, \infty}$. This follows from the Inverse Function Theorem: The derivative
\[ D \mathcal{M}_{(I_d, I_d)}: \mathfrak{w}_{d,\infty} \times \mathfrak{h}_\infty \to \mathfrak{sl}(d,\RR)\]
is the map $(v_1, v_2) \mapsto v_1 + v_2$, which is a linear isomorphism. This completes the proof of $(i)$.

We pass to $(ii)$. An homogeneous space of the form $G_0 /H_0$ with $G_0$ and $H_0$ locally compact groups and $H_0$ compact admits a unique (up to multiplication by a positive scalar) Radon measure---see \cite[p. 45]{weil_integration_1940}. Thus $\Haar{\LowTMat{d}{\infty}} \otimes \Haar{H_\infty}$ is the only $(W_{d, \infty} \times H_\infty)$-invariant measure on $W_{d, \infty} H_\infty$. Since $G'_{d, \infty}$ is unimodular, $\lambda_{G'_\infty}$ is $W_{d, \infty}$ invariant on the left and $H_\infty$-invariant on the right, hence
\[ \lambda_{G'_\infty} = c (\Haar{\LowTMat{d}{\infty}} \otimes \Haar{H_\infty}) \]
for some $c > 0$. To prove that $c=1$ is suffices to see that the two measures are defined by the same volume form on $\mathfrak{sl}(d,\RR)$. The base change matrix from the concatenation of the bases in \eqref{Haar_W_infty} and \eqref{Haar_H_infty} to the one of \eqref{Haar_SL(d,R)}
has determinant 1 because it's unipotent, upper-triangular, so we are done. 
\end{proof}

Consider the subgroup $W_{d,p}$ of lower-triangular matrices of $G_{d,p} = GL(d,\QQ_p)$. The Haar measure of $\LowTMat{d}{p}$ is fixed in \eqref{Haar_W_p}. We define $\consEll{p}$ as 2 for $p=2$ and $1$ for any odd prime $p$. 

\begin{Lem}\label{Transversal_to_H_p}
Let $H_p$ be the $p$-adic orthogonal group of a non-degenerate, diagonal quadratic form on $\QQ_p^d$. 
\begin{enumerate}[$(i)$]
	\item The multiplication map $\LowTMat{d}{p}(p^{-\consEll{p}}) \times H_p \to \GL{d}{p}$ is injective and $\LowTMat{d}{p}(p^{-\consEll{p}}) H_p$ is open in $\GL{d}{p}$. 
	\item On $\LowTMat{d}{p}(p^{\consEll{p}}) H_p$ we have $\Haar{\GL{d}{p}} = \Haar{\LowTMat{d}{p}} \otimes \Haar{H_p}$. 
\end{enumerate}
\end{Lem}

\begin{proof}
The matrices in $W_{d,p} \cap H_p$ are of the form $diag(\varepsilon_1, \ldots, \varepsilon_d)$, with $\varepsilon_i = \pm 1$, so $W_{d,p}(p^{-\consEll{p}}) \cap H_p = 1$. This shows that $W_{d,p}(p^{-\consEll{p}}) \times H_p \to G_{d,p}$ is injective. The derivative at $(I_d, I_d)$ of $W_{d,p} \times H_p \to G_{d,p}$ is the addition map
\[ \wgot_{d,p} \times \mathfrak{h}_p \to \mathfrak{gl}(d,\QQ_p), \quad (v_1, v_2) \mapsto v_1 + v_2, \]
which is a linear isomorphism. By the Inverse Function Theorem---see \cite[p. 73]{serre_lie_1992} for a proof that works also in the $p$-adic case---we get that $W_{d,p}(p^{-\consEll{p}}) H_p$ is a neighborhood of $I_d$ in $G_{d,p}$. Thus $W_{d,p}(p^{-\consEll{p}}) H_p$ is open in $G_{d,p}$ since it's a $(W_{d,p}(p^{-\consEll{p}}) \times H_p)$-orbit in $G_{d,p}$. 

By the existence and uniqueness of Haar measures, $W_{d,p}(p^{-\consEll{p}}) H_p$ admits a measure that is $W_{d,p}(p^{-\consEll{p}})$-invariant on the left and $H_p$-invariant on the right, unique up to multiplication by positive constants. Since $\Haar{\LowTMat{d}{p}} \otimes \Haar{H_p}$ and $\lambda_{G_p}$ verify this condition, they differ by multiplication by some $c > 0$. To see that $c = 1$ we use the same argument as in the proof of Lemma \ref{Transversal_to_H_real}.   
\end{proof}

Now we define the group $\LowTMat{d}{S}$ and we combine the two previous results to establish the properties of it we need. First let's introduce some notation. We endow $\mathfrak{gl}(d,\RR)$ with the operator norm $\normop{\cdot}$ with respect to $\normi{\cdot}$ on $\RR^d$. The exponential map is a bijection between the Lie algebra $\wgot_{d, \infty}$ of $\LowTMat{d}{\infty}$ and $W_{d, \infty}$. We define
\[\mathfrak{B}_{\mathfrak{w}_\infty}(r) = \{ v\in \wgot_{d, \infty} \mid \normop{v} < r \}\]
and 
\[W_{d, \infty}(r) = \exp(\mathfrak{B}_{\mathfrak{w}_\infty}(r)). \]
 Similarly, we define
\[W_{d,p}(r) = \{w \in W_{d,p} \mid \normp{w - I_d} \leq r, \normp{w\inv - I_d} \leq r \} \]
for any $r > 0$.  

Now consider
\[W_{d,S} = W_{\infty,d} \times \prod_{p \in S_f} W_{d,p}(p^{-3}),\]
and
\[W_{d,S}(r) = W_{\infty,d}(r) \times \prod_{p \in S_f} W_{d,p}(p^{-3})\]
for any $r > 0$. We endow $H_S$ and $W_{d,S}$ with their respective (left for $W_S$) Haar measures
\[ \Haar{H_S} = \otimes_{\nu \in S} \Haar{H_\nu}, \quad \Haar{\LowTMat{d}{S}} = \otimes_{\nu \in S} \Haar{\LowTMat{d}{\nu}}, \]
with $\Haar{H_\nu}$ and $\Haar{\LowTMat{d}{\nu}}$ as in Subsection \ref{subsec_triangular_gps}. 

 The next properties of $\LowTMat{d}{S}$ follow from lemmas \ref{Transversal_to_H_real} and \ref{Transversal_to_H_p}. 

\begin{Lem}\label{Transversal_H_S}
Let $S= \{\infty\} \cup S_f$ be a finite set of places of $\QQ$ and let $H_S$ be the orthogonal $\QQ_S$-group of a non-degenerate, diagonal quadratic form on $\QQ_S^d$. 
	\begin{enumerate}[$(i)$]
		\item The multiplication map $\LowTMat{d}{S} \times H_S \to \GLUnobis{d}{S}$ is injective and $\LowTMat{d}{S} H_S$ is open in $\GLUnobis{d}{S}$. 
		\item On $\LowTMat{d}{S} H_S$ we have \[\Haar{\GLUnobis{d}{S}} = \Haar{\LowTMat{d}{S}} \otimes \Haar{H_S}. \]
	\end{enumerate}	 
\end{Lem}

We'll also need an estimate of the volume of $\LowTMat{d}{S}(r)$.

\begin{Lem}\label{Volume_W_S}
Let $S = \{ \infty \} \cup S_f$ be a finite set of places of $\QQ$. For any $r \in (0, 1/2)$ we have
\[\consCoefVolHTransversalInf{d} p_S^{-3(\consExpVolHTransversal{d}+1)}
r^{\consExpVolHTransversal{d}} < \Haar{\LowTMat{d}{S}}(\LowTMat{d}{S}(r)) < \consCoefVolHTransversalSup{d} p_S^{-3 (\consExpVolHTransversal{d} + 1)} r^{\consExpVolHTransversal{d}}. \]
\end{Lem}

\begin{proof}
By Lemma \ref{Transversal_H_S} we have
\[\lambda_{W_S}(W_{d,S}(r)) = \lambda_{W_{\infty}} ( W_{\infty,d}(r)) \, \prod_{p \in S_f} \lambda_{W_p}(W_{d,p}(p^{-3})), \]
hence the result follows from the bounds for $\lambda_{W_\infty}(W_{\infty,d} (r))$ of Lemma \ref{Volume_Wr_infty_cite}---which hold for any $r \in \left( 0, \frac{1}{2} \right)$---and the value of $\Haar{\LowTMat{d}{p}}(\LowTMat{d}{p}(p^{-3}))$ obtained in Lemma \ref{Volume_W_p_cite}. 
\end{proof}	
	
We'll use two more auxiliary results to prove Lemma \ref{Transversal_recurrence}. Let $\Omega_{d,S}$ be the compact subset of $\LatSpaceUnoS{d}$ of Lemma \ref{Compact_meeting_closed_H_S-orbits}. Let $Y$ be a closed $H_S$-orbit in $X_{d,S}^1$. We denote by $\Psi^r_Y$ the map
\[ W_{d,S}(r) \times (Y \cap \Omega_{d,S}) \to X_{d,S}^1,\quad (w, y) \mapsto wy.\] 
We denote the volume of $X^1_{d,\infty}$ by $\consVolXUno{d}{\infty}$.  For the next lemma we define
\[ \consRecurrenceBox{d} = \frac{2 \consVolXUno{d}{\infty}^\frac{1}{\consExpVolHTransversal{d} }}{d(d-1)}. \]	
	
\begin{Lem}\label{Injective_implies_small_r}
Let $d\geq 3$ and let $Y$ be a closed $H_S$-orbit in $\LatSpaceUnoS{d}$. If $\Psi^r_Y$ is injective and $r < \frac{1}{2}$, then
\[ r < \consRecurrenceBox{d} p_S^4 (vol\,Y)^{-\frac{1}{\consExpVolHTransversal{d}}}. \]
\end{Lem}

\begin{proof}
Let $y_0$ be a point in $Y \cap \Omega_{d,S}$ and let $\widetilde{Y}$ be a measurable subset of $H_S$ such that
 \[ \widetilde{Y} \to Y \cap \Omega_{d,S}, \quad h \mapsto h y_0 \]
is bijective. If $\Psi^r_Y$ in injective, then 
\[ W_{d,S}(r) \widetilde{Y} \to X_{d,S}^1, \quad wh \mapsto wh y_0 \]
is also injective, hence
\[\beta_{d,S}(W_{d,S}(r) (Y \cap \Omega_{d,S})) = \Haar{\GLUnobis{d}{S}}(W_{d,S}(r) \widetilde{Y}). \]
We know that $\Haar{\GLUnobis{d}{S}} = \lambda_{\LowTMat{d}{S}} \otimes \lambda_{H_S}$ on $W_{d,S} H_S$---see Lemma \ref{Transversal_H_S}---, so 
\begin{align*}
\Haar{\GLUnobis{d}{S}}(W_{d,S}(r) \widetilde{Y}) & = \Haar{W_{d,S}}(W_{d,S}(r)) \Haar{H_S}(\widetilde{Y}) \\
		& = \lambda_{W_{d,S}}(W_{d,S}(r)) \muY(Y \cap \Omega_{d,S}) \\
		& > \left(\consCoefVolHTransversalInf{d} p_S^{-3(\consExpVolHTransversal{d}+1)} r^{\consExpVolHTransversal{d}} \right) \left( \frac{vol\,Y}{2} \right).
\end{align*}
To obtain the last inequality we used Lemma \ref{Volume_W_S} and Lemma \ref{Compact_meeting_closed_H_S-orbits}. The volume of $W_{d,S}(r) (Y \cap \Omega_{d,S})$ is strictly smaller than $2 \beta_{d,S}(X_{d,S}^1)$ and $\beta_{d,S}(X_{d,S}^1)< \consVolXUno{d}{\infty}$ by Lemma \ref{Volume_X_Sd^1}, hence
\[ \consCoefVolHTransversalInf{d} p_S^{-3 (\consExpVolHTransversal{d}+1)} (vol\, Y) r^{\consExpVolHTransversal{d}} < 2 \consVolXUno{d}{\infty}.\]
We finally get
\[r < \left( \frac{2^{d+1}}{(d(d-1))^{\consExpVolHTransversal{d}}} p_S^{3(\consExpVolHTransversal{d} + 1)} \consVolXUno{d}{\infty} (vol\, Y)\inv \right)^\frac{1}{c_d} < \consRecurrenceBox{d} p_S^4 (vol\, Y)^{- \frac{1}{\consExpVolHTransversal{d}}}. \]
\end{proof}

\begin{Lem}\label{Lolito_4}
Any $r \in (0, 1/2)$ verifies
\[e^{2r} - 1 < 4 r. \]
\end{Lem}

\begin{proof}
The function $\frac{1}{r}(e^{2r} - 1)$ is increasing on $(0, \infty)$, so 
\[\frac{e^{2r} - 1}{r} < 2(e-1) < 4 \]
for any $r \in \left( 0, \frac{1}{2} \right)$. 
\end{proof}

We are ready to prove the transversal recurrence of closed $H_S$-orbits.

\begin{proof}[Proof of Lemma \ref{Transversal_recurrence}]
Let $Y$ be a closed $H_S$-orbit in $X_{d,S}^1$. Recall that 
\[ \consBigOrbits{d} = \left(\frac{4}{d(d-1)}\right)^{\consExpVolHTransversal{d}}  \consVolXUno{d}{\infty} \hspace{1cm} \text{and} \hspace{1cm} \consRecurrenceBox{d} = \frac{2 \consVolXUno{d}{\infty}^\frac{1}{\consExpVolHTransversal{d}}}{d(d-1)}.\]
We define
\[r_Y = \consRecurrenceBox{d} p_S^4 (vol\, Y)^{-\frac{1}{\consExpVolHTransversal{d}}}. \]
Notice that $r_Y < \frac{1}{2}$ if and only if $vol\, Y > \consBigOrbits{d} p_S^{4\consExpVolHTransversal{d}}$. Suppose that this is the case. Then $\Psi^{r_Y}_Y$ isn't injective by Lemma \ref{Injective_implies_small_r}. Take $w \neq w'$ in $W_{d,S}(r_Y)$ and $y, y' \in Y \cap \Omega_{d,S}$ such that $w\inv w' y = y'$. We set $u = w\inv w'$. Then $u(Y \cap \Omega_{d,S})$ meets $Y$. We have $w_\nu \neq w'_\nu$ for some $\nu \in S$, hence $u_\nu \notin H_\nu$ by $(i)$ of Lemma \ref{Transversal_to_H_real} or Lemma \ref{Transversal_to_H_p} if $\nu = \infty$ or $\nu = p$, respectively. Thus $u$ is not in $H_S$. Notice that $\normp{u_p} = 1$ for any $p \in S_f$  because $u_p$ is in $W_{d,p}(p^{-3}) \subseteq GL(d,\ZZ_p)$. To conclude we estimate $\normi{u_\infty - I_d}$. By definition of $W_{\infty,d}(r_Y)$,
\[w = \exp v, \quad w' = \exp v' \]
for some $v, v' \in \wgot_{d, \infty}$ with $\normop{v}, \normop{v'} < r_Y$. Then
\begin{align*}
\normi{u_\infty - I_d} & \leq \normop{w\inv_\infty w'_\infty - I_d} \\
						& \leq \normop{w\inv_\infty w'_\infty - w'_\infty} + \normop{w'_\infty - I_d} \\
						& \leq \normop{w\inv_\infty - I_d} \normop{w'_\infty} + \normop{w'_\infty - I_d} \\
						& < (e^{r_Y}-1) e^{r_Y}  + (e^{r_Y}-1) \\
						& = e^{2r_Y} - 1 \\
						& < 4 r_Y  = \consCoefTRecurrence{d} p_S^4 (vol\, Y)^{-\frac{1}{\consExpVolHTransversal{d}}},
\end{align*} 
where $\consCoefTRecurrence{d} = \frac{2^3 \consVolXUno{d}{\infty}^{\consExpVolHTransversal{d}}}{d(d-1)}$. This completes the proof. 
\end{proof}

\subsection{The main proof}

Now we combine the main results of subsections \ref{subsec_Transversal_isolation}, \ref{subsec_Uniform_recurrence} and \ref{subsec_Transversal_recurrence} to obtain our upper bound for the volume of closed $H_S$-orbits in $\LatSpaceS{d}$.

\begin{proof}[Proof of Proposition \ref{Main_volume_H_S-orbits}]
Let $Q$ be a non-degenerate integral quadratic form in $d \geq 3$ variables. Suppose that $Q_S$ is $\QQ_S$-isotropic. Let $P$ be the standard quadratic form on $\QQ_S^d$ that is $\QQ_S$-equivalent to $Q_S$ and set $H_S = O(P, \QQ_S)$. Recall that the the $H_S$-orbits $Y_{Q,S}$ and $Y:=Y^1_{Q,S}$---respectively in $X_{d,S}$ and $X_{d,S}^1$---have the same volume.

Let $\consBigOrbits{d}$ be as in Lemma \ref{Transversal_recurrence}. We consider two cases:
	\begin{itemize}	
		\item \framebox{$vol\, Y > \consBigOrbits{d} p_S^{4\consExpVolHTransversal{d}}$} Recall that we defined $\widetilde{\Omega}_{d,S}(M)$ in \eqref{def_lift_big_compact}. By Lemma \ref{Transversal_recurrence} there is $g$ respectively in $\widetilde{\Omega}_{d,S}(\consDefBigCompact{d} 2^{d^4})$ and $\widetilde{\Omega}_{d,S}(\consDefBigCompact{d} p_S^{2d^4})$  if  $S = \{\infty\}$ and $S \neq \{\infty\}$, as well as $u \in G_{d,S}^1 - H_S$ with $\normp{u_p} \leq 1$ for any $p \in S_f$ and
		\[ \normi{u_\infty - I_d } \leq \consCoefTRecurrence{d} p_S^4 (vol\, Y)^{-\frac{1}{\consExpVolHTransversal{d}}},  \]
such that $g x_{d,S}^1$ and $ugx_{d,S}^1$ are in $Y$. We also know that
\[\normi{u_\infty - I_d} \geq \frac{1}{2d^3} p_S\inv \hgt_S(g)^{-2} \hgt_S(\det Q)^{-\frac{1}{d}}, \]
		by Lemma \ref{Transversal_isolation}. Let's consider first the case $S \neq \{\infty\}$. It follows that
	
	\begin{align*}
	vol\, Y & < \left( \frac{2^4 d^2}{d-1} \hgt_S(g)^2 p_S^5 \consVolXUno{d}{\infty}^{\frac{1}{\consExpVolHTransversal{d}}} \hgt_S(\det Q)^\frac{1}{d}\right)^{\consExpVolHTransversal{d}}\\
	& < (2^5 d (\consDefBigCompact{d} p_S^{2d^4})^2 p_S^5)^{\consExpVolHTransversal{d}} \consVolXUno{d}{\infty} \hgt_S(\det Q)^\frac{\consExpVolHTransversal{d}}{d}\\
	& < \Fc_d \consVolXUno{d}{\infty} p_S^{3d^6} \hgt_S(\det Q)^\frac{d+1}{2},
\end{align*}	
where $\Fc_d = (2^{2d^3+5} \cdot 3^{4d^4} d^{6d^3+1})^{\consExpVolHTransversal{d}}$. When $S = \{\infty\}$, a similar computation yields
\[vol\, Y < \Fc_d \consVolXUno{d}{\infty} 2^{2d^6} |\det Q|_\infty^\frac{d+1}{2}.\]

\item \framebox{$vol\, Y \leq \consBigOrbits{d} p_S^{4 \consExpVolHTransversal{d}}$} Since $\hgt_S(\det Q)$ is a positive integer, we have
\[ vol\, Y  \leq \consBigOrbits{d} p_S^{4 \consExpVolHTransversal{d}} \leq \consBigOrbits{d} p_S^{4 \consExpVolHTransversal{d}} \hgt_S(\det Q)^ \frac{d+1}{2}.\]	
Since $\consBigOrbits{d} p_S^{4 \consExpVolHTransversal{d}}$ is smaller than $\Fc_d \consVolXUno{d}{\infty} p_S^{3d^6}$ and $\Fc_d \consVolXUno{d}{\infty} 2^{2d^6}$, in both cases we get the desired inequality.	
	\end{itemize}

\end{proof}

\section{Effective criteria of $\ZZ_S$-equivalence}\label{sec_Zs-equiv_criteria}

The \textit{$\ZZ$-equivalence problem for integral quadratic forms} asks for a way to decide if any two given integral quadratic forms $Q_1$ and $Q_2$ are $\ZZ$-equivalent. As we mentioned in Section \ref{sec_intro}, \cite[Theorem 1]{li_effective_2016} of Li and Margulis settles the problem for quadratic forms in $d \geq 3$ variables: it suffices to look for an equivalence matrix in a finite subset of $GL(d,\ZZ)$ that depends only on $\normi{Q_1}, \normi{Q_2}$. The goal of this section is to decide in a similar way if $Q_1$ and $Q_2$ are $\ZZ_S$-equivalent, for any finite set $S = \{\infty\} \cup S_f$ of places of $\QQ$. More precisely, we will give explicit constants $c_\nu(Q_1, Q_2)$ such that $Q_1$ and $Q_2$ are $\ZZ_S$-equivalent if and only if there is an equivalence matrix $\gamma_0 \in GL(d,\ZZ_S)$ between them, such that $\normnu{\gamma_0 } \leq c_\nu(Q_1, Q_2)$ for any $\nu \in S$.  The interesting case is when the quadratic forms are $\QQ_\nu$-isotropic for some $\nu \in S$ \footnote{When $Q_1$ and $Q_2$ are $\QQ_\nu$-anisotropic, the set of $g_\nu \in GL(d,\QQ_\nu)$ taking $Q_1$ to $Q_2$ is compact, so in that case it's easy to give an upper bound $c_\nu(Q_1, Q_2)$ for the $\nu$-norm of any $g_\nu \in GL(d,\QQ_\nu)$ taking $Q_1$ to $Q_2$. Thus deciding the $\ZZ_S$-equivalence o f $Q_1$ and $Q_2$ is easy when the $Q_i$'s are $\QQ_\nu$-anisotropic for any $\nu \in S$.}. The two results of this section are Theorem \ref{Z_S-equivalence_R-isotropic} and Theorem \ref{Z_S-equivalence_R-anisotropic}, which treat respectively the case of $\RR$-isotropic and $\RR$-anisotropic quadratic forms.

\begin{Theo}\label{Z_S-equivalence_R-isotropic}
Consider a finite set $S = \{\infty\} \cup S_f$ of places of $\QQ$. Let $Q_1$ and $Q_2$ be integral quadratic forms in $d \geq 3$ variables which are non-degenerate and $\RR$-isotropic. Then $Q_1$ and $Q_2$ are $\ZZ_S$-equivalent if and only if there is $\gamma_0 \in GL(d,\ZZ_S)$ with
\begin{align*}
\normi{\gamma_0} &< \consZSEquivCritRiso{d} p_S^{19 d^6} (\normi{Q_1} \normi{Q_2})^{2d^3}, \\
\normp{\gamma_0} & \leq p |\det Q_1|_p^{-\frac{1}{2}} \quad \text{for any odd } p \in S_f, \\
\norm{\gamma_0}_2 &\leq 2^{d+2} |\det Q_1|_2^{-\frac{1}{2}} \quad \quad \text{if } 2 \in S_f,
\end{align*}
such that $Q_1 \circ \gamma_0 = Q_2$. 
\end{Theo}

\begin{proof}
If there is a $\gamma_0$ as in the statement, it is immediate that $Q_1$ and $Q_2$ are $\ZZ_S$-equivalent. Suppose now that $Q_1$ and $Q_2$ are $\ZZ_S$-equivalent. Let $P = (P_\nu)_{\nu \in S}$ be the standard quadratic form on $\QQ_S^d$ that is $\QQ_S$-equivalent to $(Q_1)_S$ and $(Q_2)_S$\footnote{Recall that $(Q_i)_S$ is the quadratic form on $\QQ_S^d$ obtained from $Q_i$ using the diagonal embedding $\QQ \to \QQ_S$.}, and let $H_S = O(P, \QQ_S)$. By Lemma  \ref{SmallStandMatForQF} there is $f \in \GLS{d}$ such that $(Q_1)_S = P \circ f$ and, for any $\nu \in S$,
\begin{equation}\label{BBB} 
\normnu{f_\nu} \leq \consSmallStandMatForQF{d}{\nu} \normnu{Q_1}^{\frac{1}{2}}.  
\end{equation}
Similarly, we choose $g \in \GLS{d}$ such that $(Q_2)_S = P \circ g$ and with the $g_\nu$ verifying the corresponding inequalities for $Q_2$. Note that $f \BasePointS{d}$ and $g \BasePointS{d}$ are in the $H_S$-orbit $Y = \OrbitS{Q_1}$, which is closed in $\LatSpaceS{d}$ by Lemma \ref{Y_Q,S_is_closed}. Since the $Q_i$ are $\RR$-isotropic, $H_\infty$ is noncompact. Thus by Proposition \ref{Dynamical_statement_RR-isotropic} there is an $h^\star \in H_S$ with
\begin{align*}
\normi{h^\star_\infty} & < \consDynStRiso{d} p_S^{9d^3} (T_\infty(f) T_\infty(g))^{\frac{3}{2}d(d-1)+6} (T_{S_f}(f) T_{S_f}(g))^{3d^2}  (vol\, Y)^6, \\
\normp{h^\star_p} & \leq p \quad  \text{for any odd } p \in S_f,\\
\norm{h^\star_2}_2 & \leq 4 \quad \text{if } 2 \in S_f,
\end{align*}
such that $h^\star g \BasePointS{d} = f \BasePointS{d}$. Hence $f\inv h^\star g$ is of the form $(\gamma_0, \ldots, \gamma_0)$ because it is in $\Gamma_{d,S}$. Moreover, $\gamma_0$ takes $Q_1$ to $Q_2$ since $f\inv h^\star g$ takes $(Q_1)_S$ to $(Q_2)_S$.

Now we relate $T_\nu(f)$ and $T_\nu(g)$ to $Q_1$ and $Q_2$.  For any $p \in S_f$, $\normp{Q_1} \leq 1$ because $Q_1$ is integral, so $\normp{f_p} \leq \consSmallStandMatForQF{d}{p}$. Since $\normp{f_p}$ is an integral power of $p$, then $\normp{f_p} \leq 1$ for any odd $p$ and $\norm{f_2}_2 \leq 2$ if $2 \in S_f$. Thus, for any odd $p \in S_f$ we get
\begin{equation}\label{Tp_bound}
T_p(f) = \frac{\normp{f_p}^d}{|\det f_p|_p} \leq 
\left(\frac{|\det P_p|_p}{|\det Q_1|_p} \right)^\frac{1}{2} 
\leq |\det Q_1|_p^{-\frac{1}{2}}, 
\end{equation}
and if $ 2 \in S_f$,
\[ T_2(f) \leq 2^d |\det Q_1|_2^{-\frac{1}{2}}. \]
For $T_\infty$ we have
\begin{equation}\label{Tinfty_bound}
T_\infty(f) = \frac{\normi{f_\infty}^d}{|\det f_\infty|_\infty} \leq d^d \left( \frac{\normi{Q_1}^d}{|\det Q_1|_\infty} \right)^\frac{1}{2}. 
\end{equation}
Similar bounds in terms of $Q_2$ hold for $T_p(g)$ and $T_\infty(g)$. By Proposition \ref{Main_volume_H_S-orbits} we know that
\begin{equation} \label{Volume_bound}
vol\, Y \leq \consVolClosedOrb{d} p_S^{3d^6} \hgt_S(\det Q_1)^\frac{d+1}{2}.
\end{equation}
Note that $\hgt_S(\det Q_1) = \hgt_S(\det Q_2)$ since $Q_1$ and $Q_2$ are $\ZZ_S$-equivalent, so we can replace $\hgt_S(\det Q_1)$ in the above inequality by $\hgt_S(\det Q_1 \det Q_2)^{\frac{1}{2}}$ to make it symmetric in $Q_1, Q_2$. Thus
\begin{align*}
\normi{h^\star_\infty} & < \consDynStRiso{d} p_S^{9d^3} 
\left( d^{2d} \frac{\normi{Q_1}^\frac{d}{2} \normi{Q_2}^\frac{d}{2}}{\sqrt{|\det Q_1 \det Q_2|_\infty}} \right)^{\frac{3}{2}d(d-1) + 6} (2^d \hgt_{S_f}(\det Q_1 \det Q_2)^{-\frac{1}{2}})^{3d^2} \cdot \\ 
& \hspace{9cm} \cdot (\consVolClosedOrb{d} p_S^{3d^6} \hgt_S(\det Q_1 \det Q_2)^{\frac{d+1}{4}})^6\\
	&\leq \mathcal{J}_d p_S^{18d^6 + 9d^3} 
	(\normi{Q_1} \normi{Q_2})^{\frac{3}{4}d^2(d-1) + 3d} 
	\frac{\hgt_S(\det Q_1 \det Q_2) ^{\frac{3}{2}(d+1)}}{|\det Q_1 \det Q_2|_\infty ^{\frac{3}{4}d(d-1) + 3} \hgt_{S_f} (\det Q_1 \det Q_2)^{\frac{3}{2}d^2}}, \\ 
\end{align*} 
where\footnote{When $2 \notin S_f$ we can take a smaller $\mathcal{J}_d$, namely $d^{3d^3 + 9} \consDynStRiso{d} (\consVolClosedOrb{d})^6 $.} $\mathcal{J}_d = 2^{6d^3} d^{3d^3 + 9} \consDynStRiso{d} (\consVolClosedOrb{d})^6 $. Since $\hgt_S(\det Q_1 \det Q_2)$ is a positive integer and $d \geq 3$, then
\begin{align*}
\hgt_S(\det Q_1 \det Q_2)^{\frac{3}{2}(d+1)} & \leq \hgt_S(\det Q_1 \det Q_2)^{\frac{3}{2}d^2} \\\
 & = |\det Q_1 \det Q_2|_\infty^{\frac{3}{2}d^2} \hgt_{S_f}(\det Q_1 \det Q_2)^{\frac{3}{2}d^2},
 \end{align*}
so 
\[\frac{\hgt_S(\det Q_1 \det Q_2) ^{\frac{3}{2}(d+1)}}{|\det Q_1 \det Q_2|_\infty ^{\frac{3}{4}d(d-1) + 3} \hgt_{S_f} (\det Q_1 \det Q_2)^{\frac{3}{2}d^2}} \leq \frac{|\det Q_1 \det Q_2|_\infty^{\frac{3}{2}d^2}}{|\det Q_1 \det Q_2|_\infty ^{\frac{3}{4}d(d-1) + 3}} \leq |\det Q_1 \det Q_2|_\infty^{d^2}.\]
Thus we obtain
\[\normi{h^\star_\infty} \leq \mathcal{J}_d p_S^{19d^6} (\normi{Q_1} \normi{Q_2})^{\frac{3}{4}d^2(d-1) + 3d} |\det Q_1 \det Q_2|_\infty^{d^2}.\]
We are ready to bound $\gamma_0$:
\begin{align*}
\normi{\gamma_0}  = \normi{f_\infty\inv h^\star_\infty g_\infty} & \leq d^2 \normi{f_\infty\inv} \normi{g_\infty} \normi{h^\star_\infty} \\
	& \leq d \cdot d! \frac{\normi{f_\infty}^{d-1}}{|\det f_\infty|_\infty} \normi{g_\infty} \normi{h^\star_\infty} \\
	& \leq d^{d+1} \cdot d! \normi{Q_1}^\frac{d-1}{2} \normi{Q_2}^\frac{1}{2} \normi{h^\star_\infty} \\
	& \leq (d^{d+1} \cdot d! \mathcal{J}_d) p_S^{19 d^6} (\normi{Q_1} \normi{Q_2})^{(\frac{3}{4}d^2 + \frac{1}{2})(d-1) + 3d}  |\det {Q_1} \det {Q_2}|_\infty^{d^2} \\
	&\leq (d^{d+1} \cdot d! \mathcal{J}_d) p_S^{19 d^6} (\normi{Q_1} \normi{Q_2})^{d^3} |\det {Q_1} \det {Q_2}|_\infty^{d^2} \\
	& \leq (2^{6d^3} d^{4d^3} \cdot d!^{2d^2 + 1} \consDynStRiso{d} (\consVolClosedOrb{d})^6)  p_S^{19 d^6} (\normi{Q_1} \normi{Q_2})^{2d^3}.
\end{align*}
Finally, for any odd $p \in S_f$ we have
\[\normp{\gamma_0} \leq \normp{f_p\inv} \normp{g_p} \normp{h^\star_p} \leq p |\det Q_1|_p^{-\frac{1}{2}}, \]
and when $2 \in S_f$, 
 \[ \norm{\gamma_0}_2 \leq 2^{d+2} |\det Q_1|_2^{-\frac{1}{2}}.\] 
\end{proof}

Now we'll consider the case of quadratic forms that are $\RR$-anisotropic and $\QQ_{p_0}$-isotropic for some $p_0 \in S_f$. The proof is almost identical to that of Theorem \ref{Z_S-equivalence_R-isotropic}, but we obtain better estimates for the equivalence matrix $\gamma_0$ thanks to the faster mixing rate of $H\ncc_{p_0} \curvearrowright \OrbitS{Q_1}$.

\begin{Theo}\label{Z_S-equivalence_R-anisotropic}
Consider a finite set $S = \{\infty\} \cup S_f$ of places of $\QQ$. Let $Q_1$ and $Q_2$ be integral quadratic forms in $d \geq 3$ variables which are $\RR$-anisotropic and $\QQ_{p_0}$-isotropic for some $p_0 \in S_f$. Then $Q_1$ and $Q_2$ are $\ZZ_S$-equivalent if and only if there is $\gamma_0 \in GL(d,\ZZ_S)$ with
\begin{align*}
\normpo{\gamma_0} &< \consZSEquivCritRani{d} p_S^{13 d^6} (\normi{Q_1} \normi{Q_2})^{\frac{1}{2}d^3 + 3d}, \\
\normp{\gamma_0} & \leq p |\det Q_1|_p^{-\frac{1}{2}} \quad \quad \text{for any odd } p \in S_f - \{p_0\},\\ 
\norm{\gamma_0}_2 & \leq 2^{d+2} |\det Q_1|_2^{-\frac{1}{2}} \quad \quad \text{if } 2 \in S_f - \{p_0\}, \\
\normi{\gamma_0} & \leq d^{d+1} \cdot d! \normi{Q_1}^{\frac{d-1}{2}} \normi{Q_2}^\frac{1}{2},
\end{align*}
such that $Q_1 \circ \gamma_0 = Q_2$.
\end{Theo}

\begin{Rem}\label{rem_pos_def_qf}
The last inequality in Theorem \ref{Z_S-equivalence_R-anisotropic} is in fact verified by any matrix in $GL(d,\RR)$ taking $Q_1$ to $Q_2$ because the quadratic forms are $\RR$-anisotropic. We included it in the statement to have a bound on $\normnu{\gamma_0}$ for any $\nu \in S$.  
\end{Rem}

\begin{proof}
If there is a $\gamma_0$ as in the statement, clearly $Q_1$ and $Q_2$ are $\ZZ_S$-equivalent. Assume now that $Q_1$ and $Q_2$ are $\ZZ_S$-equivalent. Let $P$ be the standard quadratic form on $\QQ_S^d$ that is $\QQ_S$-equivalent to $(Q_1)_S$ and $(Q_2)_S$, and let $H_S = O(P, \QQ_S)$.  By Lemma \ref{SmallStandMatForQF} there are $f,g \in G_{d,S}$ such that
\[(Q_1)_S = P \circ f \quad \text{and} \quad (Q_2)_S = P \circ g,\]
with $f_\nu$ verifying \eqref{BBB} and $g_\nu$ verifying the corresponding inequality for $Q_2$, for any $\nu \in S$. The points $f \BasePointS{d}$ and $g \BasePointS{d}$ lie in the $H_S$-orbit $Y := \OrbitS{Q_1}$, which is closed in $\LatSpaceS{d}$ by Lemma \ref{Y_Q,S_is_closed}. Note that $H_\infty$ is compact and $H_{p_0}$ is noncompact because the $Q_i$'s are $\RR$-anisotropic and $\QQ_{p_0}$-isotropic. Thus by Proposition \ref{Dynamical_statement_R-anisotropic} there is an $h^\star \in H_S$ with
\begin{align*}
 \normpo{h^\star_{p_0}} & < \consDynStRani{d} p_S^{6 d^3} (T_{p_0}(f) T_{p_0}(g))^6 (T_S(f) T_S(g))^{d(d-1)} (vol\, Y)^4, \\
\normp{h^\star_p} & \leq p \quad  \text{for any odd } p \in S_f-\{p_0\},\\
\norm{h^\star_2}_2 & \leq 4 \quad \text{if } 2 \in S_f - \{p_0\},
\end{align*} 
such that $h^\star g \BasePointS{d} = f \BasePointS{d}$. Then $f\inv h^\star g = (\gamma_0, \ldots, \gamma_0)$ for some $\gamma_0 \in GL(d,\ZZ_S)$ taking $Q_1$ to $Q_2$. Let's see that $\gamma_0$ verifies the inequalities of our statement. The inequalities \eqref{Tp_bound} and \eqref{Tinfty_bound} for $T_\nu(f)$ and $T_\nu(g)$, as well as the  bound \eqref{Volume_bound} for $vol\, Y$  hold also in the current situation, so 
\begin{align*}
\normpo{h^\star_{p_0}} & \leq \consDynStRani{d} p_S^{6 d^3} \left( 2^{2d}|\det Q_1 \det Q_2|_{p_0}^{-\frac{1}{2}} \right)^6 \left((2d^{2d}) \frac{\normi{Q_1}^\frac{d}{2} \normi{Q_2}^\frac{d}{2}}{\sqrt{\hgt_S(\det Q_1 \det Q_2)}} \right)^{d(d-1)} \cdot \\
& \hspace{9cm} \cdot (\consVolClosedOrb{d} p_S^{3d^6} \hgt_S(\det Q_1)^\frac{d+1}{2})^4 \\
	& \leq \Cc'_{a,d} p_S^{13 d^6} (\normi{Q_1} \normi{Q_2})^{\frac{1}{2}d^2(d-1)} |\det Q_1 \det Q_2|_\infty^3,
\end{align*}
where\footnote{If $2 \notin S$ we can take $\Cc'_{a,d} = d^{2d^3} \consDynStRani{d}(\consVolClosedOrb{d})^4$.} $\Cc'_{a,d} = 2^{2d^3 + 6d} d^{2d^3} \consDynStRani{d}(\consVolClosedOrb{d})^4 $. Then
\begin{align*}
\normpo{\gamma_0}  = \normpo{f_{p_0}\inv h^\star_{p_0} g_{p_0}} & \leq \frac{\normpo{f_{p_0}}^{d-1}}{|\det f_{p_0}|_{p_0}} \normpo{g_{p_0}} \normpo{h^\star_{p_0}} \\
	& \leq 2^d |\det Q_1|_{p_0}^{-\frac{1}{2}} \normpo{h^\star_{p_0}} \\
	& \leq 2^d \Cc'_{a,d} p_S^{13 d^6} (\normi{Q_1} \normi{Q_2})^{\frac{1}{2}d^2(d-1)} |\det Q_1 \det Q_2|_\infty^{\frac{7}{2}} \\
	& \leq 2^d \cdot d!^7 \Cc_{a,d} p_S^{13 d^6} (\normi{Q_1} \normi{Q_2})^{\frac{1}{2}d^2(d-1) + \frac{7}{2}d^2}\\
	& = (2^{2d^3 + 7d} d^{2d^3} \cdot d!^7 \consDynStRani{d}(\consVolClosedOrb{d})^4)  p_S^{13 d^6} (\normi{Q_1} \normi{Q_2})^{\frac{1}{2}d^3 + 3d}.
\end{align*} 
For any odd $p \in S_f$ we have
\[ \normp{\gamma_0}  \leq \frac{\normp{f_p}^{d-1}}{|\det f_p|_p} \normp{g_p} \normp{h^\star_p} \leq p |\det Q_1|_p^{-\frac{1}{2}},  \]
and if $2 \in S_f - \{p_0\}$,
\[ \norm{\gamma_0}_2 \leq \frac{\norm{f_2}_2^{d-1}}{|\det f_2|_2} \norm{g_2}_2 \norm{h^\star_2} \leq 2^{d+2} |\det Q_1|_2^{-\frac{1}{2}}. \]
To conclude we bound the $\infty$-norm of $\gamma_0$. Recall that $H_\infty = O(d,\RR)$, so $\normi{h^\star_\infty} \leq 1$.
\begin{align*}
\normi{\gamma_0}  = \normi{f_\infty \inv h^\star_\infty g_\infty} & \leq d \cdot d! \frac{\normi{f_\infty}^{d-1}}{|\det f_\infty|_\infty} \normi{g_\infty} \\
	& \leq d^{d+1} \cdot d! \normi{Q_1}^{\frac{d-1}{2}} \normi{Q_2}^\frac{1}{2}. 
\end{align*}
\end{proof}

\section{Small generators of $S$-integral orthogonal groups}\label{sec_small_gens}

In its landmark paper \cite{siegel_zur_1972}, C.L. Siegel proved that the integral orthogonal group of a rational quadratic form is finitely generated. This result was made effective by Li and Margulis in \cite[Theorem 2]{li_effective_2016}. The goal of this section is to extend the Li-Margulis theorem to $S$-integral orthogonal groups.  

Our main result is divided into two parts, depending on whether the quadratic form is $\RR$-isotropic or $\RR$-anisotropic.

\begin{Theo}\label{Small_generators_R-isotropic}
Let $S = \{\infty\} \cup S_f$ be a finite set of places of $\QQ$. For any non-degenerate, $\RR$-isotropic integral quadratic form $Q$ in $d \geq 3$ variables, the group $O(Q,\ZZ_S)$ is generated by the $\xi \in O(Q,\ZZ_S)$ with 
\begin{align*}
\normi{\xi} & \leq \consSmallGensRiso{d} p_S^{20d^7} \normi{Q}^{5 d^6}, &\\
\normp{\xi} & \leq p^{2d+2} |\det Q|_p^{-\frac{1}{2}} \quad \quad \text{for any odd } p \in S_f,\\
\norm{\xi}_2 & \leq 2^{d^2 + 3d + 3} |\det Q|_2^{-\frac{1}{2}} \quad \quad \text{if } 2 \in S_f.    
\end{align*}
\end{Theo}

\begin{Theo}\label{Small_generators_R-anisotropic}
Let $S = \{\infty\} \cup S_f$ be a finite set of places of $\QQ$. For any integral quadratic form $Q$ in $d \geq 3$ variables that is $\RR$-anisotropic and $\QQ_{p_0}$-isotropic for some $p_0 \in S_f$, the group $O(Q,\ZZ_S)$ is generated by the $\xi \in O(Q,\ZZ_S)$ with
\begin{align*}
\normpo{\xi} & \leq \consSmallGensRani{d}  p_S^{14d^7} \normi{Q}^{3d^6} |\det Q|_{p_0}^{-\frac{1}{2}},  \\
\normp{\xi} & \leq p^{2d+2} |\det Q|_p^{-\frac{1}{2}} \quad \quad \text{for any odd } p \in S_f - \{p_0\},\\
\norm{\xi}_2 & \leq 2^{d^2 + 3d + 3} |\det Q|_2^{-\frac{1}{2}} \quad \quad \text{if } 2 \in S_f - \{p_0\}, \\
\normi{\xi} & \leq d^{d+1} \cdot d! \normi{Q}^{\frac{d}{2}}.    
\end{align*}
\end{Theo}

\begin{Rem} The inequality for $\normi{\xi}$ in Theorem \ref{Small_generators_R-anisotropic} holds for any matrix in $O(Q,\RR)$. One can give a similar bound for $\normnu{\xi}$ whenever $Q$ is $\QQ_\nu$-isotropic.
\end{Rem}

Let's introduce the new notation of this section. Let $Q$ be a non-degenerate integral quadratic form in $d$ variables. We denote by $\textbf{H}^Q$ the orthogonal group $\textbf{O}(Q)$ of $Q$. For any finite set $S = \{\infty\} \cup S_f$ of places of $Q$, we denote by $\Gamma^Q_S$ the diagonal copy of $O(Q,\ZZ_S)$ in $H^Q_S$. From now on we'll work with $\Gamma^Q_S$ instead of $O(Q,\ZZ_S)$ in order to bound at the same time all the $\nu$-norms of a generating set of $\Gamma^Q_S$, for every $\nu \in S$.

This section is organized as follows: in Subsection 	\ref{subsec_two_basic_lemmas} we establish an abstract lemma, which gives a generating set of $\Gamma^Q_S$ in terms of a generating set of $H^Q_S$ and a fundamental set $U^Q_S$ of $\Gamma^Q_S$ in $H^Q_S$. Then, in Subsection \ref{subsec_fund_set_GammaQS} we explain how to obtain $U^Q_S$ form any set of equivalence matrices in $GL(d, \ZZ_S)$ between $Q$ and the \textit{reduced} quadratic forms $\ZZ_S$-equivalent to $Q$. With the help of theorems \ref{Z_S-equivalence_R-isotropic} and \ref{Z_S-equivalence_R-anisotropic}, we will choose not so big equivalence matrices in Subsection \ref{subsec_small_gen_set}. Finally, we complete the proof of our main results in Subsection \ref{subsec_main_proofs_gens_O(Q,Z_S)}.

	\subsection{Two basic lemmas}\label{subsec_two_basic_lemmas}

We say that a subset $M$ of a group $H_0$ is \textit{generating} if $H_0 = \cup_{n \geq 1} M^n$. Let $\Gamma_0$ be a subgroup of $H_0$. A \textit{fundamental set} of $\Gamma_0$ in $H_0$ is a subset $U$ of $H_0$ such that $H_0 = U \Gamma_0$. Our first lemma of this section gives a generating set of $\Gamma_0$ in terms of any $M$ and $U$ as above. We will apply it to $H_0 = H^Q_S$ and $\Gamma_0 = \Gamma^Q_S$ to obtain, from a compact generating set $M^Q_S$ of $H^Q_S$ and a carefully chosen fundamental set $U^Q_S$ of $\Gamma^Q_S$ in $H^Q_S$, the generating sets of $O(Q,\ZZ_S)$ in theorems \ref{Small_generators_R-isotropic} and \ref{Small_generators_R-anisotropic}. The second lemma of this subsection will give us $M^Q_S$, and the construction of  $U^Q_S$ will be carried out in Subsection \ref{subsec_small_gen_set}.
	
\begin{Lem}\label{Generating_set_abstract_lemma}
Let $\Gamma_0$ be a subgroup of a group $H_0$. Consider a generating set $M$ of $H_0$ and a fundamental set $U$ of $\Gamma_0$ in $H_0$. Then $\Gamma_0$ is generated by $(U\inv M U) \cap \Gamma_0$. 
\end{Lem}

\begin{proof}
Let $A_n = U\inv M^n U$ for any positive integer $n$. Since $H_0 = \cup_{n\geq 1} A_n$, to show that $\Lambda := \langle A_1 \cap \Gamma_0 \rangle$ coincides with $\Gamma_0$ it suffices to prove that $A_n \cap \Gamma_0$ is contained in $\Lambda$ for any $n\geq 1$. We show this by induction on $n$. This is true for $n=1$ by the definition of $\Lambda$. Suppose now that $A_\ell \cap \Gamma_0 \subseteq \Lambda$ for $1 \leq \ell \leq n$ and consider $\gamma_{n+1} \in A_{n+1} \cap \Gamma_0$. Take $u_1, u_2 \in U$ and $m_1,\ldots, m_{n+1} \in M$ such that 
\[\gamma_{n+1} = u_1\inv m_1 \cdots m_{n+1} u_2. \]
We write $m_{n+1} u_2$ as $u_3 \gamma_1$ for some $u_3 \in U$ and $\gamma_1 \in \Gamma_0$. Then $\gamma_1$ and $\gamma_n = u_1\inv m_1 \cdots m_n u_3$ are respectively in $A_1 \cap \Gamma_0$ and $A_n \cap \Gamma_0$. By the inductive hypothesis, $\gamma_1, \gamma_n$ belong to $\Lambda$, hence $\gamma_{n+1} = \gamma_n \gamma_1$ as well.  
\end{proof}	

Now we describe the compact generating set $M^Q_S$ of $H^Q_S$ we'll use later. Since $H^Q_S$ is conjugated in $\GLS{d}$ to the orthogonal $\QQ_S$-group $H_S$ of a standard quadratic form on $\QQ_S^d$, it suffices to exhibit a compact generating set of each $H_\nu$. The case $\nu = \infty$ can be easily settled with a standard connectivity argument. For $\nu < \infty$ we use the next lemma. We define $\mathtt{C}_2 = 8$ and $\mathtt{C}_p = p$ for any odd prime number $p$. 
\begin{Lem} \label{Gen_set_H_p}
Consider a prime number $p$ and an integer $d \geq 3$. The  orthogonal $\QQ_p$-group $H_p$ of any standard quadratic form on $\QQ_p^d$ is generated by 
\[
M_p = \{ h \in H_p \mid \normp{h} \leq \mathtt{C}_p\}.
\]  
\end{Lem}

Here is a sketch the proof of Lemma \ref{Gen_set_H_p}. When $H_p$ is compact we can show that $M_p = H_p$. Suppose now that $H_p$ is noncompact. Let $\langle M_p \rangle := \cup_{n \geq 1} M_p^n$. Note that $M_p$ has a representative of any $H\ncc_p$-coset in $H_p$ by Lemma \ref{Small_rep_H/Hncc}. Hence it suffices to show that $H\ncc_p \subseteq \langle M_p \rangle$. By \cite[Proposition 1.5.4 $(ii)$]{margulis_discrete_1991}, the unipotent subgroups of $H_p$ generate $H\ncc_p$. To conclude we show that any unipotent subgroup $U$ of $H_p$ is contained in $\langle M_p \rangle$: $M_p$ contains $U' = U \cap GL(d,\ZZ_p)$, and there is a suitable\footnote{In some maximal split torus of $H_p$.} $a \in M_p$ such that $a\inv \in M_p$ and $U = \cup_{n \geq 0} a^n U' a^{-n}$.

	\subsection{A fundamental set of $\Gamma^Q_S$ in $H^Q_S$}\label{subsec_fund_set_GammaQS}
	The goal of this subsection is to explain a general strategy to construct fundamental sets of $\Gamma^Q_S$ in $H^Q_S$. We proceed by analogy with the classical case $S = \{\infty\}$, which was first studied by Siegel in \cite{siegel_einheiten_1939}. We will describe these fundamental sets in terms of \textit{Siegel sets} $\sieS{\alpha}{\beta}$ of $\GLS{d}$. The subsection is organized as follows: First we introduce the Siegel sets $\sieS{\alpha}{\beta}$ of $\GLS{d}$, and we recall a condition on $\alpha,\beta$ which guarantees that $\sieS{\alpha}{\beta}$ is a fundamental set of $\Gamma_{d,S}$.  Then we discuss some concepts of the reduction theory of quadratic forms that we'll need. Finally, we construct fundamental sets of $\Gamma^Q_S$ in $H^Q_S$ in Lemma \ref{Siegel_set_HQS}.

Recall that $\textbf{G}_d$ denotes the $\QQ$-group $\textbf{GL}(d)$. Let $S = \{\infty\} \cup S_f$ be a finite set of places of $\QQ$. We introduce now the \textit{Siegel sets} of $G_{d,S}$. We start with $S = \{\infty\}$. Consider the following subgroups of $G_{d,\infty}$:
\begin{align*}
K &= O(d,\RR) \\
A &= \{diag(a_1,\cdots,a_d) \mid a_i >0 \text{ for any } 1 \leq i \leq d \}, \\
N &= \{\text{unipotent, upper-triangular matrices in } G_{d,\infty}\}.
\end{align*}
For any $\alpha, \beta>0$ we define
\begin{align*}
A_\alpha &= \{diag(a_1,\cdots, a_d) \in A \mid a_i \leq \alpha a_{i+1} \text{ for any } 1 \leq i \leq d-1 \}, \\
N_\beta &= \{n \in N \mid \normi{n-I_d} \leq \beta \}.
\end{align*}
The $(\alpha, \beta)$-Siegel set of $G_{d,\infty}$ is defined as
\begin{equation}\label{def_Siegel_set_R} 
\sieR{\alpha}{\beta} = K A_\alpha N_\beta. 
\end{equation}
For a general $S = \{\infty\} \cup S_f$  we define the $(\alpha,\beta)$-Siegel set of $G_{d,S}$ as
\begin{equation}\label{def_Siegel_set_QS} \sieS{\alpha}{\beta} = \sieR{\alpha}{\beta} \times \prod_{p \in S_f } GL(d,\ZZ_p). 
\end{equation}	
Recall that $\GammaS{d}$ is the diagonal copy of $GL(d,\ZZ_S)$ in $\GLS{d}$. It is a classical fact that big enough Siegel sets are fundamental sets of $\GammaS{d}$ in $\GLS{d}$.

\begin{Prop} \label{Siegel_set_GL(d)}
Consider a finite set $S = \{\infty\} \cup S_f$ of places of $\QQ$ and $d \geq 2$. For any $\alpha \geq \frac{2}{\sqrt{3}}$ and $\beta \geq \frac{1}{2}$ we have 
\[ G_{d,S} = \sieS{\alpha}{\beta} \Gamma_{d,S}. \]
\end{Prop}	
	
See \cite[Lemma 2.2]{benoist_five_2009} and \cite[Proposition 5.7 ]{platonov_algebraic_1994} for the proofs for $S=\{\infty\}$ and $GL(d,\QQ) \subseteq GL(d,\AAA)$, respectively. The proof of Proposition \ref{Siegel_set_GL(d)} goes along the same lines. 

Now we present the concepts from the reduction theory of quadratic forms that we need for the purposes of this subsection. See \cite[Chapitre I: §2, §5]{borel_introduction_1969} or \cite[Chapter 12]{cassels_rational_1978} for a more complete discussion of the topic. We say that a quadratic form on $\QQ_S^d$ is \textit{$(\alpha,\beta)$-reduced} if we can write it as $P \circ s$, for some standard quadratic form\footnote{Recall that this means that $P = (P_\nu)_{\nu \in S}$, and that $P_\nu$ is a standard quadratic form on $\QQ_\nu^d$ for every $\nu \in S$. See Subsection \ref{subsec_sqf}.} $P$ on $\QQ_S^d$ and some $s \in \sieS{\alpha}{\beta}$. Let $Q$ be a quadratic form on $\QQ^d$. Recall that $Q_S$ is the quadratic form on $\QQ_S^d$ obtained from $Q$ via the diagonal embedding $\QQ \to \QQ_S$. We say that $Q$ is \textit{$(S, \alpha, \beta)$-reduced} if $Q_S$ is $(\alpha, \beta)$-reduced. Here are some properties of reduced quadratic forms we'll need later. 

\begin{Lem}\label{TQS_general_bounds}
Let $S= \{\infty\} \cup S_f$ be a finite set of places of $\QQ$. Consider an $(S,2,1)$-reduced quadratic form $R$ in $d \geq 3$ variables with coefficients in $\ZZ_S$, and an integral quadratic form $Q$ in $d$ variables. Then:
	\begin{enumerate}[$(i)$]
	\item $R$ is integral and $p^{-2} \leq |\det R|_p \leq 1$ for any $p \in S_f$. 
	\item If $Q = R \circ \gamma$ for some $\gamma \in GL(d,\ZZ_S)$, then $|\det R|_\infty \leq p_S^2 |\det Q|_\infty$,
	\[ p_S\inv \leq |\det \gamma|_\infty \leq |\det Q|_\infty^\frac{1}{2} \hspace{1.2cm} \text{and} \hspace{1.2cm} |\det Q|_p^\frac{1}{2} \leq  |\det \gamma|_p \leq p |\det Q|_p^\frac{1}{2}\]
	for any $p \in S_f$.
	\end{enumerate}
\end{Lem}

\begin{proof}
Let $b_R, b_Q \in GL(d,\QQ)$ be the matrices of $R$ and $Q$ in the canonical basis of $\QQ^d$. Write $R_S = P \circ s$ for some $s \in \sieS{2}{1}$ and a standard quadratic form $P = (P_\nu)_{\nu \in S}$ on $\QQ_S^d$. Let $c \in G_{d,S}$ be the matrix of $P$ in the canonical basis of $\QQ_S^d$. 

Let's prove $(i)$. Recall that $R = P_p \circ s_p$ and $s_p \in GL(d,\ZZ_p)$ for any $p \in S_f$, so 
\[ |\det R|_p = |\det s_p|_p^2 \cdot |\det P_p|_p  = |(\det P)_p|_p, \]
thus\footnote{Since $P_p$ is standard, then $|\det P_p|_p$ is either $p^{-2}, p\inv$ or 1. } $p^{-2} \leq |\det R|_p \leq 1$. The matrix $b_R \in M_d(\ZZ_S)$ verifies 
\[ \normp{b_R} \leq \normp{\tra s_p} \normp{c_p} \normp{s_p} \leq 1 \]
for any $p \in S_f$, so $b_R$ is integral. 

Now suppose that we are in the situation of $(ii)$. Since $R$ and $Q$ are $\ZZ_S$-equivalent, $\hgt_S(\det R) = \hgt_S(\det Q)$. Using $(i)$ we get
\[p_S^2 |\det R|_\infty \leq \hgt_S(\det R) = \hgt_S(\det Q) \leq |\det Q|_\infty,\]
which proves the first inequality. For the second one, since $\tra \gamma b_R \gamma = b_Q$ and $|\det R|_\infty \geq 1$ because $R$ is integral, then
\[ |\det \gamma|_\infty = \left( \frac{|\det Q|_\infty}{|\det R|_\infty} \right)^\frac{1}{2} \leq |\det Q|_\infty^\frac{1}{2} \quad \text{and} \quad 
|\det \gamma|\inv_\infty = \left( \frac{|\det R|_\infty}{|\det Q|_\infty} \right)^\frac{1}{2} \leq p_S.  \]
Let's prove the third inequality. For any $p \in S_f$ we have
\[|\det \gamma|_p = \left( \frac{|\det Q|_p}{|\det R|_p} \right)^ \frac{1}{2}, \]
so $|\det Q|_p ^\frac{1}{2} \leq |\det \gamma|_p \leq p |\det Q|_p^\frac{1}{2}$ by $(i)$. 
\end{proof}

We denote by $\mathscr{R}^Q_S$ the set of rational quadratic forms that are $\ZZ_S$-equivalent to $Q$ and $(S,2,1)$-reduced.

\begin{Lem}\label{TQS_finite}
Let $Q$ be a non-degenerate integral quadratic form in $d$ variables. For any finite set $S = \{\infty\} \cup S_f$ of places of $\QQ$, the set $\mathscr{R}^Q_S$ is finite. 
\end{Lem}

\begin{proof}
Any $R \in \mathscr{R}^Q_S$ is integral by Lemma \ref{TQS_general_bounds}, so
\[ |\det R|_\infty \leq \hgt_S(\det R) = \hgt_S(\det Q). \]
Also, $R$ is $(2,1)$-reduced as real quadratic form because the real factor of $\sieS{\alpha}{\beta}$ is the $(\alpha, \beta)$ Siegel set of $GL(d,\RR)$. By Proposition \ref{Integral_reduced_qf_are_small} there are finitely many $(2,1)$-reduced integral quadratic forms on $\RR^d$ of bounded determinant. 
\end{proof}

We can finally construct a fundamental set of $\Gamma^Q_S$ in $H^Q_S$ from $\mathscr{R}^Q_S$ and $\sieS{2}{1}$.  Any $R \in \mathscr{R}^Q_S$ is in fact integral by Lemma \ref{TQS_general_bounds}. We choose $\tau_R \in \Gamma_{d,S}$ such that $R_S \circ \tau_R = Q_S$---in Subsection \ref{subsec_small_gen_set} we'll pick a convenient $\tau_R$---and we define 
\[ \mathscr{T}^Q_S = \{ \tau_R \mid R \in \mathscr{R}^Q_S \}, \]
which is finite by Lemma \ref{TQS_finite}.

\begin{Lem}\label{Siegel_set_HQS}
Let $Q$ be a non-degenerate integral quadratic form in $d \geq 2$ variables. Consider a finite set $S = \{ \infty \} \cup S_f$ of places of $\QQ$ and the standard quadratic form $P$ on $\QQ_S^d$ that is $\QQ_S$-equivalent to $Q_S$. Choose some $g \in \GLS{d}$ taking $P$ to $Q_S$, and set
\[U^Q_S = \left( g\inv \sieS{2}{1} \mathscr{T}^Q_S \right) \cap H^Q_S. \]
Then $H^Q_S = U^Q_S \Gamma^Q_S$. 
\end{Lem}

\begin{proof}
Take $h\in H^Q_S$. By Proposition \ref{Siegel_set_GL(d)} we can write $g h$ as $s \gamma\inv$ for some $s \in \sieS{2}{1}$ and $\gamma = (\gamma_0, \ldots, \gamma_0) \in \Gamma_{d,S}$. Let $R = Q \circ \gamma_0$. From $Q_S = P \circ (gh)$ we obtain $Q_S \circ \gamma = P \circ s$, so $R$ is in $\mathscr{R}^Q_S$. Consider $\tau \in \mathscr{T}^Q_S$ such that $R_S \circ \tau = Q_S$. Then $\tau \inv \gamma \inv$ is in $\Gamma^Q_S$ because 
\[Q_S \circ \gamma = R_S = Q_S \circ \tau\inv. \]
Notice also that $u = g\inv s \tau$ belongs to $U^Q_S$, and $h = u (\tau\inv \gamma\inv)$,
so we are done. 
\end{proof}
	
	\subsection{Choosing $\mathscr{T}^Q_S$}\label{subsec_small_gen_set}
	
We now know how to obtain a generating set $\mathscr{G}^Q_S$ of $\Gamma^Q_S$ from a fundamental set $U^Q_S$ of $\Gamma^Q_S$ in $H^Q_S$. In turn, there is an $U^Q_S$ for any set $\mathscr{T}^Q_S$ of equivalence matrices $(\gamma_R,\ldots,\gamma_R) \in \Gamma^Q_S$ between reduced quadratic forms $R$, and $Q_S$. The goal of this subsection is to choose the $\gamma_R$'s with $\normnu{\gamma_R}$ controlled in terms of $Q$, for any $\nu \in S$. The key to this are our effective criteria of $\ZZ_S$-equivalence: theorems \ref{Z_S-equivalence_R-isotropic} and $\ref{Z_S-equivalence_R-anisotropic}$. We state the bounds separately, depending on whether $Q$ is $\RR$-isotropic or not.   
	
For our first lemma we define $\consTQSiso{d} = 	2^{2d^5} d!^{4d^4} \consZSEquivCritRiso{d} \consReducedIntegralQF{d}^{2d^3}$.

\begin{Lem}\label{TQS_R-isotropic}
Consider a non-degenerate, $\RR$-isotropic integral quadratic form $Q$ in $d \geq 3$ variables and a finite set $S = \{\infty\} \cup S_f$ of places of $\QQ$. For any $R \in \mathscr{R}^Q_S$, there is $\gamma_R \in GL(d,\ZZ_S)$ with
\begin{align*}
\normi{\gamma_R} & \leq \consTQSiso{d} p_S^{19d^6 + 8d^4} \normi{Q}^{4d^5 + 2d^3}   \\
\normp{\gamma_R} & \leq p^2 \quad \text{for any odd }p \in S_f, \\
\norm{\gamma_R}_2 & \leq 2^{d+3}  \quad \text{if } 2 \in S_f, \\  
\end{align*}
that takes $R$ to $Q$.
\end{Lem}

\begin{proof}
Any $R \in \mathscr{R}^Q_S$ is integral by Lemma \ref{TQS_general_bounds}. Since $R$ and $Q$ are $\RR$-isotropic, by Theorem \ref{Z_S-equivalence_R-isotropic} there is $\gamma_R \in GL(d,\ZZ_S)$ with
\begin{align*}
\normi{\gamma_R} & \leq \consZSEquivCritRiso{d} p_S^{19 d^6} (\normi{R} \normi{Q})^{2d^3}, \\
\normp{\gamma_R} & \leq p |\det R|_p^{-\frac{1}{2}} \quad \text{for any odd } p \in S_f, \\
\norm{\gamma_R}_2 &\leq 2^{d+2} |\det R|_2^{-\frac{1}{2}} \quad \quad \text{if } 2 \in S_f,
\end{align*}
taking $R$ to $Q$. We will now obtain inequalities only in terms in $Q$. By Lemma \ref{TQS_general_bounds} we have
\begin{align}
\label{det_infty} |\det R|_\infty & \leq p_S^2 |\det Q|_\infty, \\
\label{det_p} |\det R|_p^{-\frac{1}{2}} & \leq p \quad \text{for any prime } p.
\end{align}
Note also that $R$ is reduced as real quadratic form since $R_S$ is (2,1)-reduced and the projection of $\sieS{2}{1}$ to $G_{d,\infty}$ is $\sieR{2}{1}$. Then, by Proposition \ref{Integral_reduced_qf_are_small} and \eqref{det_infty} we have
\begin{align}
\notag \normi{R} \leq 2^{d^2} \consReducedIntegralQF{d} |\det R|_\infty^{2d} & \leq 2^{d^2} \consReducedIntegralQF{d} p_S^{4d} |\det Q|_\infty^{2d} \\
\label{bR} & \leq 2^{d^2} d!^{2d} \consReducedIntegralQF{d} p_S^{4d} \normi{Q}^{2d^2}. 
\end{align}
We are ready to bound the $\nu$-norms of $\gamma_R$:
\begin{align*}
\normi{\gamma_R} & \leq \consZSEquivCritRiso{d} p_S^{19 d^6} (2^{d^2} d!^{2d} \consReducedIntegralQF{d} p_S^{4d} \normi{Q}^{2d^2 + 1})^{2d^3}  \\
& \leq 2^{2d^5} d!^{4d^4} \consZSEquivCritRiso{d} \consReducedIntegralQF{d}^{2d^3} p_S^{19d^6 + 8d^4} \normi{Q}^{4d^5 + 2d^3}.
\end{align*}
For any odd $p \in S_f$, by \eqref{det_p} we have
\[ \normp{\gamma_R} \leq p |\det R|_p^{-\frac{1}{2}} \leq p^2, \]
and similarly $\norm{\gamma_R}_2 \leq 2^{d+3}$ when $2 \in S_f$. 
\end{proof}

For the second lemma we define $\consTQSani{d} = 2^{d^5} d!^{2d^4} \consZSEquivCritRani{d} \consReducedIntegralQF{d}^{d^3}$.

\begin{Lem}\label{TQS_R-anisotropic}
Consider an $\RR$-anisotropic integral quadratic form $Q$ in $d \geq 3$ variables and a finite set $S = \{\infty\} \cup S_f$ of places of $\QQ$. Suppose that $Q$ is $\QQ_{p_0}$ isotropic for some $p_0 \in S_f$. For any $R \in \mathscr{R}^Q_S$, there is $\gamma_R \in GL(d,\ZZ_S)$ with
\begin{align*}
\normpo{\gamma_R} & \leq  \consTQSani{d} p_S^{13d^6 + 4d^4} \normi{Q}^{2d^5 + d^3} \\
\normp{\gamma_R} & \leq p^2 \quad \text{for any odd }p \in S_f-\{p_0\},\\
\norm{\gamma_R}_2 & \leq 2^{d+3} \quad  \text{if } 2 \in S_f-\{p_0\},    
\end{align*}
that takes $R$ to $Q$.
\end{Lem}

\begin{proof}
Any $R \in \mathscr{R}^Q_S$ is integral by Lemma \ref{TQS_general_bounds}. Since $R$ and $Q$ are $\RR$-anisotropic and $\QQ_{p_0}$-isotropic, by Theorem \ref{Z_S-equivalence_R-anisotropic} there is $\gamma_R \in GL(d,\ZZ_S)$ with
\begin{align*}
\normpo{\gamma_R} & \leq \consZSEquivCritRani{d} p_S^{13 d^6} (\normi{R} \normi{Q})^{\frac{1}{2}d^3 + 3d}, \\
\normp{\gamma_R} & \leq p |\det R|_p^{-\frac{1}{2}} \quad \quad \text{for any odd } p \in S_f - \{p_0\},\\ 
\norm{\gamma_R}_2 & \leq 2^{d+2} |\det R|_2^{-\frac{1}{2}} \quad \quad \text{if } 2 \in S_f - \{p_0\}, \\
\end{align*}
taking $R$ to $Q$. Let's obtain inequalities only in terms of $Q$. By Lemma \ref{TQS_general_bounds} we still have \eqref{det_infty} and \eqref{det_p}, hence the bounds for $\normp{\gamma_R}$ and $\norm{\gamma_R}_2$ (when $2 \in S_f-\{p_0\}$) follow. Note that \eqref{bR} also holds because $R$ is integral and $(2,1)$-reduced as real quadratic form. Then
\begin{align*}
\normpo{\gamma_R} & \leq \consZSEquivCritRani{d} p_S^{13 d^6} (2^{d^2} d!^{2d} \consReducedIntegralQF{d} p_S^{4d} \normi{Q}^{2d^2 + 1})^{d^3} \\
	& \leq 2^{d^5} d!^{2d^4} \consZSEquivCritRani{d} \consReducedIntegralQF{d}^{d^3} p_S^{13d^6 + 4d^4} \normi{Q}^{2d^5 + d^3}.
\end{align*}

\end{proof}

	\subsection{Main proofs}\label{subsec_main_proofs_gens_O(Q,Z_S)}
We are ready to complete the proofs of our two theorems on generators of $S$-integral orthogonal groups.

\begin{proof}[Proof of Theorem \ref{Small_generators_R-isotropic}]
We write $Q_S = P \circ g$ for some standard quadratic form $P$ on $\QQ_S^d$ and some $g \in G_{d,S}$. Let $H_S = O(P,\QQ_S)$.  Consider
\[M_\infty  = \{diag(a_1,\ldots, a_d) \mid a_i = \pm 1 \}\]
and 
\[M_\infty (\varepsilon) = M_\infty \cup \{ h \in H_\infty \mid \normi{h-I_d} < \varepsilon \}. \]
For any $\varepsilon > 0$, $M_\infty(\varepsilon)$ generates $H_\infty$ since it has nonempty interior and it meets every connected component of $H_\infty$. For any $p \in S_f$, let $M_p$ be the generating set of $H_p$ of Lemma \ref{Gen_set_H_p}. Hence $M_S(\varepsilon) = M_\infty(\varepsilon) \times \prod_{p \in S_f} M_p$ and $M^Q_S (\varepsilon) = g\inv M_S(\varepsilon) g$ are respectively generating sets of $H_S$ and $H^Q_S$, for any $\varepsilon > 0$. Since $Q$ is $\RR$-isotropic, for each $R \in \mathscr{R}^Q_S$ we consider $\gamma_R \in GL(d,\ZZ_S)$ as in Lemma \ref{TQS_R-isotropic}, taking $R$ to $Q$. We set $\tau_R = (\gamma_R, \ldots, \gamma_R) \in \GammaS{d}$ and
\begin{equation}\label{mp_TQS} \mathscr{T}^Q_S = \{\tau_R \mid R \in \mathscr{R}^Q_S \}.  
\end{equation}
Consider
\begin{equation}\label{mp_fund_set}
U^Q_S = (g\inv \sieS{2}{1} \mathscr{T}^Q_S) \cap H^Q_S.
\end{equation}
Then $H^Q_S = U^Q_S \Gamma^Q_S$ by Lemma \ref{Siegel_set_HQS}, and
\begin{equation}\label{gen_epsilon} 
\mathscr{G}^Q_S(\varepsilon) = ((U^Q_S)\inv M^Q_S(\varepsilon)  U^Q_S) \cap \Gamma^Q_S 
\end{equation}
generates $\Gamma^Q_S$ according to Lemma \ref{Generating_set_abstract_lemma}. We define $M_S = M_\infty \times \prod_{p \in S_f} M_p$ and $M^Q_S = g\inv M_S g$. Letting $\varepsilon$ tend to 0 in \eqref{gen_epsilon} we see that 
\begin{equation}\label{mp_gen_set} 
\mathscr{G}^Q_S = ((U^Q_S)\inv M^Q_S  U^Q_S) \cap \Gamma^Q_S 
\end{equation}
generates $\Gamma^Q_S$. For any $\widetilde{\xi} \in \mathscr{G}^Q_S$, let $\xi$ be the corresponding matrix in $GL(d,\ZZ_S)$. To conclude we'll show that any $\xi$ verifies the bounds of the statement\footnote{We choose $\mathscr{G}^Q_S$ as generating set of $\Gamma^Q_S$ instead of $\mathscr{G}^Q_S(\varepsilon)$ because $M_\infty$ is contained in $O(d,\RR)$, unlike $M_\infty(\varepsilon)$. This will allow us to bound $\normi{\xi}$ for any $\widetilde{\xi} \in \mathscr{G}^Q_S$.}. There are $\tau, \eta \in \mathscr{T}^Q_S$, $m \in M_S$ and $s,t \in \sieS{2}{1}$ such that
\[ \widetilde{\xi} = \tau\inv s\inv g (g\inv m g) g\inv t \eta = \tau\inv s\inv m t \eta. \]
 Let $b' = s\inv m t$. Note that $b'$ is in $\GammaS{d}$ because $b' = \tau \widetilde{\xi} \eta\inv$. For any odd $p \in S_f$ we have
\[ \normp{b'_p} = \normp{s_p\inv m_p t_p} \leq p^2,\]
and similarly $\norm{b'_2}_2 \leq 8$ if $2 \in S_f$. Hence $b := 2 p_S^2 b'_\infty$ has integral coefficients. The equality $s_\infty b = 2 p_S^2 m_\infty t_\infty$\footnote{This is the argument where we need $M_\infty$ to be contained in $O(d,\RR)$, because then $M_\infty \sieR{2}{1} = \sieR{2}{1}$.} shows that $\sieR{2}{1} b$ meets $\sieR{2}{1}$, so 
\begin{equation}\label{Ly}
\normi{b} \leq  \consTransSiegel{d} |\det b|_\infty^{2d}
\end{equation}
by Corollary \ref{Siegel_sets_almost_never_meet_its_translates_ap}. Note that the determinant of $\xi = \tau_\infty\inv b'_\infty \eta_\infty$ is $\pm 1$ since it preserves $Q$, so $|\det b'_\infty|_\infty = \frac{|\det \tau_\infty|_\infty}{|\det \eta_\infty|_\infty}$. By Lemma \ref{TQS_general_bounds} we get
\[ |\det b'_\infty|_\infty \leq p_S |\det Q|_\infty^{\frac{1}{2}}. \]
Putting \eqref{Ly} in terms of $b'$ we obtain

\begin{align}
\notag \normi{b'_\infty} \leq 2^{d^2 - 1} \consTransSiegel{d} p_S^{4d^2-2} |\det b'|_\infty^{2d} & \leq 2^{d^2} \consTransSiegel{d} p_S^{5 d^2} |\det Q|_\infty^d \\
	\label{mp_binfy}	& \leq 2^{d^2} d!^d \consTransSiegel{d} p_S^{5d^2} \normi{Q}^{d^2}.
\end{align}
To obtain the inequalities for the $\nu$-norms of $\xi$ we'll use the bounds of $\normnu{\tau_\nu}$ and $\normnu{\eta_\nu}$ in Lemma \ref{TQS_R-isotropic}. We start with $\nu = \infty$:
\begin{align*}
\normi{\xi} = \normi{\tau_\infty\inv b'_\infty \eta_\infty} & \leq d^2 \normi{\tau_\infty\inv} \normi{\eta_\infty} \normi{b'_\infty} \\
	& \leq d \cdot d! \frac{\normi{\tau_\infty}^{d-1} \normi{\eta_\infty}}{|\det \tau_\infty|_\infty} \normi{b'_\infty} \\
	&\leq d \cdot d! p_S (\consTQSiso{d} p_S^{19d^6 + 8d^4} \normi{Q}^{4d^5 + 2d^3} )^d (2^{d^2} d!^d \consTransSiegel{d} p_S^{5d^2} \normi{Q}^{d^2}) \\
	& \leq 2^{2d^2} d \cdot d!^{d+1} \consTQSiso{d}^d \consTransSiegel{d} p_S^{20d^7} \normi{Q}^{5 d^6}.
\end{align*}
For any odd $p \in S_f$ we have
\[\normp{\xi} = \normp{\tau_p\inv s_p\inv m_p t_p \eta_p} \leq \normp{m_p} \frac{\normp{\tau_p}^{d-1} \normp{\eta_p}}{|\det \tau_p|_p} \leq p^{2d + 2} |\det Q|_p^{-\frac{1}{2}}, \]
and similarly $\norm{\xi}_2 \leq 2^{d^2 + 3d + 3} |\det Q|_2^{-\frac{1}{2}}$.  
  
\end{proof}

\begin{proof}[Proof of Theorem \ref{Small_generators_R-anisotropic}]
Let $H_S$ be the orthogonal $\QQ_S$-group of the standard quadratic form $P$ on $\QQ_S^d$ that is $\QQ_S$-equivalent to $Q_S$, and consider $g \in G_{d,S}$ taking $P$ to $Q_S$. Let $M_\infty = H_\infty$ and for any $p \in S_f$, let $M_p$ be the generating set of $H_p$ of Lemma \ref{Gen_set_H_p}. Then $M_S = \prod_{\nu \in S} M_\nu$ generates $H_S$ and $M^Q_S = g\inv M_S g$ generates $H^Q_S = g\inv H_S g$. Since $Q$ is $\RR$-anisotropic and $\QQ_{p_0}$-isotropic, for each $R \in \mathscr{R}^Q_S$ we consider $\gamma_R \in GL(d,\ZZ_S)$ taking $R$ to $Q$ as in Lemma \ref{TQS_R-anisotropic}. Let $\tau_R = (\gamma_R,\ldots,\gamma_R) \in \GammaS{d}$.  Consider $\mathscr{T}^Q_S, U^Q_S$ and $\mathscr{G}^Q_S$ respectively as in \eqref{mp_TQS}, \eqref{mp_fund_set} and \eqref{mp_gen_set}. By Lemma \ref{Siegel_set_HQS}, the $\widetilde{\xi} \in \mathscr{G}^Q_S$ generate $\Gamma^Q_S$, so the corresponding $\xi \in GL(d,\ZZ_S)$ generate $O(Q,\ZZ_S)$. To conclude we'll show that these $\xi$ verify the inequalities of the statement. We write
\[\widetilde{\xi} = \tau\inv s\inv m t \eta  \]
for some $\tau, \eta \in \mathscr{T}^Q_S$, $m \in M_S$ and $s,t \in \sieS{2}{1}$. Consider again $b' = s\inv m t = \tau \widetilde{\xi} \eta\inv \in \Gamma_{d,S}$. Since $M_\infty \sieR{2}{1} = \sieR{2}{1}$, then \eqref{mp_binfy} still holds. Using the inequalities for $\normnu{\tau_\nu}$ and $\normnu{\eta_\nu}$ of Lemma \ref{TQS_R-anisotropic} we will now bound $\xi$:
\begin{align*}
\normpo{\xi_{p_0}} = \normpo{\tau_{p_0}\inv s_{p_0}\inv m_{p_0} t_{p_0} \eta_{p_0}} & \leq 2 p_0^2 |\det \tau_{p_0}|_{p_0}\inv \normpo{\tau_{p_0}}^{d-1} \normpo{\eta_{p_0}} \\ 		
		& \leq 2 p_0^2 |\det Q|_{p_0}^{-\frac{1}{2}} (\consTQSani{d} p_S^{13d^6 + 4d^4} \normi{Q}^{2d^5 + d^3} )^d  \\
		& \leq 2 \consTQSani{d}^d p_S^{14d^7} \normi{Q}^{3d^6} |\det Q|_{p_0}^{-\frac{1}{2}}.		
\end{align*}
For any odd $p \in S_f - \{p_0\}$ we have
\[ \normp{\xi} \leq p^2 |\det Q|_p^{-\frac{1}{2}} \normp{\tau_p}^{d-1} \normp{\eta_p} \leq p^{2d+2} |\det Q|_p^{-\frac{1}{2}}, \]
and similarly $\norm{\xi}_2 \leq 2^{d^2 + 3d + 3} |\det Q|_2^{-\frac{1}{2}}$ when $2 \in S_f - \{p_0\}$. Finally, since $Q$ is $\RR$-anisotropic, by Remark \ref{rem_pos_def_qf}, for any $h \in O(Q,\RR)$ we have
\[ \normi{h} \leq d^{d+1} \cdot d! \normi{Q}^{\frac{d}{2}}.\]
In particular the bound holds for $\normi{\xi}$.

\end{proof}

\appendix

\section{Decay of coefficients of unitary representations}\label{app_Decay_coefficients}

The purpose of this appendix is to prove Proposition \ref{Decay_smooth_vectors_almost_L2m_explicit} and Proposition \ref{Decay_speed_K_n-inv-vectors_explicit}, which give an explicit decay of coefficients of almost-$L^p$ unitary representations\footnote{The reader is referred to Section \ref{subsec_unitary_reps} for the basic definitions related to unitary representations.} of $SL(2,\QQ_\nu)$. These two statements are simple consequences of a general theorem of M. Cowling, U. Haagerup and R. Howe about the decay of coefficients of \textit{tempered} unitary representations of a semisimple group $G$ in terms of the Harish-Chandra spherical function of $G$. We'll present the statement for $G = SL(2,\QQ_\nu)$ in Section \ref{app_decay_in_terms_of_Harish}, and then we'll make explicit each term of the estimate: in Section \ref{app_asymp_Harish} we give exponential upper bounds for the Harish-Chandra function of $SL(2,\QQ_\nu)$. Finally, we prove Proposition \ref{Decay_smooth_vectors_almost_L2m_explicit} and Proposition \ref{Decay_speed_K_n-inv-vectors_explicit} respectively in Section \ref{app_proof_decay_SL(2,R)} and Section \ref{app_proof_decay_SL(2,Q_p)}.

	\subsection{Decay speed in terms of the Harish-Chandra function}\label{app_decay_in_terms_of_Harish}

Let $\nu$ be a place of $\QQ$. Recall that $K_{2,\nu}$ denotes $SO(2,\RR)$ when $\nu = \infty$, and $SL(2,\ZZ_\nu)$ when $\nu < \infty$. The Harish-Chandra function of $SL(2,\QQ_\nu)$ is the map $\Xi_\nu: SL(2,\QQ_\nu) \to [0,1]$ given by
\[\Xi_\nu(g) = \int_{K_{2,\nu}} \norm{gk e_1}\inv \dd k,\]
where $e_1 = (1,0) \in \QQ_\nu^2$, and the norm $\norm{\cdot}$ is $\normnu{\cdot}$ when $\nu < \infty$, and the standard euclidean norm $\norm{\cdot}_{euc}$ of $\RR^2$ when $\nu = \infty$. We integrate with respect to the Haar probability measure on $K_{2,\nu}$. If $\pi$ is a unitary representation of $SL(2,\QQ_\nu)$ and $v \in \Hc_\pi$, we denote by $\delta(v)$ the square-root of the dimension of the $\CC$-linear span of $\pi(K_{2,\nu})v$. The next proposition is a particular case of \cite[Theorem 2]{cowling_almost_1988}.

\begin{Prop}\label{Effective_decay_tempered}
Consider a place $\nu$ of $\QQ$. Let $\pi$ be a tempered unitary representation of $SL(2,\QQ_\nu)$. For any $v_1, v_2 \in \mathcal{H}_\pi$ we have
\[|\scalar{\pi(g)v_1}{v_2}| \leq \Xi_\nu (g) \norm{v_1} \, \norm{v_2} \delta(v_1) \delta(v_2),\]
for any $g \in SL(2,\QQ_\nu)$. 
\end{Prop}

	\subsection{Asymptotics of the Harish-Chandra function of $\textbf{SL}(2)$}\label{app_asymp_Harish}

The purpose of this section is to give exponential decay estimates of $\Xi_\nu$. To lighten the notation, in this section we denote respectively $SL(2,\QQ_\nu)$ and $K_{2,\nu}$ by $G_\nu$ and $K_\nu$.  Recall that
\[a_{\infty, t} = diag(e^{\frac{t}{2}}, e^{-\frac{t}{2}}) \quad \text{and} \quad a_{p,m} = diag(p^{-m}, p^{m}) \]
for any $t \in \RR$ and any $m \in \ZZ$. Consider
\[ A^+_\infty = \{a_{\infty,t} \mid t \geq 0 \} \quad \text{and} \quad A^+_p = \{a_{p,m} \mid m \in \mathbb{N} \}. \]  
The function $\Xi_\nu$ is $K_\nu$ bi-invariant, so its decay speed depends only on the values it takes on $A_\nu^+$, since $G_\nu = K_\nu A_\nu ^+ K_\nu$ according to the \textit{Cartan decomposition of} $G_\nu$.
%Note that $\Xi_\nu$---as well as any $K_\nu$ bi-invariant function on $G_\nu$---is determined by its values in $A^+_\nu$, since $G_\nu = K_\nu A^+_\nu K_\nu$ according to the Cartan decomposition of $G_\nu$. Let's see how fast $\Xi_\nu$ decays along $A^+_\nu$.

Let's work first with $\Xi_\infty$. Let $r_\theta \in K_\infty$ be the rotation of angle $\theta$. The map $\theta \mapsto r_\theta$ is a parametrization $[0, 2\pi) \to K_\infty$, and the Haar probability measure $\lambda_{K_\infty}$ of $K_\infty$ in the $\theta$-coordinate is $(2\pi)\inv \frac{\dd }{\dd \theta}$. Thus,  
\begin{align*}\label{H-C_function_explicit}
\Xi_\infty (a_{\infty,t}) &= \frac{1}{2\pi} \int_0^{2\pi} \normeuc{a_{\infty, t} r_\theta e_1}\inv \dd \theta \\
&= \frac{1}{2 \pi} \int_0 ^{2\pi} \left( e^t \cos^2 \theta + e^{-t} \sin^2 \theta \right) ^ {-\frac{1}{2}} d\theta.
\end{align*}  
It is known that the functions $t \mapsto \Xi_\infty(a_{\infty,t})$ and $t \mapsto t e^{-\frac{t}{2}}$ are equivalent as $t \to \infty$---see \cite[p. 236]{howe_nonabelian_1992}. This implies right away the next exponential decay for $\Xi_\infty$. 

\begin{Cor}\label{Xi_infty_exponential_bound}
There is a positive constant $\consDecayHarishReal$ such that 
\[
\Xi_\infty(a_{\infty,t}) \leq \consDecayHarishReal e^{-\frac{t}{3}}.
\]
for any $t \geq 0$. 
\end{Cor}

We pass to the $p$-adic case. We'll prove an explicit formula for $\Xi_p$ from which the exponential decay will easily follow.  
\begin{Lem}\label{formula_Xi_p}
For any prime $p$ and any integer $m \geq 0$ we have
\[ \Xi_p (a_{p,m}) = \frac{p^{-m}}{p+1} ((2m+1)(p-1) + 2). \] 
\end{Lem} 

Before proving Lemma \ref{formula_Xi_p} we state the exponential decay of $\Xi_p$ that we'll use in practice.

\begin{Cor}\label{Xi_p-exponential_bound} 
For any prime number $p$ we have
\[\Xi_p(a_{p,m}) < 10 p^{-\frac{m}{2}}\] 
for $m\geq 1$. 
\end{Cor}
\begin{proof}
Note that
\begin{align*}
\Xi_p(a_{p,m}) &= \frac{1}{p^m} \left( (2m+1) \frac{p-1}{p+1} + \frac{2}{p+1} \right) \\
			& \leq \frac{1}{p^m} \left( \left(2 - \frac{4}{p+1}\right) m + 1 \right)   \leq 3 \frac{m}{p^m},
\end{align*}
and $\frac{m}{p^m} < \frac{2}{\log 2} p^{-\frac{m}{2}}$ \footnote{Indeed:
\[ \frac{p^{m/2}}{m}  \geq \frac{2^{m/2}}{m} > \frac{1 + \frac{\log 2}{2}m}{m} > \frac{\log 2}{2}.\]}. Thus
\[ \Xi_p (a_{p,m}) < \frac{6}{\log 2} p^{-\frac{m}{2}} <  10 p^{-\frac{m}{2}}. \]
\end{proof}

To compute $\Xi_p(a_{p,m})$ we use the next well-adapted measurable partition of $K_p$. For any integer $n \geq 0$ we define 
\[F_n = \left\{ (k_{ij}) \in K_p \mid |k_{11}|_p = p^{-n}, |k_{21}|_p = 1 \right\} \]
and
\[F_{-n} = \left\{ (k_{ij}) \in K_p \mid |k_{11}|_p = 1, |k_{21}|_p = p^{-n} \right\}.\]
Let $\lambda_{K_p}$ be the Haar probability measure of $K_p$.  

\begin{Lem}\label{partition_K}
The subset $\bigcup_{n\in \ZZ} F_n$ of $K_p$ has full measure and
	\begin{equation}\label{Tt1}
	\lambda_{K_p}(F_n) =  \frac{p-1}{p+1} p^{-|n|} 
	\end{equation}
for any $n\in \ZZ$.
\end{Lem}

\begin{proof}
Let $(e_1,e_2)$ be the standard basis of $V = \QQ_p^2$.
We denote by $\Psi$ be the map $k \mapsto ke_1$ from $K_p$ to the unit sphere $\SSS_V$ of $V$.  Note that $\mu = \Psi_* \lambda_{K_p}$ is the unique $K_p$-invariant probability measure on $\SSS_V$, thus
\begin{equation}\label{Tt0}
\mu(A) = \lambda_{V}(\ZZ_p A),
\end{equation} 
for any measurable subset $A$ of $\SSS_V$. Consider $C_n = \Psi(F_n)$. Then
\[ \lambda_{K_p}(F_n) = \mu(C_n), \]
since $F_n = \Psi\inv (C_n)$. Note that $\bigcup_{n\in \ZZ} C_n$ is conull in $\SSS_V$ because it consists of the points $(x_1, x_2) \in \SSS_V$ with $x_1 \neq 0 \neq x_2$. Since  $C_n = diag(p^n,1) C_0$ for any $n \geq 0$, from \eqref{Tt0} we get 
\[ \mu(C_n) = \left| \det \begin{pmatrix}
p^n & 0 \\
0 & 1
\end{pmatrix} \right|_p \mu(C_0) = p^{-n} \mu(C_0). 
\]
In the same way one shows that $\mu(C_n) = p^{-|n|} \mu(C_0)$
for any $n\in \ZZ$. Thus
\begin{equation*}
 1 = \mu(\SSS_V) = \sum_{n \in \ZZ} \mu (C_n) = \frac{p + 1}{p-1} \mu(C_0),
\end{equation*}
so
\[ \lambda_{K_p} (F_n) = \mu (C_n) = \frac{p-1}{p+1} p^{-|n|} \]
for any $n \in \ZZ$. 
\end{proof}
Let's prove the formula of $\Xi_p(a_{p,m})$.
\begin{proof}[Proof of Lemma \ref{formula_Xi_p}]

From the definition of $F_n$ we easily see that  
\[ \normp{a_{p,m} k e_1}\inv = \begin{cases}
p^{-m} &\text{ if } n \leq 0, \\
p^{n-m} &\text{ if } 0 < n < 2m, \\
p^{m} &\text{ if } n \geq 2m, 
\end{cases}
\]
for $k \in F_n$ and $m \geq 0 $. Then
\begin{align*}
\Xi_p(a_{p,m}) &=  \sum_{n \in \ZZ} \int_{F_n} \normp{a_{p,m} k e_1}\inv \dd \lambda_{K_p}(k) \\
		&= \left( \sum_{n \leq 0} \lambda_{K_p}(F_n) \right) p^{-m} + \sum_{0 < n < 2m} \lambda_{K_p}(F_n) p^{n-m} + \left( \sum_{n\geq 2m} \lambda_{K_p}(F_n)\right) p^{m} \\
		&= \frac{p-1}{p+1} \left[ \sum_{n \leq 0} p^{n-m} + \sum_{0< n < 2m} p^{-m} + \sum_{n\geq 2m} p^{m-n} \right] \\
		&= \frac{p^{-m}}{p+1} ((2m+1)(p-1) + 2),
\end{align*}
as we wanted. 
\end{proof}

	\subsection{Decay speed of almost $L^{2m}$ unitary representations of $SL(2,\RR)$}\label{app_proof_decay_SL(2,R)}

\begin{proof}[Proof of Proposition \ref{Decay_smooth_vectors_almost_L2m_explicit}]
Consider an almost $L^{2m}$ unitary representation $\pi$ of $SL(2,\RR)$. 
Let $r_\theta \in K_{2, \infty}$ be the rotation of angle $\theta$. We decompose $\pi$ as Hilbert sum of irreducible unitary representations of $K_{2, \infty}$
  \[\Hc_\pi = \widehat{\bigoplus_{j \in \ZZ}} \, \Hc_j, \]
where $K_{2, \infty}$ acts on any $v' \in \Hc_j$ as $\pi(r_\theta) v' = e^{\theta j i} v'$. We write $v$ and $w$ respectively as $\sum_{j \in \ZZ} v_j$ and $\sum_{k \in \ZZ} w_k$ with $v_\ell, w_\ell \in \Hc_\ell$. Note that $\pi^{\otimes m}$ is tempered because the coefficient of any two pure tensors $v'_1 \otimes \cdots \otimes v'_m, w'_1 \otimes \cdots \otimes w'_m$ is the product of $m$ coefficients of $\pi$. Applying Proposition \ref{Effective_decay_tempered} to $\pi^{\otimes m}$ and the $K_{2, \infty}$-invariant vectors $v_j^{\otimes m}$ and $w_k^{\otimes m}$ we obtain 
\begin{align*}
|\scalar{\pi^{\otimes m}(g)v_j^{\otimes m}}{w_k^{\otimes m}}| & \leq \Xi_\infty(g) \norm{v_j^{\otimes m}} \, \norm{w_k^{\otimes m}} \delta(v_j^{\otimes m}) \delta(w_k^{\otimes m})  \\
	& = \Xi_\infty (g) (\norm{v_j} \, \norm{w_k})^m
\end{align*}
for any $g \in SL(2,\RR)$, or equivalently 
\[|\scalar{\pi(g)v_j}{w_k}| \leq \Xi_\infty^{\frac{1}{m}} (g) \norm{v_j} \, \norm{w_k}.\] 
The last inequality for $g = a_{\infty,t}$ and Corollary \ref{Xi_infty_exponential_bound} yield
\begin{equation}\label{abc}
|\scalar{\pi(a_{\infty,t}) v_j}{w_k}| \leq e^{-\frac{t}{3m}} (\consDecayHarishReal ^{1/m} \norm{v_j} \, \norm{w_k}),
\end{equation}
for any $j,k \in \ZZ$. Recall that
\[\Zc = \begin{pmatrix}
0 & -1 \\
1 & 0
\end{pmatrix} \quad \text{and} \quad \norm{v'}_{\Zc} = (\norm{v'}^2 + \norm{\pi(\Zc)v'}^2)^\frac{1}{2}, \]
for any $v' \in \Hc_\pi$. For any $j \in \ZZ$ we have
\[\pi(\Zc) v_j = \frac{\dd}{\dd \theta}\Big |_{\theta = 0} e^{ \theta ji} v_j = ji v_j.\]
Since $v$ is $K_{2,\infty}$-smooth, then
\[ \norm{\pi(\Zc)v}^2 =   \sum_{j \in \ZZ} j^2 \norm{v_j}^2 < \infty,  \]
and similarly for $w$. We are ready to prove the bound for the coefficient of $v$ and $w$. We use \eqref{abc} and the Cauchy-Schwartz inequality as follows:
\begin{align*}
|\scalar{\pi(g) v}{w}| & \leq \sum_{j,k \in \ZZ}|\scalar{\pi(g) v_j}{w_k}| \\
			& \leq e^{-\frac{t}{3m}} \consDecayHarishReal^{\frac{1}{m}} \left(\sum_{j \in \ZZ} \norm{v_j} \right) \left(\sum_{k \in \ZZ} \norm{w_k} \right) \\
			& =  e^{-\frac{t}{3m}} \consDecayHarishReal^{\frac{1}{m}} \left(\norm{v_0} + \sum_{j \in \ZZ-\{0\}} \frac{1}{j} \, \norm{j v_j} \right)  \left(\norm{w_0} + \sum_{k \in \ZZ-\{0\}} \frac{1}{k} \, \norm{k w_k} \right)\\
			& \leq e^{-\frac{t}{3m}} \consDecayHarishReal^{\frac{1}{m}} (1 + 2\zeta(2)) \left( \norm{v_0}^2 + \norm{\pi(\Zc)v}^2 \right)^\frac{1}{2} \left( \norm{w_0}^2 + \norm{\pi(\Zc)w}^2 \right)^\frac{1}{2}\\
			& \leq e^{-\frac{t}{3m}} (5 \consDecayHarishReal^{\frac{1}{m}}  \norm{v}_\Zc \norm{w}_\Zc).
\end{align*}

\end{proof}

	\subsection{Decay speed of tempered unitary representations of $SL(2,\QQ_p)$}\label{app_proof_decay_SL(2,Q_p)}

Recall that $K_{2,p} = SL(2,\ZZ_p)$ for any prime number $p$. For any positive integer $n$ we denote by $K_{2,p}(n)$ the kernel of the natural map $K_{2,p} \to SL(2,\ZZ/ p^n\ZZ)$. Let $\pi$ be a unitary representation of $SL(2,\QQ_p)$. A vector $v \in \Hc_\pi$ invariant under some $K_{2,p}(n)$ is $K_{2,p}$-finite, and the next result gives an upper bound of $\delta(v)$.

\begin{Lem}\label{Size_SL(2,Z_mod_pnZ}
For any positive integer $n$ we have
\[ \# SL(2,\ZZ / p^n\ZZ) = p^{3n} - p^{3n-2}. \]
\end{Lem}

\begin{proof}
Let $A_n$ be a free $(\ZZ / p^n\ZZ)$-module with basis $(e_1, e_2)$. The $SL(2,\ZZ / p^n\ZZ)$-orbit of $e_1$ has size $p^{2n}-p^{2n-2}$ because it consists of the elements $x_1e_1 + x_2e_2$ of $A_n$ such that $p$ does not divide $x_1$ and $x_2$ simultaneously. The stabilizer of $e_1$ in $SL(2,\ZZ / p^n\ZZ)$ is 
\[S_n := \begin{pmatrix}
1 & \ZZ / p^n\ZZ \\
0 & 1
\end{pmatrix}. 
\]
Thus
\begin{equation}
\# SL(2,\ZZ / p^n\ZZ) = \# (SL(2,\ZZ / p^n\ZZ)e_1) \, \# S_n = p^{3n} - p^{3n-2},  
\end{equation}
as claimed. 
\end{proof}

\begin{proof}[Proof of Proposition \ref{Decay_speed_K_n-inv-vectors_explicit}]
Recall that $\pi$ is a tempered unitary representation of $SL(2,\QQ_p)$ and that $v_1, v_2 \in \Hc_\pi$ are invariant under $K_{2,p}(n_1)$ and $K_{2,p}(n_2)$, respectively. Note that $\pi(K_{2,p})v_i$ has at most $[K_{2,p} : K_{2,p}(n_i)] = \# SL(2,\ZZ / p^{n_i}\ZZ)$ elements, so  
\begin{equation}\label{qwerty}
\delta (v_i) \leq \left( \# SL(2,\ZZ / p^{n_i}\ZZ) \right)^\frac{1}{2} < p^{\frac{3}{2}n_i} 
\end{equation} 
 by Lemma \ref{Size_SL(2,Z_mod_pnZ}. For any $m \in \ZZ$ we have
\begin{equation}\label{azerty}
|\scalar{\pi(a_{p,m}) v_1}{v_2}| \leq \Xi_p(a_{p,m}) \norm{v_1} \, \norm{v_2} \delta(v_1) \delta(v_2) 
\end{equation}
by Proposition \ref{Effective_decay_tempered}. To obtain the desired inequality we use in \eqref{azerty} the upper bound of $\Xi_p(a_{p,m})$ in Corollary \ref{Xi_p-exponential_bound}, and \eqref{qwerty}.
\end{proof}

\section{Volume computations on Lie groups and homogeneous spaces}\label{app_volume_computations}

This appendix gathers the computations that give the explicit constants of the main results of the article. The appendix is divided into six parts. In Subsection \ref{subsec_Haar} we fix the Haar measures we'll work with. Subsections \ref{subsec_ROG}, \ref{subsec_p-adic_og} and \ref{app_SBF} are consecrated to orthogonal groups. The computations on lower-triangular groups are carried out in Subsection \ref{subsec_triangular_gps}. We close by establishing in Subsection \ref{subsec_vol_X1} a formula for the volume of the space of covolume 1 lattices of $\QQ_S^d$.

	\subsection{Haar measures}\label{subsec_Haar} 
In this subsection we fix the Haar measures used throughout the article. More precisely, we explain how we normalize the Haar measure of a Lie group once we fix a basis of the Lie algebra.

 Let $\nu$ be a place of $\QQ$ and let $H_0$ be a closed subgroup of $G_{d,\nu} = GL(d,\QQ_\nu)$.  Let $(y_1, \ldots, y_k)$ be the coordinates on the Lie algebra $\hgot_0$ of $H_0$ with respect to a basis $\beta$. We take $\Haar{\mathfrak{h}_0}$ such that
\[ \Haar{\hgot_0} ( \{(y_1, \ldots, y_k) \in \hgot_0 \mid |y_1|_\nu, \ldots, |y_k|_\nu \leq 1\} = \begin{cases}
1 & \text{if } \nu < \infty, \\
2^k & \text{if } \nu = \infty. \\
\end{cases} \] 
Let $\omega$ be the left-invariant volume form on $H_0$ such that  
 \[\omega_{I_d} = (dy_1 \wedge \cdots \wedge dy_k)_0.\]
We denote by $\lambda_{H_0}$ the left Haar measure on $H_0$ given by integration with respect to $\omega$.  We'll say that a Haar measure $\nu_{H_0}$ on $H_0$ and a Lebesgue measure $\nu_{\hgot_0}$ on $\hgot_0$ are \textit{compatible} if they can be obtained as above from the same basis $\beta$ of $\hgot_0$.

Here is the relation between compatible Haar measures of $H_0$ and $\hgot_0$. Let $\psi(z)$ be the power series $\frac{1-e^{-z}}{z}$. 

\begin{Lem}\label{Density_Haar_H_P}
Let $H_0$ be a Lie subgroup of $G_{d,\nu}$ and let $\nu_{H_0}, \nu_{\mathfrak{h}_0}$ be compatible Haar measures on $H_0$ and $\mathfrak{h}_0$. Suppose that the exponential map of $H_0$ is bijective between neighborhoods $\mathfrak{U}$ and $U$ of $0 \in \mathfrak{h}_0$ and $I_d \in H_0$, respectively, and let $\log: U \to \mathfrak{U}$ be its inverse. The map
\[ D_{H_0}(v) = |\det \psi( ad_{\mathfrak{h}_0} v)|_\nu\]
is a density of $\log_* \lambda_{H_0}$ with respect to $\lambda_{\mathfrak{h}_0}$ on $\mathfrak{U}$. 
\end{Lem}

\begin{proof}
Since $\nu_{H_0}$ and $\nu_{\hgot_0}$ are compatible, there are coordinates $(y_1,\ldots, y_k)$ on $\hgot_0$ with respect to a basis of $\hgot_0$ such that $\nu_{\hgot_0}$ and $\nu_{H_0}$  are respectively given by the integration with respect to $dy_1 \wedge \cdots \wedge dy_k$ and the left-invariant volume form $\omega$ on $H_0$ with $\omega_{I_d} = (dy_1 \wedge \cdots \wedge dy_k)_0$. We just have to prove that 
\[(\exp^* \omega)_v = \det \psi( ad_{\mathfrak{h}_0} v) (dy_1 \wedge \cdots \wedge dy_k)_v \]
for any $v \in \mathfrak{U}$. The derivative of $\exp: \mathfrak{h}_0 \to H_0$ at $v$---see \cite[p. 99]{godement_introduction_2017}---is given by
\[ D \exp_v = L_{h} \circ \psi(\ad_{\hgot_0} v), \]
where $L_h: H_0 \to H_0$ is the left multiplication by $h = \exp v $. Thus
\begin{align*}
(\exp^* \omega)_v &= \psi(\ad_{\hgot_0} v)^* L_h^* \omega_h \\
			&=  \psi(\ad_{\hgot_0} v)^* \omega_{I_d} \\
			& = \det \psi(\ad_{\hgot_0} v) (dy_1 \wedge \cdots \wedge dy_k)_0.
\end{align*} 

\end{proof}

	\subsection{Real orthogonal groups}\label{subsec_ROG}
The goal of this section is to prove Lemma \ref{Volume_small_balls_real_orthogonal_groups}, an estimate of the volume of small balls in real orthogonal groups. 

Let $\normi{\cdot}$ be the norm on $M_d(\RR)$ of the maximum of the absolute values of the entries. If $P$ is a non-degenerate quadratic form on $\RR^d$ we denote by $H_P$ the group $O(P,\RR)$. Let $b_P$ be the matrix of $P$ in the canonical basis $e_1, \ldots, e_d$ of $\RR^d$. The Lie algebra of $H_P$ is
\[ \mathfrak{h}_P = \{v \in \mathfrak{gl}(d,\RR) \mid \tra v b_P + b_P v = 0 \}. \]
Let $e_1^*, \cdots, e_d^*$ be the standard basis of the dual space $(\RR^d)^*$ and let $E_{ij}$ be the matrix $e_j^* \otimes e_i$. If $P(x) = a_1 x_1^2 + \cdots + a_d x_d^2$, we consider the basis of $\hgot_P$ formed by the
\begin{equation}\label{Haar_H_infty}
H_{ij} = E_{ij} - a_i a_j\inv E_{ji}
\end{equation}
with $1 \leq i < j \leq d$. We denote by $\Haar{H_P}$ and $\Haar{\hgot_P}$ the Haar measures of $H_P$ and $\hgot_P$ induced by this basis. We'll estimate the volume of small symmetric balls of $H_P$ centered at the identity. For any $r > 0$ we define
\begin{equation}\label{Sym_ball_real}
H_{P}(r) = \{ h \in H_P \mid \normi{h - I_d} < r, \normi{h\inv - I_d} < r \}. 
\end{equation}
For any integer $d \geq 2$ consider
\[\consInfVolROG{d} = \left( \frac{1}{3d} \right)^{\frac{d(d-1)}{2}} \quad \text{and} \quad 
 \consSupVolROG{d} = \left(\frac{20 d}{3} \right)^{\frac{d(d-1)}{2}}.\]
\begin{Lem}\label{Volume_small_balls_real_orthogonal_groups}
Let $P(x) = a_1 x_1^2 + \cdots + a_d x_d^2$ with $d \geq 3$ and each $a_i \in \{ \pm 1 \}$. Then
\[ \consInfVolROG{d} r^{\frac{1}{2} d(d-1)} < \lambda_{H_P}(H_P(r)) < \consSupVolROG{d} r^{\frac{1}{2} d(d-1)}\]
for any $r \in \left(0, \frac{2}{5d} \right]$.
\end{Lem}

The idea to prove Lemma \ref{Volume_small_balls_real_orthogonal_groups} is simple: for any small $r$ we can parametrize $H_P(r)$ via the exponential map of $H_P$. We'll see that the volumes of $H_P(r)$ and $\exp\inv H_P(r)$ are comparable. We break the proof into several auxiliary lemmas. 

Let $G_{d, \infty} = GL(d,\RR)$ and $\ggot_{d, \infty} = \mathfrak{gl}(d,\RR)$. To compare the sizes of $v \in \ggot_{d, \infty}$ and $\exp v$ it is convenient to work with a submultiplicative norm. Let $\normop{\cdot}$ be the operator norm on $\ggot_{d, \infty}$ with respect to $\normi{\cdot}$ on $\RR^d$. For any linear subspace $\mathfrak{w}$ of $\ggot_{d, \infty}$ we define
\[\mathfrak{w} (r) = \{ v \in \mathfrak{w} \mid \normop{v} < r \}. \]
The next lemma gives open subsets of $\ggot_{d, \infty}$ and $G_{d, \infty}$ where $\exp$ restricts to a diffeomorphism.

\begin{Lem}\label{Where_exp_is_inversible}
For any $d \geq 2$, the exponential map of $G_{d, \infty}$ is a diffeomorphism between $\mathfrak{g}_{d, \infty} (\log 2)$ and an open subset of $G_{d, \infty}$. 
\end{Lem}
\begin{proof}
The inverse of $\exp$, that we'll denote by $\log$, is defined by the power series
\[\log g = \sum_{i = 1}^\infty \frac{(-1)^{n+1}}{n}(g- I_d)^n, \]
that converges when $\normop{g - I_d} < 1$. If $v$ is in $\mathfrak{g}_{d, \infty} (\log 2)$, then
\[ \normop{\exp (v)- I_d} < e^{\log 2} - 1 = 1, \]
so we are done. 
\end{proof}

The next result is useful to estimate the volume of $G_{d, \infty} (r)$.

\begin{Lem}\label{Ensanwichando_bola_simétrica}
Let $d \geq 3$. For any $r \in \left(0, \frac{2}{5d} \right]$ we have
\[\exp \mathfrak{g}_{d, \infty} \left(\frac{9}{10}r \right) \subseteq G_{d, \infty}(r) \subseteq \exp  \mathfrak{g}_{d, \infty} \left(\frac{5d}{3}	r\right) \subseteq \exp \mathfrak{g}_{d, \infty} (\log 2) .\] 
\end{Lem}

We'll use the next two inequalities in the proof of Lemma \ref{Ensanwichando_bola_simétrica}:

\begin{align}
\label{Bou_up} \frac{s_1}{1-s_1} &\leq \frac{5}{3}s_1 \quad \text{for any } s_1 \in \left[0, \frac{2}{5} \right],\\
\label{Bou_low} \frac{9}{10} s_2 &\leq \log(1+s_2) \quad \text{for any } s_2 \in \left[0, \frac{2}{15}\right].
\end{align}

\begin{proof}[Proof of Lemma \ref{Ensanwichando_bola_simétrica}]
Take $r \in \left(0, \frac{2}{5d}\right]$ and $g = \exp v \in G_{d, \infty}(r)$. We have
\[\normop{g - I_d} \leq d \normi{g-I_d} \leq dr < 1, \]
so $\log g = \sum_{n \geq 1} \frac{(-1)^{n+1}}{n} (g-I_3)^n$ converges. Moreover
\begin{equation*}
\normop{\log g}  \leq \sum_{n\geq 1} \normop{g - I_d}^n \leq \frac{dr}{1-dr}, 
,
\end{equation*}
so $\normop{\log g} \leq \frac{5d}{3}r$ by \eqref{Bou_up}. This proves the inclusion
\[G_{d, \infty} (r) \subseteq \exp \mathfrak{g}_{d, \infty} \left(\frac{5d}{3} r\right). \]
Since $r \leq \frac{2}{5d} \leq \frac{2}{15}$, then $\frac{5d}{3}r \leq \frac{2}{3} < \log 2 = 0.693 \ldots$. So $\log$ is a diffeomorphism from $G_{d, \infty}(r)$ to an open subset of $\mathfrak{g}_{d, \infty}$---see Lemma \ref{Where_exp_is_inversible}.

Now take $v \in \mathfrak{g}_{d, \infty}(9r/10)$ and set $g = \exp v$. By \eqref{Bou_low} we have
\[\normop{v} \leq \log(1+r), \]
so
\[\normi{g - I_d} \leq \normop{g - I_d} \leq e^{\normop{v}} - 1 < r. \]
The same argument with $-v$ gives the same upper bound for $\normi{g\inv - I_d}$. This proves the inclusion
\[\exp \mathfrak{g}_{d, \infty} \left( \frac{9}{10}r \right) \subseteq G_{d, \infty}(r). \] 
\end{proof} 

Recall that $\psi(z) = \frac{1-e^{-z}}{z}$ and that, near the identity, the map 
\[D_{H_P}(v) = \det \psi(ad_{\hgot_P} \, v)\] 
is a density of $\log_* \lambda_{H_P}$ with respect to $\lambda_{\hgot_P}$ by Lemma \ref{Density_Haar_H_P}. The next lemma bounds $D_{H_P}$ near $0$.  
 
\begin{Lem}\label{Density_Haar_effective}
Let $P$ be a non-degenerate quadratic form on $\RR^d$. For any $v \in \hgot_P(1/2)$ we have
\[5^{-\frac{d(d-1)}{2}} < D_{H_P}(v) < 2^{\frac{d(d-1)}{2}}. \]
\end{Lem}

We introduce the function $\ftt(r) = \frac{1}{r}(e^r - 1 -r)$. To prove Lemma \ref{Density_Haar_effective} we use the next inequality. 
\begin{Lem}\label{Cota_psi}
For any $z \in \CC$ with $|z|_\infty < r$ we have
\[ 1 - \ftt(r) < |\psi(z)|_\infty < 1 + \ftt(r). \]
\end{Lem}

\begin{proof}
Note that
\begin{equation}\label{eee1}	
 |\psi(z)-1|_\infty  = \left| \sum_{n = 1}^{\infty} (-1)^n \frac{z^n}{(n+1)!} \right|_\infty  \leq \sum_{n=1}^\infty \frac{|z|^n_\infty}{(n+1)!} = \ftt(|z|_\infty) < \ftt(r).
\end{equation}
By the triangle inequality we have
\begin{equation}\label{eee2}
1 - |\psi(z)-1|_\infty \leq |\psi(z)|_\infty \leq 1 + |\psi(z) - 1|_\infty. 
\end{equation}
The inequality of the statement follows from \eqref{eee1} and \eqref{eee2}. 
\end{proof}

\begin{proof}[Proof of Lemma \ref{Density_Haar_effective}]
Note that $D_{H_P}(v) = \prod_\eta \psi(\eta)$, where $\eta$ runs through all the eigenvalues---with multiplicity---of $ad_{\hgot_P} \, v$. Since $\psi(0) = 0$, the $\eta = 0$ don't contribute to $D_{H_P}(v)$, so we'll neglect them.  Each $\eta$ is the sum of two eigenvalues  of $v$. Let $\normop{\cdot}$ be the operator norm on $\mathfrak{gl}(d,\CC)$ with respect to $\normi{\cdot}$ on $\CC^d$. Suppose that $v \in \hgot_P(1/2)$ and let $\lambda$ be an eigenvalue of $v$. Then
\[|\lambda|_\infty \leq \normop{v}  < \frac{1}{2}. \]
It follows that $|\eta|_\infty < 1$ for any $\eta$, and
\begin{equation}\label{bbk}
\frac{1}{5} < 0.281\ldots =  - \ftt(1) \leq |\psi(\eta)|_\infty \leq 1 + \ftt(1) = 1.718\ldots < 2 
\end{equation}
by Lemma \ref{Cota_psi}. To obtain the inequality of the statement we multiply \eqref{bbk} for all $\eta \neq 0$. There are less than $\frac{1}{2}d(d-1)$ these\footnote{Since $v$ is antisymmetric with respect to a non-degenerate symmetric bilinear form, the eigenvalues of $v$ come in pairs: $\pm \lambda_1 , \ldots, \pm \lambda_{\frac{d}{2}}$ if $d$ is even and $\pm \lambda_1 , \ldots, \pm \lambda_{\frac{d-1}{2}}, 0$ if $d$ is odd.}.

\end{proof}

The last thing we need to estimate the volume of $H_P(r)$ is an approximation of the volume of $\hgot_P(1)$.
 
\begin{Lem}\label{Volume_unit_ball_in_h_P}
Consider a quadratic form $P(x) = a_1 x_1^2 + \cdots + a_d x_d^2$ with each $a_i \in \{ \pm 1 \}$. 
Then 
\[ \left(\frac{2}{d}\right) ^{\frac{d(d-1)}{2}}  \leq \lambda_{\mathfrak{h}_P}(\mathfrak{h}_P(1)) \leq 2^{\frac{d(d-1)}{2}}.\] 
\end{Lem}

\begin{proof}
We define
\[ \mathfrak{B}_P(r) = \{ v \in \hgot_P \mid \normi{v} < r \}. \]
Let $v = \sum_{i<j} v_{ij} H_{ij} \in \hgot_P$. Since $a_i = \pm 1$ for every $i$, we have $\normi{v} = \max_{i<j} |v_{ij}|_\infty$. Then $\lambda_{\hgot_P}(\mathfrak{B}_P(r)) = (2r)^{dim \hgot_P}$ by our choice of $\lambda_{\hgot_P}$. 

Note that
\begin{equation}\label{chan} 
\mathfrak{B}_P(1/d) \subseteq \hgot_P(1) \subseteq \mathfrak{B}_P(1) 
\end{equation}
since $\frac{1}{d} \normop{\cdot} \leq \normi{\cdot} \leq \normop{\cdot}$ on $\mathfrak{gl}(d,\RR)$. The result follows by comparing the volumes in \eqref{chan}.
 
\end{proof}

\begin{proof}[Proof of Lemma \ref{Volume_small_balls_real_orthogonal_groups}]
Since $r \leq \frac{2}{5d}$, then
\[ \exp \hgot_P \left( \frac{9}{10}r \right) \subseteq H_P(r) \subseteq \exp \hgot_P \left( \frac{5dr}{3} \right) \subseteq \exp \hgot_P (\log 2) \]
by Lemma \ref{Ensanwichando_bola_simétrica}. Recall that
\[D_{H_P}(v) \leq 2^{\frac{d(d-1)}{2}} \]
by Lemma \ref{Density_Haar_effective} since $r \leq \frac{2}{5d} <  \frac{1}{2}$. Thus
\begin{align*}
\lambda_{H_P}(H_P(r)) & \leq \lambda_{H_P} \left( \exp \hgot_P \left( \frac{5dr}{3} \right) \right) \\
			& = \int_{\hgot_P(5dr/3)} D_{H_P}(v) \dd \lambda_{\hgot_P}(v) \\
			& < 2^{\frac{d(d-1)}{2}} \lambda_{\hgot_P}(\hgot_P(1)) \left( \frac{5dr}{3} \right)^{\frac{d(d-1)}{2}} \\
			& \leq  \left(\frac{20 d}{3} \right)^{\frac{d(d-1)}{2}} r^{\frac{d(d-1)}{2}},
\end{align*}
where we used Lemma \ref{Volume_unit_ball_in_h_P} to obtain the last line. A similar argument gives the lower bound:
\begin{align*}
\lambda_{H_P}(H_P(r)) & > 5^{-\frac{d(d-1)}{2}} \lambda_{\hgot_P} \left( \hgot_P \left( \frac{9r}{10} \right) \right) \\
			& \geq  \left( \frac{1}{3d} \right)^{\frac{d(d-1)}{2}} r^{\frac{d(d-1)}{2}}. 
\end{align*}

\end{proof}

	\subsection{$p$-adic orthogonal groups}\label{subsec_p-adic_og}

Now we prove a formula---Lemma \ref{Volume_small_balls_in_p-adic_Lie_groups}---for the volume of small balls in $p$-adic orthogonal groups.  

For any Lie subgroup $H_0$ of $G_{d,p} = GL(d,\QQ_p)$ and any $r > 0$, we define
\begin{equation}\label{p-adic_ball}
H_0(r) = \{ h\in H_0 \mid \normp{h - I_d} \leq r, \normp{h\inv - I_d}\leq r \}. 
\end{equation}
Let $P(x) = a_1 x_1^2 + \cdots + a_d x_d^2$ with $a_1,\ldots, a_d \in \QQ_p^\times$, and let $H_P = O(P, \QQ_p)$. Like in the real case, we consider also the Haar measures $\lambda_{H_P}$ and $\lambda_{\hgot_P}$ of $H_P$ and $\mathfrak{h}_P$ induced by the basis 
\begin{equation}\label{Haar_H_p}
H_{ij} = E_{ij} - a_i a_j\inv E_{ji}, \quad i<j,
\end{equation}
of $\hgot_P$. We define
\[ \mathscr{D}_P = \prod_{i<j} \min \{1, |a_i a_j\inv|_p\}. \]
Here is our volume formula. 

\begin{Lem}\label{Volume_small_balls_in_p-adic_Lie_groups}
Let $p$ be a prime number and let $P$ be a non-degenerate diagonal quadratic form on $\QQ_p^d$. For any integer $n \geq 3$ we have
\[\lambda_{H_P}(H_P(p^{-n})) = \mathscr{D}_P \cdot p^{-\frac{1}{2}d(d-1)n}. \]
\end{Lem}

\begin{Cor}\label{Volume_balls_standard_p-adic-orthogonal_groups}
Let $d \geq 3$ and let $H$ be the orthogonal $\QQ_p$-group of a standard quadratic form on $\QQ_p^d$. Then
\[ \lambda_{H_P}(H(p^{-n})) = p^{-\frac{1}{2}d(d-1)n}. \]
\end{Cor}

\begin{proof}
If $P(x) = a_1 x_1^2 + \cdots +  a_d x_d^2$ is a standard quadratic form on $\QQ_p^d$, then $|a_k|_p = 1$ for any $k \leq d-2$ and $p\inv \leq |a_{d-1}|_p, |a_d|_p \leq 1$. It follows that $|a_i a_j\inv|_p \geq 1$ if $i < j$, so $\mathscr{D}_P = 1$. 
\end{proof}

The strategy to prove Lemma \ref{Volume_small_balls_in_p-adic_Lie_groups} is the same as in the real case---see Lemma \ref{Volume_small_balls_real_orthogonal_groups}. First we establish the auxiliary results we'll use for the main proof. Let $\ggot_{d,p} = \mathfrak{gl}(d,\QQ_p) \simeq M_d(\QQ_p)$ and let $\normp{\cdot}$ be the norm on $\ggot_{d,p}$ of the maximum of the $p$-adic absolute values of the entries. For any linear subspace $\wgot$ of $\ggot_{d,p}$ we define 
\[\mathfrak{w}(r) = \{ v \in \mathfrak{w} \mid \normp{v} \leq r \}. \]

\begin{Lem}\label{Exponential_GL(d,Q_p)}
Let $p$ be a prime number and $d \geq 2$. The exponential map is a bijection between $\ggot_{d,p} (p^{-n})$ and $G_{d,p}(p^{-n})$ for any integer $n \geq 3$. 
\end{Lem}

One has to be careful because the exponential map doesn't converge in all of $\QQ_p$. We take care of this with the next lemma. 

\begin{Lem}\label{lolito}
Consider $t\in \QQ_p$. If $0 < |t|_p \leq p^{-3}$, then:
	\begin{enumerate}[$(i)$]
	\item $\left|\frac{t^m}{m!}\right|_p < |t|_p $ for any integer $m > 1$.
	\item $\frac{t^m}{m!} \to 0$ as $m \to \infty$.
	\end{enumerate}
\end{Lem}

\begin{proof}
Notice that $\frac{m}{p-1} < 3(m-1)$ for any integer $m \geq 2$ and any prime number $p$. Then

\[3(m-1)  > \frac{m}{p-1}   
= \sum_{j\geq 1} \frac{m}{p^j}  
\geq \sum_{j \geq 1} \left \lfloor \frac{m}{p^j} \right \rfloor = - \log_p |m!|_p,
\]
so
\[|m!|\inv_p < p^{3(m-1)} \leq |t|_p^{-(m-1)}.  \]
It follows that $\left |\frac{t^m}{m!} \right|_p < |t|_p$.

Since
\begin{equation}\label{aaaa}
\left |\frac{t^m}{m!} \right|_p  \leq p^{-\log_p |m!|_p} p^{-3m} \leq p^{\frac{m}{p-1}} p^{-3m} = p^{m\left( \frac{1}{p-1} - 3 \right)}, 
\end{equation}
and $ \frac{1}{p-1} - 3 <0$, the last term of \eqref{aaaa}, and hence also the first, tend to 0 as $m \to \infty$. 
\end{proof}

\begin{proof}[Proof of Lemma \ref{Exponential_GL(d,Q_p)}]
Consider $n\geq 3$ and $v \in \mathfrak{g}_{d,p}$ with $\normp{v} \leq p^{-n}$. By Lemma \ref{lolito} we have
\[\left|\left| \frac{v^m}{m!} \right|\right|_p \leq \frac{\normp{v}^m}{|m!|_p} < \normp{v} \]
for any $m \geq 2$, so 
\[\exp (v) - I_d = v + \sum_{m \geq 2} \frac{v^m}{m!} \]
converges and $\normp{\exp(v) - I_d} = \normp{v}$. This shows that $\exp$ sends $\mathfrak{g}_{d,p}(p^{-n})$ to $G_{d,p}(p^{-n})$. 

Now consider $g \in G_{d,p} (p^{-n})$. We have
\[  \left|\left| \frac{(g-I_d)^m}{m} \right|\right|_p \leq \frac{\normp{g-I_d}^m}{|m|_p} \leq \frac{\normp{g-I_d}^m}{|m!|_p} < \normp{g- I_d}\]
for any $m \geq 2$, so 
\[\log g  = (g-I_d) + \sum_{m\geq 2} \frac{(-1)^{m+1}}{m} (g - I_d)^m \]
converges and $\normp{\log g} = \normp{g - I_d}$. Thus $\log = \exp\inv$ sends $G_{d,p}(p^{-n})$ to $\mathfrak{g}_{d,p} (p^{-n})$, which proves our claim.  
\end{proof}

\begin{Lem}\label{Haar_is_Lebesgue_p-adic_orthogonal_groups}
Consider a prime number $p$ and $d \geq 2$. Let $H$ be the orthogonal $\QQ_p$-group of a non-degenerate diagonal quadratic form on $\QQ_p^d$. Then $\log_* \lambda_H = \lambda_\hgot$ on $\hgot(p^{-3})$. 
\end{Lem}

\begin{proof}
Since $\lambda_H$ and $\lambda_\hgot$ are compatible, then
 \[ \frac{\dd \log_* \lambda_H}{\dd \lambda_\hgot} (v) = |\det \psi ( ad_\hgot v)|_p \]
 on $\hgot(p^{-3})$ by lemmas \ref{Exponential_GL(d,Q_p)} and \ref{Density_Haar_H_P}. Thus it suffices to prove that $|\psi(\eta)|_p = 1$ for any eigenvalue $\eta$ of $ad_\hgot v$ when $v \in \hgot_P(p^{-3})$. 
 
Let's fix $v \in \hgot_P(p^{-3})$. Let $K$ be a finite extension of $\QQ_p$ that has the eigenvalues $\lambda$ of $v$. The $p$-adic absolute value extends uniquely to an ultrametric absolute value on $K$ that we denote also by $|\cdot|_p$---see \cite[Theorem 11, Chapter III]{koblitz_p-adic_1984}. On $K^d$ we consider the norm 
\[ \normp{(y_1,\ldots,y_d)} = \max_i |y_i|_p. \]
Let $y \in K^d$ be an eigenvector of $v$ corresponding to $\lambda \in K$, with $\normp{y} = 1$. Then
\[ |\lambda|_p = \normp{v y} \leq \normp{v} \leq p^{-3}. \]
An eigenvalue $\eta$ of $ad_\hgot \, v$ is the sum of two eigenvalues of $v$, hence $|\eta|_p \leq p^{-3}$. By Lemma \ref{lolito}, $|1-e^{-\eta}|_p = |\eta|_p$, so 
\[|\psi(\eta)|_p = \left| \frac{1-e^{-\eta}}{\eta} \right|_p = 1. \] 
\end{proof}

Now we compute the volume of $\hgot_P(1)$.
\begin{Lem}\label{Volume_unit_ball_in_Lie_algebra_p-adic}
Let $P(x)$ be a non-degenerate diagonal quadratic form on $\QQ_p^d$. Then
\[ \lambda_{\mathfrak{h}} ( \mathfrak{h}_P(1)) = \mathscr{D}_P. \]
\end{Lem}

\begin{proof}
We write $P(x) = a_1 x_1^2 + \cdots + a_d x_d^2$. Recall that  the matrices $H_{ij} = E_{ij} - a_i a_j\inv E_{ji}$, $i<j$ form a basis of $\hgot_P$. Take $v = \sum_{i<j} v_{ij} H_{ij} \in \mathfrak{h}_P$. Consider the norm
\[ \norm{v}' = \max_{i<j} |v_{ij}|_p\]
and let 
\[ \mathfrak{B}' = \{ v \in \mathfrak{h}_P \mid \norm{v}' \leq 1 \}. \]
Then $\lambda_{\mathfrak{h}}( \mathfrak{B}') = 1$ by our choice of Haar measure on $\hgot_P$. The entries of $v$ are $v_{ij}$ and $a_i a_j\inv v_{ij}$ with $i<j$, in particular $\norm{v}' \leq \normp{v}$.  The ball $\mathfrak{h}_P(1)$ is an open subgroup of $\mathfrak{B}'$, hence
\[ [\mathfrak{B}' : \hgot_P(1)] \lambda_{\mathfrak{h}_P}(\hgot_P(1)) = 1. \]
Notice that $v$ is respectively in $\mathfrak{B}'$ and $\hgot_P(1)$ if and only if $|v_{ij}|_p \leq 1$ and $|v_{ij}|_p \leq \min \{1, |a_i a_j\inv|_p\}$ for every $i<j$. Hence
\[ \frac{1}{[\mathfrak{B}' : \hgot_P(1)]} = \prod_{i<j} \min \{1, |a_i a_j\inv|_p \} =  \mathscr{D}_P. \]
\end{proof}

We are ready to compute the volume of $H_P(p^{-n})$.

\begin{proof}[Proof of Lemma \ref{Volume_small_balls_in_p-adic_Lie_groups}]
Let $n \geq 3$. Then $\exp \hgot_P(p^{-n}) = H_P(p^{-n})$ by Lemma \ref{Exponential_GL(d,Q_p)}. By Lemma \ref{Haar_is_Lebesgue_p-adic_orthogonal_groups} we know that $\log_* \lambda_{H_P} = \lambda_{\hgot_P}$ on $\hgot_P(p^{-3})$, so 
\[\lambda_{H_P}(H_P(p^{-n})) = \lambda_{\hgot_P} (\hgot_P(p^{-n}))
 = \lambda_{\hgot_P} ( \hgot_P(1)) p^{-n \dim H_P}. \]
 We are done thanks to Lemma \ref{Volume_unit_ball_in_Lie_algebra_p-adic}.
\end{proof}

	\subsection{Bump functions on real orthogonal groups}\label{app_SBF}

Let $P(x) = a_1x_1^2 +\cdots + a_d x_d^2$ with $a_i \in \{\pm 1\}$ and let $H = O(P, \RR)$. The objective of this section is to establish Lemma \ref{Smooth_bump_functions}, where we build smooth bump functions on $H$ and we estimate the $L^2$-norm of them and their first-order derivatives.

See the beginning of Section \ref{subsec_ROG} for the definition of $H(r)$ and the choice of Haar measures $\Haar{H}$ and $\Haar{\hgot}$ of $H$ and its Lie algebra $\hgot$, respectively. We define $\consBumpFuncRealOG{d}$ as $10^{d^2} d^{\frac{1}{4}(d+2)^2}$ for any integer $d \geq 0$. Here is our main statement. 

\begin{Lem}\label{Smooth_bump_functions} 
Consider $d \geq 3$. Let $P(x) = a_1 x_1^2 + \cdots + a_d x_d^2$ with each $a_i \in \{ \pm 1 \}$, and $H = O(P,\RR)$. For any $r \in \left(0, \frac{2}{5d} \right]$ there is a smooth function $\psi_r: H \to [0, \infty)$ with support in $H(r)$ such that $\norm{\psi_r}_{L^1} = 1,$ 
\[ \norm{\psi_r}_{L^2} < \consBumpFuncRealOG{d} r^{-\frac{1}{4}d(d-1)}, \]
and for any $v \in \hgot$,
\[  \norm{v(\psi_r)}_{L^2} \leq \consBumpFuncRealOG{d} \normi{v} r^{-(\frac{1}{4}d(d-1)+1)}.\]
\end{Lem}

The maps $\psi_r$ will be obtained by precomposing bump functions on $\hgot$ with the logarithm map of $H$. Before proving Lemma \ref{Smooth_bump_functions} we establish four auxiliary results. The first one is a simple lemma of analysis on $\RR^m$, so we omit the proof. We denote
\[\BB^m(r) = \{ x \in \RR^m \mid \normi{x} < r \}. \]
We endow the space of linear maps $\RR^m \to \RR$ with the operator norm with respect to $\normi{\cdot}$ on $\RR^m$ and $\RR$. For any $\mathcal{C}^1$ map $F':\RR^m \to \RR$ and any $r > 0$, we define
 \[ F'_{[r]}(x) = r^{-m} F'(r\inv x) \quad \text{and} \quad M_{F'} = \sup_{x \in \BB^m(1)} \normop{D_x F'}.\]
For any vector field $V$ on $\mathbb{B}^m(1)$ we set
\[ M_V = \sup_{x \in \BB^m(1)} \normi{V_x}.\]

\begin{Lem}\label{SBF_on_Rm}
Let $F':\RR^m \to [0, \infty)$ be a $\mathcal{C}^1$ function with support  in $\BB^{m}(1)$ and let $r > 0$.

	\begin{enumerate}[$(a)$]
		\item The map $F'_{[r]}$ has support in $\BB^{m}(r)$.
		\item $\norm{F'_{[r]}}_{L^1} = \norm{F'}_{L^1}$.
		\item $|| F'_{[r]}||_{L^2} = r^{-\frac{m}{2}} ||F' ||_{L^2}$. 
		\item Suppose that $r \leq 1$. Let $V$ be a vector field on $\BB_{m}(1)$. Then
		\[ ||V(F'_{[r]})||_{L^2} \leq 2^\frac{m}{2} M_{F'} M_V r^{ -(\frac{m}{2} +1 )}. \]
	\end{enumerate}
\end{Lem}

Let $P$ and $H$ be as in Lemma \ref{Smooth_bump_functions}. We give now the bump function on $\hgot$ that we'll use to construct the $\psi_r$'s.  Recall that any $y \in \hgot$ is of the form
\[ y = \sum_{i<j} y_{ij} H_{ij},\]
where $H_{ij} = E_{ij} - a_i a_j\inv E_{ji}$---see \eqref{Haar_H_infty}. We define $F: \hgot \to [0,1]$ as
\[ F(y) = \prod_{i<j} \btt(y_{ij}), \]
where $\btt: \RR \to [0,1]$ is a smooth function with support in $[-1,1]$ such that $\int_{-1}^1 \btt(t) \dd t = 1$ and $|\btt'(t)|_\infty \leq 2$ for any $t \in \RR$. Recall the notation
\[ \mathfrak{B}(r) = \{ y \in \hgot \mid \normi{y} < r \}. \]
The map $F$ is smooth and has support in $\mathfrak{B}(1)$. Let's estimate $M_F$. 

\begin{Lem}\label{MF}
For any $y \in \hgot$ we have $\normop{D_y F} < d^2$. 
\end{Lem}

\begin{proof}
We have
\[ \left|\frac{\partial F}{\partial y_{i_0,j_0}} \right|_\infty = \left| \frac{b'(y_{i_0,j_0})}{b(y_{i_0,j_0})} \prod_{i<j} b(y_{ij})\right|_\infty \leq 2, \]
hence
\[ |(D_yF)v|_\infty = \left|\sum_{i<j} \frac{\partial F}{\partial y_{ij}} v_{ij} \right|_\infty \leq \sum_{i<j} 2 \normi{v} = d(d-1) \normi{v}. \]
The conclusion follows from this inequality. 
\end{proof}

Recall that $\normop{\cdot}$ is the operator norm on $\mathfrak{gl}(d,\RR)$ with respect to $\normi{\cdot}$ on $\RR^d$ and that 
\[\hgot(r) = \{ v \in \hgot \mid \normop{v} < r\}. \]
For any $v \in \hgot$ we denote by $\widetilde{v}$ the vector field 
\[y \mapsto \frac{Id - e ^{-ad_\hgot y}}{ad_\hgot y} (v)\] 
on $\hgot$. This is simply, near the identity, the left-invariant vector field determined by $v$ in exponential coordinates. Let's estimate $M_{\widetilde{v}}$.  

\begin{Lem}\label{Mv}
For any $v \in \hgot$ and any $y \in \hgot(1)$ we have
\[ \normi{\widetilde{v}_y} \leq 5d \normi{v}. \]
\end{Lem}

\begin{proof}
Recall that $\normop{\cdot}$ is the operator norm on $\hgot$ with respect to $\normi{\cdot}$ on $\RR^d$. We denote also by $\normop{\cdot}$ the operator norm on $\mathfrak{gl}(\hgot)$ with respect to $\normop{\cdot}$ on $\hgot$. Notice that
\[ \normop{ad\, y(y')} = \normop{yy' - y'y} \leq 2 \normop{y} \normop{y'},\]
so $\normop{ad\, y} \leq 2 \normop{y}$. We conclude as follows:
\begin{align*}
\normi{\widetilde{v}_y} \leq \normop{\widetilde{v}_y} & \leq \left| \left| Id - \frac{ad\, y }{2!} + \frac{(ad\, y)^2}{3!} - \cdots \right|\right|_{op} \normop{v} \\
	& \leq \left(1 + \frac{1}{2}(e^{\normop{ad\, y}} - 1) \right) \normop{v} \\
	& \leq \frac{1}{2}(e^2 + 1) d \normi{v} \leq 5d \normi{v}. 
\end{align*}

\end{proof}
For any $r \in \left( 0, \frac{2}{5d} \right]$ we define $r_1 = \frac{9}{10 d} r$ and $\psi_r': H \to [0, \infty)$ as 
\[ \psi'_r(h) = F_{[r_1]}(\log h) \ind_{H(r)}(h).\]
This function verifies almost all the properties we want in Lemma \ref{Smooth_bump_functions}. Let $\consBumpFuncRealOGbis{d} = 5d^3 (20d)^{\frac{1}{4}d(d-1) + 1}$. 

\begin{Lem}\label{SBF_bis}
For any $r \in \left( 0, \frac{2}{5d} \right]$ the map $\psi'_r: H \to [0, \infty)$ is smooth, has support in $H(r)$ and verifies the following
	\begin{enumerate}[$(i)$]
		\item $5^{-\frac{1}{2}d(d-1)} \leq \norm{\psi'_r}_{L^1(H)} \leq 2^{\frac{1}{2}d(d-1)}$,
		\item $ \norm{\psi'_r}_{L^2(H)} \leq \consBumpFuncRealOGbis{d} r^{-\frac{1}{4}d(d-1)}$,
		\item $\norm{v(\psi'_r)}_{L^2(H)} \leq \consBumpFuncRealOGbis{d} \normi{v} r^{-(\frac{1}{4}d(d-1) + 1)}$ for any $v \in \hgot$.  
	\end{enumerate}
\end{Lem}

\begin{proof}
Since $r \in \left( 0, \frac{2}{5d} \right]$, $\exp: \hgot \left(\frac{9r}{10} \right) \to H(r)$ is injective by Lemma \ref{Where_exp_is_inversible}. Note that $\mathfrak{B}(r_1) \subseteq \hgot \left(\frac{9r}{10} \right)$ since $\normop{v} \leq d \normi{v}$. The map $F_{[r_1]}: \hgot \to [0, \infty)$ is smooth and has support in $\mathfrak{B}(r_1)$, so $\psi'_r$ is smooth and has support in $\exp \mathfrak{B}(r_1)$, which is contained in $H(r)$. 

In the computations that follow we'll use the properties of $F_{[r_1]}$ in Lemma \ref{SBF_on_Rm}. By Lemma \ref{Density_Haar_H_P} we have
\[\int_H \psi'_r(h) \dd \lambda_H(h) = \int_{\mathfrak{B}(r_1)}  F_{[r_1]}(v) D_H(v) \dd \lambda_{\hgot}(v), \]
so $(i)$ results from the fact that $5^{-\frac{1}{2}d(d-1)} < D_H < 2 ^{\frac{1}{2}d(d-1)}$ on $\hgot(1/2)$---see Lemma \ref{Density_Haar_effective}. Now note that
\begin{align*}
\norm{\psi'_r}_{L^2(H)} & = \left( \int_{\mathfrak{B}(r_1)}  F_{[r_1]}^2(v) D_H(v) \dd \lambda_{\hgot}(v) \right)^\frac{1}{2} \\
	& \leq 2^{\frac{1}{4}d(d-1)} \norm{F_{[r_1]}}_{L^2(\hgot)} \\
	& = (2 r_1\inv)^{\frac{1}{4}d(d-1)} \norm{F}_{L^2(\hgot)}.
\end{align*}
We have $\norm{F}_{L^2(\hgot)}  = \norm{\btt}_{L^2(\RR)}^{dim\hgot} \leq 1$ since $\btt^2 \leq \btt$ and $\norm{\btt}_{L^1(\RR)} = 1$. Thus
\[ \norm{\psi'_r}_{L^2(H)} \leq (2 r_1\inv)^{\frac{1}{4}d(d-1)} = \left(\frac{20 d}{9} \right)^{\frac{1}{4}d(d-1)} r^{-\frac{1}{4}d(d-1)} < \consBumpFuncRealOGbis{d} r^{-\frac{1}{4}d(d-1)},\]
so $(ii)$ is established. For $v \in \hgot$ we have
\begin{align*}
\norm{v(\psi'_r)}_{L^2(H)} & \leq 2^{\frac{1}{4}d(d-1)} \norm{\widetilde{v}(F_{[r_1]})}_{L^2(\hgot)} \\
	& \leq 2^{\frac{1}{4}d(d-1)} \left( 2^{\frac{1}{4}d(d-1)} M_F M_{\widetilde{v}} \cdot r_1^{-(\frac{1}{4}d(d-1) + 1)} \right) \\
	& = 2^{\frac{1}{2}d(d-1)} \left( \frac{10 d}{9} \right)^{\frac{1}{4}d(d-1)} M_F M_{\widetilde{v}} \cdot r^{-(\frac{1}{4}d(d-1)+1)}.
\end{align*} 
Recall that $M_F < d^2 $ and $M_{\widetilde{v}} \leq 5d \normi{v}$ by lemmas \ref{MF} and \ref{Mv}, so 
\[ \norm{v(\psi'_r)}_{L^2(H)} < \consBumpFuncRealOGbis{d} r^{-(\frac{1}{4}d(d-1)+1)}. \]
\end{proof}

To obtain Lemma \ref{Smooth_bump_functions} we just have to normalize $\psi'_r$.

\begin{proof}[Proof of Lemma \ref{Smooth_bump_functions}]
Consider $r \in \left( 0, \frac{2}{5d} \right]$ and $\psi'_r: H \to [0,\infty)$ as in  Lemma \ref{SBF_bis}. We set $I_r = \norm{\psi'_r}_{L^1(H)} \inv$ and $\psi_r = I_r \psi'_r$. Then $\norm{\psi_r}_{L^1(H)} = 1$. By Lemma \ref{SBF_bis} we have $I_r \leq 5^{\frac{1}{2}d(d-1)}$, thus
\[ \norm{\psi_r}_{L^2(H)} \leq 5^{\frac{1}{2}d(d-1)} \consBumpFuncRealOGbis{d} r^{-\frac{1}{4}d(d-1)} < \consBumpFuncRealOG{d} r^{-\frac{1}{4}d(d-1)}, \]
and for any $v \in \hgot$
\[ \norm{v(\psi_r)}_{L^2(H)} \leq 5^{\frac{1}{2}d(d-1)} \consBumpFuncRealOGbis{d} \normi{v} r^{-(\frac{1}{4}d(d-1)+1)} \leq \consBumpFuncRealOG{d} \normi{v} r^{-(\frac{1}{4}d(d-1)+1)}. \]
\end{proof}

\subsection{Lower-triangular groups}\label{subsec_triangular_gps}

Let $H_S$ be the orthogonal $\QQ_S$-group of a diagonal quadratic form on $\QQ_S^d$. To prove the transversal recurrence of closed $H_S$-orbits in $X_{d,S}^1$---Lemma \ref{Transversal_recurrence}---in Section \ref{sec_vol_closed_orbits}, we thickened any such orbit using a subgroup $W_{d,S} = \prod_{\nu \in S} W_{d, \nu}$ of lower-triangular matrices in $GL(d,\QQ_S)$. Here we prove the volume estimates for the open subsets $W_{d,\nu}(r)$ of $W_{d,\nu}$---Lemma \ref{Volume_Wr_infty_cite} for $\nu = \infty$, and Lemma \ref{Volume_W_p_cite} for $\nu = p$---used in Lemma \ref{Volume_W_S}.

We start with $\nu = \infty$. Let $W_{d,\infty}$ be the group of lower-triangular matrices in $GL(d,\RR)$ with positive entries in the main diagonal. Let $E_{ij}$ be the $d\times d$ matrix with a 1 in the entry $i,j$ and zeros elsewhere, and let $F_k = E_{kk} - E_{dd}$ for any $1 \leq k < d$. The Haar measure of $W_{d,\infty}$ determined by the basis
\begin{equation}\label{Haar_W_infty}
\beta_{d,W} =  (F_1, \ldots, F_{d-1}, E_{21}, E_{32}, \ldots, E_{d,d-1}, E_{3,1}, \ldots, E_{d,d-2}, \ldots, E_{d1}) 
\end{equation}
of its Lie algebra $\wgot_{d,\infty}$, will be denoted by $\lambda_{W_{\infty}}$.   

The exponential map is a bijection between $\wgot_{d, \infty}$ and $W_{d, \infty}$. For any $r>0$ we define
\[\mathfrak{w}_{d,\infty}(r) = \{ v\in \wgot_{d, \infty} \mid \normop{v} < r \}\]
and 
\[W_{d, \infty}(r) = \exp(\mathfrak{w}_{d,\infty}(r)). \]
 We introduce $\consExpVolHTransversal{d} = \frac{d(d+1)}{2}-1$.

\begin{Lem}\label{Volume_Wr_infty_cite} 
For any $0< r < \frac{1}{2}$ we have 
\[ \consCoefVolHTransversalInf{d} r^{\consExpVolHTransversal{d}} < \Haar{\LowTMat{d}{\infty}} (\LowTMat{d}{\infty}(r)) < \consCoefVolHTransversalSup{d} r^{\consExpVolHTransversal{d}}, \]
where $V^-_d = \frac{2^{d-1}}{d^{2 c_d}}$ and $V^+_d = 2^{d^2-1}$.
\end{Lem}

To prove Lemma \ref{Volume_Wr_infty_cite} we'll use the next two auxiliary results.

\begin{Lem}\label{Eigenvalues_adv_W_infty}
Let $v = \sum_{j<i} v_{ij} E_{ij} \in \wgot_{d, \infty}$. The eigenvalues of $\ad v: \wgot_{d, \infty} \to \wgot_{d, \infty}$ are $0$, with multiplicity $d-1$, and $\eta_{ij} = v_{ii} - v_{jj}$ for any $1 \leq j < i \leq d$.  
\end{Lem}

\begin{proof}
Consider
\[\mathfrak{a} = \bigoplus_{k = 1}^{d-1} \RR F_i \quad \text{and} \quad \mathfrak{n} = \bigoplus_{i>j} \RR E_{ij}. \]
Notice that $\wgot_{d, \infty} = \mathfrak{a} \oplus \mathfrak{n}$. Write $v = v_1 + v_2$ with $v_1 \in \mathfrak{a}$ and $v_2 \in \mathfrak{n}$. The matrices of $ad\, v_1$ and $ad\, v_2$ in the basis $\beta_{d,W}$ are diagonal and strictly lower-diagonal. Hence the eigenvalues of $ad\, v$ are the diagonal entries of $ad\, v_1$. Since $[v_1, F_k] = 0$ for $1 \leq k \leq d-1$, $0$ is an eigenvalue with multiplicity (at least) $d-1$. For $i>j$ we have $[v_1, E_{ij}] = (v_{ii} - v_{jj})E_{ij}$, which gives the eigenvalues $\eta_{ij}$.  
\end{proof}

\begin{Lem}\label{Volume_unit_ball_Lie_W_infty}
We have
\[\left(\frac{2}{d^2}\right)^{c_d} \leq \lambda_{\wgot_{d, \infty}}(\mathfrak{w}_{d,\infty}(1)) \leq 2^{c_d}.\]
\end{Lem}

\begin{proof}
For any 
\[ v = \sum_{k=1}^{d-1} v_{kk} F_k + \sum_{i>j} v_{ij} E_{ij} =  \begin{pmatrix}
v_{11} &  \cdots & 0 & 0 \\
\vdots & \ddots & \vdots & \vdots \\
v_{d-1,1}  & \cdots & v_{d-1,d-1} & 0\\
v_{d1} & \cdots & v_{d,d-1} & -(v_{11}+\cdots+v_{d-1,d-1})
\end{pmatrix} \in \wgot_{d, \infty} \]
we define
\[\norm{v}' = \max_{i \geq j} |v_{ij}|_\infty \]
and 
\[ \mathfrak{B}' (r) = \{ v \in \wgot_{d, \infty} \mid \norm{v}' < r \}, \]
so $\lambda_{\mathfrak{w}_{d,\infty}}( \mathfrak{B}'(1)) = 2^{c_d}$ by the choice of $\lambda_{\mathfrak{w}_{d,\infty}}$.  Notice that
\[ \norm{\cdot}' \leq \normi{\cdot} \leq \normop{\cdot} \leq d \normi{\cdot} \leq d^2 \norm{\cdot}', \]
so 
\[\mathfrak{B}'\left( \frac{1}{d^2} \right) \subseteq \mathfrak{w}_{d,\infty}(1) \subseteq \mathfrak{B}'(1).\]
The comparison of the volumes of these balls gives the inequality of the statement. 
\end{proof}

We are ready to estimate the volume of $W_{d, \infty}(r)$.

\begin{proof}[Proof of Lemma \ref{Volume_Wr_infty_cite}]
We consider again the analytic map $\psi(z) = \frac{1}{z}(1-e^{-z})$. The exponential map is a bijection $\wgot_{d, \infty} \to W_{d,\infty}$ and, like in the proof of Lemma \ref{Density_Haar_H_P}, the positive function
\[D(v) = \det \psi(\ad v) \]
is a density of $log_* \lambda_{LowTMat{d}{\infty}}$ with respect to $\lambda_{\mathfrak{w}_{d,\infty}}$. 

Consider $v = \sum_{i\geq j} v_{ij} E_{ij} \in \wgot_{d, \infty}$ with $\normop{v} < \frac{1}{2}$. Aside from the 0 with multiplicity $d-1$\footnote{They don't contribute to the density since $\psi(0) = 1$.}, the eigenvalues of $\ad v$ are $\eta_{ij} = v_{ii}-v_{jj}$ for $1 \leq j < i \leq d$
according to Lemma \ref{Eigenvalues_adv_W_infty}, so 
\[D(v) = \prod_{i>j} \psi(\eta_{ij}).\]
For $i>j$ we have
\[|\eta_{ij}|_\infty = |v_{ii} - v_{jj}|_\infty \leq 2 \normi{v} \leq 2 \normop{v} < 1. \]
Since $\psi$\footnote{From the identity $z^2 e^z \psi'(z) = z + 1 - e^z$ we readily see that $\psi' < 0$ on $\RR^\times$} is decreasing on $\RR$, we have
\[\frac{1}{2} < 0.632 \ldots = \psi(1) < \psi(\eta_{ij}) < \psi(-1) = 1.718 \ldots < 2, \]
hence
\[ 2^{-\frac{d(d-1)}{2}} \leq D(v) \leq 2^\frac{d(d-1)}{2}. \]
For any $0 < r \leq \frac{1}{2}$ we have
\begin{align*}
\Haar{\LowTMat{d}{\infty}} (W_{d, \infty}(r)) & = \int_{\mathfrak{w}_{d,\infty}(r)} D(v) \dd \Haar{\mathfrak{w}_{d,\infty}}(v) \\
			& < 2^\frac{d(d-1)}{2} \Haar{\mathfrak{w}_{d,\infty}}(\mathfrak{w}_{d,\infty}(1)) r^{\consExpVolHTransversal{d}} \\
			& < 2^{d^2-1} r^{\consExpVolHTransversal{d}}.
\end{align*} 
We used Lemma \ref{Volume_unit_ball_Lie_W_infty} to get the last line. In the same fashion we obtain
\[ \Haar{\LowTMat{d}{\infty}} ( W_{d, \infty}(r)) > 
2^{-\frac{d(d-1)}{2}} \left( \frac{2}{d^2} \right)^{\consExpVolHTransversal{d}}  r^{\consExpVolHTransversal{d}} =
 \frac{2^{d-1}}{d^{2 \consExpVolHTransversal{d}}} r^{\consExpVolHTransversal{d}}. \] 
\end{proof}

Now we work with the group $\LowTMat{d}{p}$ of lower-triangular matrices in $GL(d,\QQ_p)$. Recall that the goal is to establish Lemma \ref{Volume_W_p_cite} below.

We endow $W_{d,p}$ with the Haar measure $\lambda_{W_p}$ determined by the basis
\begin{equation}\label{Haar_W_p}
(E_{11}, \ldots , E_{dd}, E_{21}, E_{32}, \ldots, E_{d,d-1}, \ldots, E_{d1})
\end{equation}
of its Lie algebra $\wgot_{d,p}$. We'll compute the measure of small compact-open subgroups of $W_{d,p}$ of the following form: For any $r > 0$, set
\[W_{d,p}(r) = \{w \in W_{d,p} \mid \normp{w - I_d} \leq r, \normp{w\inv - I_d} \leq r \}. \]
We consider also the orthogonal $\QQ_p$-group $H_p$ of a non-degenerate diagonal quadratic form $P(x) = a_1 x_1^2 + \cdots + a_d x_d^2$ on $\QQ_p^d$. We define $\ell_p$ as $2$ for $p=2$, and $1$ for any odd $p$.

\begin{Lem}\label{Volume_W_p_cite} 
Let $p$ be a prime number. For any $n \geq 3$ we have 
\[\Haar{\LowTMat{d}{p}}(W_{d,p}(p^{-n})) = p^{-(\consExpVolHTransversal{d}+1) n}. \]
\end{Lem}

To compute the volume of $W_{d,p}(p^{-n})$ we use the next two lemmas. The proof of the first one is the same as in Lemma \ref{Eigenvalues_adv_W_infty}. 
 
\begin{Lem}\label{Eigenvalues_adv_Wp}
 Consider $v = \sum_{i,j} v_{ij} E_{ij} \in \wgot_{d,p}$.
The eigenvalues of $\ad v : \wgot_{d,p} \to \wgot_{d,p}$ are  
$v_{ii}-v_{jj}$, for any $1\leq j < i \leq d$, and $0$ with multiplicity $d$. 
\end{Lem}

We use once more the analytic function $\psi(\theta) = \frac{1}{\theta}(1 - e^{-\theta})$. 

\begin{Lem}\label{Lolito_3}
Let $p$ be a prime number. Then $|\psi(\theta)|_p = 1$ for any $\theta \in p^3 \ZZ_p$.  
\end{Lem} 

\begin{proof}
Notice that
\[\psi(\theta) = \sum_{j = 0}^ \infty \frac{(-1)^j}{(j+1)!} \theta^j. \]
We have
\[\left| \frac{(-1)^j}{(j+1)!} \theta^j \right|_p < \left| \frac{\theta^{j+1}}{(j+1)!} \right|_p, \]
and the right-hand side term tends to 0 as $j \to \infty$ by $(ii)$ of Lemma \ref{lolito}, so $\psi(\theta)$ converges. We also have
\[ \left| \frac{(-1)^j}{(j+1)!} \theta^j \right|_p < 1 \]
for any $j \leq 1$ by $(i)$ of Lemma \ref{lolito}, thus $|\psi(\theta)|_p = 1$. 
\end{proof}

We are ready for the main proof.

\begin{proof}[Proof of Lemma \ref{Volume_W_p_cite}]
The exponential map is a bijection between
\[ \wgot_{d,p}(p^{-n}) = \{ v \in \wgot_{d,p} \mid \normp{v} \leq p^{-n} \} \]
and $W_{d,p}(p^{-n})$ for any $n \geq 3$ by Lemma \ref{Exponential_GL(d,Q_p)}, and the map 
\[D(v) = |\det \psi(\ad v)|_p \]
is a density of $\log_* \lambda_{W_{d,p}}$ with respect to $\lambda_{\mathfrak{w}_{d,p}}$ on $\wgot_{d,p}(p^{-n})$. If $v = \sum_{i,j} v_{ij} E_{ij}$, then
\[D(v) =  \prod_{j < i} |\psi(v_{ii}-v_{jj}) |_p \]
by Lemma \ref{Eigenvalues_adv_Wp}. If $\normp{v} \leq p^{-3}$, then $D(v) = 1$ by Lemma \ref{Lolito_3}. Hence
\[ \Haar{\LowTMat{d}{p}}(W_{d,p}(p^{-n})) = \Haar{\mathfrak{w}_{d,p}}(\wgot_{d,p}(p^{-n})) = p^{-(\consExpVolHTransversal{d}+1)n} \]
for any $n \geq 3$. 
\end{proof}

\subsection{The volume of $X_{d,S}^1$}\label{subsec_vol_X1}

Here we compute the volume—Lemma \ref{Volume_X_Sd^1}—of the space $X_{d,S}^1$ of covolume 1 lattices of $\QQ_S^d$. This is used in the proof of Lemma \ref{Injective_implies_small_r}. 

First let's fix the normalizations of the relevant measures. For any place $\nu$ of $\QQ$ and any $d \geq 1$, let
\[G'_{d, \nu} = \{g \in GL(d,\QQ_\nu) \mid |\det g|_\nu = 1 \}. \]
Note that $G'_{d,\infty} = SL^\pm(d,\RR)$, and $G'_{d,p}$ is an open subgroup of $G_{d,p} = GL(d,\QQ_p)$. Let $\Haar{G'_{d,\infty}}$ be the Haar measure of $G'_{d,\infty}$ determined by the basis
\begin{equation}\label{Haar_SL(d,R)}
(E_{11}-E_{dd}, \ldots, E_{d-1,d-1} - E_{dd}) \cup (E_{ij})_{i \neq j}
\end{equation}
of $\mathfrak{sl}(d,\RR)$. Let $\Haar{G_{d,p}}$ be the Haar measure of $G_{d,p}$ relative to the basis
\begin{equation}\label{Haar_G_d,p}
(E_{ij})_{i,j=1}^d
\end{equation}
of $\mathfrak{gl}(d,\QQ_p)$, and let $\Haar{G'_{d,p}}$ be the restriction of $\Haar{G_{d,p}}$ to $G'_{d,p}$.
We define $G'_{d,S} = \prod_{\nu \in S} G'_{d,\nu}$ and $\Gamma'_{d,S} = \GammaS{d} \cap G'_{d,S}$, for any finite set $S$ of places of $\QQ$. Let $\Haar{G'_{d,S}} = \otimes_{\nu \in S} \Haar{G'_{d,\nu}}$. Recall that $\LatSpaceUnoS{d}$ can be identified with $G'_{d,S} / \Gamma'_{d,S}$ by Lemma \ref{G'_covers_X1}. We endow $\LatSpaceUnoS{d}$ with the $G'_{d,S}$-invariant measure $\beta_{d,S}$ determined by $\Haar{G'_{d,S}}$.

\begin{Lem}\label{Volume_X_Sd^1}
For any  finite set $S = \{\infty\} \cup S_f$ of places of $\QQ$ we have
\[ vol\, \LatSpaceUnoS{d} = vol\, X^1_{d,\infty} \prod_{p \in S_f} \prod_{j = 1}^d \left( 1- \frac{1}{p^j} \right). \]
\end{Lem} 

\begin{Rem}
See \cite[p. 145]{siegel_lectures_1989} for the explicit value of $vol \, X^1_{d, \infty}$.
\end{Rem}

We start by showing that $vol \, \LatSpaceUnoS{d}$ and $vol \, X^1_{d,\infty}$ differ only by a factor coming from the $GL(d,\ZZ_p)$'s, for $p \in S_f$.  

\begin{Lem}\label{Volume_Xuno_as_product}
For any finite set $S = \{\infty\} \cup S_f$ of places of $\QQ$ we have 
\[vol\, \LatSpaceUnoS{d} = vol\, X^1_{d,\infty} \prod_{p \in S_f} \Haar{G_{p,d}}(GL(d,\ZZ_p)). \] 
\end{Lem}

\begin{proof}
Let $\Gamma$ be a lattice of a locally compact group $G$. We say that a measurable subset $U$ of $G$ is a \textit{fundamental domain} of $\Gamma$ in $G$ if and only if any $g \in G$ can be written uniquely as $u \gamma$, for some $u \in U$ and $\gamma \in \Gamma$. Hence $vol(G/\Gamma) = \Haar{G}(U)$ for any such $U$. Consider 
\[G''_{d,S} := SL^\pm(d,\RR) \times \prod_{p \in S_f} GL(d,\ZZ_p), \]
and $\Gamma''_{d,S} = \GammaS{d} \cap G''_{d,S}$. We will see that $\LatSpaceUnoS{d}$ can be identified with $G''_{d,S} / \Gamma''_{d,S}$. Thus, to establish the result it suffices to show that $U_{d,\infty} \times \prod_{p \in S_f} GL(d,\ZZ_p)$ is a fundamental domain of $\Gamma''_{d,S}$ in $G''_{d,S}$, for any fundamental domain $U_{d,\infty}$ of $\Gamma_{d,\infty} = GL(d,\ZZ)$ in $SL^\pm(d,\RR)$. 

First let's show that $G''_{d,S}$ acts transitively on $X_{d,S}^1$. By Lemma \ref{G'_covers_X1}, any lattice $\Delta$ of $\QQ_S^d$ of covolume 1 is of the form $g' \ZZ_S^d$ for some $g' \in G'_{d,S}$. Suppose that $S_f = \{p_1, \ldots, p_k\}$. Let's see that we may take $g'$ with $g'_{p_1} \in GL(d,\ZZ_{p_1})$. Note that for any prime $p$\footnote{Since $\ZZ[1/p]$ and $SL(d,\ZZ[1/p])$ are dense in $\QQ_p$ and $SL(d,\QQ_p)$, respectively.},
\[GL(d,\QQ_p) = GL(d,\ZZ_p) GL(d,\ZZ[1/p]).\] 
We write $g'_{p_1}$ as $k_{1} \gamma_{1}$, for some $k_{1} \in GL(d,\ZZ_{p_1})$ and $\gamma_{1} \in GL(d,\ZZ[1/p_1])$. Thus $\det \gamma_{1}$ is both a unit of $\ZZ[1/p_1]$ and $\ZZ_{p_1}$, since it's equal to $\det (k_{1}\inv g'_{p_1})$. Hence $\gamma_{1}$ belongs to  $SL^\pm(d,\ZZ[1/p_1])$. Let $\widetilde{\gamma_{1}} = (\gamma_{1}, \ldots, \gamma_{1}) \in \Gamma'_{d,S}$ and $g^\star = g' \widetilde{\gamma_1}\inv$. Then $g^\star_{p_1}$ is in $GL(d,\ZZ_{p_1})$, and $\Delta = g^\star \ZZ_S^d$. Remark that if for some $i>1$, $g'_{p_i}$ already was in $GL(d,\ZZ_{p_i})$, the same is true for $g^\star_{p_i}$ since $\gamma_{1}$ belongs to $GL(d,\ZZ_{p_i})$. Continuing this process with $p_2,\ldots,p_k$ we manage to express $\Delta$ as $g'' \ZZ_S^d$ for some $g'' \in G''_{d,S}$. 

Now we identify $X^1_{d,S}$ with $G''_{d,S} / \Gamma''_{d,S}$. Note that $\Gamma''_{d,S}$ is the diagonal copy of $GL(d,\ZZ)$ in $G''_{d,S}$. Let $U_{d,\infty}$ be a fundamental domain of $GL(d,\ZZ)$ in $SL^\pm(d,\RR)$, and set $U_{d,S} = U_{d,\infty} \times \prod_{p \in S_f} GL(d,\ZZ_p)$. Since $SL^\pm(d,\RR) = U_{d,\infty} GL(d,\ZZ)$ and $GL(d,\ZZ) \subseteq GL(d,\ZZ_{p_i})$, then $G''_{d,S} = U_{d,S} \Gamma''_{d,S}$. Consider now $u,v \in U_{d,S}$, $\gamma_1, \gamma_2 \in GL(d,\ZZ)$ and $\widetilde{\gamma_i} = (\gamma_i, \ldots, \gamma_i) \in \Gamma''_{d,S}$. If $u \widetilde{\gamma_1} = v \widetilde{\gamma_2}$, comparing the real coordinates we see that $\gamma_1 = \gamma_2$, so $u = v$. Thus $U_{d,S}$ is a fundamental domain of $\Gamma''_{d,S}$ in $G''_{d,S}$.   
\end{proof}

The next lemma gives the volume of $GL(d,\ZZ_p)$. We will prove it after using it to set Lemma \ref{Volume_X_Sd^1}.

\begin{Lem}\label{Vol_GLdZp}
For any prime $p$ and any integer $d \geq 2$ we have
\[ \Haar{G_{d,p}}(GL(d,\ZZ_p)) =  \prod_{j = 1}^d \left( 1- \frac{1}{p^j} \right).  \]
\end{Lem}

\begin{proof}[Proof of Lemma \ref{Volume_X_Sd^1}]
The result follows from Lemma \ref{Volume_Xuno_as_product} and Lemma \ref{Vol_GLdZp}.
\end{proof}

To prove Lemma \ref{Vol_GLdZp} we will compute the volume and the index of a (finite-index) subgroup of $GL(d,\ZZ_p)$.  This is carried out in the next three auxiliary results. We omit the proof of the first since it's very similar to Lemma \ref{Haar_is_Lebesgue_p-adic_orthogonal_groups}. For the definition of $G_{d,p}(r)$ see \eqref{p-adic_ball}.

\begin{Lem}\label{Vol_in_GL(d)}
For any prime number $p$ and any integer $n \geq 3$ we have
\[ \lambda_{G_{d,p}} G_{d,p}(p^{-n}) = p^{-d^2n} \]
\end{Lem}

Let's now compute the index of $G_{d,p}(p^{-n})$ in $GL(d,\ZZ_p)$, or in other words, the cardinality of $GL(d, \ZZ / p^n \ZZ)$.  Consider positive integers $d$ and $N$. A complete flag of $(\ZZ / N \ZZ)^d$ is a sequence
\[ 0 = A_0 \subseteq A_1 \subseteq \ldots \subseteq A_d = (\ZZ / N \ZZ)^d, \]
where $A_i$ is a free $\ZZ / N\ZZ$-submodule of  $(\ZZ / N\ZZ)^d$ of rank $i$. We denote by $F_N(d)$ the number of complete flags of $(\ZZ / N\ZZ)^d$. In the following lemma, $\varphi (N) = \# ( \ZZ / N\ZZ)^\times $.

\begin{Lem}\label{Count_of_flags}
For any prime $p$ and any integers $n, d > 0$ we have
\[F_{p^n}(d) =  \frac{p^{\frac{1}{2} d(d+1)n}}{\varphi(p^n)^d} \prod_{j=1}^d  \left( 1-  \frac{1}{p^{j}} \right).  \]
\end{Lem}

\begin{proof}
We'll prove the result by induction on $d$. The base case $d = 1$ is immediate. 

Suppose that the formula holds for $d-1$. The number of flags of $M_d = (\ZZ / p^n \ZZ)^d$ having $A_1$ equal to a fixed line $\ell$ of $M_d$ is $F_{p^n}(d-1)$, since $M_d / \ell$ is a free $(\ZZ / p^n \ZZ)$-module of rank $d-1$. Thus,
\[F_{p^n}(d) = \# \{\text{lines in } M_d\} \cdot F_{p^n}(d-1).\]
Note that $(a_1,\ldots,a_d) \in M_d$ generates a line if and only if some $a_i$ is invertible in $\ZZ / p^n \ZZ$. There are $p^{dn} - p^{d(n-1)}$ such elements. A line in $M_d$ has $p^n$ generators, thus
\[\# \{\text{lines in } M_d\} = \frac{p^{dn}}{\varphi(p^n)}(1-p^{-d}),\]
which proves the formula for $F_{p^n}(d)$.
\end{proof}

\begin{Lem}\label{Card_GL(d)}
For any prime $p$ and any $d \geq 2$ we have
\[ \# GL(d,\ZZ / p^n \ZZ) = p^{d^2n} \prod_{j = 1}^d \left( 1- \frac{1}{p^j} \right).  \]
\end{Lem} 

\begin{proof}
The group $GL(d,\ZZ / p^n \ZZ)$ acts transitively on the set of complete flags of $M_d = (\ZZ / p^n \ZZ)^d$. Let $e_1, \ldots, e_d$ be the standard basis of $M_d$. The stabilizer of
\[0 \subseteq \langle e_1 \rangle \subseteq \ldots \subseteq \langle e_1,\ldots,e_{d-1} \rangle \subseteq M_d \]
is the subgroup of upper-triangular matrices in $GL(d,\ZZ / p^n \ZZ)$, which has\footnote{The entries in the main diagonal are in $(\ZZ / p^n \ZZ)^\times$ and the entries above the main diagonal can be chosen freely in $\ZZ / p^n \ZZ$.} $\varphi(p^n)^d p^{\frac{d(d-1)}{2}n}$ elements. Then
\[  \# GL(d,\ZZ / p^n \ZZ) = \varphi(p^n)^d p^{\frac{d(d-1)}{2}n} F_{p^n}(d),  \]
so the formula follows from Lemma \ref{Count_of_flags}.
\end{proof}
Now we can compute the volume of $GL(d,\ZZ_p)$.

\begin{proof}[Proof of Lemma \ref{Vol_GLdZp}]
Consider an integer $n \geq 3$. We have 
\begin{align*}
\lambda_{G_{d,p}}(GL(d,\ZZ_p)) &=  [G_{d,p}(1) : G_{d,p}(p^{-n})] \, \lambda_{G_{d,p}}(G_{d,p}(p^{-n})) \\ 
	&=  \# GL(d,\ZZ / p^n \ZZ) \, \lambda_{G_{d,p}}(G_{d,p}(p^{-n})),
	\end{align*}
so the formula is obtained from  Lemma \ref{Card_GL(d)} and Lemma \ref{Vol_in_GL(d)}.
\end{proof}

\section{Effective $S$-adic Mahler's Criterion}\label{app_Mahler}
The classic Mahler's Criterion says that a set of covolume 1 lattices of $\RR^d$ is relatively compact if and only if its nonzero vectors are uniformly bounded away from 0. The goal of this appendix is to establish an effective $S$-adic Mahler's criterion—Lemma \ref{CS_aux_2}—. It is used in the proof of Proposition \ref{Compact_meeting_closed_H_S-orbits}.

We begin by proving a lemma that essentially is an effective Mahler's criterion for lattices of $\RR^d$. We denote $GL(d,\RR)$ by $G_{d,\infty}$. Let's recall the definition of Siegel set of $G_{d,\infty}$. Consider the following subgroups of $G_{d,\infty}$:
\begin{align*}
K &= O(d,\RR), \\
A &= \{diag(a_1,\cdots,a_d) \in G_{d,\infty} \mid a_i>0 \text{ for any } i \}, \\
N &= \{\text{unipotent, upper-triangular matrices in } G_{d,\infty} \}.
\end{align*}
For any $\alpha, \beta>0$ we define
\begin{align*}
A_\alpha &= \{diag(a_1,\cdots, a_d) \in A \mid a_i \leq \alpha a_{i+1} \text{ for any } i \}, \\
N_\beta &= \{n \in N \mid \normi{n-I_d} \leq \beta \}.
\end{align*}
The $(\alpha, \beta)$-Siegel set of $G_{d,\infty}$ is defined as
\[ \sieR{\alpha}{\beta} = K A_\alpha N_\beta. \]
We denote by $\normeuc{\cdot}$ the standard euclidean norm of $\RR^d$, and let $SL^\pm(d,\RR)$ be the group of real, $d \times d$ matrices with determinant $\pm 1$.

\begin{Lem}\label{CS_aux_1}
Consider two positive numbers $\beta \leq 1 \leq \alpha$. Any $g \in \sieR{\alpha}{\beta} \cap SL^{\pm} (d,\RR)$ verifies
\[ \normi{g} \leq \sqrt{d} \, \alpha^{\frac{(d-1)^2}{2}} \max\{1, \normeuc{ge_1}^{-(d-1)} \}. \]
\end{Lem}
\begin{proof}
We write $g = kan$ with $k \in O(d,\RR), a = diag(a_d, \ldots, a_d) \in A_\alpha$, and $n \in N_\beta$. Notice that $\normeuc{g e_1} = a_1$, and 
\begin{equation}\label{Mahler1}
\normi{an} = \normi{(a_1, \ldots, a_d)}
\end{equation} 
because $an$ is upper-triangular and
\[|(an)_{ij}|_\infty = |a_i n_{ij}|_\infty \leq \beta |a_i|_\infty \leq |a_i|_\infty \]
for any $i < j$. We'll bound from above $a_k$ in terms of $a_1$ and $\alpha$. By the definition of $A_\alpha$, for any $i < j$ we have $a_i \leq \alpha^{j-i} a_j$ and $a_j\inv \leq \alpha^{j-i} a_i\inv$. Then
\begin{align*}
1 = [a_1 \cdots a_{k-1}] a_k [a_{k+1} \cdots a_d] & \geq [a_1 (\alpha\inv a_1) \cdots (\alpha^{-(k-2)} a_1)] a_k [(\alpha\inv a_k) \cdots (\alpha^{-(d-k)} a_k)] \\
			& = \alpha^{-\frac{(k-2)(k-1)}{2}} a_1^{k-1} \alpha^{-\frac{(d-k)(d-k+1)}{2}} a_k^{d-k+1},
\end{align*}
hence
\begin{align}
\notag a_k & \leq \alpha^{\frac{(k-2)(k-1)}{2(d-k+1)}} \alpha^{\frac{d-k}{2}} a_1^{-\frac{k-1}{d-k+1}} \\
\notag			& \leq \alpha^{\frac{(d-2)(d-1)}{2}} \alpha^{\frac{d-1}{2}} \max \{ 1, a_1^{-(d-1)} \} \\
\label{Mahler2}			& = \alpha^{\frac{(d-1)^2}{2}} \max \{ 1, a_1^{-(d-1)} \}.
\end{align}   
Combining \eqref{Mahler1} and \eqref{Mahler2} we obtain the upper bound for $\normi{an}$ that proves the result, since
\[
\normi{g} = \normi{kan}  \leq \sqrt{d} \, \normi{an}.
\]

\end{proof}

We pass to the $S$-adic setting. Let $S= \{\infty\} \cup S_f$ be a finite set of places of $\QQ$. We define the height of any $v \in \QQ_S^d$ as
\[ \hgt_S(v) = \normeuc{v_\infty} \prod_{p \in S_f} \normp{v_p}. \]
Note that we are using the euclidean norm of $\RR^d$ instead of $\normi{\cdot}$. We define the systole of a lattice $\Delta$ of $\QQ_S^d$ as
\[ \alpha_1(\Delta) = \min_{v \in \Delta-\{0\}} \hgt_S(v). \]
Recall that the covolume of $\Delta$, denoted $\text{cov} \Delta$, is the volume of $\QQ_S^d / \Delta$. If we write $\Delta$ as $g \ZZ_S^d$ for some $g \in GL(d,\QQ_S)$, it's easy to see that
 \[ \text{cov } g\ZZ_S^d = \hgt_S(\det g).	 \]
We are ready to prove our effective $S$-adic Mahler's criterion.

\begin{Lem}\label{CS_aux_2}
Consider a finite set $S = \{ \infty \} \cup S_f$ of places of $\QQ$  and an integer $d\geq 2$. For any lattice $\Delta$ of $\QQ_S^d$ of covolume 1, there is $g \in G_{d,\infty} \times \prod_{p \in S_f} GL(d,\ZZ_p)$ with  
\begin{equation}\label{Choco}
\normi{g_\infty} \leq \sqrt{d} \cdot \left(\frac{2}{\sqrt{3}} \right) ^{\frac{(d-1)^2}{2}} \max \{ 1, \alpha_1(\Delta)^{-(d-1)} \},
\end{equation}
such that $\Delta = g \ZZ_S^d$.
\end{Lem}

\begin{proof}
Thanks to Proposition \ref{Siegel_set_GL(d)}, $\Delta = g \ZZ_S^d$ for some $g \in \sieS{\frac{2}{\sqrt{3}}}{\frac{1}{2}}$. Further, we choose $g$ such that $\hgt_S(ge_1) = \alpha_1(\Delta).$ Note that $|\det g_p|_p  = \normp{g_p e_1} = 1$ for any $p \in S_f$ since $g_p$ is in $GL(d,\ZZ_p)$. Thus $g_\infty$ is in $SL^{\pm}(d,\RR)$ because $\hgt_S(\det g) =  cov(\Delta) = 1$. Since
\[ \alpha_1(\Delta) = \hgt_S(g e_1) = \normeuc{g_\infty e_1}, \]
we obtain \eqref{Choco} by applying Lemma \ref{CS_aux_1} to $g_\infty$.  
\end{proof}

To close the appendix we state a noneffective version of Lemma \ref{CS_aux_2}.

\begin{Cor}[$S$-adic Mahler's Criterion]\label{Mahlers_crit}
A subset $\Omega$ of $\LatSpaceUnoS{d}$ is relatively compact if and only if 
\[\inf \{ \alpha_1(\Delta) \mid \Delta \in \Omega \} > 0.  \]
\end{Cor}

\section{Effective $S$-adic recurrence of unipotent flows}\label{app_recurrence_unipotent_flows}
The goal of this appendix is to establish Proposition \ref{Effective_recurrence_unipotent_flows_p-adic}, an  effective recurrence of unipotent flows on the space of covolume 1 lattices of $\QQ_S^d$. We used it to prove Proposition \ref{Compact_meeting_closed_H_S-orbits}. Proposition \ref{Effective_recurrence_unipotent_flows_p-adic} will follow from \cite[Theorem 9.3]{kleinbock_flows_2007}, a more general result of D. Kleinbock and G. Tomanov, which we restate here. Before doing so, we introduce three concepts needed for the statement: Besicovitch spaces, doubling measures, and $(C,\vartheta)$-good functions.

Let $Z$ be a metric space. We denote by $B_Z(z,r)$ the open ball with center $z \in Z$ and radius $r$. We say that $Z$ is a \textit{Besicovitch space} if there exist a positive integer $N_Z$ with the following property: For any bounded subset $A$ of $Z$ and any function $r: A \to (0, \infty)$, there is a finite or countable subset $\mathscr{B}$ of 
\begin{equation}\label{Besicovitch_property}
\mathscr{B}_r := \{B_Z(a, r(a)) \mid a \in A \} 
\end{equation}
 that still covers $A$, and such that any point of $Z$ belongs to at most $N_Z$ elements of $\mathscr{B}$.
 
For example, $\QQ_p$---more generally any ultrametric space---is a Besicovitch space with $N_{\QQ_p} = 1$. Indeed, for any pair of open balls in $\QQ_p$, either they are disjoint or one is contained in the other. Consider a bounded subset $A$ of $\QQ_p$ and a positive function $r$ on $A$. If $r$ is unbounded, let $a_0 \in A$ with $r(a_0) > diam(A)$. We can choose $\mathscr{B} = \{ B_Z(a_0, r(a_0))\}$. If $r$ is bounded, any point of $A$ is in a unique maximal---with respect to the inclusion---element of $\mathscr{B}_r$. We take $\mathscr{B}$ as the subset of maximal elements of $\mathscr{B}_r$. Notice that $\mathscr{B}$ is at most countable because any two distinct elements of it are disjoint and $\QQ_p$ is second-countable. A second example of Besicovitch space is $\RR^d$ with its standard metric, according to Besicovitch's Covering Theorem. For a proof see \cite[p. 30]{mattila_geometry_1995}. It's easy to see that if three intervals of $\RR$ meet, one of them is contained in the union of the other two, so $N_\RR = 2$.  

Next, we introduce a measure-theoretic analog of Besicovitch spaces. We say that a Borel measure $\lambda$ on a metric space $Z$ is \textit{doubling} if for any $c>1$
\begin{equation}\label{doubling_constant} D_\lambda(c) = \sup \left \{ \frac{\lambda(B_Z(z, cr))}{\lambda( B_Z(z, r))} \mid z \in supp \, \lambda , r > 0 \right \} 
\end{equation}
is finite. 

The Haar measure of $\RR$ is doubling since $D_{\lambda_\RR}(c) = c$. Let's see that the Haar measure $\lambda_{\QQ_p}$ of $\QQ_p$ is also doubling. We normalize $\Haar{\QQ_p}$ so that $\Haar{\QQ_p}(\ZZ_p) = 1$. Then the measure of a closed ball of radius $p^n$ is $p^{n}$, for any $n \in \ZZ$. It follows that
\[ rp\inv \leq \lambda_{\QQ_p}( B_{\QQ_p}(z,r)) < r \]
for any $z \in \QQ_p$ and any $r>0$, so
\begin{equation}\label{D_mu}
\frac{\lambda_{\QQ_p}(B_{\QQ_p}(z,cr))}{\lambda_{\QQ_p}(B_{\QQ_p}(z,r))} < cp 
\end{equation}
for any $c > 1$. This shows that $D_{\lambda_{\QQ_p}}(c) \leq cp$. 

Lastly, we are interested in a class of functions that can't take small values for a long time. Let's formalize this intuition. Let $Z$ be a metric space and let $K$ be a field endowed with an absolute value $|\cdot|$. Consider a non-empty subset $B$ of $Z$ and a measurable function $F: Z \to K$. We define
\begin{equation}\label{BFepsilon}
B(F, \varepsilon) = \{ b \in B \mid |F(b)| < \varepsilon \} 
\end{equation}
for any $\varepsilon > 0$. If $\lambda$ is a Borel measure on $Z$ and $B$ meets $supp \, \lambda$, we define
\[\norm{F}_{B,\lambda} = \sup \{ |F(b)| \mid b \in B \cap supp \, \lambda \}. \]
Let $C, \vartheta > 0$. We say that $F$ is $(C,\vartheta)$-good with respect to $\lambda$ if and only if for any open ball $B$ of $Z$ centered at a point in $supp \, \lambda$ we have
\[ \lambda(B(F, \varepsilon)) \leq C \left( \frac{\varepsilon}{\norm{F}_{B, \lambda}} \right)^\vartheta \lambda(B). \]
When $Z$ is a completion $\QQ_\nu$ of $\QQ$, we'll simply call $(C,\vartheta)$-good a $(C,\vartheta)$-good function with respect to the Haar measure $\lambda_{\QQ_\nu}$, and we'll write $\norm{\cdot}_{B}$ instead of $\norm{\cdot}_{B,\lambda_{\QQ_\nu}}$.

The main example of $(C,\vartheta)$-good functions are $\QQ_\nu$-polynomial maps in one variable. The next lemma of Kleinbock and Margulis gives the constants $(C,\vartheta)$ for real polynomials---see \cite[Proposition 3.2]{kleinbock_flows_1998}. 

\begin{Lem}\label{real_polynomials_are_good}
Consider a non-zero polynomial $q(t) \in \RR[t]$ of degree $d$. If $d \leq d_0$, then $q(t)$ defines a $\left( d_0(d_0+1)^\frac{1}{d_0}, 1/d_0 \right)$-good function on $\RR$. 
\end{Lem}

For $p$-adic polynomials, we will prove at the end of the appendix the next result.

\begin{Lem}\label{p-adic_polynomials_are_good}
Consider a non-zero polynomial $q(t) \in \QQ_p[t]$ of degree $d$. If $d \leq d_0$, then $q(t)$ defines a $(d_0^2p, 1/d_0)$-good function on $\QQ_p$. 
\end{Lem}

We'll need later two simple property of $(C, \vartheta)$-good functions.

\begin{Lem}\label{Max_of_good_is_good}
Let $\nu$ be a place of $\QQ$ and consider two measurable functions $F, F_1: \QQ_\nu \to \QQ_\nu$.
\begin{enumerate}[$(i)$]
\item If $F$ and $F_1$ are $(C, \vartheta)$-good, then $\max \{|F|_\nu, |F_1|_\nu\}$ is $(C, \vartheta)$-good. 
\item If $F^2$ is $(C, \vartheta)$-good, then $F$ is $(C, 2\vartheta)$-good.  
\end{enumerate}
\end{Lem}

\begin{proof}
We start with $(i)$. Set $F_m = \max \{|F|_\nu, |F_1|_\nu \}$ and let $B$ be a ball in $\QQ_\nu$. It's easy to see that
\[B(F_m, \varepsilon) = B(F, \varepsilon) \cap B(F_1, \varepsilon), \]
and $\norm{F_m}_B = \max\{\norm{F}_B, \norm{F_1}_B \}$. Suppose that $\norm{F_m}_B = \norm{F}_B$. 
Since $B(F_m, \varepsilon)$ is contained in $B(F, \varepsilon)$ and $F$ is $(C, \vartheta)$-good, then 
\[\lambda_{\QQ_\nu} (B(F_m, \varepsilon)) \leq \lambda_{\QQ_\nu}(B(F, \varepsilon)) \leq C \left(\frac{\varepsilon}{\norm{F}_B} \right)^\vartheta \lambda_{\QQ_\nu}(B). \]
Thus $F$ is $(C, \vartheta)$-good.

We pass to $(ii)$. Notice that $||F^2||_B = ||F||_B^2$ and $B(F^2, \varepsilon^2)  = B(F, \varepsilon)$ for any $\varepsilon > 0$. Since $F^2$ is $(C, \vartheta)$-good, then
\begin{equation}
\lambda_{\QQ_\nu}(B(F,\varepsilon)) = \lambda_{\QQ_\nu}(B(F^2, \varepsilon^2))  \leq C \left( \frac{\varepsilon^2}{||F||^2_B} \right)^\vartheta \lambda_{\QQ_\nu}(B),
\end{equation} 
so $F$ is $(C, 2 \vartheta)$-good. 
\end{proof}

Let $\Delta'$ be a discrete $\ZZ_S$-submodule of $\QQ_S^d$. Let's give a formula to compute $\text{cov }\Delta'$ that will be useful later. Let $e_1,\ldots,e_d$ be the standard basis of $\QQ^d$ and let $I = (i_1, \ldots, i_k)$ be a $k$-tuple of integers $1 \leq i_1< \cdots < i_k \leq d$. We denote $e_{i_1}\wedge \cdots \wedge e_{i_k}$ simply by $e_I$. On $\bigwedge^k \RR^d$ we consider the only euclidean norm $\normeuc{\cdot}$ such that $(e_I)_I$ is an orthonormal basis, and on $\bigwedge^k \QQ_p^d$ we consider the ultrametric norm
\[\left|\left| \sum_I a_I e_I \right| \right|_p = \max_I |a_I|_p. \]
Let $v_1,\ldots, v_k \in \QQ_S^d$ be a $\ZZ_S$-basis of $\Delta'$. Then
\begin{equation}\label{covolume_formula}
\text{cov } \Delta' = \normeuc{(v_1 \wedge \cdots \wedge v_k)_\infty} \prod_{p \in S_f} \normp{(v_1 \wedge \cdots \wedge v_k)_p}. 
\end{equation}  
For any lattice $\Delta$ of $\QQ_S^d$, we denote $\Sigma(\Delta)$ the set of nonzero $\ZZ_S$-submodules of $\Delta$. Let $\Sigma_{<1}(\Delta)$ be the set of nonzero $\ZZ_S$-submodules of $\Delta$ of covolume $<1$.

The following is a restatement of \cite[Theorem 9.3]{kleinbock_flows_2007} with slightly different notation.  

\begin{Theo}\label{Effective_recurrence_good_flows}
Consider a Besicovitch metric space $Z$, a doubling measure $\lambda$ on $Z$ and a finite set  $S = \{\infty\} \cup S_f$ of places of $\QQ$. Let $B = B_Z(z_0, r)$, $\widetilde{B} = B_Z(z_0, 3^d r)$, and let $F$ be a continuous function $\widetilde{B} \to GL(d,\QQ_S)$. Suppose that $C, \vartheta > 0$ and $\rho \in (0,1)$ verify the following: for every $\Delta' \in \Sigma(\ZZ_S^d)$
	\begin{enumerate}[$(i)$]
	 \item The map $\psi_{\Delta'}: z \mapsto cov(F(z)\Delta')$ is $(C, \vartheta)$-good with respect to $\lambda$ on $\widetilde{B}$;
	 \item $\norm{\psi_{\Delta'}}_{B, \lambda} \geq \rho$. 
	\end{enumerate}
Then, for any $0< \varepsilon \leq \rho$ one has
\[\lambda(\{z \in B \mid \alpha_1(F(z) \ZZ_S^d) < \varepsilon \}) \leq d C (N_Z D_\lambda(3)^2)^d 
\left(\frac{\varepsilon}{\rho} \right)^\vartheta \lambda(B), \]
with $N_Z$ and $D_\lambda(3)$ as in \eqref{Besicovitch_property} and \eqref{doubling_constant}, respectively.	  
\end{Theo}

We can now establish the statement of effective recurrence of unipotent flows that we used in Section \ref{sec_vol_closed_orbits}.

\begin{proof}[Proof of Proposition \ref{Effective_recurrence_unipotent_flows_p-adic}]
We write $\Delta$ as $g \ZZ_S^d$ for some $g \in GL(d,\QQ_S)$ and we consider the map $F: \QQ_\nu \to GL(d,\QQ_S), t \mapsto u_t g$. For any $\Delta' \in \Sigma(\ZZ_S^d)$ we define $\psi_{\Delta'}(t) = cov\, (F(t) \Delta')$. The result we seek will be established by applying Theorem \ref{Effective_recurrence_good_flows} to $Z  = \QQ_\nu, \lambda = \Haar{\QQ_\nu}$ and $F$. In order to do so, we need show that $\psi_{\Delta'}$ verifies conditions $(i)$ and $(ii)$ of Theorem \ref{Effective_recurrence_good_flows}.

Let's start with $(i)$. We define
\[\Ctt_{d,\nu} = \begin{cases}
		(d-1)^2 p & \text{if } \nu = p, \\
		4(d-1)^2 & \text{if } \nu = \infty.
\end{cases}\]
We claim that $\psi_{\Delta'}$ is $(\Ctt_{d,\nu}, \consAlfaRecurrence{d})$-good. Since $u_t = \exp(tv)$ for some nilpotent, $d \times d$ matrix $v$, then 
\[F(t)_\nu = (q_{ij}(t))_{1\leq i,j \leq d} \]
for some $q_{ij}(t) \in \QQ_\nu[t]$ of degree at most $d-1$. Let $v_1, \ldots, v_k$ be a basis of $\Delta'$. By \eqref{covolume_formula} we have 
\[\psi_{\Delta'} (t) = cov(g \Delta', S-\{\nu\}) \, \normnu{(F(t)v_1)_\nu \wedge \cdots \wedge (F(t)v_k)_\nu}, \]
where $cov(g \Delta', S-\{ \nu \}) =  \prod_{\nu \in S-\{ \nu \}} \normnu{(gv_1)_\nu \wedge \cdots (g v_k)_\nu}$. Writing the $(F(t)v_i)_\nu$ in terms of the canonical basis $e_1, \ldots, e_d$ of $\QQ_\nu^d$ and expanding the wedge product we see that
\[ (F(t)v_1)_\nu \wedge \cdots \wedge (F(t)v_k)_\nu = \sum_{J} Q_J(t) e_J \]
for some $Q_J(t) \in \QQ_\nu[t]$ of degree at most $(d-1)^2$. Here $J$ runs over all the $k$-tuples of integers $(j_1, \ldots, j_k)$ with $1 \leq j_1 < \ldots < j_k \leq d$. Hence, 
\[\frac{\psi_{\Delta'} (t)}{cov(g \Delta', S-\{ \nu \})} = \begin{cases}
\max_J |Q_J(t)|_p & \text{if } \nu = p, \\
\left(\sum_J Q_J(t)^2 \right)^\frac{1}{2} & \text{if } \nu = \infty.
\end{cases}\]
Lemmas \ref{real_polynomials_are_good}, \ref{p-adic_polynomials_are_good} and \ref{Max_of_good_is_good} imply that $\psi_{\Delta'}$ is $(\Ctt_{d,\nu}, \consAlfaRecurrence{d})$-good \footnote{In the case $\nu = \infty$, we use Lemma \ref{real_polynomials_are_good} with $d_0 = 2(d-1)^2$. Note that $d_0(d_0 + 1)^2 < 2 d_0 = 4(d-1)^2$.}.

Consider any $\rho \in (0,1)$ and
\[ B_\nu(T) = \{ t \in \QQ_\nu \mid |t|_\nu < T \}. \]
Let's show that condition $(ii)$ of Theorem \ref{Effective_recurrence_good_flows} holds for $\lambda = \Haar{\QQ_\nu}$ and $B = B_\nu(T)$ whenever $T$ is big enough. If $F(0) = cov(g \Delta') \geq 1$, we are done. Otherwise, $g \Delta'$ belongs to $\Sigma_{<1}(\Delta)$, which is finite by Lemma \ref{Sigma<1_finite}. We write 
\[\Sigma_{<1}(\Delta) = \{ g\Delta'_1, \ldots, g \Delta'_\ell \}. \]
As we showed earlier, there are $Q_{i,J}(t) \in \QQ_\nu[t]$ such that
\[\frac{ \psi_{\Delta'_i}(t)}{cov\, (g \Delta_i', S-\{\nu\})} = \begin{cases}
\max_J |Q_{i,J}(t)|_p & \text{if } \nu = p, \\
\left(\sum_J Q_{i,J}(t)^2 \right)^\frac{1}{2} & \text{if } \nu = \infty.
\end{cases}\]
Since $U_\nu = (u_t)_{t \in \QQ_\nu}$ doesn't preserve the $\QQ_S$-module generated by $g \Delta'_i$, some $Q_{i,J}(t)$ is nonconstant. Thus, there is $t_i \in \QQ_\nu$ such that $\psi_{\Delta'_i}(t_i) > 1$. Let 
\[T_0 = T_0(U_\nu, \Delta) = \max_{i} |t_i|_\nu. \]
Then condition $(ii)$ is satisfied for any $T \geq T_0$. 

We have now the right to invoke Theorem \ref{Effective_recurrence_good_flows}. Recall that 
\[D_{\lambda}(3) \begin{cases} 
= 3 & \text{if } \nu = \infty,\\
\leq 3p & \text{if } \nu = p,
\end{cases}\]
and that the Besicovitch constants of $\QQ_p$ and $\RR$ are $1$ and $2$, respectively. By Theorem \ref{Effective_recurrence_good_flows}, for any $\varepsilon \in (0, \rho]$ we have
\[\frac{\lambda(\{t \in B_{\nu}(T) \mid \alpha_1(u_t \Delta) < \varepsilon \})}{\lambda(B_\nu(T))} \leq 
\begin{cases}
d((d-1)^2p)(3p)^{2d} \left( \varepsilon / \rho \right)^{\consAlfaRecurrence{d}}  & \text{if } \nu = p, \\
d(4 (d-1)^2)(2\cdot 3^2)^d \left( \varepsilon / \rho \right)^{\consAlfaRecurrence{d}} & \text{if } \nu = \infty.
\end{cases}
 \] 
Letting $\rho$ tend to 1 we obtain the result we sought. 
\end{proof}

To close this appendix we prove Lemma \ref{p-adic_polynomials_are_good} with the aid of the next three auxiliary results. We fix a nonzero $q(t) \in \QQ_p[t]$ of degree $d$. Let $m$ be a positive integer. Let $I_m^\varepsilon$ be the set of integers $0 \leq a \leq p^m - 1$ such that $(p^m \ZZ_p + a) \cap \ZZ_p(q, \varepsilon)$ is nonempty.

\begin{Lem}\label{L1}
Let $m \geq 0$. If $\# I_m^\varepsilon \geq d+1$, then
\[\norm{q}_{\ZZ_p} \leq \varepsilon p^{d(m-1)}. \]
\end{Lem}

\begin{proof}
Consider pairwise different elements $a_0, \ldots, a_d$ in $I_m^\varepsilon$ and $t_i \in a_i + p^m \ZZ_p$ with $|q(t_i)|_p < \varepsilon$. Notice that
\[ |a_i - a_j|_p \geq p^{-(m-1)}. \]
Using Lagrange's Interpolation Formula we write $q(t)$ as 
\[q(t) = \sum_{i = 0}^d q(t_i) \prod_{j \neq i} \frac{t-t_j}{t_i - t_j}.\]
For any $z \in \ZZ_p$ and any $i$ we have 
\[\left| q(t_i) \prod_{j \neq i} \frac{z-t_j}{t_i - t_j} \right|_p \leq \varepsilon p^{d(m-1)}, \]
hence $|q(z)|_p \leq \varepsilon p^{d(m-1)}$.  
\end{proof}

\begin{Lem}\label{L2}
For any $m \geq 1$ and any $\varepsilon > 0$ we have
\[ \lambda_{\QQ_p}(\ZZ_p(q,\varepsilon)) \leq p^{-m} \# I_m^\varepsilon.\]
\end{Lem}

\begin{proof}
The measure of $\ZZ_p(q, \varepsilon)$---defined in \eqref{BFepsilon}---is less or equal than the measure of 
\[\bigcup_{a \in I_m^\varepsilon} a + p^m \ZZ_p \]
because the first set is contained in the second. 
\end{proof}

\begin{Lem}\label{L3}
Suppose that $q(t) \in \QQ_p[t]$ is non-zero and has degree $\leq d_0$. Then
\[ \lambda_{\QQ_p}(\ZZ_p(q, \varepsilon)) \leq d_0^2 p \left( \frac{\varepsilon}{\norm{q}_{\ZZ_p}} \right)^{\frac{1}{d_0}}. \]
\end{Lem}

\begin{proof}
We choose $m_0 \geq 1$ such that
\[ p^{m_0-1} < d_0 + 1 \leq p^{m_0}. \]
Then $p^{m_0-1} \leq d_0$ and $p^{m_0} \leq d_0 p$. 

If $\lambda_{\QQ_p}(\ZZ_p(q, \varepsilon)) = 0$, the inequality we want is true. Suppose now that $\lambda_{\QQ_p}(\ZZ_p(q, \varepsilon))$ is positive and choose $m \geq 1$ such that
\[p^{-m} < \frac{\lambda_{\QQ_p}(\ZZ_p(q, \varepsilon))}{d_0} \leq p^{-(m-1)}. \]
By Lemma \ref{L2} we have
\[ \lambda_{\QQ_p}(\ZZ_p(q,\varepsilon)) \leq p^{-(m + m_0)} \# I_{m + m_0}^\varepsilon, \]
and we also know that $d_0 p^{-(m + m_0)} < \lambda_{\QQ_p}(\ZZ_p(q,\varepsilon))$, thus
\[\# I_{m + m_0}^\varepsilon \geq d_0 + 1, \]
 so we can use Lemma \ref{L1}:
\[\norm{q}_{\ZZ_p} \leq \varepsilon p^{d_0 m_0} p^{d_0(m - 1)} 
\leq \varepsilon (d_0 p)^{d_0} \left( \frac{d_0}{\lambda_{\QQ_p}(\ZZ_p(q, \varepsilon))} \right)^{d_0}.\]
This is equivalent to the inequality of the statement.   
\end{proof}

We are ready to prove that polynomial maps on $\QQ_p$ are $(C, \vartheta)$-good. 

\begin{proof}[Proof of Lemma \ref{p-adic_polynomials_are_good}]
Let $B$ be a ball in $\QQ_p$. We write it as $z + p^n \ZZ_p$ for some $z \in \QQ_p$ and $n \in \ZZ$. The degree of $Q(t) = q(z + p^n t)$ is also $d$, and $\norm{Q}_{\ZZ_p} = \norm{q}_B$. By Lemma \ref{L3} we have
\begin{equation}\label{zap}
\lambda_{\QQ_p}( \ZZ_p(Q,\varepsilon)) \leq d_0^2 p \left( \frac{\varepsilon}{||q||_B} \right)^{\frac{1}{d_0}}.
\end{equation} 
From the equality
\[ B(q, \varepsilon) = z + p^n \ZZ_p(Q, \varepsilon) \]
we deduce that
\[\lambda_{\QQ_p}(\ZZ_p(Q, \varepsilon)) = p^n \lambda_{\QQ_p}(B(q,\varepsilon)) = \frac{\lambda_{\QQ_p}(B(q,\varepsilon))}{\lambda_{\QQ_p}(B)}, \]
which combined with \eqref{zap} yields
\[\lambda_{\QQ_p}(B(q,\varepsilon)) \leq d_0^2 p \left(\frac{\varepsilon}{\norm{q}_B} \right) ^{\frac{1}{d_0}} \lambda_{\QQ_p}(B). \] 
\end{proof}

\section{Effective Reduction Theory}\label{app_reduction_theory}

In this appendix we prove quantitative refinements of some results from the reduction theory of quadratic forms. The main result is Proposition \ref{Integral_reduced_qf_are_small}, which bounds the norm of a reduced integral quadratic form in terms of its determinant. It was used in the proofs of lemmas \ref{TQS_R-isotropic} and \ref{TQS_R-anisotropic} in Section \ref{sec_small_gens}. The appendix has two subsections. In Subsection  \ref{subsec_int_equiv} we bound the norm of any integral equivalence matrix between positive definite, reduced quadratic forms. This will allow us to establish Proposition \ref{Integral_reduced_qf_are_small} in Subsection \ref{subsec_riqf}. The proofs follow closely the exposition of Cassels in \cite[Chapter 12]{cassels_rational_1978} of the reduction theory, as well as sections 9.2, 9.3 and 9.4 of \cite{li_effective_2016}. 

Let's start by recalling the basic notation and definitions. We denote the group $GL(d,\RR)$ by $G_{d,\infty}$. In \eqref{def_Siegel_set_R} we defined the Siegel set $\sieR{\alpha}{\beta}$, which is a fundamental set of $\Gamma_{d,\infty} = GL(d,\ZZ)$ in $G_{d,\infty}$ for any $\alpha \geq \frac{2}{\sqrt{3}}$ and $\beta \geq \frac{1}{2}$---see Proposition \ref{Siegel_set_GL(d)}. Let $Q_{p,q}$ be the quadratic form $x_1^2 + \cdots + x_q^2 - x_{q+1}^2 - \cdots - x_{p+q}^2$ and set $d = p + q$. We'll say that a quadratic form $R$ on $\RR^{d}$ is \textit{$(\alpha, \beta)$-reduced} if $R = Q_{p,q}\circ s$ for some $s \in \sieR{\alpha}{\beta}$, where $p,q$ is the signature of $R$. 

	\subsection{Integral equivalence between reduced quadratic forms}\label{subsec_int_equiv}

The purpose of this subsection is to prove an upper bound of the norm of an integral equivalence matrix between two positive definite, reduced quadratic forms.
	
\begin{Prop}\label{Change_matrix_reduced_pdqf}
For any $i \in \{1,2\}$, let $R_i$ be an $(\alpha_i, \beta_i)$-reduced, positive definite quadratic form on $\RR^d$, where $\alpha_i, \beta_i \geq 1$. If $b$ is an integral $d \times d$ matrix such that $R_1 \circ b = R_2$, then 
\[ \normi{b} \leq \consChangeReducedPDQF{d} \alpha_1^{d-1} \alpha_2^{(d-1)^2} \beta_1^d \beta_2^{d(d-1)}  |\det b|_\infty^{2d}, \]
where $\consChangeReducedPDQF{d} = d^\frac{3d}{2} (d+1)^{d^2} \cdot d!^{d+1}$. 
\end{Prop}

We will now state two lemmas and use them to establish Proposition \ref{Change_matrix_reduced_pdqf}. We prove the two lemmas afterwards. 

Consider a positive definite quadratic form $R$ on $\RR^d$ and a lattice $\Delta$ of $\RR^d$. For any $r > 0$ we define $E^-_r(\Delta, R)$ and $E^\circ_r(\Delta,r)$ as the respective linear spans of the $v \in \Delta$ with $R(v) \leq r$ and $R(v) < r$.  A vector $v \in \Delta$ is said to be \textit{$R$-extremal} if $v$ does not belong to $E\ncc_{R(v)}(\Delta, R)$. Here is the first lemma.

\begin{Lem}\label{Delta_extremal_vectors_are_small}
Let $R$ be a positive definite, $(\alpha, \beta)$-reduced quadratic form on $\RR^d$, where $\alpha, \beta \geq 1$, and let $\Delta \subset \ZZ^d$ be a lattice of $\RR^d$. Any $R$-extremal vector $w$ of $\Delta$ verifies
 \[  \normi{w} \leq \consExtremalVectorsBound{d} \alpha^{d-1} \beta^d [\ZZ^d : \Delta]^{2d}, \]
 where $\consExtremalVectorsBound{d} = d^\frac{3}{2} (d+1)^d \cdot d! $.
\end{Lem}	

 The dimensions of $E^-_r(\Delta, R)$ and $E\ncc_r(\Delta, R)$ will be respectively denoted by $d^-_r(\Delta, R)$ and $d\ncc_r (\Delta, R)$. For any $1 \leq i \leq d$, the $i$-th $R$-minima of $\Delta$ is defined as 
\[\mathscr{M}_i(\Delta,R) = \inf \{r>0 \mid d^-_r(\Delta, R) \geq i\}. \]
We say that the vectors $v_1, \cdots, v_d \in \Delta$ realize the $R$-minima of $\Delta$ if
\[R(v_i) = \mathscr{M}_i(\Delta, R),\]
for every $1\leq i \leq d$. Here is the second lemma we'll use to prove Proposition \ref{Change_matrix_reduced_pdqf}.

\begin{Lem}\label{Basis_realizing_minima}
Let $v_1,\cdots, v_d$ be linearly independent vectors in $\Delta$ realizing the $R$-minima of $\Delta$. Then each $v_i$ is an $R$-extremal vector of $\Delta$. 
\end{Lem}

\begin{Rem}
There are always linearly independent $v_1, \ldots, v_d \in \Delta$ realizing the $R$-minima of $\Delta$: we choose an $R$-shortest nonzero $v_1 \in \Delta$. If we already have $v_1, \ldots, v_j$, we choose an $R$-shortest $v_{j+1}$ in $\Delta - (\RR v_1 \oplus \cdots \oplus \RR v_j)$. It's possible to do this since any subset of $\Delta$ is closed.  
\end{Rem}

\begin{proof}[Proof of Proposition \ref{Change_matrix_reduced_pdqf}]
Consider linearly independent vectors $v_1,\cdots,v_d \in \ZZ^d$ realizing the $R_2$-minima of $\ZZ^d$, and let $\tau_2 = (v_1,\cdots,v_d) \in M_d(\ZZ)$. By lemmas \ref{Delta_extremal_vectors_are_small} and \ref{Basis_realizing_minima} we know that
\[ \normi{\tau_2} \leq \consExtremalVectorsBound{d} \alpha_2^{d-1} \beta_2^d.  \]
Let $\Delta$ be the lattice $b \ZZ^d$ of $\RR^d$ and set $w_i = bv_i$. Since $R_1 \circ b = R_2$, the linearly independent vectors $w_1,\cdots, w_d$ realize the $R_1$-minima of $\Delta$. Let $\tau_1$ be the $d \times d$ integral matrix $(w_1,\cdots,w_d)$. Using lemmas \ref{Delta_extremal_vectors_are_small} and \ref{Basis_realizing_minima} once more we get
\[ \normi{\tau_1} \leq \consExtremalVectorsBound{d} \alpha_1^{d-1} \beta_1^d |\det b|_\infty^{2d}.\]
Note that $b\tau_2 = \tau_1$, so 
\begin{align*}
\normi{b} = \normi{\tau_1 \tau_2\inv}  & \leq d \normi{\tau_1}  \normi{\tau_2\inv} \\
& \leq d!(W_{1,d} \alpha_1^{d-1} \beta_1^d |\det b|_\infty^{2d}) (W_{1,d} \alpha_2^{d-1} \beta_2^d)^{d-1} \\
& = (d^\frac{3d}{2} (d+1)^{d^2} \cdot d!^{d+1} ) \alpha_1^{d-1} \alpha_2^{(d-1)^2} \beta_1^{d} \beta_2^{d(d-1)} |\det b|_\infty ^{2d}. 
\end{align*}
 
\end{proof}

It will be convenient to reformulate Proposition \ref{Change_matrix_reduced_pdqf} in terms of right translates of Siegel sets by integral matrices. 

\begin{Cor}\label{Siegel_sets_almost_never_meet_its_translates_ap}
Let $b$ be a $d \times d$ integral matrix. If $\sieR{2}{1} b$ meets $\sieR{2}{1}$, then
\[ \normi{b} \leq \consTransSiegel{d} |\det b|_\infty ^{2d}, \]
where $\consTransSiegel{d} = 2^{d(d-1)} d^\frac{3d}{2} (d!)^{d+1} (d+1)^{d^2}$. 
\end{Cor}

\begin{proof}
Take $s_1, s_2 \in \sieR{2}{1}$ such that $s_1 b = s_2$. The positive definite quadratic form $R_i = Q_{d,0} \circ s_i$ is $(2,1)$ reduced and $b$ takes $R_1$ to $R_2$, so Proposition \ref{Change_matrix_reduced_pdqf} implies
\begin{align*}
\normi{b} & \leq W_d 2^{(d-1)^2 + d - 1} |\det b|_\infty^{2d} \\
		& = 2^{d(d-1)} d^{\frac{3d}{2}} (d!)^{d+1} (d+1)^{d^2} |\det b|_\infty^{2d}. 
\end{align*}
\end{proof}
	
We turn now to the proof of Lemma \ref{Delta_extremal_vectors_are_small}, which will use the next two auxiliary results. We say that a $d \times d$ matrix $b$ with real coefficients \emph{has big diagonal} if $b_{ii} \geq |b_{ij}|$ for any $1 \leq i, j \leq d$.

\begin{Lem}\label{Triangulation_integral_matrix}
Let $c$ be a nonsingular $d \times d$ matrix with integral coefficients. There is $\gamma \in \Gamma_{d,\infty}$ such that $c\gamma$ is an upper-triangular matrix with big diagonal.
\end{Lem}

\begin{proof}
Using repeatedly the euclidean algorithm, we transform $c$ into an upper-triangular matrix with big diagonal performing elementary column operations\footnote{These are permuting columns or adding to a column an integral multiple of another.}, which correspond to multiplying $c$ on the right by some $\gamma \in \Gamma_{d,\infty}$. 

\end{proof}

\begin{Lem}\label{Right_translate_Siegel_set}
Consider a nonsingular, upper-triangular matrix $b\in M_d(\ZZ)$ with  big diagonal. Let $\delta = |\det b|_\infty$, consider $\alpha > 0$ and $\beta \geq 1$. Then $\sieR{\alpha}{\beta} b$ is contained in $\sieR{\alpha \delta}{\beta(\delta + d)}$. 	
\end{Lem}

\begin{proof}
Take any $a \in A_\alpha$ and $n \in N_\beta$. It suffices to prove that $anb = a'n'$ for some $a' \in A_{\alpha \delta}$ and $n' \in N_{\beta(\delta + d)}$. We set $c = diag (b_{11},\cdots, b_{dd})$. Then 
\[a' = ac = diag(a_{11}b_{11}, \cdots, a_{dd}b_{dd}),\] 
and 
\[ \frac{a_{i+1,i+1} b_{i+1, i+1}}{a_{ii} b_{ii}} \leq \alpha b_{i+1,i+1} \leq \alpha \delta, \]
hence $a'$ is in $A_{\alpha \delta}$. Since $n' = c\inv n b$, for any $i < j$ we have
\begin{align*}
|n'_{ij}|_\infty = \frac{1}{b_{ii}} \left| \sum_{k=1}^j n_{ik}b_{kj} \right|_\infty & \leq \beta \sum_{k=1}^j |b_{kj}|_\infty \\
			& \leq \beta \sum_{k=1}^j b_{kk} < \beta (\delta + d). 
\end{align*}
This shows that $n'$ is in $N_{\beta(\delta + d)}$. 
\end{proof}

\begin{proof}[Proof of Lemma \ref{Delta_extremal_vectors_are_small}]
Let's first suppose that $\Delta = \mathbb{Z}^d$. Consider $a = diag(a_1,\cdots,a_d) \in A_\alpha$ and $n = (n_{ij}) \in N_\beta$ such that $R = Q_{d,0} \circ (an)$. We set 
\[ w = nv = (w_1,\cdots, w_d). \]
First we bound $|w_k|_\infty$ for $1 \leq k \leq d$. Consider two cases:
	\begin{itemize}
	\item \textbf{Case I: there is $j \leq k$ such that $R(v) \leq R(e_j)$.} Then
	\[a_k^2 w_k^2 \leq R(v) \leq R(e_j),\]
	which implies that
	\begin{align*}
	w_k^2 &\leq \frac{a_1^2}{a_k^2} n_{1j}^2 + \cdots + \frac{a_{j-1}^2}{a_k^2}n_{j-1,j}^2 + \frac{a_j^2}{a_k^2} \\
	& \leq (\alpha^{2(k-1)} +  \cdots  +\alpha^{2(k-j)}) \beta^2 \\
	& \leq d \alpha^{2(d-1)} \beta^2.\\
	\end{align*}
	Thus $|w_k|_\infty \leq \sqrt{d} \, \alpha^{d-1} \beta$. 
	\item \textbf{Case II: $R(e_j) < R(v)$ for every $j \leq k$.} Then, since $v$ is an $R$-extremal vector of $\ZZ^d$, $R(v) \leq R(v')$ for every $v'$ of the form $v + c_1e_1 + \cdots + c_k e_k$ with $c_1, \cdots, c_k \in \ZZ$. Set $w' = nv'$, and choose $c_k, c_{k-1}, \cdots , c_1$ so that $|w'_j| \leq \frac{1}{2}$ for every $j \leq k$. Since $w_i = w'_i$ for $k< i \leq d$, from $R(v) \leq R(v')$ we deduce that
	\begin{align*}
	w_k^2 & \leq \frac{a_1^2}{a_k^2} (w'_1)^2 + \cdots + \frac{a_k^2}{a_k^2} (w'_k)^2 \\
	& \leq \frac{1}{4} (\alpha^{2(k-1)} + \cdots + \alpha^2 + 1) \\
	& \leq \frac{d}{4} \alpha^{2(d-1)} < d \alpha^{2(d-1)} \beta^2,
	\end{align*}
so $|w_k|_\infty < \sqrt{d} \alpha^{d-1} \beta$.
	\end{itemize}
In both cases we have
\[ \normi{w} \leq \sqrt{d} \alpha^{d-1} \beta, \]
so
	\begin{align}
\notag	\normi{v} = \normi{n\inv w} & \leq d \normi{n\inv} \normi{w} \\
\notag			& \leq d((d-1)! \normi{n}^{d-1}) (\sqrt{d} \alpha^{d-1} \beta)\\
\label{Z-extremal_bound}			& \leq \sqrt{d} \cdot d! \alpha^{d-1} \beta^d.
	\end{align}

Suppose now that $\Delta$ is any lattice of $\RR^d$ contained in $\ZZ^d$. By Lemma \ref{Triangulation_integral_matrix}, we can write $\Delta$ as $b\ZZ^d$ for some upper-triangular matrix $b \in M_d(\ZZ)$ with a big diagonal . Then $[\ZZ^d: \Delta] = |\det b|_\infty$, which we denote by $\delta$. Consider an $R$-extremal vector $w = bv$ of $\Delta$. Then $v$ is an $(R \circ b)$-extremal vector of $\ZZ^d$. The positive definite quadratic form $R \circ b$ is $(\alpha \delta, \beta(\delta + d))$-reduced by Lemma \ref{Right_translate_Siegel_set}, so \eqref{Z-extremal_bound} yields 
\begin{align*}
\normi{v} & \leq \sqrt{d} \cdot d!  (\alpha \delta)^{d-1} (\beta(\delta + d))^d \\
& \leq \sqrt{d} (d+1)^d \cdot d!  \alpha^{d-1} \beta^d \delta^{2d-1}.
\end{align*} 
Hence
\begin{align*}
\normi{w} & \leq d \normi{b} \normi{v} \\
			& \leq d^\frac{3}{2} (d+1)^d \cdot d!  \alpha^{d-1} \beta^d \delta^{2d}.
\end{align*}
\end{proof}
	We turn to the the proof of Lemma \ref{Basis_realizing_minima}, which requires another auxiliary result. For $R,\Delta$ fixed and varying $r > 0$, the subspaces $E^-_r(\Delta, R)$ form a (not necessarily complete) flag of $\RR^d$
\[\{0\} = E_0 \subsetneq \cdots \subsetneq E_\ell = \RR^d.\]
Let $\textbf{d}_i$ be the dimension of $E_i$ for any $0 \leq i \leq \ell$, and let $\textbf{r}_i$ be the smallest nonnegative real number such that
\[ E_i = E^-_{\textbf{r}_i} (\Delta, R). \]
To lighten the notation we write $\mathscr{M}_j$ instead of $\mathscr{M}_j(\Delta, R)$ in the next lemma. 

\begin{Lem}\label{AuxMin}
Let $1 \leq j \leq d$ and $k\geq 0$ be integers such that $\textbf{d}_k < j \leq \textbf{d}_{k+1}$. Then $E\ncc_{\mathscr{M}_j} (\Delta, R) = E_k$. 
\end{Lem} 

\begin{proof}
From the definition of $\mathscr{M}_j$ follows that $d\ncc_{\mathscr{M}_j}(\Delta, R) < j$. But $d\ncc_{\mathscr{M}_j}(\Delta, R)$ is one of the $\textbf{d}_i'$s, hence its value cannot exceed $\textbf{d}_k$. This means that $E\ncc_{\mathscr{M}_j}(\Delta, R)$ is contained in $E_k$. Now, $E_k = E^-_{\textbf{r}_k}(\Delta, R)$ has dimension $\textbf{d}_k < j$, hence $\textbf{r}_k < \mathscr{M}_j$. This implies that $E_k$ is contained in $E\ncc_{\mathscr{M}_j} (\Delta, R)$. 
\end{proof} 

\begin{proof}[Proof of Lemma \ref{Basis_realizing_minima}]
Consider any integer $1 \leq j \leq d$ and choose $k\geq 0 $ such that 
\[\textbf{d}_k < j \leq \textbf{d}_{k+1}.\] 
Then $E\ncc_{\mathscr{M}_j} (\Delta, R) = E_k$ by Lemma \ref{AuxMin}. Since $E_k = E^-_{\textbf{r}_k} (\Delta,R)$ has dimension $\textbf{d}_k$, then $\textbf{r}_k \geq \mathscr{M}_{\textbf{d}_k} = R(v_{\textbf{d}_k})$. It follows that $v_1, \cdots, v_{\textbf{d}_k}$ belong to $E_k$. Since $v_1,\ldots, v_d$ are linearly independent, $(v_1, \cdots, v_{\textbf{d}_k})$ is a basis of $E_k$ and thus $v_j$ is not in $E_k = E\ncc_{R(v_j)} (\Delta, R)$. In other words, $v_j$ is an $R$-extremal vector of $\Delta$. 
\end{proof}
	
	\subsection{Reduced integral quadratic forms}\label{subsec_riqf}
Let $Q$ be an $(\alpha, \beta)$-reduced integral quadratic form on $\RR^d$—we are no longer assuming that $Q$ is positive definite—. The goal of this subsection is to bound the norm of $Q$ in terms of $\alpha, \beta$ and $|\det Q|_\infty$. The main result is Proposition \ref{Integral_reduced_qf_are_small}. It  improves slightly \cite[Lemma 12.3, p. 325]{cassels_rational_1978} and \cite[Corollary 3, p. 902]{li_effective_2016}.  From it we recover in Corollary \ref{Finiteness_reduced_qf_in_Gamma_class} the main finiteness result of the reduction theory of integral quadratic forms---see \cite[Lemme 5.7, p. 38]{borel_introduction_1969}. 
\begin{Prop}\label{Integral_reduced_qf_are_small}
Let $Q$ be an integral, $(\alpha, \beta)$-reduced quadratic form on $\RR^d$, for some $\alpha, \beta \geq 1$. Then
\[ \normi{Q} \leq \consReducedIntegralQF{d} \alpha^{d^2} \beta^{2d^2} |\det Q|_\infty^{2d}, \]
where $\consReducedIntegralQF{d} = d^\frac{d}{2} (d+1)^{d^2} (d!)^{2d+1}  $.
\end{Prop}

The proof of Proposition \ref{Integral_reduced_qf_are_small} is based on Proposition \ref{Change_matrix_reduced_pdqf} and the next lemma. We denote by $J = (J_{ij})$ the $d\times d$ matrix with entries $J_{ij} = \delta_{i+j, d+1}$.  

\begin{Lem}\label{Stability_Siegel_sets}
Consider real numbers $\alpha> 0$ and $\beta \geq 1$. If $s$ belongs to the Siegel set $\sieR{\alpha}{\beta}$, then $\tra s \inv J$ is in $\sieR{\alpha}{(d-1)!\beta^{d-1}}$. 
\end{Lem}

\begin{proof}
Write $s = kan$ with $k \in K, a = diag(a_1,\cdots, a_d) \in A_\alpha,$ and $n \in N_\beta$. Then $\tra s \inv J = (kJ)(J a\inv J)(J \tra n \inv J)$. Note that $kJ$ is in $K$, 
\[J a\inv J = diag(a_d\inv,\cdots, a_1\inv) \] 
is in $A_\alpha$, and $J \tra n \inv J$ is in $N_{(d-1)!\beta}$ because it is unipotent, upper triangular and 
\[\normi{J \tra n \inv J} = \normi{n\inv} \leq (d-1)! \normi{n}^{d-1} \leq (d-1)!\beta^{d-1}.\]
\end{proof}

\begin{proof}[Proof of Proposition \ref{Integral_reduced_qf_are_small}]
Consider $s_2 \in \sieR{\alpha}{\beta}$ such that $Q = Q_{p,q} \circ s_2$ and define 
\[ s_1 = I_{p,q} \tra s_2\inv J,  \]
where $I_{p,q}$ is the matrix of $Q_{p,q}$ in the canonical basis of $\RR^d$. Notice that $s_1$ is in $\sieR{\alpha}{(d-1)!\beta^{d-1}}$ by Lemma \ref{Stability_Siegel_sets}. Then, the positive definite quadratic forms $R_1 = Q_{d,0} \circ s_1$ and $R_2 = Q_{d,0}\circ s_2$ are respectively $(\alpha, (d-1)! \beta^{d-1})$ and $(\alpha, \beta)$-reduced. One easily checks that $s_2 = s_1 J b_Q$, hence $R_1 \circ (Jb_Q) = R_2$. Proposition \ref{Change_matrix_reduced_pdqf} gives 
\begin{align*}
\normi{Q} = \normi{Jb_Q} & \leq W_d \alpha^{d-1} \alpha^{(d-1)^2} ((d-1)! \beta^{(d-1)})^d \beta^{d(d-1)} |\det J b_Q|_\infty^{2d} \\
& \leq d^\frac{3d}{2} (d!)^{d+1} ((d-1)!)^d (d+1)^{d^2} \alpha ^{d^2} \beta^{2d^2} |\det Q|_\infty^{2d}.
\end{align*}
\end{proof}

From Proposition \ref{Integral_reduced_qf_are_small} we easily recover the finiteness of $\ZZ$-equivalence classes of integral quadratic forms with given determinant.

\begin{Cor}\label{Finiteness_reduced_qf_in_Gamma_class}
Let $m$ be a nonzero integer. There are finitely many $\ZZ$-equivalence classes of integral quadratic forms $Q$ in $d$ variables with $\det Q = m$.  
\end{Cor}

\begin{proof}
Any such class has a $\left( \frac{2}{\sqrt{3}}, \frac{1}{2} \right)$-reduced representative $Q$ by Proposition \ref{Siegel_set_GL(d)}, and there are finitely many $\left( \frac{2}{\sqrt{3}}, \frac{1}{2} \right)$-reduced integral quadratic forms on $\RR^d$ with determinant $m$  by Proposition \ref{Integral_reduced_qf_are_small}. 
\end{proof}

\section{Explicit constants}\label{app_constants}

	\subsection{Section \ref{sec_dynamics_Z_S-equiv}}
	
\begin{center}
\begin{tabular}{|c|c|}
\hline
$\consDynStRiso{d} = 12 \cdot 2^{3d^2(d-1)} d^2 \consMixingReal^6 \consSmoothBump{d}^{12}$ & Proposition \ref{Dynamical_statement_RR-isotropic} \\
\hline
$\consDynStRani{d} = 10^4 \cdot 2^{2d^3} (9d^3 \cdot d!)^{2d(d-1)} $ & Proposition \ref{Dynamical_statement_R-anisotropic} \\
\hline
$\consDecayHarishReal$ non explicit, see Corollary \ref{Xi_infty_exponential_bound} & Proposition \ref{Decay_smooth_vectors_almost_L2m_explicit} \\
\hline 
$\consMixingReal = 5 \consDecayHarishReal^{\frac{1}{2}}$ & Proposition \ref{Mixing_speed_R-isotropic} \\
\hline
$\consSmoothBump{d} = 3 (3d^2 \cdot d!)^{\frac{1}{4}d(d-1) + 1} \consBumpFuncRealOG{d} $ &  Lemma \ref{S-adic_smooth_bump_functions} \\
\hline
   
\end{tabular}
\end{center}	
	
	\subsection{Section \ref{sec_vol_closed_orbits}}
		
\begin{center}
\begin{tabular}{|c|c|}
\hline
	$\consVolXUno{d}{\infty} = vol\, X_{d,\infty}^1$ & \\
	\hline
 $\consVolClosedOrb{d} = (2^{2d^3 + 5} \cdot 3^{4d^4} d^{6d^3 + 1})^{\consExpVolHTransversal{d}} \consVolXUno{d}{\infty}$ & Proposition \ref{Main_volume_H_S-orbits} \\
 \hline
 $\consDefBigCompact{d} = 2^{d^3} \cdot 3^{2d^4} d^{3d^3}$ & Proposition \ref{Compact_meeting_closed_H_S-orbits}\\
 \hline
  $\consCRecurrence{\nu}{d} = \begin{cases}
	3^{2d} d^3 p^{2d+1} &\text{if } \nu = p, \\
	2^{d+2} \cdot 3^{2d} d^3 & \text{if } \nu = \infty. 
	\end{cases}$ & Proposition \ref{Effective_recurrence_unipotent_flows_p-adic} \\
	\hline
	$\consAlfaRecurrence{d} = \frac{1}{(d-1)^2}$ & Proposition \ref{Effective_recurrence_unipotent_flows_p-adic} \\
	\hline
	$\consCoefTRecurrence{d} = \frac{2^3 \consVolXUno{d}{\infty}^{\consExpVolHTransversal{d}}}{d(d-1)}$ & Lemma \ref{Transversal_recurrence} \\
	\hline
	$\consBigOrbits{d} = \left(\frac{4}{d(d-1)}\right)^{\consExpVolHTransversal{d}}  \consVolXUno{d}{\infty}$ & Lemma \ref{Transversal_recurrence} \\
	\hline
	$\consExpVolHTransversal{d} = \frac{d(d+1)}{2} - 1$ & Lemma \ref{Transversal_recurrence}\\
	\hline
	$\consCoefVolHTransversalInf{d} = \frac{2^{d-1}}{d^{2 \consExpVolHTransversal{d}}}$ & Lemma \ref{Volume_W_S} \\
	\hline
	$\consCoefVolHTransversalSup{d} = 2^{d^2 - 1}$ & Lemma \ref{Volume_W_S} \\
	\hline
	$\consRecurrenceBox{d} = \frac{2 \consVolXUno{d}{\infty}^\frac{1}{\consExpVolHTransversal{d}}}{d(d-1)}$ & Lemma \ref{Injective_implies_small_r} \\
	 \hline
\end{tabular}
\end{center}

	\subsection{Section \ref{sec_Zs-equiv_criteria}}
	
	\begin{center}
	\begin{tabular}{|c|c|}
	\hline 
	$\consZSEquivCritRiso{d} = 2^{6d^3} d^{4d^3} \cdot d!^{2d^2 + 1} \consDynStRiso{d} (\consVolClosedOrb{d})^6$ & Theorem \ref{Z_S-equivalence_R-isotropic} \\
	\hline
	$\consZSEquivCritRani{d} = 2^{2d^3 + 7d} d^{2d^3} \cdot d!^7 \consDynStRani{d}(\consVolClosedOrb{d})^4$ & Theorem \ref{Z_S-equivalence_R-anisotropic}  \\
	\hline  
	\end{tabular}
	\end{center}
	
	\subsection{Section \ref{sec_small_gens}}
	
	\begin{center}
	\begin{tabular}{|c|c|}
	\hline
	$\consTQSiso{d} = 2^{2d^5} d!^{4d^4} \consZSEquivCritRiso{d} \consReducedIntegralQF{d}^{2d^3}$ & Lemma \ref{TQS_R-isotropic} \\
	\hline
	$\consTQSani{d} = 2^{d^5} d!^{2d^4} \consZSEquivCritRani{d} \consReducedIntegralQF{d}^{d^3} $ & Lemma \ref{TQS_R-anisotropic} \\
	\hline
	$\consSmallGensRiso{d} = 2^{2d^2} d \cdot d!^{d+1} \consTQSiso{d}^d \consTransSiegel{d} $ & Theorem \ref{Small_generators_R-isotropic} \\
	\hline
	$\consSmallGensRani{d} = 2 \consTQSani{d}^d $ & Theorem \ref{Small_generators_R-anisotropic} \\
	\hline
	\end{tabular}
	\end{center}

	\subsection{Appendix \ref{app_volume_computations}}

\begin{center}
\begin{tabular}{|c|c|}
\hline
$\consInfVolROG{d} = \left( \frac{1}{3d} \right)^{\frac{d(d-1)}{2}}$ & Lemma \ref{Volume_small_balls_real_orthogonal_groups} \\
\hline
$\consSupVolROG{d} = \left(\frac{20 d}{3} \right)^{\frac{d(d-1)}{2}}$ & Lemma \ref{Volume_small_balls_real_orthogonal_groups} \\
\hline
$\consBumpFuncRealOG{d} = 10^{d^2} d^{\frac{1}{4}(d+2)^2}$ & Lemma \ref{Smooth_bump_functions} \\
\hline
$\consBumpFuncRealOGbis{d} = 5d^3 (20d)^{\frac{1}{4}d(d-1) + 1}$ & Lemma \ref{SBF_bis} \\
\hline
$\consExpVolHTransversal{d} = \frac{d(d+1)}{2}-1$ & Lemmas \ref{Volume_Wr_infty_cite}, \ref{Volume_unit_ball_Lie_W_infty}, \ref{Volume_W_p_cite} \\
\hline
$ \consCoefVolHTransversalInf{d} = \frac{2^{d-1}}{d^{2 \consExpVolHTransversal{d}}}$ & Lemma \ref{Volume_Wr_infty_cite} \\
\hline
$\consCoefVolHTransversalSup{d} = 2^{d^2-1}$ & Lemma \ref{Volume_Wr_infty_cite}\\
\hline
$\consVolXUno{d}{\infty} = vol\, X_{d,\infty}^1$ & Lemma \ref{Volume_X_Sd^1} \\
\hline
\end{tabular}
\end{center}

	\subsection{Appendix \ref{app_reduction_theory}}
	
	\begin{center}
	\begin{tabular}{|c|c|}
	\hline	
	$\consChangeReducedPDQF{d} = d^\frac{3d}{2} (d+1)^{d^2} \cdot d!^{d+1} $ & Proposition \ref{Change_matrix_reduced_pdqf} \\
	\hline
	$\consExtremalVectorsBound{d} = d^\frac{3}{2} (d+1)^d \cdot d! $ & Lemma \ref{Delta_extremal_vectors_are_small} \\
	\hline
	$\consTransSiegel{d} = 2^{d(d-1)} d^\frac{3d}{2} (d+1)^{d^2} d!^{d+1} $ & Corollary \ref{Siegel_sets_almost_never_meet_its_translates_ap} \\
	\hline 
	$\consReducedIntegralQF{d} = d^\frac{d}{2} (d+1)^{d^2} d!^{2d+1}  $ & Proposition \ref{Integral_reduced_qf_are_small} \\
	\hline
	
	\end{tabular}
	\end{center}

\bibliography{Articulo_biblio}{}
\bibliographystyle{alpha}

Irving Calderón

Department of Mathematical Sciences, 

Durham University, 

Lower Mountjoy, 

DH1 3LE Durham, 

United Kingdom

{\tt irving.d.calderon-camacho@durham.ac.uk}

\end{document}